\documentclass[11pt]{memo-l}
\usepackage{amsxtra}
\usepackage{amssymb}
\usepackage{graphicx}
\usepackage{comment}
\usepackage{url}
\usepackage{color}

\usepackage{amsfonts}
\usepackage{textcomp}
\usepackage{caption}
\usepackage{sidecap}
\usepackage{subcaption}
\usepackage{tikz}
\usepackage{pgfplots}

\DeclareMathAlphabet{\mathpzc}{OT1}{pzc}{m}{it}

\DeclareMathOperator*{\esssup}{ess\,sup}
\DeclareMathOperator*{\essinf}{ess\,inf}
\DeclareMathOperator*{\oph}{Op_h}
\newcommand{\opht}[1]{\operatorname{Op_{h,#1}}}
\newcommand{\model}{\operatorname{mod}}
\DeclareMathOperator*{\weyl}{Op_{h}}
\DeclareMathOperator*{\diam}{diam}
\DeclareMathOperator*{\op}{Op}
\DeclareMathOperator*{\supp}{supp}
\DeclareMathOperator*{\WF}{{WF}}
\DeclareMathOperator*{\WFh}{{WF_h}}

\DeclareMathOperator*{\WFhp}{{WF_{h,\Psi}}}
\DeclareMathOperator*{\WFhpi}{{WF^i_{h,\Psi}}}
\DeclareMathOperator*{\WFhpf}{{WF^f_{h,\Psi}}}
\DeclareMathOperator*{\loc}{loc}
\DeclareMathOperator*{\comp}{comp}
\DeclareMathOperator*{\Ell}{ell}
\DeclareMathOperator*{\MS}{MS_h}
\DeclareMathOperator*{\MSp}{{MS_{h,\Psi}}}
\DeclareMathOperator*{\graph}{graph}
\DeclareMathOperator*{\sgn}{sgn}
\DeclareMathOperator*{\Id}{Id}
\DeclareMathOperator*{\tr}{tr}
\DeclareMathOperator*{\Arg}{Arg}

\newcommand{\m}{$$}

\newcommand{\Ph}[2]{\Psi^{#1}_{#2}}
\renewcommand{\Re}{\operatorname{Re}}
\renewcommand{\Im}{\operatorname{Im}}

\newtheorem{theorem}{Theorem}[chapter]
\newtheorem{lemma}[theorem]{Lemma}

\theoremstyle{definition}
\newtheorem{defin}[theorem]{Definition}

\theoremstyle{remark}
\newtheorem{remark}[theorem]{Remark}

\numberwithin{section}{chapter}
\numberwithin{equation}{chapter}

\newtheorem{prop}{Proposition}
\numberwithin{prop}{section}
\newtheorem{corol}[prop]{Corollary}
\newtheorem{conjecture}{Conjecture}
\numberwithin{conjecture}{section}

\newenvironment{remarks}[1][]{\begin{remark}\begin{trivlist}
\item[\hskip \labelsep {\bfseries #1}]\end{trivlist}\begin{itemize}}{\end{itemize}\end{remark}}

\numberwithin{figure}{section}

\newcommand{\quotient}[2]{{\left.\raisebox{.2em}{$#1$}\middle/\raisebox{-.2em}{$#2$}\right.}}

\newcommand{\bl}{\begin{flushleft}}
\newcommand{\el}{\end{flushleft}}
\newcommand{\br}{\begin{flushright}}
\newcommand{\ert}{\end{flushright}}
\newcommand{\bc}{\begin{center}}
\newcommand{\ec}{\end{center}}

\newcommand{\mc}[1]{\mathcal{#1}}

\newcommand{\recip}[1]{\frac{1}{#1}}
\newcommand{\imply}{\Rightarrow}

\newcommand{\complex}{\mathbb{C}}
\newcommand{\ints}{\mathbb{Z}}
\newcommand{\numList}{\begin{enumerate}}
\newcommand{\enumList}{\end{enumerate}}

\newcommand{\composed}{\text{\textopenbullet}}

\newcommand{\e}{\epsilon}

\newcommand{\re}{\mathbb{R}}

\newcommand{\nn}{\nonumber\\}
\newcommand{\la}{\langle}
\newcommand{\ra}{\rangle}
\newcommand{\Dloc}{\mathcal{D}_{\loc}}

\newcommand{\Cc}{C_c^\infty}

\renewcommand{\O}[1]{\mathpzc{O}_{#1}}
\renewcommand{\o}[1]{\mathpzc{o}_{#1}}
\renewcommand{\sharp}{\#}

\newcommand{\Deltap}{\Delta_{\partial\Omega,\delta'}}
\newcommand{\Deltad}[1]{\Delta_{#1,\delta}}
\newcommand{\Scl}[2]{S^{#1}_{#2,cl}}
\renewcommand{\div}{\operatorname{div}}
\newcommand{\resd}{\Lambda(h,\delta)}
\newcommand{\resp}{\Lambda(h,\delta')}

\DeclareMathOperator*{\ess}{ess}

\newcommand{\hf}{\frac 12}

\newcommand{\LOmega}{D_\Omega}
\newcommand{\LGamma}{D_\Gamma}
\newcommand{\Lchi}{D_\chi}

\newcommand{\Dl}{\tilde{N}}
\newcommand{\dDl}{\partial_{\nu}\mc{D}\ell}
\renewcommand{\S}{\mc{S}\ell}
\newcommand{\D}{\mc{D}\ell}
\newcommand{\Qs}{Q^G_\lambda}
\newcommand{\Qa}{Q^{\Dl}_\lambda}
\newcommand{\Qdd}{Q^{\dDl}_{\lambda}}

\DeclareMathOperator*{\Fried}{Fried}
\DeclareMathOperator*{\ad}{ad}

\newcommand{\PsiHom}{\Psi_{\textup{con}}}
\newcommand{\SHom}{S_{\textup{con}}}
\newcommand{\IHom}{I_{\textup{con}}}

\newcommand{\pO}{\partial\Omega}
\newcommand{\LO}{D_{\pO}}

\makeindex

\begin{document}

\frontmatter

\title [Distribution of Resonances in Scattering by Thin Barriers]{Distribution of Resonances in Scattering by Thin Barriers}


\author[J. Galkowski]{Jeffrey Galkowski}
\address{Mathematics Department, University of California, Berkeley, CA 94720 USA}
\curraddr{Mathematics Department, Stanford University, 
CA 94305, USA}
\email{jeffrey.galkowski@stanford.edu}
\thanks{The author is grateful to the National Science Foundation for support under the National Science Foundation Graduate Research Fellowship Grant No. DGE 1106400 and grant DMS-1201417}


\subjclass[2010]{Primary 35P20, 31B10, 31B20, 31B25, 31B35;\\ Secondary 45C05, 47F05}


\keywords{transmission problems, layer potentials, layer operators, high frequency, resonances, potential theory, hypersingular operator}

\dedicatory{To Opa}

\maketitle

\tableofcontents

\begin{abstract}
We study high energy resonances for the operators 
$$-\Deltad{\pO}:=-\Delta +\delta_{\partial\Omega}\otimes V\quad \text{and}\quad -\Deltap:=-\Delta+\delta_{\pO}'\otimes V\partial_\nu$$ 
where $\Omega\subset\re^d$ is strictly convex with smooth boundary, $V:L^2(\pO)\to L^2(\pO)$ may depend on frequency, and $\delta_{\pO}$ is the surface measure on $\pO$. These operators are model Hamiltonians for the quantum corrals studied in \cite{Aligia, Heller, Crommie} and for leaky quantum graphs \cite{Exner}. We give a quantum version of the Sabine Law \cite{Sabine} from the study of acoustics for both $-\Deltad{\pO}$ and $-\Deltap$.  It characterizes the decay rates (imaginary parts of resonances) in terms of the system's ray dynamics. In particular, the decay rates are controlled by the average reflectivity and chord length of the barrier. 

For $-\Deltad{\pO}$ with $\Omega$ smooth and strictly convex, our results improve those given for general $\pO$ in \cite{GS} and are generically optimal. Indeed, we show that for generic domains and potentials there are infinitely many resonances arbitrarily close to the resonance free region found by our theorem. In the case of $-\Deltap$, the quantum Sabine law gives the existence of a resonance free region that converges to the real axis at a fixed polynomial rate. The size of this resonance free region is optimal in the case of the unit disk in $\re^2$. As far as the author is aware, this is the only class of examples that is known to have resonances converging to the real axis at a fixed polynomial rate but no faster.

The proof of our theorem requires several new technical tools. We adapt intersecting Lagrangian distributions from \cite{UhlMel} to the semiclassical setting and give a description of the kernel of the free resolvent as such a distribution. We also construct a semiclassical version of the Melrose--Taylor parametrix \cite{MelTayl} for complex energies. We use these constructions to give a complete microlocal description of the single, double, and derivative double layer operators in the case that $\pO$ is smooth and strictly convex. These operators are given respectively for $x\in \pO$ by 
\begin{gather*} 
G(\lambda)f(x):=\int_{\partial\Omega}R_0(\lambda)(x,y)f(y)dS(y)\,,\\
 \Dl(\lambda)f(x):=\int_{\pO}\partial_{\nu_y}R_0(\lambda)(x,y)f(y)dS(y)\\
 \dDl(\lambda)f(x):=\int_{\partial\Omega}\partial_{\nu_x}\partial_{\nu_y}R_0(\lambda)(x,y)f(y)dS(y)\,.
 \end{gather*}
This microlocal description allows us to prove sharp high energy estimates on $G$, $\Dl$, and $\dDl$ when $\Omega$ is smooth and strictly convex, removing the log losses from the estimates for $G$ in \cite{GS,GalkSLO} and proving a conjecture from \cite[Appendix A]{GalkSLO}. 
\end{abstract}
  

\chapter*{Acknowledgements}
The author would like to thank Maciej Zworski for invaluable guidance and discussions. Thanks also to Semyon Dyatlov, Hart Smith and Mike Zaletel for their interest and helpful comments and to the anonymous referee for careful reading and many helpful comments. The author is grateful to the National Science Foundation for support under the National Science Foundation Graduate Research Fellowship Grant No. DGE 1106400 and grant DMS-1201417

\mainmatter


\chapter{Introduction}
We seek to understand the long-term behavior of waves trapped by thin barriers. Thin barriers are a model for quantum corrals, which are physical systems assembled from individual atoms using a scanning tunneling microscope (see for example \cite{Aligia,Heller,fiete2002theory, Crommie} and references therein). The corral formed by the atoms partially confines electron waves to its interior (see Figure \ref{fig:quantumCorral}).  From our point of view, the important features of quantum corrals are:

\begin{enumerate}
\item The potential produced by the confining atoms is intense and localized to a thin region hereafter referred to as the boundary.
\item The potential can vary along the boundary.
\end{enumerate}

\begin{SCfigure}
\includegraphics[width=.5\textwidth]{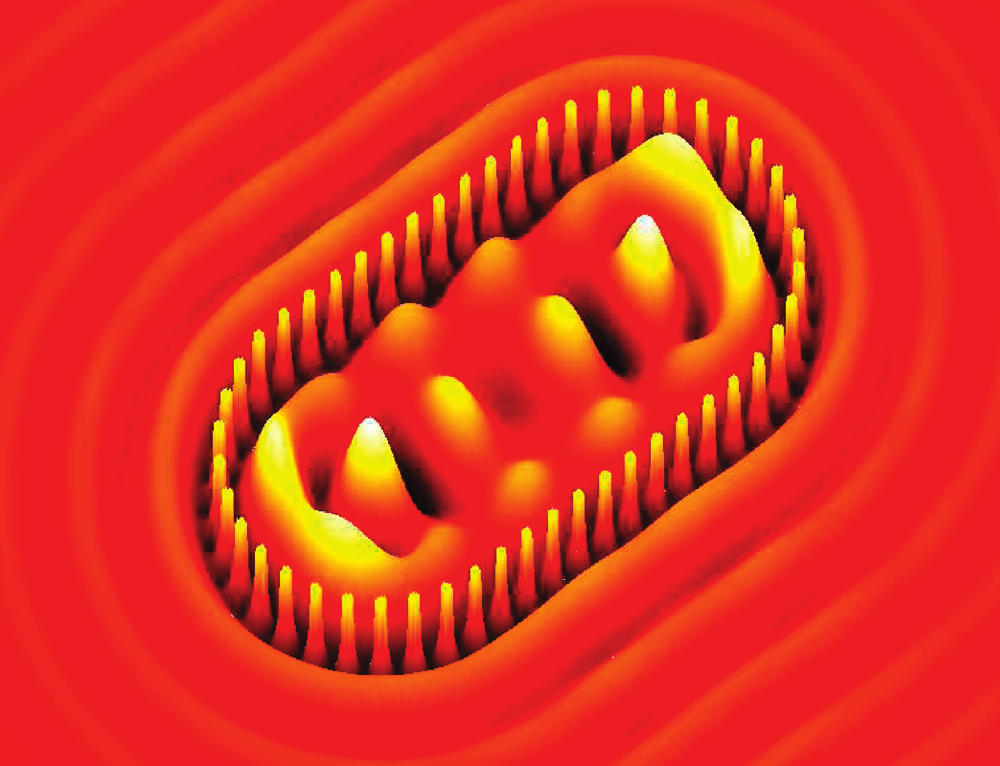}
\vspace{1cm}
\caption[Image of a quantum corral]{This figure shows an image of a quantum corral taken using a scanning tunneling microscope. The atoms produce large spikes in the potential around the boundary of a Bunimovich stadium. The smaller peaks represent the probability density function of the electron. While most of the wavefunction is confined inside the corral, there are smaller ripples in the exterior. Hence, the electron is only partially confined. This image is included from \cite{Heller} with the permission of the authors.\label{fig:quantumCorral}}
\end{SCfigure}

Sound propagation in a concert hall also has the above properties. Moreover, the strength of the interaction with the walls varies as a function of the frequency of the interacting wave. Because of this, we allow our potentials to depend on frequency.

In order to model these systems mathematically, we replace the physical potential with a delta function potential $ \delta_{\Gamma}\otimes V_{\model}$ where $\delta_\Gamma$ is the Hausdorff $d-1$ measure on some hypersurface $\Gamma\Subset \re^d$. This model was suggested in \cite{crommie1995waves} and we show in \cite[Section 7.3]{thesis} that it is an accurate approximation of the physical potential.  

We study the decay of solutions to 
\begin{gather}
\label{eqn:waveIntro}
(\partial_t^2+P)u=0
\end{gather}
where
\begin{gather}
\label{eqn:opDelta}
P=-\Deltad{\Gamma}:=-\Delta + \delta_{\Gamma}\otimes V_{\model},\quad\quad \delta_\Gamma(u)=\int_{\Gamma}ud\mc{H}_{d-1}
\end{gather}
with $\mc{H}_{d-1}$ denoting the Hausdorff $d-1$ measure. The delta function potential is also used to study leaky quantum graphs: models used in the theoretical understanding of waveguides and thin wires in nanotechnology. (See for example the summary article \cite{Exner}.)

Motivated by interest in $\delta'$ interactions from mathematical physics, \cite{SolveModels,boundStatesGeneralSingular, Munoz} spectral theory \cite{seba, SpecDeltaPrime, Kurasov}, and another model of leaky quantum graphs \cite{Exner}, we also study solutions to \eqref{eqn:waveIntro} with $P$ given by
\begin{equation}
\label{eqn:opDeltap}
P=-\Deltap:=-\Delta +\delta'_{\partial\Omega}\otimes V_{\model}\partial_\nu.\quad \quad \delta'_{\partial\Omega}(u):=\int_{\partial\Omega}-\partial_\nu ud\mc{H}_{d-1}.
\end{equation}
We have replaced $\Gamma$ by the boundary of the domain $\Omega$ since we only consider hypersurfaces of that type for the $\delta'$ interaction. 

Solutions to \eqref{eqn:waveIntro} with $P=-\Deltad{\Gamma}$ (as in \eqref{eqn:opDelta}) have \emph{resonance expansions} roughly of the form \begin{equation} 
\label{eqn:expand}u(t,x)\sim \sum_{z\in \text{Res}}e^{-itz}u_z(x)
\end{equation}
where $\text{Res}\subset \mathbb{C}$ is the set of scattering resonances of the operator $-\Deltad{\Gamma}$ (to be defined below).
This was first proved in \cite[Theorem 1.4]{GS}, where one can find a more precise statement. To motivate \eqref{eqn:expand}, we recall the useful fact that solutions to wave equations on compact manifolds have expansions similar to \eqref{eqn:expand} in terms of eigenvalues. Thus, the role played by scattering resonances in leaky systems is similar to that played by eigenvalues in the closed setting.

As seen from \eqref{eqn:expand}, the real and (negative) imaginary part of $z\in \text{Res}$ respectively give the frequency and exponential decay rate of the associated resonant state $e^{-itz}u_z$. Hence, resonances close to the real axis give information about the long term behavior of solutions to \eqref{eqn:waveIntro}. In their seminal works, Lax--Phillips \cite{Lax} and Vainberg \cite{Vain} understood the relation between propagation of singularities for the wave equation and the presence of scattering resonances near the real axis. We study the singularities of solutions to \eqref{eqn:waveIntro} in Chapter \ref{ch:lowerBound} to demonstrate the existence of resonances with prescribed decay rates. 

To get a quantitative heuristic for the decay of waves (the imaginary part of resonances), we imagine solving the wave equation 
$$(\partial_t^2-P)u=0\,,\quad u|_{t=0}=u_0,\quad u_t|_{t=0}=0$$
where $P$ is either $-\Deltap$ or $-\Deltad{\partial\Omega}$ with initial data $u_0$ a wave packet (that is a function localized in frequency and space up to the scale allowed by the uncertainty principle) localized at position $x_0\in \Omega$ and frequency $\xi_0\in \re^d\setminus\{0\}$. The solution, $u$, then propagates along the billiard flow starting from $(x_0,\xi_0)$. At each intersection of the billiard flow with the boundary, the amplitude inside of $\Omega$ will decay by a factor, $R$, depending on the point and direction of intersection. Suppose that the billiard flow from $(x_0,\xi_0)$ intersects the boundary at $(x_n,\xi_n)$ $n>0$. Let $l_n=|x_{n+1}-x_n|$ be the distance between two consecutive intersections with the boundary (see Figure \ref{fig:reflectPic}). Then the amplitude of the wave decays by a factor $\prod_{i=1}^nR_i$ in time $\sum_{i=1}^nl_i$ where $R_i=R(x_i,\xi_i)$. The energy scales as amplitude squared and since the imaginary part of a resonance gives the exponential decay rate of $L^2$ norm, this leads us to the heuristic that resonances should occur at 
\begin{equation} \label{eqn:heurRes}\Im \lambda=\left.\overline{\log |R|^2}\right/(2\bar{l})\end{equation}
where the map $\bar{\cdot}$ is defined by $\bar{f}=\frac{1}{N}\sum_{i=1}^Nf_i$. 
In the early 1900s, Sabine \cite{Sabine} postulated that the decay rate of acoustic waves in a region with leaky walls is determined by the average decay over billiards trajectories. Such a model has also been used for quantum corrals \cite{Heller}. The expression \eqref{eqn:heurRes} provides a precise statement of Sabine's idea. In Chapter \ref{ch:resFree} we show that a quantum Sabine law of the form \eqref{eqn:heurRes} holds for both $-\Deltad{\partial\Omega}$ and $-\Deltap$ (see Theorems \ref{thm:qSabine1} and \ref{thm:qSabine2}). 

\begin{figure}
\centering
\begin{tikzpicture}
\draw (0,0) to [out=0,in=-90](2,2) to[out=90, in=0](1,3)to[out=180, in =95](-2,2)to [out=275,in=180](0,0);
\draw [dashed](0,0)to(2,2)node[right]{$R_1$}to (1,3)node[above]{$R_2$}to (1-2.4,3-2.4)node[left]{$R_3$};
\draw (0.9,1)node[left]{$l_1$} (1.4,2.6)node[below]{$l_2$}(1-1.3,3-1.2)node[left]{$l_3$};
\draw (-1.4,2.4)node{$\Omega$};
\end{tikzpicture}
\caption[Schematic of a wave packet inside a thin barrier]{\label{fig:reflectPic}The figure shows the path of a wave packet along with the lengths between each intersection ($l_i$) and the reflection coefficient at each point of intersection with the boundary ($R_i$). After each reflection with the boundary, the amplitude of the wave packet inside $\Omega$ decays by a factor of $R_i$. The time between reflections is given by $l_i$.}
\end{figure}
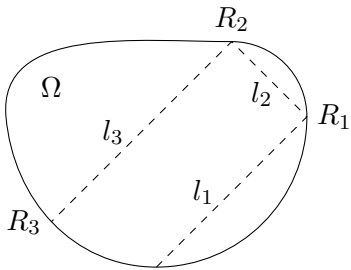

Although the appearance of scattering poles in the expansion \eqref{eqn:expand} is more intuitive, a mathematically more useful characterization of \emph{scattering resonances} of an operator, $P$, is as the poles of the meromorphic continuation of the resolvent
$$R_{P}(\lambda):=(P-\lambda^2)^{-1}$$
from $\Im \lambda \gg 1$. In order to give an expansion of the form \eqref{eqn:expand} (and hence prove exponential decay for waves), we find a resonance free region near the real axis. It suffices to study $P-\lambda^2$ with $|\Re\lambda|\geq C$ and it is convenient to write $\lambda=z/h$ with $h\ll 1$. This converts problems of the form
\begin{equation}\label{eqn:transform} -\Delta+V(x)-(z/h)^2\to -h^2\Delta+h^2V(x)-z^2.
\end{equation}

Since $|\Re \lambda|\gg1$, we are interested in high frequencies and our main intuition comes from the quantum-classical correspondence: high energy waves inherit many properties of the corresponding classical dynamics. In particular, in scattering by smooth compactly supported potentials, the dynamics corresponding to the operator 
\begin{equation}\label{eqn:opHamiltonian}-h^2\Delta +V(x)\end{equation}
is given by the Hamiltonian flow for the Hamiltonian 
\begin{equation}\label{eqn:HamiltonianSmoothPotential}|\xi|^2+V(x).\end{equation}
Our operators are of the form
$$-\Delta + \delta_{\Gamma}\otimes V_{\model}-\lambda^2\quad\quad \text{ and }\quad\quad -\Delta+\delta'_{\partial\Omega}\otimes V_{\model}\partial_\nu-\lambda^2 $$
so, using \eqref{eqn:transform}, we replace $V(x)$ in \eqref{eqn:opHamiltonian} and \eqref{eqn:HamiltonianSmoothPotential} with $h^2V_{\model}\otimes \delta_{\Gamma}\,\,(\text{or } \delta'_{\partial\Omega})$. Then
we can think of a potential of the form $h^2V_{\model}\otimes \delta_\Gamma$ as the distributional limit of a sequence of potentials $\{h^2V_n\}\subset \Cc(\re^d)$. As $n$ increases, $V_n$ narrows and increases in intensity. Because of the $h^2$ scaling, for each fixed $n$, a wave with energy $E>0$ will pass through the potential $h^2V_n$. Thus, the potential will not produce confinement at any positive energy. However, as $V_n$ increases without bound, we expect the corresponding classical dynamics to approach the billiard ball flow (see Section \ref{sec:billiard}). Thus, if $\re^d\setminus \Gamma$ has a bounded component, we expect classical confinement at any energy $E$. Using this naive analysis, we might expect very slow decay of waves at any frequency. However, as the potential $V_n$ narrows, tunneling effects decrease the strength of confinement. In fact, the precise analysis of scattering by delta functions, $\Deltad{\Gamma}$, presented in this article shows that if $V_{\model}$ grows mildly with frequency, then the confinement produced is only slightly stronger than that for $V\in \Cc(\re^d)$. However, if $V_{\model}$ is allowed to depend strongly on frequency, then we demonstrate in \cite{GalkCircle} that as a result of effects coming from paths $x(t)$ nearly tangent to the submanifold $\Gamma$, confinement can become much stronger than that for $V\in \Cc(\re^d)$. Similarly, if the potential is more singular than $\delta_{\Gamma}$, then confinement becomes stronger than that for $V\in \Cc(\re^d)$. 

A key step in understanding the distribution of resonances for thin barriers is relating the poles of $R_{P}(\lambda)$ to the existence of nonzero $\lambda$-outgoing solutions to
\begin{equation}\label{eqn:outgoingSoln}(P-\lambda^2)u=0.\end{equation}
By \emph{$\lambda$-outgoing} we mean that there exist $M>0$ and $\varphi\in \Cc(\re^d)$ with 
$$u(x)=(R_0(\lambda)\varphi)(x)\,,\quad\text{ for }|x|\geq M$$
where, $R_0(\lambda)$, the \emph{free resolvent}, is the meromorphic continuation of 
$$R_0(\lambda):=(-\Delta-\lambda^2)^{-1}$$
from $\Im \lambda\gg 1$. This was done for $-\Deltad{\Gamma}$ by Smith and the author in \cite{GS}. There, it is shown that for the case of $-\Deltad{\partial\Omega}$ this is equivalent to solving 
\begin{equation} 
\label{eqn:transmitnoPrime}
\begin{cases}
(-\Delta -\lambda^2)u_1=0&\text{in }\Omega\\
(-\Delta-\lambda^2)u_2=0&\text{in }\re^d\setminus\overline{\Omega}\\
u_1|_{\partial\Omega}=u_2|_{\partial\Omega}\\
\partial_\nu u_1-\partial_\nu u_2+Vu_1=0&\text{on }\partial\Omega\\
u_2\text{ is }\lambda\text{-outgoing}
\end{cases}.
\end{equation}
In Chapter \ref{ch:mer}, we examine the case of $-\Deltap$ and show that solving \eqref{eqn:outgoingSoln} is equivalent to solving
\begin{equation}
\label{eqn:transmitPrime}
\begin{cases}
(-\Delta -\lambda^2)u_1=0&\text{in }\Omega\\
(-\Delta-\lambda^2)u_2=0&\text{in }\re^d\setminus\overline{\Omega}\\
\partial_\nu u_1|_{\partial\Omega}=\partial_\nu u_2|_{\partial\Omega}\\
u_1-u_2+V\partial_\nu u_1=0&\text{on }\partial\Omega\\
u_2\text{ is }\lambda\text{-outgoing}
\end{cases}.
\end{equation}
Equations \eqref{eqn:transmitnoPrime} and \eqref{eqn:transmitPrime} are examples of transmission problems i.e. problems which involve solving partial differential equations on two different domains with boundary conditions transferring information between the two domains. Another example of a transmission problem is that given by a transparent obstacle having differing wave speeds inside and outside $\Omega$. Resonances in this case were studied by Popov--Vodev \cite{PopVod} and Cardoso--Popov--Vodev \cite{Card, Card2}. However, the methods employed in these instances are quite different than those appearing in the current article

Although we will consider only compact hypersurfaces, it is instructive to look at the case where $\partial\Omega=\{x_1=0\}\subset \re^d$ to gain some heuristic understanding of how resonances behave for $-\Deltap$ and $-\Deltad{\pO}$,. We consider a plane wave with frequency $h^{-1}$, $e^{\frac{i}{h}\la x,\xi\ra }$, approaching $x_1=0$ from the left. (See Figure \ref{fig:1dModelReflection} for a depiction of the setup.) We are then interested in what fraction of the wave is reflected by the barrier and what fraction is transmitted. Let $R_{\delta}$ and $R_{\delta'}$ denote the reflection coefficients and $T_{\delta}$, $T_{\delta'}$ the transmission coefficients.

By using the boundary conditions in \eqref{eqn:transmitnoPrime} (or a formal computation), one can see that the appropriate transmission condition for $V\delta(x_1)$ is
$$u_+(0,x')=u_-(0,x')\quad \quad \partial_{x_1}u_-(0,x')-\partial_{x_1}u_+(0,x')+Vu_+(0,x')=0.$$
This leads to 
\begin{equation}\label{eqn:reflect} R_\delta=\frac{hV}{2i\xi_1-hV}\,,\quad\quad\quad T_\delta=\frac{2i\xi_1}{2i\xi_1-hV}.
\end{equation}
By a similar computation, one can see that the appropriate transmission condition for $V\delta'(x_1)$ is given by 
$$\partial_{x_1}u_+(0,x')=\partial_{x_1}u_-(0,x')\quad\quad u_-(0,x')-u_+(0,x')+V\partial_{x_1}u_-'(0,x')=0,$$
which leads to 
\begin{equation} \label{eqn:reflectPrime} R_{\delta'}=\frac{Vi\xi_1}{Vi\xi_1-2h}\,,\quad\quad\quad T_{\delta'}=\frac{2h}{2h-Vi\xi_1}.
\end{equation}
For general $V$ and $\Omega$, $R_\delta$ and $R_{\delta'}$ are the symbols of certain pseudodifferential operators. However, the definition of these pseudodifferential operators is involved and we postpone it until \eqref{eqn:reflectionOperator} and \eqref{eqn:reflectionOperatorPrime}.

\begin{figure}
\centering
\begin{tikzpicture}
\draw[very thin,-] (1,-4) -- (1,1) ;
\node at (1, 1.5) {$V\delta(x_1)$} ;
\node at (-2.5,0){$u_1=e^{\frac{i}{h}\la x,\xi\ra}+ R_{\delta}e^{\frac{i}{h}(-x_1\xi_1+\la x',\xi'\ra)}$} ;
\node at (2.5,0){$u_2=T_{\delta}e^{\frac{i}{h}\la x,\xi\ra}$} ;
\node at (-0.5,1){$x_1<0$} ;
\node at (2,1){$x_1>0$} ;

\draw[thick]
(-2.5,-1.5)--(-2.5,-1);
\draw[thick]
(-2.5,-1.5)--(-2.3,-1.5);
\draw[thick]
(-2.5,-1)--(-2.3,-1);
\node at (-3,-1.25) {$1$};
\draw[ultra thick, ->, red] 
        (-2,-1.5) sin (-1.5,-1) ;
\draw[ultra thick,->,red]
(-1.5,-1)cos (-1,-1.5) sin (-.5,-2) ;
\draw[ultra thick,->,red]
(-.5,-2)cos(0,-1.5)sin(.5,-1);
\draw[ultra thick,red]
(.5,-1)cos (1,-1.5);
\draw[dashed]
(-2,-1.5)--(1,-1.5);
\node at (-.7,-2.6){$+$};
\begin{scope}[shift={(.5,-2)}]
\draw[thick]
(-3,-1.5)--(-3,-1.8);
\draw[thick]
(-3,-1.5)--(-2.8,-1.5);
\draw[thick]
(-3,-1.8)--(-2.8,-1.8);
\node at (-3.5,-1.65) {$|R_{\delta}|$};
\draw[ultra thick, ->, blue, dashed] 
      (-1.5,-1.2)  cos  (-2,-1.5)sin (-2.5,-1.8);
\draw[ultra thick,->,blue, dashed]
(-.5,-1.8) cos (-1,-1.5) sin (-1.5,-1.2);
\draw[ultra thick,->,blue, dashed]
(.5,-1.2)cos(0,-1.5)sin(-.5,-1.8);
\draw[dashed]
(-2.5,-1.5)--(.5,-1.5);
\end{scope}
\begin{scope}[shift={(2,-.75)}]
\draw[ultra thick,->,red]
 (-1,-1.5) sin (-.5,-1.9) ;
\draw[ultra thick,->,red]
(-.5,-1.9)cos(0,-1.5)sin(.5,-1.1);
\draw[ultra thick,red]
(.5,-1.1)cos (1,-1.5);
\draw[thick]
(1.5,-1.5)--(1.5,-1.1);
\draw[thick]
(1.5,-1.5)--(1.3,-1.5);
\draw[thick]
(1.5,-1.1)--(1.3,-1.1);
\draw[dashed]
(-1,-1.5)--(1,-1.5);
\node at (2,-1.3) {$|T_{\delta}|$};
\end{scope}

\end{tikzpicture}
\caption[Model for calculating reflection coefficients]{\label{fig:1dModelReflection} The setup for plane wave interactions. Here, $u=u_1\oplus u_2$ is a solution to \eqref{eqn:transmitnoPrime} and $R_\delta$ and $T_{\delta}$ are respectively the reflection and transmission coefficients. The waves shown in the red solid lines are traveling to the right. On the left and right they are respectively the initial wave and the transmitted wave. The wave shown in the blue dashed line is the reflected wave traveling to the left. The setup for the $\delta'$ interaction is similar.}
\end{figure}
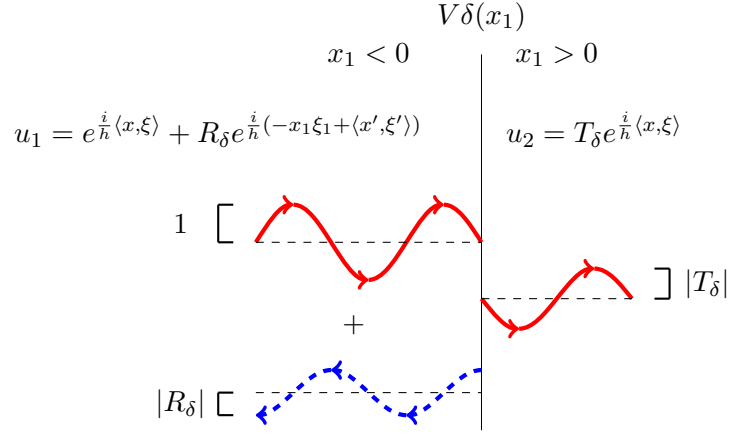

In this computation, we consider waves with frequency equal to $h^{-1}$ and hence have $\xi\in S^{d-1}$. When $\xi_1$ is near $0$, the plane wave travels nearly tangent to $x_1=0$. Our first observation is that as $\xi_1\to 0$, $|R_\delta|\to 1$ while $|R_{\delta'}|\to 0$. This reflects the fact that frequencies tangent to $\Gamma=\{x_1=0\}$ are annihilated by the normal derivative to $\Gamma.$ Thus, we expect glancing (tangent) trajectories to contribute less to the resonances close to the real axis for $-\Deltap$ than for $-\Deltad{\Gamma}$.  

Equation \eqref{eqn:heurRes} with $R=R_{\delta}$ suggests that the resonances of $-\Deltad{\partial\Omega}$ lie in regions with $\Im z\sim h\log h^{-1}.$ On the other hand, if we assume that $V\sim h^\alpha$ for $\alpha<1$ and let $R=R_{\delta'}$, then we obtain for $-\Deltap$ that $\Im z \sim h^{3-2\alpha}.$ Thus, the resonances of $-\Deltap$ are much closer to the real axis than those of $-\Deltad{\partial\Omega}$. Indeed, we show that when written in terms of $\lambda$, the resonances of $-\Deltap$ converge to the real axis at a fixed polynomial rate while those of $-\Deltad{\partial\Omega}$ diverge logarithmically from the real axis. 

The paper \cite{GS} analyzes $-\Deltad{\Gamma}$ when $\Gamma $ is any finite union of subsets of compact $C^{1,1}$ hypersurfaces. In this paper, in order to give a detailed analysis of resonances for the operators $-\Deltad{\Gamma}$ and to introduce $-\Deltap$, we focus on the case where $\Gamma=\pO$ is smooth and strictly convex. 

We now state schematic versions of our main theorems
\begin{theorem}[Quantum Sabine Law: $\delta$ potential]
\label{thm:qSabine1}
Let $\Omega\Subset \re^d$ be strictly convex with smooth boundary. Let $\lambda=z/h$ be a resonance of $-\Deltad{\pO}$. Then there exists $h_0>0$ so that for $0<h<h_0$ and $\Re z\sim 1$, we have
\begin{equation}\label{eqn:1}\frac{\Im z}{h}\leq \sup_{B^*\pO} \left.\overline{\log |R_\delta|^2}\right/(2\bar{l})=:\mc{I}(h)\end{equation}
where $B^*\pO$ denotes the unit coball bundle of $\pO$.
Moreover, for $\Omega=B(0,1)\subset \re^2$ and $V$ constant, \eqref{eqn:1} is sharp.
\end{theorem}
We also show that Theorem \ref{thm:qSabine1} is sharp for a generic strictly convex $\Omega$ and potential $V$.
\begin{theorem}
\label{thm:qSabineOptimal}
For a generic strictly convex $\Omega\Subset \re^d$ with smooth boundary and $V$, \eqref{eqn:1} is sharp. In particular, there exists a sequence $\{(h_n,z_n)\}_{n=1}^\infty$ with $h_n\to 0$ such that $\Re z_n\sim1$, $z_n/h_n$ is a resonance of $-\Deltad{\pO}$, and for any $\delta>0$ there exists $N>0$ such that for $n>N$
$$\frac{\Im z_n}{h_n}\geq \mc{I}(h_n)-\delta\log h_n^{-1}$$
where $\mc{I}$ is as in \eqref{eqn:1}.
\end{theorem}
We also give a quantum Sabine law for the $\delta'$ potential
\begin{theorem}[Quantum Sabine Law: $\delta'$ potential]
\label{thm:qSabine2}
Let $\lambda=z/h$ be a resonance of $-\Deltap$. Then there exists $h_0>0$ so that for $0<h<h_0$ and $\Re z\sim 1$, we have
\begin{equation}\label{eqn:2}\frac{\Im z}{h}\leq \sup_{B^*\pO} \left.\overline{\log |R_{\delta'}|^2}\right/(2\bar{l}).\end{equation}
Moreover, for $\Omega=B(0,1)\subset \re^2$ and $V$ constant, \eqref{eqn:2} is sharp.
\end{theorem}
The precise version of these theorems can be found respectively in Theorems \ref{thm:resFree}, \ref{thm:lowerBoundNumRes}, and \ref{thm:resFreePrime}.

\begin{figure}
\centering
\begin{tabular}{ll}
\raisebox{-.2cm}{$\Im \lambda$}\\
\includegraphics[width=4.3in]{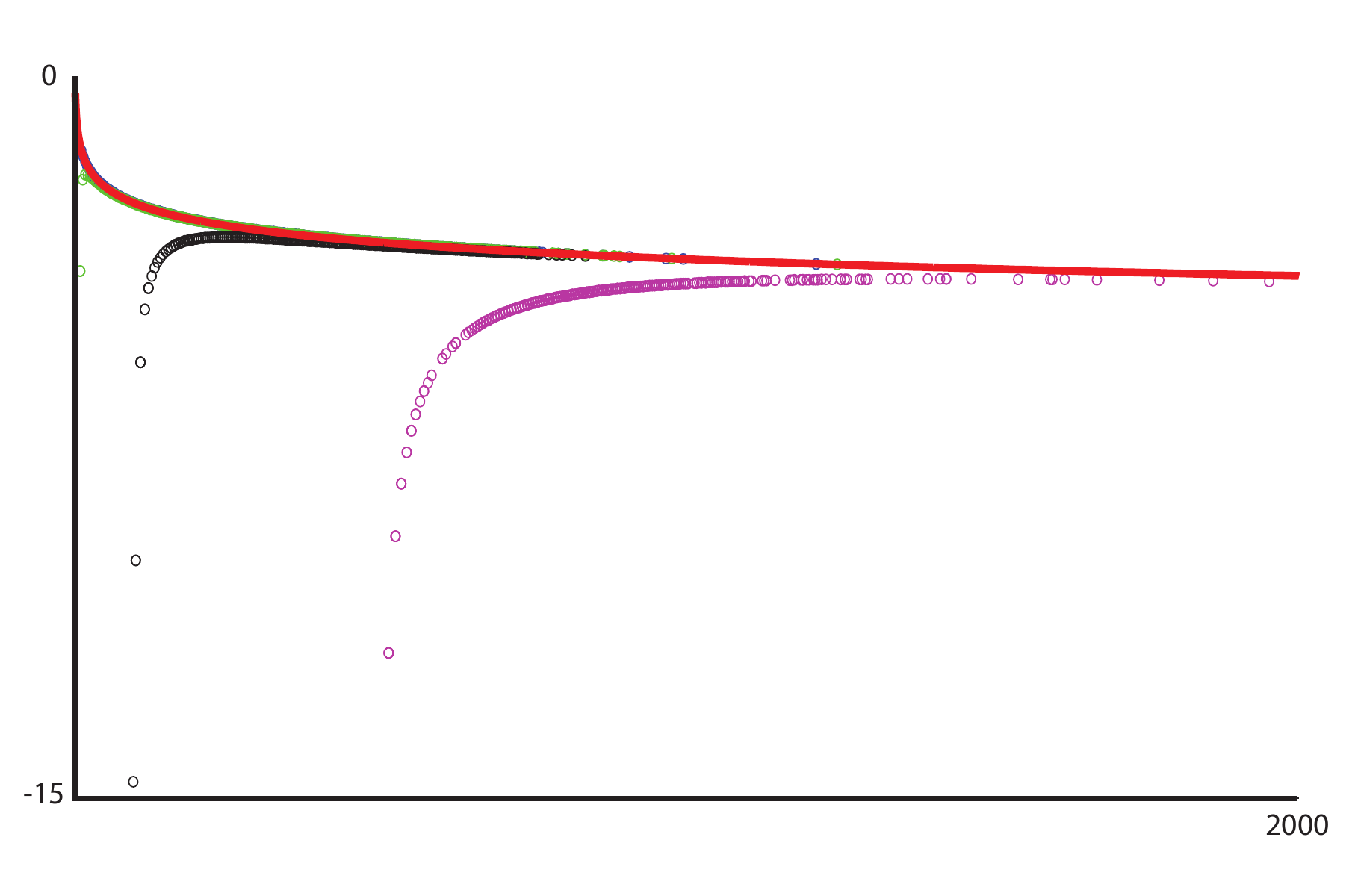}&\raisebox{.75cm}{$\Re \lambda$}
\end{tabular}
\caption[Resonances computed for a circle]{\label{fig:circleRes}When $\Omega=B(0,1)\subset \re^2$, the boundary values of resonance states can be expanded in a Fourier series $\sum a_n e^{inx}.$ We show the resonances for $V\equiv 1$ corresponding to the $n=0,\,10,\,100,$ and $500$ modes. The solid line shows the bound given by Theorem \ref{thm:qSabine1}.}
\end{figure}

One important part of the analysis leading to Theorems \ref{thm:qSabine1}, \ref{thm:qSabineOptimal} and \ref{thm:qSabine2} is the study of boundary layer operators.
The single layer, double layer, and derivative double layer (hypersingular) operators given respectively for $x\in \pO$ by
\begin{gather*} 
G(\lambda)f(x):=\int_{\partial\Omega}R_0(\lambda)(x,y)f(y)dS(y)\,,\\
 \Dl(\lambda)f(x):=\int_{\pO}\partial_{\nu_y}R_0(\lambda)(x,y)f(y)dS(y)\\
 \dDl(\lambda)f(x):=\int_{\partial\Omega}\partial_{\nu_x}\partial_{\nu_y}R_0(\lambda)(x,y)f(y)dS(y).
 \end{gather*}
Resonances for $-\Deltad{\pO}$ and $-\Deltap$ are given respectively by poles of the operators 
 $$(I+G(\lambda)V)^{-1}\quad \quad \text{ and }\quad\quad (I-\dDl(\lambda)V)^{-1}.$$
 The boundary layer operators are also of interest in the numerical solution of the Helmholtz equation. In fact, high frequency estimates on these operators are used to prove convergence and stability of certain numerical schemes (see for example \cite{baskin1504sharp,chandler2009condition,chandler2012numerical} and references therein).
 
 In Chapter \ref{ch:layer}, we give a complete microlocal description of these operators when $\Omega$ is strictly convex with smooth boundary (see Sections \ref{sec:decomposed2} and \ref{sec:BLONearGlance}). We then use this description to prove the following sharp high energy estimate.
 \begin{theorem}
 \label{thm:layerEstimates}
 Let $\Omega\Subset \re^d$ be strictly convex with smooth boundary. Then there exists $\lambda_0>0$ such that for $|\lambda|>\lambda_0$, 
 \begin{align*}
 \|G(\lambda)\|_{L^2(\pO)\to L^2(\pO)}&\leq C\la \lambda\ra^{-2/3}e^{\LO(\Im \lambda)_-}\\
 \|\Dl(\lambda)\|_{L^2(\pO)\to L^2(\pO)}&\leq Ce^{\LO(\Im \lambda)_-}\\
 \|\dDl(\lambda)\|_{H^1(\pO)\to L^2(\pO)}&\leq C\la \lambda\ra e^{\LO(\Im \lambda)_-}
 \end{align*}
 where $(\Im \lambda)_-=\max(-\Im \lambda, 0)$ and $\LO$ is the diameter of $\Omega$. Moreover, all of the above estimates are sharp for $\lambda\in \re$.
 \end{theorem} 
\noindent This estimate removes a log loss from the estimate for $G$ in \cite{GS,GalkSLO} and proves the sharp estimate for $\Dl$ conjectured in \cite[Appendix A]{GalkSLO}.

\section{Outline of the article}

We begin in Chapter \ref{ch:prelim} with a review of the geometric and analytical tools that are used in the analysis of $-\Deltad{\pO}$ and $-\Deltap$. In addition to this review, we develop a notion of a sheaf-valued symbol that is sensitive to local changes of semiclassical order.  

In order to analyze $-\Deltad{\partial\Omega}$ and $-\Deltap$, we need a semiclassical analog of the Melrose--Uhlmann \cite{UhlMel} notion of intersecting lagrangian distributions and a semiclassical version of the Melrose--Taylor parametrix for both the interior and exterior of a strictly convex domain that is adapted for use with complex energies. 
We postpone the development of these tools to the Appendices and instead begin the analysis of $-\Deltad{\pO}$ and $-\Deltap$. 

In Chapter \ref{ch:mer}, we give the formal definition of $-\Deltap$ and $-\Deltad{\pO}$. We then give a proof of the meromorphic continuation of the resolvent of $-\Deltap$. However, unlike for $-\Deltad{\pO}$, we already use some microlocal understanding of $\dDl$ to give a proof of the meromorphic continuation and so we restrict our attention to $\Omega$ with smooth boundary.  In Chapter \ref{ch:mer} we also show that (except for $d=1$ and $\lambda=0$) the resonances of $-\Deltap$ occur at $\lambda$ for which there exist nontrivial solutions $\varphi\in H^1(\partial\Omega)$ to 
\begin{equation} (I-\dDl(\lambda)V)\varphi=0\label{eqn:boundTemp2}\end{equation}
where $\dDl$ is the derivative double layer (or hypersingular) operator.


Since the meromorphic continuation for $-\Deltad{\pO}$ was given in \cite{GS} and is analogous to that for $-\Deltap$, we omit the proof here.
However, we recall that in \cite{GS}, the author and Smith show that resonances of $-\Deltad{\pO}$ occur at $\lambda$ for which there exist nontrivial solutions $\varphi\in L^2(\pO)$ to 
\begin{equation}\label{eqn:boundTemp1}(I+G(\lambda)V)\varphi =0\end{equation}
where $G$ is the single layer operator. 

Thus, our next step is to analyze the boundary layer operators $G$, $\Dl$, and $\dDl$, which we do in Chapter \ref{ch:layer}. 
Restricting our attention to the case where $\pO$ is piecewise smooth and Lipschitz, we prove nearly sharp (i.e. modulo a $\log \lambda$ loss) high energy estimates for these operators using restriction bounds for eigenfunctions and their derivatives.  

Our next task is to give a microlocal description of the boundary layer operators in the case that $\pO$ is smooth and strictly convex. To do this we use the semiclassical intersecting Lagrangians developed in Appendix \ref{ch:iLagrange} to give a microlocal description of the free resolvent. Next we use the calculus of semiclassical Fourier integral operators to give a microlocal description of the boundary layer operators away from glancing i.e. away from momenta $\xi$ that are tangent to the boundary. In the case that $\Omega$ is strictly convex, we use the semiclassical Melrose--Taylor parametrix to understand the boundary layer operators near glancing. Finally, we use this microlocal model to prove sharp high energy estimates for $G$, $\Dl$, and $\dDl$.

In Chapter \ref{ch:resFree}, we prove the quantum Sabine law (Theorems \ref{thm:qSabine1} and \ref{thm:qSabine2}) for the resonance free regions of the operators $-\Deltad{\pO}$ and $-\Deltap$ when $\Omega$ is strictly convex with smooth boundary. To prove the theorem, we perform a microlocal analysis of \eqref{eqn:boundTemp2} and \eqref{eqn:boundTemp1} to give a dynamical characterization of the size of the resonance free region for $-\Deltad{\partial\Omega}$. In Appendix \ref{ch:model}, we show that this characterization is sharp for both $-\Deltad{\pO}$ and $-\Deltap$ when $\Omega $ is the unit disk $\re^2$ and $V$ is constant.  

Finally, in Chapter \ref{ch:lowerBound}, we prove Theorem \ref{thm:qSabineOptimal} to show that the quantum Sabine law for $-\Deltad{\pO}$. is generically sharp. The analysis used to prove Theorem \ref{thm:qSabineOptimal} is essentially a rigorous version of the discussion resulting in \eqref{eqn:heurRes},

Appendix \ref{ch:model} demonstrates the sharpness of the above estimates on the size of the resonance free region for $\Omega=B(0,1)\subset \re^2$. For a more complete analysis in the case of the disk, see \cite[Chapter 2]{thesis}, \cite{GalkCircle}.
Appendix \ref{ch:iLagrange} adapts the Melrose-Uhlmann \cite{UhlMel} notion of an intersecting Lagrangian distribution to the semiclassical setting. Appendix \ref{ch:semiclassicalDirichletParametrices} constructs the Melrose-Taylor parametrix \cite{MelTayl}. The parametrix was developed to understand the wave equation near curved boundaries and was adapted by Gerard and Stefanov--Vodev for use in the semiclassical Dirichlet problem outside a strictly convex obstacle in \cite{Gerard, StefVod}. In Appendix \ref{ch:semiclassicalDirichletParametrices}, we adapt this construction to the interior and exterior of a convex domain and to perturbative ($\Im z\leq Mh\log h^{-1}$) complex energies. In particular, we construct operators describing boundary layer operators and potentials for use in Chapter \ref{ch:layer}.


\chapter{Preliminaries}
\label{ch:prelim}
\section{The billiard ball flow and map}
\label{sec:billiard}
We need notation for the billiard ball flow and billiard ball map. Write $\nu$ for the outward pointing unit normal to $\pO$. Then
\m S^*\re^d|_{\partial\Omega}=\partial\Omega_+\sqcup\partial\Omega_-\sqcup\partial\Omega_0\,\,\m
where $(x,\xi)\in \partial\Omega_+$ if $\xi$ is pointing out of $\Omega$ (i.e. $\nu(\xi)>0$), $(x,\xi)\in\partial\Omega_-$ if it points inward (i.e $\nu(\xi)<0$), and $(x,\xi)\in\partial\Omega_0$ if $(x,\xi)\in S^*\partial\Omega$. The points $(x,\xi)\in \partial\Omega_0$ are called \emph{glancing} points. Let $B^*\partial\Omega$ be the unit coball bundle of $\partial\Omega$ and denote by $\pi_{\pm}:\partial\Omega_{\pm}\to B^*\partial\Omega$ and $\pi:S^*\re^d|_{\partial\Omega}\to \overline{B^*\partial\Omega}$ the canonical projections onto $\overline{B^*\partial\Omega}$. Then the maps $\pi_{\pm}$ are invertible. Finally, write 
\m t_0(x,\xi)=\inf\{t>0:\exp_t(x,\xi)\in T^*\re^d|_{\partial\Omega}\}\,\,\m
where $\exp_t(x,\xi)$ denotes the lift of the geodesic flow to the cotangent bundle. That is, $t_0$ is the first positive time at which the geodesic starting at $(x,\xi)$ intersects $\partial\Omega$.

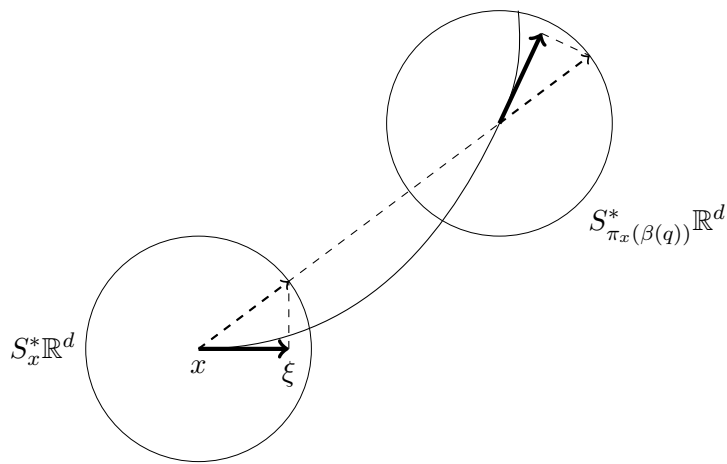
\begin{figure}
\centering
\begin{tikzpicture}
\begin{scope}[shift={(4,3)}, scale =1.5]
\draw[ultra thin] (0,0) circle [radius=1];
\draw[dashed, thick, ->] (0,0) to (.8,.6);
\draw[ultra thick,->] (0,0)to (.373,.799);
\draw[thin, dashed] (0.373,0.799)to (0.8,0.6);
\draw (.7,-.9)node[right]{$S_{\pi_x(\beta(q))}^*\re^d$};
\end{scope}
\begin{scope}[scale=1.5]
\draw[ultra thin] (0,0) circle [radius=1];
\draw[ultra thick,->] (0,0)node[below]{$x$} to (.8,0)node[below]{$\xi$};
\draw[thin, dashed](.8,0) to (.8,.6);
\draw[dashed, thick, ->] (0,0) to (.8,.6);
\draw (-1,0)node[left]{$S_x^*\re^d$};
\end{scope}
\draw [thin](0,0) to [out=0, in =-115] (4,3) to [out=65, in =-85] (4.25,4.5);
\draw[ dashed, thin] (0,0) to (4,3);
\end{tikzpicture}
\caption[Billiard ball map]{\label{fig:billiardBallmap} The figure shows how the billiard ball map is constructed. Let $q=(x,\xi)\in B^*\partial\Omega$. The solid black arrow on the left denotes the covector $\xi\in B_x^*\partial\Omega$ and that on the right $\xi(\beta(q))\in B_{\pi_x(\beta(q))}^*\partial\Omega.$ The center of the left circle is $x$ and that of the right is $\pi_x(\beta(q)).$ }
\end{figure}

We define the broken geodesic flow as in \cite[Appendix A]{DyZw}. Without loss of generality, we assume $t_0> 0$. Fix $(x,\xi)\in S^*\re^d$ and denote $t_0=t_0(x,\xi)$. If $\exp_{t_0}(x,\xi)\in \partial\Omega_0$, then the billiard flow cannot be continued past $t_0$. Otherwise there are two cases: $\exp_{t_0}(x,\xi)\in \partial\Omega_+$ or $\exp_{t_0}(x,\xi)\in \partial\Omega_-$. We let 
$$(x_0,\xi_0)=\begin{cases}
\pi_-^{-1}(\pi_+(\exp_{t_0}(x,\xi)))\in \partial\Omega_-\,,&\text{if }\exp_{t_0}(x,\xi)\in \partial\Omega_+\\
\pi_+^{-1}(\pi_-(\exp_{t_0}(x,\xi)))\in \partial\Omega_+\,,&\text{if }\exp_{t_0}(x,\xi)\in \partial\Omega_-
\end{cases}.$$
We then define $\varphi_t(x,\xi)$, the \emph{broken geodesic flow}, inductively by putting 
$$\varphi_t(x,\xi)=\begin{cases}\exp_t(x,\xi)&0\leq t<t_0\\
\varphi_{t-t_0}(x_0,\xi_0)&t\geq t_0
\end{cases}.$$

We introduce notation from \cite{Saf1} for the billiard flow. Let $K$ be the set of ternary fractions of the form $0.k_1k_2,\dots$, where $k_j=0$ or $1$ and $S$ denote the left shift operator 
\m S(0.k_1k_2\dots)=0.k_2k_3\dots.\m

\noindent For $k\in K$, we define the billiard flow of type $k$, $G_k^t:S^*\re^d\to S^*\re^d$ as follows. For $0\leq t\leq t_0$,
\begin{equation}
\label{eqn:billiardFlow}G_k^t(x,\xi)=\begin{cases}\varphi_t(x,\xi)&\text{if }k_1=0\\
\exp_t(x,\xi)&\text{if }k_1=1
\end{cases}
\end{equation}
Then, we define $G_k^t$ inductively for $t>t_0$ by 
\begin{equation}
\label{eqn:billiardFlow2}
G_k^t(x,\xi)=G_{Sk}^{t-t_0}(G_k^{t_0}(x,\xi)).
\end{equation}
We call $G_k^t$ the billiard flow of type $k$. By \cite[Proposition 2.1]{Saf1}, $G_k^t$ is measure preserving.

\begin{remarks}
\item In \cite{Saf1}, geodesics could be of multiple types when total internal reflection occurred. However, in our situation, the metrics on either side of the boundary match, so there is no total internal reflection and geodesics are uniquely identified by their starting points and $k\in K$. 
\item In general, there exist situations where $G_k^t$ intersects the boundary infinitely many times in finite time. However, since we work in convex domains, we need not consider this situation.
\end{remarks}
Now, for $k\in K $ and $T>0$, we define the set $\mc{O}_{T,k}\subset S^*\re^d$ to be the complement of the set of $(x,\xi)$ such that one can define the flow $G_k^t$ for $t\in[0,T]$. That is, $\mc{O}_{T,k}$ is the set for which the billiard flow of type $k$ is glancing in time $0\leq t\leq T.$ Last, define the set
\begin{equation}
\label{eqn:glancingSet}
\mc{O}_T=\bigcup_{k\in K}\mc{O}_{k,T}.
\end{equation}

The billiard ball map reduces the dynamics of $G_0^k$ to the boundary. We define the billiard ball map as in \cite{HaZel}. Let $(x,\xi')\in B^*\partial\Omega$ and $(x,\xi)=\pi_-^{-1}(x,\xi')\in \partial\Omega_-$ be the unique inward pointing covector with $\pi(x,\xi)=(x,\xi')$. Then, the billiard ball map $\beta:B^*\partial\Omega\to \overline{B^*\partial\Omega}$ maps $(x,\xi')$ to the projection onto $T^*\partial\Omega$ of the first intersection of the billiard flow with the boundary. That is,
\begin{equation}
\label{eqn:billiardMap}\beta:(x,\xi')\mapsto \pi(\exp_{t_0(x,\xi)}(x,\xi)).\end{equation}

\begin{remarks} 
\item Just like the billiard flow, the billiard ball map is not defined for $(x,\xi')\in \pi(\partial\Omega_0)=S^*\partial\Omega$. However, since we consider convex domains, $\beta:B^*\Omega\to B^* \Omega$ and $\beta^n$ is well defined on $B^*\partial\Omega$.
\item Figure \ref{fig:billiardBallmap} shows the process by which the billiard ball map is defined.
\end{remarks}

The billiard ball map is symplectic. This follows from the fact that the Euclidean distance function $|x-x'|$ is locally a generating function for $\beta$; that is, the graph of $\beta$ in a
neighborhood of $(x_0, \xi_0, y_0,\eta_0)$ is given by
\begin{equation}
\label{eqn:locBillRelation}
\{(x\,,\,-d_x|x-y|\,,\,y\,,\,d_y|x-y|\,)\,:(x,y)\in \partial\Omega\times \partial\Omega\}.
\end{equation}
 We denote the graph of $\beta$ by $C_b$. For strictly convex $\Omega$, $C_b$ is given globally by \eqref{eqn:locBillRelation}.

\subsection{Dynamics in Strictly Convex Domains}
\label{sec:dynamicsConvex}
Let $\partial \Omega$ be strictly convex near a point $x_0$ and let $\gamma:[0,\delta)\to \partial \Omega$ be a smooth geodesic parametrized by arc length with $\gamma(0)=x_0$. We are interested in how $|\xi'|_g$ changes under the billiard ball map for $|\xi'|_g$ close to 1. Our interest in this region comes from a desire to understand how the reflection coefficients $R_{\delta}$ and $R_{\delta'}$ from \eqref{eqn:reflect} and \eqref{eqn:reflectPrime} behave when a wave travels nearly tangent to a strictly convex boundary.

We start by examining how the normal component to $\partial \Omega$ changes under the billiard ball map.
Notice that for $|\xi'|_g$ sufficiently close to 1, the strict convexity of $\partial \Omega$ at $x_0$ implies that there is a geodesic connecting $x_0$ to $\pi_x(\beta(x_0))$ which lies inside a small neighborhood of $x_0$. (Here $\pi_x$ denotes projection to the base.)
Hence, we consider
\begin{align*}
\Delta_{\xi_d}&=\frac{((\gamma(s)-\gamma(0))\cdot \nu(0)-(\gamma(0)-\gamma(s))\cdot\nu(s))}{|\gamma(s)-\gamma(0)|}\\
&=\frac{(\gamma(s)-\gamma(0))\cdot(\nu(0)+\nu(s))}{|\gamma(s)-\gamma(0)|}.
\end{align*}
Here $|\cdot|$ is the euclidean norm in $\re^d$ and $\nu$ is the inward pointing unit normal.

First, note that 
\begin{gather*} 
\gamma''(s)=k(s)\nu(s),\qquad\quad \nu'(s)\cdot \gamma'(s)=-k(s),\\
\gamma'(s)\cdot \nu(s)=0,\qquad \quad\|\gamma'(s)\|=\|\nu(s)\|=1
\end{gather*}
where $k(s)$ is the curvature of $\gamma$. 
Then, expanding in Taylor series gives
\begin{align}
\Delta_{\xi_d}\left[s+\O{}(s^2)\right]&=\left[\gamma'(0)s+\gamma''(0) \tfrac{s^2}{2} +\gamma^{(3)}(0)\tfrac{s^3}{6} +\mc O(s^4)\right]\cdot\\
&\quad \quad \left[2\nu(0)+\nu'(0)s+\nu''(0)\tfrac{s^2}{2}+\mc O(s^3)\right]\nn
\Delta_{\xi_d}\left[1+\O{}(s)\right]&=2\gamma'(0)\cdot \nu(0)+\left(\gamma'\cdot \nu\right)'(0)s\nn
&\quad\quad+\left(2\gamma^{(3)}(0)\cdot \nu(0)+3(\gamma'\cdot \nu')'(0)\right)\tfrac{s^2}{6}+\O{}(s^3)\nn
\Delta_{\xi_d}&=\left[2(k'(0)\nu(0)-k(0)\nu'(0))\cdot\nu(0)-3\kappa'(0)\right]\tfrac{s^2}{6}+\mc O(s^3)\nn
\label{eqn:dynamicsApprox}\Delta_{\xi_d}&=(2k'(0)-3k'(0))\tfrac{s^2}{6}+\mc O(s^3)=-k'(0)\tfrac{s^2}{6}+\mc O(s^3).
\end{align}
Next observe that
$$\frac{\gamma(s)-\gamma(0)}{|\gamma(s)-\gamma(0)|}\cdot\nu(0) =\frac{k(0)}{2}s+\O{}(s^2).$$ 

So, using that for $\Omega$ is strictly convex, $k(0)>c>0$, if $\sqrt{1-|\xi'|_g^2}=r$, $cs\leq r\leq Cs$. Using \eqref{eqn:dynamicsApprox}, we have $\sqrt{1-|\xi'(\beta(x_0,\xi'))|_g^2}=r+\mc O(r^2)$.
Summarizing, we have
\begin{lemma}
\label{lem:dynStrictlyConvex}
Let $\Omega\subset \re^d$ be strictly convex. Then, letting $q\in B^*\partial\Omega$ and denote $$r:=\sqrt{1-|\xi'(q)|_g^2},$$ we have
$$\sqrt{1-|\xi'(\beta(q))|^2_g}=r+\O{}(r^2).$$
\end{lemma}
 
By the calculations above, the set of $\O{}(h^\e)$ near glancing points is stable under the billiard ball map. This also follows from the equivalence of glancing hypersurfaces \cite{MelroseGlanceSurf}. Moreover, we have the following lemma:
\begin{lemma}
\label{lem:iteratedStability}
Fix $\e>0$. Suppose that $\Omega$ is strictly convex and $(x',\xi')\in T^*\partial\Omega$ with $1-|\xi'|_g=\O{}(h^\e).$ Then, for $N=\O{}(h^{-\e})$, 
$$1-|\xi'(\beta^N(x',\xi'))|_g=\O{}(h^\e).$$
\end{lemma}
\begin{proof}
Suppose that $\sqrt{1-|\xi'|^2_g}=r$. Then, by \eqref{eqn:dynamicsApprox}, 
$$\sqrt{1-|\xi'(\beta(x',\xi'))|^2_g}\leq r+C_1r^2\quad \text{for }r\text{ small enough}$$
where $C_1>0$ is uniform in $B^*\partial\Omega$. Let $a_n=\sqrt{1-|\xi'(\beta^n(x',\xi'))|^2_g}.$
Then, $a_n\leq a_{n-1}+C_1a_{n-1}^2.$ Therefore, we need only examine the sequence 
$$x_{n}=x_{n-1}+C_1x_{n-1}^2,\quad x_1=Ch^\e.$$

First, observe that if $x_j\leq Cjh^\e$, then, 
$$x_{j+1}=x_j(1+C_1x_j)\leq Cjh^\e(1+CC_1jh^\e)=C(j+1)h^\e\left(\frac{j}{j+1}+\frac{CC_1jh^\e}{j+1}\right).$$
Therefore, for $j\leq C^{-1}C_1^{-1}h^{-\e},$ $x_{j+1}\leq C(j+1)h^\e$. 

Now, we have
\begin{align*} 
\frac{x_n}{x_1}&=\prod_{j=1}^{n-1}(1+C_1x_j)\leq \prod_{j=1}^{n-1}(1+CC_1jh^\e)\\
&\leq (1+CC_1(n-1)h^\e)^{n-1}=\left(1+\frac{(n-1)^2CC_1h^\e}{n-1}\right)^{n-1}.
\end{align*} 
As long as $(n-1)^2=\O{}(h^{-\e})$ and $n-1\leq C_1^{-1}C^{-1}h^{-\e}$, we have
$x_{n}=x_1\O{}(1).$
\end{proof}


\section{Semiclassical preliminaries}
\label{sec:semiclassicalPreliminaries}
In this section, we review the methods of semiclassical analysis which are needed throughout the rest of our work. The theories of pseudodifferential operators, wavefront sets, and the local theory of Fourier integral operators are standard and our treatment follows that in \cite{EZB} and \cite{DyGui}.  We introduce the notion \emph{shymbol} from \cite{Galk} which is a notion of sheaf-valued symbol that is sensitive to local changes in semiclassical order of a symbol. 

\subsection{Notation}
We review the relevant notation from semiclassical analysis in this section. For more details, see \cite{d-s} or \cite{EZB}.
\subsubsection{Big $\O{}$ notation}
The $\O{}(\cdot)$ and $\o{}(\cdot)$ notations are used in this paper in the following ways:
we write $u=\O{\mc X}(F)$ if the norm of $u$ in the functional space $\mc X$ is bounded
by the expression $F$ times a constant. We write $u=\o{\mc{X}}(F)$ if the norm of $u$  has
$$\lim_{s\to s_0}\frac{\|u(s)\|_{\mc{X}}}{F(s)}=0$$
where $s$ is the relevant parameter. If no space $\mc{X}$ is specified, then $u=\O{}(F)$ and $u=\o{}(F)$ mean 
$$|u(s)|\leq C_z|F(s)|\quad \text{and} \quad \lim_{s\to s_0}\frac{|u(s)|}{F(s)}=0$$
respectively. 
\subsubsection{Phase space}
Let $M$ be a $d$-dimensional manifold without boundary. Then, we denote an element of the cotangent bundle to $M$, $(x,\xi)$ where $\xi\in T_x^*M$.
\subsection{Symbols and Quantization}
We start by defining the exotic symbol class $f(h)S^m_\delta(M)$.
\begin{defin}
\label{def:symbol}
Let $a(x,\xi; h)\in C^\infty (T^*M\times [0,h_0) )$, $f\in C^\infty((0,h_0))$, $m\in \re$, and $\delta\in[0,1/2)$. Then, we say that $a\in f(h)S^m_{\delta}(T^*M)$ if for every $K\Subset M$ and $\alpha,\,\beta$ multiindeces, there exists $C_{\alpha\beta K}$ such that
\begin{equation} |\partial_x^\alpha\partial_\xi^\beta a(x,\xi; h)|\leq C_{\alpha\beta K}f(h)h^{-\delta(|\alpha|+|\beta|)}\la \xi\ra^{m-|\beta|}\}\label{eqn:defSymbol}\end{equation}
We denote $S_\delta^\infty:=\cup_m S_\delta^m$, $S_\delta^{-\infty}:=\cap_mS_\delta^m$ and when one of the parameters $\delta$ or $m$ is 0, we suppress it in the notation. 

We say that $a(x,\xi;h)\in S_\delta^{\comp}(M)$  if $a\in S_{\delta}(M)$ and $a$ is supported in some $h$-independent compact set.
\end{defin}
This definition of a symbol is invariant under changes of variables (see for example \cite[Theorem 9.4]{EZB} or more precisely, the arguments therein). 

\subsection{Pseudodifferential operators}
We follow \cite[Section 14.2]{EZB} to define the algebra $\Ph{m}{\delta}(M)$ of pseudodifferential operators with symbols in $S^m_\delta(M)$. (For the details of the construction of these operators, see for example \cite[Sections 4.4, 14.12]{EZB}. See also \cite[Chapter 18]{HOV3} or \cite[Chapter 3]{gr-s}.) Since we have made no assumption on the behavior of our symbols as $x\to \infty$, we do not have control over the behavior of $\Ph{k}{\delta}(M)$ near infinity in $M$. However, we do require that all operators $A\in \Ph{m}{\delta}(M)$ are \emph{properly supported}. That is, the restriction of each projection map $\pi_x,\pi_{x'}M\times M\to M$ to the support of $K_A(x,x;h)$, the Schwartz kernel of $A$, is a proper map. For the construction of such a quantization procedure, see for example \cite[Proposition 18.1.22]{HOV3}. An element in $A\in \Ph{m}{\delta}(M)$ acts $H^s_{h,\loc}(M)\to H^{s-m}_{h,\loc}(M)$ where $H^s_{h,\loc}(M)$ denotes the space of distributions locally in the semiclassical Sobolev space $H_h^s(M)$. The definition of these spaces can be found for example in \cite[Section 7.1]{EZB}. Finally, we say that a properly supported operator, $A$, with
$$A:\mc{D}'(M)\to C^\infty(M)$$ 
and each seminorm $\O{}(h^\infty)$ is $\O{\Ph{-\infty}{}}(h^\infty)$. We include operators that are $\O{\Ph{-\infty}{}}(h^\infty)$  in all pseuodifferential classes. 

With this definition, we have the semiclassical principal symbol map
\begin{equation}\label{eqn:princSymbol}\sigma:\Psi^m_{\delta}(M)\to \quotient{S^m_{\delta}(M)}{h^{1-2\delta}S^{m-1}_\delta(M)}\end{equation}
and a non-canonical quantization map
$$\oph:S^m_\delta(M)\to \Ph{m}{\delta}(M)$$
with the property that $\sigma\composed \oph$ is the natural projection map onto  
$$\quotient{S^m_{\delta}(M)}{h^{1-2\delta}S^{m-1}_\delta(M)}.$$
When pseudodifferential operators act on half-densities, we also have the subprincipal symbol map
$$\sigma_1:\Psi^m_{\delta}(M)\to \quotient{h^{1-2\delta} S^{m-1}_{\delta}(M)}{h^{2-3\delta}S^{m-2}_{\delta}(M)}$$
and one can find a quantization, based locally on the Weyl quantization, satisfying
$$(\sigma+\sigma_1)\composed \oph =\pi: S^m_{\delta}(M)\to  \quotient{S^m_{\delta}(M)}{h^{2-3\delta}S^{m-2}_\delta(M)}$$
for $\pi$ the natural projection map.
These mapping properties follow from keeping careful track of the errors the proof of \cite[Theorem  9.10]{EZB} together with \cite[Section 14.12]{EZB} (see also, \cite[Section 4.2]{thesis}).

Henceforward, we will write $\sigma(A)$ to write any representative of the corresponding equivalence class in the right-hand side of \eqref{eqn:princSymbol}. We do not include the sub-principal symbol because then the calculus of pseudodifferential operators would be more complicated. With this in mind, the standard calculus of pseudodifferential operators with symbols in $S^m_\delta$ gives for $A\in \Ph{m_1}{\delta}(M)$ and $B\in \Ph{m_2}{\delta}(M)$, 
\begin{gather*}
\sigma(A^*)=\overline{\sigma(A)}+\O{S^{m_1-1}_\delta(M)}(h^{1-2\delta})\\
\sigma(AB)=\sigma(A)\sigma(B)+\O{S^{m_1+m_2-1}_{\delta}(M)}(h^{1-2\delta})\\
\sigma([A,B])=-ih\{\sigma(A),\sigma(B)\}+\O{S^{m_1+m_2-2}(M)}(h^{2(1-2\delta)}).
\end{gather*}
Here $\{\cdot,\cdot\}$ denotes the Poisson bracket and we take adjoints with respect to $L^2(M)$. 

\subsubsection{Wavefront sets and microsupport of pseudodifferential operators}
In order to define a notion of wavefront set that captures both $h$-microlocal and $C^\infty$ behavior, we define the \emph{fiber radially compactified cotangent bundle}, $\overline{T}^*M$, by $\overline T
^*M=T^*M\sqcup S^*M$ where 
$$S^*M:=\quotient{\left(T^*M\setminus\{M\times 0\}\right)}{\re_+}$$
and the $\re_+$ action is given by $(t,(x,\xi))\mapsto(x,t\xi).$ Let $|\cdot|_g$ denote the norm induced on $T^*M$ by the Riemannian metric $g$. Then a neighborhood of a point $(x_0,\xi_0)\in S^*M$ is given by $V\cap \{|\xi|_g\geq K\}$ where $V$ is an open conic neigbhorhood of $(x_0,\xi_0)\in T^*M$. 

For each $A\in \Ph{m}{\delta}(M)$ there exists $a\in S^m_\delta(M)$ with $A=\oph(a)+\O{\Ph{-\infty}{}}(h^\infty).$ Then the \emph{semiclassical wavefront set of $A$}, $\WFhp(A)\subset \overline T^*M$, is defined as follows. A point $(x,\xi)\in \overline T^*M$ does not lie in $\WFhp(A)$ if there exists a neighborhood $U$ of $(x,\xi)$ such that each $(x,\xi)$ derivative of $a$ is $\O{}(h^\infty \la \xi\ra^{-\infty})$ in $U$. As in \cite{Alexandrova}, we write 
$$\WFhp(A)=:\WFhpf(A)\sqcup\WFhpi(A)$$
where $\WFhpf(A)=\WFh(A)\cap T^*M$ and $\WFhpi(A)=\WFh(A)\cap S^*M.$

Operators with compact wavefront sets in $T^*M$ are called \emph{compactly microlocalized}. These are operators of the form 
$$\oph(a)+\O{\Ph{-\infty}{}}(h^\infty)$$ for some $a\in S^{\comp}_\delta(M).$ The class of all compactly microlocalized operators in $\Ph{m}{\delta}(M)$ are denoted by $\Ph{\comp}{\delta}(M)$. 

We will also need a finer notion of microsupport on $h$-dependent sets. 

\begin{defin}\label{d-microlocal-vanishing}
An operator $A\in \Ph{\comp}{\delta}(M)$ is \emph{microsupported} on an $h$-dependent family of sets $V(h)\subset T^*M$ if we can write $A=\oph(a)+\O{\Ph{-\infty}{}}(h^\infty)$, where for each compact set $K\subset T^*M$, each differetial operator $\partial^\alpha$ on $T^*M$, and each $N$, there exists a constant $C_{\alpha N K}$ such that for $h$ small enough,
$$\sup_{(x,\xi)\in K\setminus V(h)}|\partial^\alpha a(x,\xi;h)|\leq C_{\alpha NK}h^N.$$
We then write $$\MSp(A)\subset V(h).$$
\end{defin}

%
The change of variables formula for the full symbol of a
pseudodifferential operator~\cite[Theorem~9.10]{EZB} contains an asymptotic expansion in powers
of $h$ consisting of derivatives of the original symbol. Thus
definition~\ref{d-microlocal-vanishing} does not depend on the choice
of the quantization procedure $\oph$.  Moreover, since we take $\delta<1/2$, if
$A\in\Ph{\comp}{\delta}$ is microsupported inside some $V(h)$ and
$B\in\Ph{m}{\delta}$, then $AB$, $BA$, and $A^*$ are also microsupported
inside $V(h)$. This implies the following.
\begin{lemma}
Suppose that $A,B\in \Ph{\comp}{\delta}$ and $\MSp(A)\cap\MSp(B)=\emptyset.$ Then 
$$\WFhp(AB)=\emptyset.$$
\end{lemma}

For $A\in \Ph{\comp}{\delta}(M)$, $(x,\xi)\notin \WFh(A)$ if and only if there exists an $h$-independent neighborhood $u$ of $(x,\xi)$ such that $A$ is microsupported on the complement of $U$. However, $A$ need only be microsupported on any $h$-independent neighborhood of $\WFhp(A)$, not on $\WFhp(A)$ itself. Also, notice that by Taylor's formula if $A\in \Ph{\comp}{\delta}(M)$ is microsupported in $V(h)$ and $\delta'>\delta$, then $A$ is also microsupported on the set of all points in $V(h)$ which are at least $h^{\delta'}$ away from the complement of $V(h)$.

\begin{remark} Notice that since we are working with $A\in \Ph{\comp}{\delta}(M)$ for $0\leq\delta<1/2$ we have $a\in S_\delta^{\comp}(T^*M)$ and $a$ can only vary on a scale $\sim h^{-\delta}$. This implies that the set $\MSp(A)$ will respect the uncertainty principle.
\end{remark}

\subsubsection{Ellipticity and $L^2$ operator norm}

For $A\in\Ph{m}{\delta}(M)$, define its \emph{elliptic set}
$\Ell(A)\subset T^*M$ as follows: $(x,\xi)\in\Ell(A)$ if and
only if there exists a neighborhood $U$ of $(x,\xi)$ in $\overline
T^*M$ and a constant $C$ such that $|\sigma(A)|\geq
C^{-1}\langle\xi\rangle^m$ in $U\cap T^*M$.  The following statement
is the standard semiclassical elliptic estimate;
see~\cite[Theorem~18.1.24']{HOV3} for the closely related microlocal
case and for example~\cite[Section~2.2]{zeeman} for the semiclassical
case.
%
%
\begin{lemma}
\label{lem:microlocalElliptic}
Suppose that $P\in\Ph{m}{\delta}(M)$ and $A\in\Ph{m'}{\delta}(M)$ with $\WFhp(A)\subset\Ell(P)$. Then for each $\chi\in\Cc(M)$,  there exist $Q_i\in \Ph{m'-m}{\delta}(M)$ such that 
$$\chi A=\chi Q_1P+\O{\Ph{-\infty}{\delta}}(h^\infty)=\chi PQ_2+\O{\Ph{-\infty}{}}(h^\infty).$$
In particular, for each $s\in \re$ and $u\in H_h^{s+m'}$ there exists $C>0$ such that for all $N>0$, and $\chi_1\in C^\infty(M)$ with $\chi_1 \equiv 1 $ on $\supp \chi$, 
$$\|\chi Au\|_{H_h^s}\leq C\|\chi Pu\|_{H_h^{s+m'-m}}+\O{}(h^\infty)\|\chi_1 u\|_{H_h^{-N}}.$$
\end{lemma}

We also recall the estimate for the $L^2\to L^2$ norm of a pseudodifferential operator (see for example \cite[Chapter 13]{EZB}).
\begin{lemma}
\label{lem:goodL2Bound}
Suppose that $A\in \Ph{}{\delta}(M)$. Then there exists $C>0$ such that 
$$\|A\|_{L^2\to L^2}\leq \sup_{T^*M} |\sigma(A)|+Ch^{1-2\delta}.$$
\end{lemma}

\subsection{Semiclassical microlocalization of distributions and operators}
\subsubsection{Semiclassical wavefront sets and microsupport for distributions}
An $h$-dependent
family $u(h):(0,h_0)\to \mc D'(M)$ is called \emph{h-tempered} if
for each open $U\Subset M$, there exist constants
$C$ and $N$ such that 
\begin{equation}
  \label{tempered}
\|u(h)\|_{H^{-N}_{h}(U)}\leq Ch^{-N}.
\end{equation}
For a tempered distribution $u$, we say that $(x_0,\xi_0)\in \overline T^*M$ does not lie in the wavefront set $\WFh(u)$, if there exists a
neighborhood $V$ of $(x_0,\xi_0)$ such that for each
$A\in\Ph{}{}(M)$ with $\WFhp(A)\subset V$, we have $Au=\O{C^\infty}(h^\infty)$. As above, we write 
$$\WFh(u)={\WFh}^f(u)\sqcup{\WFh}^i(u)$$
where ${\WFh}^i(u)=\WFh(u)\cap S^*M$.
By Lemma~\ref{lem:microlocalElliptic},
$(x_0,\xi_0)\not\in\WFh(u)$ if and only if there exists compactly
supported $A\in\Ph{}{}(M)$ elliptic at $(x_0,\xi_0)$ such that
$Au=\O{C^\infty}(h^\infty)$.  The wavefront set of $u$ is a
closed subset of $\overline T^*M$. It is empty if and only if
$u=\O{C^\infty(M)}(h^\infty)$. We can also verify that for
$u$ tempered and $A\in\Ph{m}{\delta}(M)$,
$\WFh(Au)\subset\WFhp(A)\cap\WFh(u)$.

\begin{defin}
A tempered distribution $u$ is said to be \emph{microsupported} on an $h-$dependent family of sets $V(h)\subset T^*M$ if for $\delta\in[0,1/2)$, $A\in \Ph{}{\delta}(M)$, and $\MSp(A)\cap V=\emptyset$, $\WFh(Au)=\emptyset.$
\end{defin}

\subsubsection{Semiclassical wavefront sets of tempered operators}
An $h$- dependent family of operators $A(h):\mc{S}(M)\to \mc{S}'(M')$ is called \emph{h-tempered} if for each $U\Subset M$, there exists $N\geq 0$ and $k\in \ints^+$, such that
\begin{equation}
\label{eqn:temperedOp}
\|A(h)\|_{H_h^k(U)\to H_{h,\loc}^{-k}(M')}\leq Ch^{-N}
\end{equation}

For an $h$-tempered family of operators, we write that the wavefront set of $A$ is given by
$${\WFh}'(A):=\{(x,\xi,y,\eta)\,|,(x,\xi,y,-\eta)\in \WFh(K_A)\}$$
where $K_A$ is the Schwartz kernel of $A$. 
\begin{defin}
\label{d:microlocal-vanishingOps}
A tempered operator $A$ is said to be
\emph{microsupported} on an $h$-dependent family of sets
$V(h)\subset T^*M\times T^*M'$, if for all $\delta\in [0,1/2)$ and each $B_1\in \Ph{}{\delta}(M')$ and $B_2\in \Ph{}{\delta}(M)$ with $(\MSp(B_1)\times\MSp(B_2))\cap V=\emptyset$, we have $\WFh(B_1AB_2)=\emptyset.$ We then write 
$${\MS}'(A)\subset V(h).$$
\end{defin}

\begin{remark}
With the definitions above, we have for $A\in \Ph{m}{\delta}(M)$, $${\WFh}'(A)=\{(x,\xi,x,\xi)\,:\, (x,\xi)\in \WFhp(A)\}.$$ 
In addition, we have that if $A\in \Ph{\comp}{\delta}$, then $\MSp(A)\subset V(h)$ if and only if 
$${\MS}'(A)\subset \{(x,\xi,x,\xi)\,:\, (x,\xi)\in V(h)\}.$$
Since there is a simple relationship between $\WFhp$ and $\WFh$, as well as $\MSp$ and $\MS$, we will only use the notation without $\Psi$ from this point forward and the correct object will be understood from context.
\end{remark}

\subsection{Semiclassical Lagrangian distributions}
  \label{s:prelim.lagrangian}

In this subsection, we review some facts from the theory of
semiclassical Lagrangian distributions.  See~\cite[Chapter~6]{g-s}
or~\cite[Section~2.3]{svn} for a detailed account,
and~\cite[Section~25.1]{HOV4} or~\cite[Chapter~11]{gr-s} for the microlocal case. We do not attempt to define the principal symbol as a globally invariant object. Indeed, it is not always possible to do so in the semiclassical setting. When it is possible to do so, i.e. when the Lagrangian is exact, we define the symbol modulo the Maslov bundle. When the Lagrangian is not exact, a factor $e^{iA/h}$ with $A$ a constant depending on the choice of phase function appears. Taking symbols modulo the Maslov bundle makes the theory considerably simpler. We can make this simplification since for all of our symbolic computations, we work only in a single coordinate chart and, moreover, except when $\Lambda$ is conic and hence exact, we are concerned only with the absolute value of the symbol.

\subsubsection{Phase functions}
Let $M$ be a manifold without boundary.  We denote its dimension
by $d$.  Let $\varphi(x,\theta)$ be a smooth real-valued function on
some open subset $U_\varphi$ of $M\times \mathbb R^L$, for some $L$;
we call $x$ the \emph{base variable} and $\theta$ the \emph{oscillatory
variable}. As in \cite[Section 21.2]{HOV3}, we say that $\varphi$ is a \emph{clean phase
function} with excess $e$ if the number of linearly independent differentials $d
(\partial_{\theta_1}\varphi),\dots,d(\partial_{\theta_L}\varphi)$ on the \emph{critical set}
\begin{equation}
  \label{e:c-varphi}
C_\varphi:=\{(x,\theta)\mid \partial_\theta \varphi=0\}\subset U_\varphi
\end{equation}
is equal to $L-e$ where $e=\dim C_\varphi-\dim X.$
Note that
\[
\Lambda_\varphi:=\{(x,\partial_x\varphi(x,\theta))\mid (x,\theta)\in C_\varphi\}\subset T^*M
\]
is an immersed Lagrangian submanifold (we will shrink the domain of $\varphi$ to make it
embedded). We say that $\varphi$
\emph{generates} $\Lambda_\varphi$. We call $\varphi$ a \emph{non-degenerate} phase function if $e=0$.

\subsubsection{Symbols}
Let $\delta\in [0,1/2)$. A smooth function $a(x,\theta;h)$ is called a
compactly supported symbol of type $\delta$ on $U_\varphi$, if it is supported in some compact $h$-independent subset of $U_\varphi$, and for each differential operator $\partial^\alpha$ on $M\times \mathbb R^L$, there exists a constant $C_\alpha$ such that
$$
\sup_{U_\varphi}|\partial^\alpha a|\leq C_\alpha h^{-\delta|\alpha|}.
$$
As above, we write $a\in
S^{\comp}_\delta(U_\varphi)$ and denote $S^{\comp}:=S^{\comp}_0$. 

\subsubsection{Lagrangian distributions}
Given a clean phase function $\varphi$ with excess $e$ and a symbol $a\in S_\delta^{\comp}(U_{\varphi})$, consider the $h$-dependent family of
functions
\begin{equation}
  \label{e:lagrangian-basic}
u(x;h)=(2\pi h)^{-(d+2L-2e)/4}\int_{\mathbb R^L} e^{i\varphi(x,\theta)/h}a(x,\theta;h)\,d\theta.
\end{equation}
We call $u$ a \emph{Lagrangian distribution} of type $\delta$ generated by $\varphi$ and denote this by $u\in I^{\comp}_{\delta}(\Lambda_{\varphi})$.

By the method of non-stationary phase, if
$\supp a$ is contained in some $h$-dependent compact
set $K(h)\subset U_\varphi$, then
\begin{equation}
  \label{e:lag-wf}
\MS(u)\subset\{(x,\partial_x\varphi(x,\theta))\mid (x,\theta)\in C_\varphi\cap K(h)\}\subset\Lambda_\varphi.
\end{equation}
\begin{remark} We are using the fact that $a\in S_\delta(U_\varphi)$ for some $\delta<1/2$ here.
\end{remark}

The \emph{phase dependent principal symbol} $\sigma_\varphi(u)\in S^{\comp}_\delta(\Lambda_\varphi)$ of $u$ 
is defined modulo $\O{}(h^{1-2\delta})$ by the expression
\begin{equation}
  \label{e:lagrangian-symbol}
\sigma_\varphi(u)(x,\partial_x\varphi(x,\theta);h)= a(x,\theta;h),\quad \text{ for }\quad
(x,\theta)\in C_\varphi.
\end{equation}
That $\sigma_\varphi(u)$ does not depend (modulo $\O{}(h^{1-2\delta})$) on the choice of $a$
producing $u$ will
follow from Lemma~\ref{l:lagrangian-basic}. However, it does depend on the choice of $\varphi$ parameterizing $\Lambda$.

\subsubsection{Principal Symbols}
We define the principal symbol of a Lagrangian distribution independently of the choice of $\varphi$. To do this, we will need to use half-densities on $\Lambda_{\varphi}$ (see, for example \cite[Chapter 9]{EZB} for a definition).

Following \cite[Section 25.1]{HOV4}, we split the $\theta$ variables into $\theta=(\theta',\theta'')$ such that the map $C_{\xi}\ni (x,\theta)\mapsto \theta''$ has bijective differential where 
$$C_{\xi}=\{(x,\theta):\partial_{\theta}\varphi(x,\theta)=0,\, \partial_x\varphi (x,\theta)=\xi\}.$$
Then, letting 
$$\Phi=\left(\begin{array}{cc}\varphi_{xx}''&\varphi_{x\theta'}''\\ 
\varphi_{\theta'x}''&\varphi_{\theta'\theta'}''\end{array}\right),$$ 
\begin{lemma}
\label{l:lagrangian-basic}
Modulo Maslov factors, and a factor $e^{iA/h}$ for some constant $A\in\re$ depending on $\varphi$, the \emph{principal symbol} 
$$\sigma(u)\in \quotient{S_{\delta}^{\comp}(\Lambda_\varphi;\Omega^{1/2})}{h^{1-2\delta}S_{\delta}^{\comp}(\Lambda_{\varphi};\Omega^{1/2})}$$ is a half density given by 
$$\sigma(u)(x,\xi)=|d\xi|^{1/2}\int_{C_{\xi}}a(x,\theta) e^{i\pi/4\sgn\Phi}|\det \Phi|^{-1/2}d\theta'' .$$
\end{lemma}

\begin{remark}
In the case that $\Lambda_{\varphi}$ is exact the factor $e^{iA/h}$ can be removed.
\end{remark}

\begin{defin}
  \label{d:lagrangian}
Let $\Lambda\subset T^*M$ be an embedded Lagrangian submanifold.
We say that an $h$-dependent family of functions
$u(x;h)\in \Cc(M)$ is a (compactly supported
and compactly microlocalized) 
Lagrangian distribution of type $\delta$ associated to $\Lambda$, if
it can be written as a sum of finitely many functions
of the form~\eqref{e:lagrangian-basic}, for different phase functions
$\varphi$ parametrizing open subsets of $\Lambda$, plus an
$\O{\Cc}(h^\infty)$ remainder. 
Denote by $I^{\comp}_\delta(\Lambda)$ the space of all such distributions,
and put $I^{\comp}(\Lambda):=I^{\comp}_0(\Lambda)$.
\end{defin}
%
%
By Lemma~\ref{l:lagrangian-basic}, if $\varphi$ is a phase function and $u\in
I^{\comp}_\delta(\Lambda_\varphi)$, then $u$ can be written in the
form~\eqref{e:lagrangian-basic} for some symbol $a$, plus an $
\O{\Cc}(h^\infty)$ remainder.  The symbol~$\sigma_\varphi(u)$,
given by~\eqref{e:lagrangian-symbol}, is well-defined modulo $\O{}(h^{1-2\delta})$.

The action of a pseudodifferential operator on a Lagrangian
distribution is given by the following Lemma, following from
 the method of stationary
phase:
%
%
\begin{lemma}\label{l:lagrangian-mul}
Let $u\in I^{\comp}_\delta(\Lambda)$ and $P\in \Ph{m}{\delta}(M)$.
Then $Pu\in I^{\comp}_\delta(\Lambda)$. Moreover, if $\Lambda=\Lambda_\varphi$ for some phase function $\varphi$,
then
$$
\sigma(Pu)=\sigma(P)|_{\Lambda}\cdot\sigma(u)
+\O{}(h^{1-2\delta})_{S^{\comp}_\delta(\Lambda)}.
$$
\end{lemma}

\subsection{Fourier integral operators}
\label{sec:semiFIO}

A special case of Lagrangian distributions are Fourier integral
operators associated to canonical relations.  Let $M,M'$ be two
manifolds of dimension $d_1$ and $d_2$, respectively. Consider a Lagrangian submanifold $\Lambda\subset T^*M'\times T^*M.$

A compactly supported operator $U:\mc D'(M')\to \Cc(M)$ is
called a (semiclassical) \emph{Fourier integral operator} of type
$\delta$ associated to the canonical relation 
$$C=\{(x,\xi,y,\-\eta):(x,\xi,y,-\eta)\in \Lambda\}$$
if its Schwartz kernel $K_U(x,x')$ lies
in $I^{\comp}_\delta(\Lambda)$. We write $U\in
I^{\comp}_\delta(C)$.  The numerology $h^{-(d+2L-2e)/4}$ in \eqref{e:lagrangian-basic} is explained by the fact that
the normalization for Fourier integral operators is chosen so that
$$\|U\|_{L^2(M')\to L^2(M)}\sim 1$$ when $C$ is the generated by a symplectomorphism.

The main lemma in the calculus of Fourier integral operators is as follows \cite[Theorem 25.2.3]{HOV4}
\begin{lemma}
\label{lem:FIOcomp}
Let $A_1\in I_{\delta}^{\comp}(M_2\times M_1, C_1)$ and $A_2\in I_{\delta}^{\comp}(M_3\times M_2,C_2)$ and suppose that the composition $C=C_2\composed C_1$ is clean with excess $e$. Then, $A_1\composed A_2\in h^{-e/2}I_{\delta}^{\comp}(C)$. For $\gamma\in C$, let $C_\gamma$ denote the fiber over $\gamma$ of the intersection of $C_1\times C_2$ with $T^*M_3\times \Delta(T^*M_2)\times T^*M_1.$ Then, if $a_1$ and $a_2$, and $a$ are the principal symbols of $A_1$, $A_2$ and $A_2A_1$ respectively, then 
$$a=(2\pi h)^{-e/2}\int_{C_\gamma}a_2\times a_1.$$
\end{lemma}

We will also need the following version of Lemma~\ref{l:lagrangian-mul} for pseudodifferential compositions.
%
%
%
%

\begin{lemma}
\label{l:lieDerivativeFIO}
Assume that $U\in I_\delta^{\comp}(\Lambda)$ with $\Lambda\subset T^*X\times T^*Y$ and $P\in \Ph{m}{\e}(M')$ with $p=\sigma(P)$ vanishing on the projection of $\Lambda$ onto $T^*(X)$. Denote by $p_1$ the subprincipal symbol of $P$. Then $PA\in h^{1-(\delta+\e)}I_{\max(\delta,\e)}^{\comp}(\Lambda)$ with principal symbol
$$i^{-1}h\mc{L}_{H_p}a+p_1a$$
where $H_p$ is the Hamiltonian vector field generated by $p$.
\end{lemma}

\subsection{Semiclassical wavefront set calculus}

We give some facts which are standard in the homogeneous setting. The following lemma is the analog of \cite[Theorem 8.2.4]{HOV1}
\begin{lemma}
\label{lem:pullback}
Suppose that $u$ is a tempered distribution on $M'$ and $f:M\to M'$ is a $C^\infty$ map. Let 
$$N_f=\{(f(x),\eta)\in M'\times S_{f(x)}^*M'; (df_x)^t\eta=0\}.$$
Then the pullback $f^*u$ can be defined in one and only one way for all $u\in \mc{D}'(M')$ tempered with 
$$N_f\cap {\WFh}^i(u)=\emptyset$$
so that $f^*u=u\composed f$ when $u\in C^\infty$. Moreover, 
$$\WFh(f^*u)\subset \{(x,(df_x)^t\eta): (f(x),\eta)\in \WFh(u)\}=:f^*\WFh(u).$$
\end{lemma}
\begin{proof}
The proof of the first statement follows that in \cite[Theorem 8.2.4]{HOV1} as does the statement for ${\WFh}^i(u).$

To see that ${\WFh}^f(f^*u)\subset f^*{\WFh}^f(u)$ we may assume that $M$ and $M'$ are subsets of $\re^d$ and $\re^{d'}$ respectively. We may also assume that $u$ is supported in a small neighborhood of a point $f(x_0)$ and has ${\WFh}^i(u)=\emptyset$ and hence that $u$ is compactly microlocalized. In this case, we observe that 
$$f^*u=(2\pi h)^{-d'}\int u(x')e^{\frac{i}{h}\left(\la f(x),\xi\ra-\la x',\xi\ra\right)}d\xi dx'$$
Then, since $u$ is compactly microlocalized, we can write
$$f^*u=(2\pi h)^{-d'}\int u(x')a(x',x,\xi)e^{\frac{i}{h}\left(\la f(x),\xi\ra-\la x',\xi\ra\right)}d\xi dx'+\mc{O}_{C^\infty}(h^\infty)$$
where $a\in S^{\comp}.$
Thus, $f^*u=F_a u$ where $F_a$ is a Fourier integral operator associated to the relation
$$C=\{(x,\xi,x',\eta)\in T^*M\times T^*M': x=f(x'),\, \xi=(d f_x)^t\eta\}.$$
The wave front set statement follows.
\end{proof}

Combining Lemma \ref{lem:pullback} with \cite[Lemma 5]{Alexandrova}, we have 
\begin{lemma}
\label{lem:tensor}
Suppose that $u$ and $v$ are tempered distributions on $M$. Then the product $uv$ can be defined as the pullback of $u\otimes v$ by the diagonal map $\delta:X\to X\times X$ provided that 
$${\WFh}^i(u)\cap \{(x,\xi):(x,-\xi)\in {\WFh}^i(v)\}=\emptyset.$$
Moreover,
$${\WFh}(uv)\subset\{(x,\xi+\eta): (x,\xi)\in \WFh(u),(x,\eta)\in \WFh(v)\}.$$
\end{lemma}

We also need the following simple lemma
\begin{lemma}
\label{lem:pushforward}
Suppose that $u$ is a tempered distribution on $M$ and $f:M\to M'$ is a $C^\infty$ map. Then the pushforward, $f_*u$ has
$$\WFh(f_*u)\subset(df)^t\WFh(u)$$
where 
$$(df)^tA:=\{(x',\eta):\text{ there exists }(x,\xi)\in A \text{ with }\,x'=f(x),\,(df_{x})^t\xi=\eta\}.$$
\end{lemma}
\begin{proof}
The case of ${\WFh}^i$ follows from the standard proof in the homogeneous setting (see for example \cite[Chapter 6]{g-s}). 

For the case of ${\WFh}^f$, we observe similar to lemma \ref{lem:pullback} that 
$$f_*u=(2\pi h)^{-d'}\int e^{\frac{i}{h}\left(\la f(x),\xi)-\la x',\xi\ra\right)}u(x)d\xi d x.$$
Then, as above, we may assume that $u$ is compactly microlocalized and hence that $f_*$ is a Fourier integral operator associated to 
$$C=\{(x',\eta,x,\xi)\in T^*M'\times T^*M: x'=f(x),\, (df_x)^t\eta=\xi\}.$$
\end{proof}

\subsection{The Conic Calculi}
\subsubsection{Notation for the Kohn-Nirenberg Calculus}
\label{sec:kohn}
We refer the reader to \cite{HOV3} and \cite{HOV4} for the theory of Kohn-Nirenberg pseudodifferential operators and Fourier integral operators. We denote the standard Kohn-Nirenberg symbol classes by $\SHom^m$ where for each $\chi \in \Cc(\re^d)$, 
$$\SHom^m:=\{a\in C^\infty(T^*\re^d)\,:\,|\partial_x^\beta\partial_\xi^\alpha\chi(x)a(x,\xi)|\leq C_{\chi}\la \xi\ra^{m-|\alpha|}.$$
We denote the corresponding pseudodifferential operators, and Fourier integral operators of order $k$ by $\PsiHom^m$ and $\IHom^m$. Furthermore, we denote by 
 \m\sigma:\PsiHom^m(M)\to \SHom^m(M)/\SHom^{m-1}(M)\m 
the symbol map and its right inverse, a non-canonical quantization map 
\m\op:\SHom^m(M)\to \PsiHom^m(M).\m 
We also use the notation $\WF$ to denote the $C^\infty$ wave front set of distributions and ${\WF}'$ to denote the $C^\infty$ wave front set of operators.

\subsubsection{Conic Semiclassical Lagrangian Distributions and FIOs} 
We also need a notion of semiclassical Fourier integral operators associated to conic Lagrangians. Let $\varphi(x,\theta)$ be a clean phase function with excess $e$ that is homogeneous of degree 1 in the $\theta$ variables.

We say that a smooth function $a(x,\theta;h )$ is a symbol of order $k$ on $U_\varphi$ if $\supp a\subset K\times \{|\theta|>C\}$ for some $h$ independent $C$ and $K\Subset M$ and if for $\alpha$, and $\beta$, there exist a constants $C_{\alpha\beta}$ such that
\m
|\partial_x^\alpha\partial_{\theta}^\beta a|\leq C_{\alpha\beta}\la \theta\ra^{k-|\beta|}.
\m
We write $a\in S^m(U_\varphi).$

Then we consider the $h-$dependent family of functions 
\begin{equation}
\label{eqn:semiLagrangeConic}
u(x;h)=(2\pi h)^{-(d+2L-2e)/4}\int_{\re^L}e^{i\varphi(x,\theta)/h}a(x,\theta;h)d\theta.
\end{equation}
We call $u$ a \emph{Lagrangian distribution} of order $k$ generated by $\varphi$ and denote this by $u\in I^m(\Lambda_{\varphi}).$ The properties of such distributions follow from those of the standard homogeneous Lagrangian distributions since \eqref{eqn:semiLagrangeConic} corresponds to a rescaling in the phase variable of a homogeneous Lagrangian distribution.  
\begin{defin}
Let $\Lambda\subset T^*M$ be a Lagrangian submanifold that is conic outside of a compact set in the fiber. We say that a distribution $u(x)\in \mc{D}'(M)$ is a semiclassical Lagrangian distribution of order $k$ and type $\delta$ associated to $\Lambda$ if it can be written as a sum of finitely many distributions of the form \eqref{eqn:semiLagrangeConic} for different phase functions $\varphi$, homogeneous of degree 1 in $\theta$, parametrizing open sets of $\Lambda$ plus an element of $I_\delta^{\comp}(\Lambda)$. Denote by $I_\delta^m(\Lambda)$ the space of all such distributions.
\end{defin}

The notion of conic semiclassical Fourier integral operators follows analogous to that in Section \ref{sec:semiFIO} and the calculus of conic semiclassical integral operators analogous to Lemma \ref{lem:FIOcomp} follows from the proof in the homogeneous setting.

\section{The shymbol}
\label{sec:shymbol}
In Chapters \ref{ch:layer} and \ref{ch:resFree} we will need to compute symbols of operators whose semiclassical order may vary from point to point in $T^*M$.
One can often handle this type of behavior by using weights to compensate for the growth. However, this requires some a priori knowledge of how the order changes and limits the allowable size in the change of order. In this section, we will develop a notion of a sheaf valued symbol, the \emph{shymbol}, that can be used to work in this setting without such a priori knowledge. 

Let $M$ be a compact manifold. Let $\mc{T}(T^*M)$ be the topology on $T^*M$. For $s\in \re$, denote the symbol map 
$$ \sigma_s:h^{s}\Psi^{\comp}_\delta\to h^sS^{\comp}_\delta/h^{s+1-2\delta}S^{\comp}_\delta.\,\,$$
Suppose that for some $N>0$ and $\delta \in [0,1/2)$, $A\in h^{-N}\Psi^{\comp}_\delta(M)$. We define a finer notion of symbol for such a pseudodifferential operator. Fix $0<\e\ll 1-2\delta$. For each open set $U\in \mc{T}(T^*M)$, define \emph{the $\e$-order of $A$ on $U$} 
$$ I_A^\e(U):=\sup_{s\in \mc{S}_\e}s+1-2\delta\,\,$$
where
$$\mc{S}_\e:=\left\{s\in \e\mathbb{Z} \,\left|\,\begin{gathered}\text{ there exists }\chi\in \Cc(T^*M),\,\chi|_U=1,\\\sigma_s(\oph(\chi) A\oph(\chi))|_U\equiv 0\end{gathered}\right.\right\}.$$
Then it is clear that for any $V\Subset U$ there exists $\chi\in \Cc(U)$ with $\chi\equiv 1$ on $V$ such that $\oph(\chi) A\oph(\chi)\in h^{I^\e_A(U)}\Ph{\comp}{\delta}(M).$

Give $\mc{T}(T^*M)$ the ordering that $U\leq V$ if $V\subset U$ with morphisms $U\to V$ if $U\leq V$. Notice that $U\leq V$ implies $I_A^\e(U)\leq I_A^\e(V).$ Then define the functor $F_A^\e:\mc{T}(T^*M)\to \textbf{Comm} $ (the category of commutative rings)
by 
$$F_A^\e(U)=
\begin{cases}
h^{I_A^\e(U)}S^{\comp}_\delta(M)|_U\,/\,h^{I_A^\e(U)+1-2\delta}S^{\comp}_\delta(M)|_U  & I_A^\e(U)\neq \infty\\
\{0\}  &I_A^\e(U)=\infty 
\end{cases},
$$
$$F_A^\e(U\to V)=
\begin{cases}
h^{I_A^\e(V)-I_A^\e(U)}|_V&I_A^\e(V)\neq \infty \\ 
0 &I_A^\e(V)=\infty
\end{cases}.$$

Then $F_A^\e$ is a presheaf on $T^*M$. We sheafify $F_A^\e$, still denoting the resulting sheaf by $F_A^\e$, and say that {\emph{$A$ is of $\e$-class $F_A^\e.$}} We define the \emph{stalk} of the sheaf at $q$ by $F_A^\e(q):=\varinjlim_{q\in U}F_A^\e(U).$

Now, for every $U\subset \mc{T}(T^*M)$, $I_A^\e(U)\neq \infty$, there exists $\chi_U\in \Cc(T^*M)$ with $\chi_U\equiv 1$ on $U$ such that  
$ \sigma_{I_A^\e(U)}(\oph(\chi_U) A\oph(\chi_U))|_U\neq 0.$
Then we define the \emph{$\e$-shymbol of $A$} to be the section of $F_A^\e$, $\tilde\sigma^\e_{(\cdot)}(A):\mc{T}(T^*M)\to F_A^\e(\cdot)$, given by
$$\tilde{\sigma}^\e_U(A) :=\begin{cases}\sigma_{I_A^\e(U)}(\oph(\chi_U) A\oph(\chi_U))|_U&I_A^\e(U)\neq \infty\\0&I_A^\e(U)=\infty\end{cases}.$$
Define also the \emph{$\e$-stalk shymbol}, $\tilde{\sigma}^\e(A)_q$ to be the germ of $\tilde{\sigma}^\e(A)$ at $q$ as a section of $F_A^\e.$ 

Now, define $I_A^\e(q):=\sup_{q\in U}I_A^\e(U).$  We then define the simpler \emph{compressed shymbol} 
\begin{equation}
\label{def:symbFunc}
\begin{gathered}
\tilde{\sigma}^\e(A):T^*M\to \bigsqcup_q\quotient{h^{I_A^\e(q)}\complex}{h^{I_A^\e(q)+1-2\delta}\complex}\quad \text{by}\\
\tilde{\sigma}^\e(A)(q):=\begin{cases} 0 &I_A^\e(q)=\infty\\
\lim\limits_{q\in U}^{}\tilde{\sigma}^\e_U(A)(q)&I_A^\e(q)<\infty\end{cases}\end{gathered}
\end{equation}

The limit in \eqref{def:symbFunc} exists since if $I_A^\e(q)<\infty$, then there exists $U\ni q$ such that for all $V\subset U$, $I_A^\e(V)=I_A^\e(U).$ This also shows that it is enough to take any sequence of $U_n\downarrow q.$
It is easy to see from standard composition formulae that the compressed shymbol has
$$\tilde{\sigma}^\e(AB)(q)=\tilde{\sigma}^\e(A)(q)\tilde{\sigma}(B)(q),\quad A\in h^{-N}\Psi_{\delta}^{\comp}\text{ and }B\in h^{-M}\Psi_{\delta}^{\comp}.$$
Moreover, 
\m \tilde{\sigma}^\e([A,B])(q)=-ih\left\{\tilde{\sigma}^\e(A)(q),\tilde{\sigma}^\e(B)(q)\right\}.\m

The following lemma follows from Lemma \ref{lem:FIOcomp} combined with the definitions above:
\begin{lemma}
\label{lem:EgorovSheaf}
Suppose that $A\in \Psi^{\comp}_\delta$ and let $T$ be a semiclassical FIO associated to the symplectomorphism $\kappa$ with elliptic symbol $t\in S_\delta$. Then for $0<N$ independent of $h$ 
$(AT)_N:=(T^*A^*)^N\left(AT\right)^N$ has $$ \tilde{\sigma}^\e((AT)_N)(q)=\prod\limits_{i=1}^N\left(|\tilde{\sigma}^\e(A)t|^2\composed \kappa^i(q)+\O{}\left(h^{I_{A_i}^\e(\beta^k(q))+1-2\delta}\right)\right).$$
\end{lemma}
\begin{proof}
Fix $q\in T^*M$. Let $\chi_k\in \Cc$ have $\chi_k\equiv 1$ on $B_q\left(\recip{k}\right)$ and $\supp \chi_k\subset B_q\left(\frac{2}{k}\right).$ Then let $D:=\oph(\chi_k)(AT)_N\oph(\chi_k).$ We have that 
\m D=\oph(\chi_k)(BT)_N\oph(\chi_k)+\O{\Psi^{\comp}_{\delta}}(h^\infty)\,\,\m
where $B_i=\oph(\psi_{k,i})A_i\oph(\psi_{k,i})$ and $\psi_{k,i}\equiv 1$ in some neighborhood of $\beta^i(q)$ and is supported inside a neighborhood $U_{k,i}$ of $\beta^i(q)$ such that $U_{k,i}\downarrow q.$ Then the result follows from standard composition formulae in Lemma \ref{lem:FIOcomp}.
\end{proof}
Now, since $\e>0$ is arbitrary, we define the \emph{semiclassical order of $A$ at $q$} by $I_A(q):=\sup_{\e>0}I_A^\e(q)$ with the understanding that $f=\O{}(h^{I_A(q)})$ means that for any $\e>0$, 
$$|f(q)|\leq C_\e h^{I_A(q)-\e}.$$
Furthermore, we suppress the $\e$ in the notation $\tilde{\sigma}^\e(A)(q)$ and denote the \emph{compressed shymbol}, $\tilde{\sigma}(A)(q)$, again with the understanding that for any $\e>0$,
$$\tilde{\sigma}(A)(q)\in \quotient{h^{I_A(q)-\e}\mathbb{C}}{h^{I_A(q)+1-2\delta-\e}\mathbb{C}}.$$


\chapter{Meromorphic Continuation of the Resolvent}
\label{ch:mer}
In this chapter, we begin our analysis of $-\Deltad{\pO}$ and $-\Deltap$. We start by giving the formal definition of the operators using quadratic forms. We then prove the meromorphic continuation of the resolvent for $-\Deltap$. The proof for $-\Deltad{\pO}$ can be found in \cite[Section 6]{GS}.
Then, in addition to describing resonances as poles of the meromorphic continuation of the resolvent, we give a more concrete description of resonances as solutions to transmission problems. In particular, 
\begin{theorem}
\label{thm:meromorphyPrime}
Let $\Omega\Subset \re^d$ have smooth boundary. Suppose that $V:L^2(\partial\Omega)\to L^2(\partial\Omega)$ is self adjoint and invertible.
Then $$R_V(\lambda):=(-\Deltap-\lambda^2)^{-1}$$
has a meromorphic continuation from $\Im \lambda\gg 1$ to $\mathbb{C}$ if $d$ is odd and to the logarithmic cover of $\mathbb{C}\setminus\{0\}$ if $d$ is even.
\end{theorem}

Moreover, the poles of  $R_V(\lambda)$ are in 1-1 correspondence with solutions $u\in H_{\Delta,\loc}^{3/2}(\re^d\setminus \partial\Omega)$ to 
\begin{equation}
\label{eqn:mainPrime}
\begin{cases}(-\Deltap -\lambda^2)u=0\\
u \text{ is $\lambda$-outgoing.}
\end{cases}
\end{equation}
Here, 
$$H_{\Delta}^{3/2}(U):=\{u\in H^{3/2}(U)\,:\,\Delta u\in L^2(U)\}.$$

We also show
\begin{theorem}  
\label{thm:ResDomainPrime}
Suppose that $\Omega \Subset \re^d$ has a smooth boundary. Then $\lambda$ is a resonance of $-\Deltap$ if and only if the following has a solution $u=u_1\oplus u_2\in H^{3/2}_{\Delta}(\Omega)\oplus H^{3/2}_{\Delta,\loc}(\re^d\setminus\overline{\Omega})$
\begin{equation}
\label{eqn:mainPrimeTransmit}
\left\{\begin{aligned}(-\Delta-\lambda^2)u_1=0 &&(-\Delta -\lambda^2)u_2=0\\
\partial_\nu u_1|_{\partial\Omega}=\partial_\nu u_2|_{\partial\Omega}&&u_1-u_2+V\partial_\nu u_1=0&\text{ on }\pO\\
u_2\text{ is $\lambda$-outgoing.}
\end{aligned}\right.
\end{equation}
Furthermore, (except for in $d=1$ with $\lambda=0$ which is always a pole) $\lambda$ is a resonance if and only if there is a nonzero solution $\psi\in L^2(\pO)$ of 
$$(I-\dDl(\lambda)V)\psi=0.$$
\end{theorem}

Once we have defined the operator, we give the proof of Theorem \ref{thm:meromorphyPrime} and \ref{thm:ResDomainPrime}. This is done similarly to the analysis for $-\Deltad{\Gamma}$ in \cite{GS}, however, the process is complicated by the lower regularity of $\delta'_{\partial\Omega}.$ We start by showing that $I-V\dDl$ has a meromorphic inverse and by writing a formula for $R_V$ in terms of $I-V\dDl$ we obtain Theorem \ref{thm:meromorphyPrime}. 

\section{Formal definition of the operators}

\label{sec:formalDefinition}
\subsection{Definition of $-\Deltad{\pO}$}
We define the operator $-\Deltad{\pO}$ using the symmetric quadratic form, with dense domain $H^1(\re^d)\subset L^2(\re^d)$,
$$
Q_{V,\pO}(u,w):=\la \nabla u,\nabla w\ra_{L^2(\re^d)} + \la V\gamma u,\gamma w\ra_{L^2(\pO)}\,.
$$
where $\gamma:H^s(\re^d)\to H^{s-1/2}(\pO)$, $s>1/2$ denotes the restriction operator.

Using the Sobolev embedding and H\"older's inequality,  we can bound 
\begin{equation}
\label{eqn:restrictionBound}
\|\gamma u\|_{L^2(\pO)}\le C\,\|u\|^{\frac 12}_{L^2}\|u\|^{\frac 12}_{H^1}\le C\,\e\, \|u\|_{H^1}+C\,\e^{-1}\|u\|_{L^2}\,.
\end{equation}
It follows that there exist $c\,, C>0$ such that 
$$
|Q_{V,\pO}(u,w)|\leq C\,\|u\|_{H^1}\|w\|_{H^1}\quad\text{ and }\quad c\,\|u\|_{H^1}^2\leq Q_{V,\pO}(u,u)+C\|u\|_{L^2}^2\,.
$$
By Reed-Simon \cite[Theorem VIII.15]{RS}, $Q_{V,\pO}(u,w)$ is determined by a unique self-adjoint operator
$-\Deltad{\pO}$, with domain $\mc{D}_\delta$ consisting of $u\in H^1$ such that $Q_{V,\pO}(u,w)\le C\|w\|_{L^2}$ for all $w\in H^1(\re^d)$.
By Rellich's embedding lemma, the potential term is compact relative to $H^1$.
It follows by Weyl's essential spectrum theorem, see \cite[Theorem XIII.14]{RS2}, that $\sigma_{\ess}(-\Deltad{\pO})=[0,\infty)$.
Additionally, there are at most a finite number of eigenvalues in $(-\infty,0]$, each of finite multiplicity.

If $u\in \mc D_\delta$, by the Riesz representation theorem we then have $Q_{V,\pO}(u,w)=\la g,w\ra$ for some $g\in L^2(\re^d)$, and taking $w\in \Cc(\re^d)$ shows that in the sense of distributions 
\begin{equation}
\label{eqn:LaplaceDistributional}
-\Delta u+(V \gamma u)\delta_{\pO}=g\,.
\end{equation}
Conversely, if $u\in H^1(\re^d)$ and \eqref{eqn:LaplaceDistributional} holds for some $g\in L^2(\re^d)$, then by density of $\Cc\subset H^1$ we have $Q_{V,\pO}(u,w)=\la g,w\ra$ for $w\in H^1(\re^d)$, hence $u\in\mc D_\delta$, and $-\Deltad{\pO} u$ is given by the left hand side of \eqref{eqn:LaplaceDistributional}. We thus can define
$$
\|u\|_{\mc D_\delta} = \|u\|_{H^1}+ \|\Deltad{\pO}u\|_{L^2}\,,
$$
where finiteness of the second term carries the assumption that $\Deltad{\pO}u\in L^2$.

Suppose that $\chi\in \Cc(\re^d\setminus\pO)$ and that $u\in H^1(\re^d)$ solves \eqref{eqn:LaplaceDistributional}. Then, 
$$
\Delta (\chi u)= \chi g +2\la \nabla \chi, \nabla\ra u+(\Delta \chi)u\in L^2(\re^d)\,.
$$
Hence,
$$
\|\chi u\|_{H^2}\le C_\chi \|u\|_{\mc D}\,.
$$
That is, $\mc D_\delta\subset H^1(\re^d)\cap H^2_{\loc}(\re^d\setminus\pO)$, with continuous inclusion.

For general $\pO$ and $V$, elements of $\mc{D}$ may be more singular near $\pO$. However, we assume that $\partial\Omega$ is a $C^{\infty}$ hypersurface and that $V:H^{\frac 12}(\partial\Omega)\to H^{\frac 12}(\partial\Omega)$. Then since $u\in H^1(\re^d)$, and $\gamma:H^s(\re^d)\rightarrow H^{s-\hf}(\partial\Omega)$ for $s\in (\hf,2]$, we have $V\gamma u\in H^{\frac 12}(\partial\Omega)$. By
\eqref{eqn:LaplaceDistributional} we can write $u$ as $(-\Delta)^{-1}g$ plus the single layer potential of a $H^{\hf}(\partial\Omega)$ function, hence estimates such as \cite[Proposition 8.1]{GS}  (see also \cite[Theorems 9, 10]{Epstein}) show that
$$\mc D_\delta\subset \mc E_2=H^1(\re^d)\cap (H^2(\Omega)\oplus H^2(\re^d\setminus\overline{\Omega}))\,,
$$
with continuous inclusion.
We remark that
$H^2(\Omega)$ and $H^2(\re^d\setminus\overline{\Omega})$ can be identified as restrictions of $H^2(\re^d)$ functions; see \cite{Calderon} and \cite[Theorem VI.5]{Stein}. Thus, if $u\in \mc D$ then $u$ has a well defined trace on $\partial\Omega$ of regularity $H^{\frac 32}(\partial\Omega)$, and the first derivatives of $u$ have one-sided traces from the interior and exterior, of regularity $H^{\frac 12}(\partial\Omega)$.

For $w \in H^1(\re^d)$ and $u\in \mc E_2$,
it follows from Green's identities that
$$
Q_{V,\partial\Omega}(u,w)=\la -\Delta u,w\ra_{L^2(\re^d\setminus\pO}+
\la \partial_{\nu}u+\partial_{\nu'}u+V\gamma u,\gamma w\ra_{L^2(\partial\Omega)}\,,
$$
where $\partial_\nu$ and $\partial_{\nu'}$ denote the exterior normal derivatives from $\Omega$ and $\re^d\setminus\overline{\Omega}$.
Thus, in the case that $V$ is bounded from $H^{\frac 12}(\partial\Omega)\to H^{\frac 12}(\partial\Omega)$, we can completely characterize the domain $\mc D$ of the self-adjoint operator $-\Deltad{\partial\Omega}$ as
\begin{equation}\label{eqn:domainChar}
\mc D_\delta=\bigl\{u\in \mc E_2
\quad \text{such that}\quad \partial_{\nu}u+\partial_{\nu'}u+V\gamma u=0\,\bigr\}\,,
\end{equation}
in which case $\Deltad{\partial\Omega}u=\Delta u|_\Omega\oplus \Delta u|_{\re^d\setminus\overline\Omega}$.

\subsection{Definition of $-\Deltap$}
\label{sec:defDeltap}
We assume that $V:L^2(\pO)\to L^2(\pO)$ is self-adjoint and invertible and define $-\Deltap$ using the quadratic form 
$$Q_{V,\delta',\pO}(u,w)=\la \nabla u,\nabla w\ra_{\re^d\setminus\pO}  +\la V^{-1}(\gamma u_2-\gamma u_1),(\gamma w_1-\gamma w_2)\ra_{\pO}$$
with form domain $H^{1}(\re^d\setminus\pO)\cap L^2(\re^d)=:\mc{Q}$ (see also \cite{Berhndt}). Here $u_1=u|_{\Omega}$ and $u_2=u|_{\re^d\setminus\overline{\Omega}}.$

Using the Sobolev embedding and H\"older's inequality as in \eqref{eqn:restrictionBound} that 
\begin{align*} 
\|\gamma u_1\|_{L^2(\pO)}&\leq C\|u\|_{L^2(\re^d)}^{\hf}\|u\|_{H^1(\Omega)}^{\hf}\leq C\e \|u\|_{H^1(\Omega)}+C\e^{-1}\|u\|_{L^2(\re^d)}\\
\|\gamma u_2\|_{L^2(\pO)}&\leq C\|u\|_{L^2(\re^d)}^{\hf}\|u\|_{H^1(\re^d\setminus\overline{\Omega})}^{\hf}\leq C\e \|u\|_{H^1(\re^d\setminus\overline{\Omega})}+C\e^{-1}\|u\|_{L^2(\re^d)}
\end{align*}
So, there exist $c\,,\,C>0$ such that 
\begin{gather*} 
|Q_{V,\delta',\pO}(u,w)|\leq C\|u\|_{\mc{Q}}\|w\|_{\mc{Q}},\\
c\|u\|^2_{\mc{Q}}\leq Q_{V,\delta',\pO}(u,u)+C\|u\|_{L^2}^2.
\end{gather*}
By Reed-Simon \cite[Theorem VIII.15]{RS}, $Q_{V,\delta',\pO}(u,w)$ is determined by a unique self-adjoint operator
$-\Deltap$, with domain $\mc{D}_{\delta'}$ consisting of $u\in \mc{Q}$ such that $Q_{V,\delta',\pO}(u,w)\le C\|w\|_{L^2}$ for all $w\in \mc{Q}$.

By Rellich's embedding lemma, the potential term is compact relative to $\mc{Q}$.
It follows by Weyl's essential spectrum theorem, see \cite[Theorem XIII.14]{RS2}, that $\sigma_{\ess}(-\Deltap)=[0,\infty)$.
Additionally, there are at most a finite number of eigenvalues in $(-\infty,0]$, each of finite multiplicity.

Now, if $u\in \mc{D}_{\delta'}$, then by the Riesz representation theorem 
$$Q_{V,\delta',\pO}(u,w)=\la g,w\ra$$
 for some $g\in L^2(\re^d)$. Then, taking $w\in \Cc(\re^d)$ shows that in the sense of distributions
\begin{equation} \label{eqn:distPrime} -\Delta u -\delta_{\partial\Omega}'\otimes (u_1-u_2)=g.
\end{equation}
Conversely, if $u\in \mc{Q}$ and \eqref{eqn:distPrime} holds for $g\in L^2(\re^d)$, then by density of $\Cc\subset\mc{Q}$, $Q_{V,\delta',\pO}(u,w)=\la g,w\ra$ for $w\in\mc{Q}$ and hence $u\in \mc{D}_{\delta'}$ with $-\Deltap u=g.$ Then we can define
$$\|u\|_{\mc{D}_{\delta'}}=\|u\|_{\mc{Q}}+\|\Deltap u\|_{L^2}.$$
Moreover, by identical arguments to those above, $\mc{D}_{\delta'}\subset H^2_{\loc}(\re^d\setminus\pO).$ 

Now, suppose $u\in \mc{D}_{\delta'}$. Then, applying $(-\Delta)^{-1}$ to \eqref{eqn:distPrime} gives 
\begin{equation}\label{eqn:applyDistEqn} u=-\D(0)(u_1-u_2)+(-\Delta)^{-1}g\end{equation}
where $\D(0)$ denotes the double layer potential.
 Notice that \eqref{eqn:applyDistEqn} implies that 
$$-\Delta u(x)= g(x)\,,\,x\in \re^d\setminus \pO.$$
Hence, $u\in H^1_{\Delta}(\re^d\setminus \pO)\cap L^2(\re^d)$ where for $s<2$, and $U\subset \re^d$ open 
$$H^s_{\Delta}(U)=\{u\in H^s(U)\,:\,\Delta u\in L^2(U)\}$$
which has norm
$$\|u\|_{H^s_{\Delta}(U)}:=\|u\|_{H^s(U)}+\|\Delta u\|_{L^2(U)}.$$
That is, $\mc{D}_{\delta'}\subset H^{1}_{\Delta}(\re^d\setminus \pO)\cap L^2(\re^d).$

\begin{lemma}
Let $U\subset \re^d$ be open with smooth boundary. Then the map 
$$H^s_\Delta(U)\ni u\to \partial_\nu u|_{\partial U}\in H^{s-3/2}$$
can be extended continuously from the corresponding map $C^\infty(U)\to C^\infty(\partial U).$
\end{lemma}
\begin{proof}
Let $E:H^{s}(\partial U)\to H^{s+1/2}(U)$, $s\geq -1/2$ denote the extension operator. Then fix $w\in H^{s}(\partial U)$. Observe that for $v\in C^\infty(U)$, 
\begin{align*} 
|\la \partial_\nu v,w\ra_{\partial U}|&\leq |\la \nabla v,\nabla Ew\ra_{L^2(U)}|+\|\Delta  v\|_{L^2(U)}\|Ew\|_{L^2(U)}\\
&\leq C\|v\|_{H^{-s+\frac{3}{2}}(U)}\|Ew\|_{H^{s+\hf}(U)}+\|\Delta  v\|_{L^2(U)}\|Ew\|_{L^2(U)}\\
&\leq C\|v\|_{H^{-s+\frac{3}{2}}_{\Delta}(U)}\|Ew\|_{H^{s+\hf}(U)}\leq C\|v\|_{H^{-s+\frac{3}{2}}_{\Delta}(U)}\|w\|_{H^{s}(\partial U)}.
\end{align*}
Hence, since $w|_{\partial U}\in H^{s}(\partial U)$ is arbitrary, 
$$\|\partial_\nu v\|_{H^{-s}(U)}\leq C\|v\|_{H^{-s+3/2}_{\Delta}(U)}$$
and the lemma follows from density of $C^\infty(U).$ 
\end{proof}

Now,  by Lemma \ref{lem:LayerPotentialInverses} or rather its analog for $\lambda=0$ (see, e.g \cite[Proposition 7.11.4]{Taylor}) we can take the normal derivative of the right hand side of \eqref{eqn:applyDistEqn} from either inside or outside $\pO$ and the limits agree. Hence, 
$$\partial_{\nu}u_1=\partial_\nu u_2=-\partial_{\nu}\Dl(0)(u_1-u_2)+\partial_{\nu}(-\Delta)^{-1}g.$$ Moreover, $\dDl(0)$ is a homogeneous pseudodifferential operator of order 1 and hence maps $H^s(\pO)\to H^{s-1}(\pO)$ for all $s$. Therefore, for $u\in \mc{D}_{\delta'}$,
\begin{equation}\label{eqn:derEqual} 
\partial_{\nu}u_1=\partial_\nu u_2.
\end{equation}

\begin{lemma}
Let $U\subset \re^d$ be open with smooth boundary. Then for $u\in H^{1}_\Delta(U)$ and $w\in H^1(U)$, 
\begin{equation} 
\label{eqn:greens}\la \nabla u,\nabla w\ra_{U}=\la -\Delta u,w\ra_{U}+\la \partial_\nu u,w\ra_{\partial U}.
\end{equation}
where the last pairing is interpreted as the dual pairing of $H^{-\hf}(\partial U)$ and $H^{\hf}(\partial U).$
\end{lemma}
\begin{proof}
Let $u_n\in C^\infty(U)$ have $u_n\to u$ in $H^1_{\Delta}(U)$. Then, by the previous lemma and the definition of $H^1(U)$, 
$$-\Delta u_n\underset{L^2(U)}{\longrightarrow} -\Delta u\,,\quad u_n\underset{H^1(U)}{\longrightarrow} u\,,\quad \partial_\nu u_n\underset{H^{-1/2}(\partial U)}{\longrightarrowß} \partial_\nu u.$$
Therefore, 
\begin{align*}
\la \nabla u,\nabla w\ra_U&=\la \nabla (u-u_n),\nabla w\ra_U+\la \nabla u_n,\nabla w\ra_U\\
&=\la \nabla (u-u_n),\nabla w\ra_U+\la -\Delta u_n,w\ra_U+\la \partial_\nu u_n, w\ra_{\partial U}\\
&\to \la -\Delta u,w\ra_U+\la \partial_\nu u,w\ra_{\partial U}
\end{align*}
where we have used the fact that $w|_{\partial U}\in H^{1/2}.$
\end{proof}

Now, let $u\in \mc{D}_{\delta'}$ and $w\in \mc{Q}$. Then using \eqref{eqn:derEqual} and \eqref{eqn:greens} we have that for some $g\in L^2(\re^d)$, 
\begin{align*} \la g,w\ra_{L^2(\re^d)}&=Q_{V,\delta',\pO}(u,w)\\
&=\la -\Delta u,w\ra_{L^2(\re^d\setminus \pO)} +\la \partial_\nu u_1+V^{-1}(\gamma u_1-\gamma u_2),\gamma w_1-\gamma w_2\ra_{\pO} 
\end{align*}
Thus, for $u\in \mc{D}_{\delta'}$,
$$\partial_\nu u_1+V^{-1}(\gamma u_1-\gamma u_2)=0$$
and hence $\partial_\nu u_1=\partial_\nu u_2\in L^2(\pO)$.  Now, let $\mc{D}_{N_j}(i)$, $j=1,2$ denote the Neumann to Dirichlet map at $\lambda=i$ with $j=1$ and $2$ corresponding to $\Omega$ and $\re^d\setminus \overline\Omega$ respectively.  Then, $\mc{D}_{N_j}(i):L^2(\pO)\to H^1(\pO)$. and hence there exists $v\in L^2(\re^d)\cap H_{\Delta}^{3/2}(\re^d\setminus \pO)$ solving 
$$(-\Delta +1)v=0,\quad \text{in }\re^d\setminus \pO\quad \partial_\nu v_1=\partial_\nu v_2\,,\,v_j|_{\pO}=\mc{D}_{N_j}(i)\partial_{\nu_j}u_j.$$
So, $u-v$ solves
$$(-\Delta)(u-v)= v-\Delta u\quad\text{ in }\re^d\setminus \pO\,\quad \partial_{\nu_j}(u_j-v_j)=0$$
and hence $u-v\in H^2(\re^d\setminus \pO)$. Together, this implies that $u\in H_{\Delta}^{3/2}(\re^d\setminus \pO).$ 
Thus, we have  
$$\mc{D}_{\delta'}\subset \mc{E}_{\delta'}:=\left\{ \begin{gathered}u\in H^{\frac{3}{2}}_{\Delta}(\re^d\setminus \pO)\cap L^2(\re^d)\,:\\
 \partial_\nu u_1=\partial_\nu u_2\,,\,V\partial_\nu u_1+\gamma u_1-\gamma u_2=0\end{gathered}\right\}.$$

Now, let $u\in \mc{E}_{\delta'}$. Then for $w\in \mc{Q}$
\begin{align*} Q_{V,\delta',\pO}(u,w)&=\la \nabla u,\nabla w\ra_{\re^d\setminus \pO}+\la V^{-1}(\gamma u_2-\gamma u_1),\gamma w_1- \gamma w_2\ra\\
&=\la -\Delta u,w\ra_{\re^d\setminus \pO} +\la \partial_\nu u_1+V^{-1}(\gamma u_1-\gamma u_2),\gamma w_1-\gamma w_2\ra \\
&=\la -\Delta u,w \ra_{\re^d\setminus \pO}. 
\end{align*}
Hence, $u\in \mc{D}_{\delta'}$ i.e. $\mc{E}_{\delta'}\subset \mc{D}_{\delta'}$
and we have fully characterized the domain $\mc{D}_{\delta'}$ of $-\Deltap$ as
\begin{equation}
\label{eqn:domainDeltaPrime} 
\mc{D}_{\delta'}=\{ u\in H^{\frac{3}{2}}_{\Delta}(\re^d\setminus \pO)\cap L^2(\re^d)\,:\, \partial_\nu u_1=\partial_\nu u_2\,,\,V\partial_\nu u_1+\gamma u_1-\gamma u_2=0\}.
\end{equation}
As above, for $u\in \mc{D}_{\delta'}$, $\Deltap u=\Delta u|_{\Omega}\oplus \Delta u|_{\re^d\setminus\overline{\Omega}}.$
Furthermore, using this in \eqref{eqn:distPrime}, as a distribution
$$-\Delta u+\delta_{\pO}'\otimes (V\partial_\nu u)=g.$$
Thus, if $u$ is $\lambda$-outgoing and solves $(-\Deltap-\lambda^2)u=0$, then 
$$-\Delta u+\delta_{\pO}'\otimes(V\partial_\nu u)-\lambda^2u=0$$
and hence, applying $R_0(\lambda)$ on the left,
\begin{equation}
\label{eqn:bveqn}
\begin{aligned} 
u+R_0(\lambda)\delta_{\pO}'\otimes (V\partial_\nu u)&=0\\
u-\D V\partial_\nu u&=0\\
(I-\dDl V)\partial_\nu u&=0
\end{aligned} 
\end{equation}

\section{Meromorphic continuation of the resolvent}
\label{sec:Meromorphy}
We assume that $V$ is self adjoint and invertible. Hence,
\begin{equation} \label{eqn:Vassume} \|u\|_{L^2(\partial\Omega)}\leq C\|Vu\|_{L^2(\partial\Omega)}.\end{equation}
\begin{remark}
We have in mind the situation that $V$ is a self-adjoint pseudodifferential operator that is elliptic. 
\end{remark}

\noindent Under these assumptions, we show that 
$(I-V\dDl (\lambda))^{-1}$ is a meromorphic family of Fredholm operators on the domain of $R_0(\lambda)$.

We start by analyzing $I-V\dDl (\lambda)$. Observe that by Proposition \ref{prop:layerInverse} $\dDl ^{-1}:H^{s}(\partial\Omega)\to H^{s+1}(\partial\Omega)$ exists as a meromorphic family of operators and we can write
$$(I-V\dDl (\lambda))=(\dDl ^{-1}(\lambda)-V)\dDl (\lambda).$$
 Hence, since $\dDl ^{-1}:H^s\to H^{s+1}$, it is compact on $L^2$ and $-V+\dDl ^{-1}$ is Fredholm by \eqref{eqn:Vassume}.

To conclude that $(-V+\dDl ^{-1}(\lambda))^{-1}$ is a meromorphic family of operators, we need only show that there exists $\lambda_0$ with $\Im \lambda_0>0$ such that $(-V+\dDl (\lambda)^{-1})^{-1}$ exists. 

To see this, let $\lambda=ih^{-1}$. Then, $\dDl (ih^{-1})$ is a semiclassical pseudodifferential operator with small parameter $h$ and symbol
\begin{equation}\label{eqn:symbdDlEll}\sigma(\dDl )=-h^{-1}\frac{\sqrt{1+|\xi'|^2_g}}{2}.
\end{equation}
Since we work in the physical half plane we can see that $R_0(ih^{-1})\in h^{-2}\Ph{-2}{}$ elliptic with symbol 
$$\sigma(R_0(ih^{-1}))=\frac{h^2}{|\xi|^2+1}$$
and $R_0=\oph(a)$ where $a$ has terms depending polynomially on $\xi$.

\begin{remark}
In fact, since we work on $\re^d$ 
$$R_0(ih^{-1})=\weyl(h^2(|\xi|^2+1)^{-1}).$$
\end{remark}
\noindent Thus, \eqref{eqn:symbdDlEll} follows from Lemma \ref{lem:diagPiece}.

Using, \eqref{eqn:symbdDlEll}, we have
$\|\dDl ^{-1}u\|_{L^2}\leq Ch\|u\|_{L^2}$
and for $h$ small enough, \eqref{eqn:Vassume} implies
\begin{equation}
\label{eqn:Fredholm}
\|u\|_{L^2}\leq C\|(-V+\dDl ^{-1})u\|_{L^2}.
\end{equation}
Together with the Fredholm property and the Analytic Fredholm Theorem (see for example \cite[Appendix C]{ZwScat}), this implies that for $h$ sufficiently small, $(-V+\dDl ^{-1}(e^{3\pi i/4}h^{-1}))$ is invertible and hence that $(-V+\dDl ^{-1}(\lambda))$ is a meromorphic family of Fredholm operators on $L^2$. Putting this together with the meromorphy of $\dDl (\lambda)^{-1}:H^s\to H^{s+1}$, this implies that 
$$(I-V\dDl (\lambda))^{-1}=\dDl (\lambda)^{-1}(-V+\dDl (\lambda)^{-1})^{-1}:L^2\to H^{1}$$
is a meromorphic family of operators.

We now prove the meromorphy of $R_V(\lambda)$. Let $L$ be a vector field with $L|_{\pO}=\partial_\nu$. Then let $K=-L^*\gamma^*V\gamma LR_0(\lambda)$ and suppose $\rho \in \Cc(\re^d)$ with $\rho\equiv 1 $ on $\Omega$. We have $\dDl (\lambda)=\gamma_1R_0(\lambda)\gamma_1^*$ where $\gamma_1=\gamma L$. 

\begin{remark}
Notice that in this formula for $\dDl$, one must interpret $\gamma L$ as either $\gamma^+L$ or $\gamma^-L$ where $\gamma^{\pm}$ denote restriction as a limit from either inside or outside of $\pO$. However, in the case of $\dDl$, it does not matter how one chooses from $\pm$. (See Proposition \ref{prop:layerInverse})
\end{remark}

Then,
\begin{equation}\label{eqn:relatedDlK} (I+K(\lambda)\rho)^{-1}\gamma_1^*=\gamma_1^*(I-V\dDl (\lambda))^{-1}.
\end{equation}
So, 
\begin{align*} (I+K(\lambda)\rho)^{-1}&=I-(I+K(\lambda)\rho)^{-1}K(\lambda)\rho\\
&=I+\gamma_1^*(I-V\dDl (\lambda))^{-1}V\gamma_1R_0(\lambda)\rho.
\end{align*}
Hence, 
$$K(\lambda)(1-\rho)(I+K(\lambda)\rho)^{-1}=K(\lambda)(1-\rho)$$
and we have that 
\begin{align*}
R_V(\lambda)&=R_0(\lambda)(I+K(\lambda)\rho)^{-1}(I-K(\lambda)(1-\rho))\\
&\!\!\!\!=\left(R_0(\lambda)+R_0(\lambda)\gamma_1^*(I-V\dDl (\lambda))^{-1}V\gamma_1R_0(\lambda)\rho\right)(I-K(\lambda)(1-\rho))
\end{align*}
and hence for $g\in L^2_{\comp}$ we can take $\rho g=g$ to obtain
\begin{equation}\label{eqn:resolveForm}R_V(\lambda)g=R_0(\lambda)g+R_0(\lambda)\gamma_1^*(I-V\dDl (\lambda))^{-1}V\gamma_1R_0(\lambda)g.
\end{equation}
Thus, the meromorphy of $(I+V\dDl (\lambda))^{-1}$ from $L^2 \to H^{1}$ implies that 
$$R_V(\lambda):L^2_{\comp}\to H^{1/2-\e}_{\loc}$$ 
is meromorphic with $\lambda$-outgoing image for $\lambda$ in the domain of $R_0(\lambda)$.

To see that in $d=1$, we have meromorphy of the resolvent at $\lambda=0$, we observe that in that case $I-V\dDl (\lambda)$ is a matrix valued meromorphic function and hence $\det(I-V\dDl (\lambda))$ is a meromorphic function. Together with the invertibility of $I-V\dDl (\lambda)$ for some $\lambda_0$ with $\Im \lambda_0>0$, we have that $\det((I-V\dDl (\lambda))^{-1})$ is a meromorphic family of operators on $\mathbb{C}$. Hence, the meromorphy of $R_0(\lambda)$ at $0$ together with \eqref{eqn:resolveForm} implies the meromorphy of $R_V(\lambda)$ at 0. Moreover, since $\gamma_1^*$ has finite dimensional range, the singular terms of $R_V(\lambda)$ also have this property. 

We use a version of Rellich uniqueness theorem to show that resonances occur at $\lambda$ such \eqref{eqn:mainPrimeTransmit}
has a solution.
\begin{lemma}
\label{lem:rellich}
If $\lambda$ lies in the domain of $R_0(\lambda)$, then a global $\lambda$-outgoing solution to $(-\Delta-\lambda^2)u=0$ vanishes identically. 
\end{lemma}
\begin{lemma}
For $\lambda$ in the domain of $R_0(\lambda)$, $\lambda$-outgoing solutions $u\in H_{\Delta,\loc}^{3/2}(\re^d\setminus\partial\Omega).$ to \eqref{eqn:mainPrime} lie in 1-1 correspondence with solutions $f\in L^2$ of $(I-V\dDl (\lambda))f=0$ given by $u=R_0(\lambda)(\gamma_1^*f)$ and $f=V\gamma_1u.$
\end{lemma}
\begin{proof}
Let $f\in L^2$ with $(I-V\dDl (\lambda))f=0$ and define $u=R_0(\gamma_1^*f)$. Then, by \cite[Section 7.11]{Taylor}, together with $\eqref{eqn:Vassume}$, $(-\Delta-\lambda^2)u=0$ in $\Omega_i$ and $\partial_\nu u|_{\partial\Omega}\in L^2$. Now, let $\mc{D}_{N_i}(i)$ denote the Neumann to Dirichlet map at $\lambda=i$. Then $\mc{D}_{N_i}(i):L^2\to H^1$ and hence, there exists $u_1\in L^2\cap H_{\Delta,\loc}^{3/2}(\re^d\setminus \partial\Omega)$ solving,
$$(-\Delta+1)u=0\,,\quad \text{in }\Omega_i\quad \partial_\nu u_1|_{\partial\Omega}=\partial_\nu u|_{\partial\Omega}\,,\, u|_{\partial\Omega_i}=\mc{D}_{N_i}(i)\partial_{\nu_i} u|_{\partial\Omega}.$$
So, $u-u_1$ solves
$$(-\Delta-\lambda^2)(u-u_1)=-(-\Delta-\lambda^2)u_1\in L^2(\Omega_i)\,,\quad \partial_{\nu_i}(u-u_1)=0$$
and hence $u-u_1\in H_{\loc}^2(\re^d\setminus \partial\Omega).$ Together, this implies $u\in H_{\Delta,\loc}^{3/2}(\re^d\setminus\partial\Omega)$.  By definition $u$ is $\lambda$-outgoing and
$(-\Delta-\lambda^2)u=\gamma_1^*f.$
Now, 
$$\gamma_1^*V\gamma_1u=\gamma_1^*V\dDl (\lambda)f=\gamma_1^*f$$
so $u$ solves \eqref{eqn:mainPrime}.

Now, suppose that $u\in H_{\Delta,\loc}^{3/2}$ solves \eqref{eqn:mainPrime}. Then, by Lemma \ref{lem:rellich}, 
$$\gamma_1 u=\gamma_1 R_0(\lambda)(\gamma_1^*V\gamma_1u)=\dDl (\lambda)V\gamma_1u.$$
So, letting $f=V\gamma_1u$, we have $f-V\dDl (\lambda)f=0$ and $f\in L^2(\partial\Omega).$ 
\end{proof}
\begin{remark}
Notice that the kernel of $I-V\dDl (\lambda)$ is in 1-1 correspondence with that of $I-\dDl (\lambda)V$. In particular, suppose 
\begin{equation}
\label{eqn:boundaryPrelimPrime1}
(I-\dDl (\lambda)V)f=0.
\end{equation}
Then, letting $g=Vf$, $(I-V\dDl (\lambda))g=0.$ Next, suppose that $(I-V\dDl (\lambda))f=0$. Then by \eqref{eqn:Vassume}, $f=V\dDl f=Vg$ for some $g\in L^2$. Hence, $(I-\dDl (\lambda)V)g=\dDl (\lambda)(I-V\dDl (\lambda))f=0.$ 
\end{remark}

The case $d=1$ and $\lambda=0$ must be treated differently. In particular, we need to show that $R_V(\lambda)$ is singular at $\lambda=0$ if and only if there is a nontrivial solution $u\in H_{\Delta,loc}^{3/2}(\re\setminus \partial\Omega)\cap L^\infty(\re)$ to $\Delta u=-\gamma_1^*V\gamma_1 u$ since $u\in L^\infty(\re)$ is equivalent to $0$ outgoing for such $u$. In fact, it is easy to see that $u\equiv a$ is a nontrivial $0$-outgoing solution to  $\Delta u=-\gamma_1^*V\gamma_1 u$. Moreover, let $\partial\Omega=\{x_1,\dots x_m\,:\,x_1<x_2<\dots <x_m\}\subset \re$, (see Figure \ref{fig:omega1-d}) and $V\gamma_1 u=(c_1,\dots c_m)\in \complex^m$.  Then if $\Delta u=-\gamma_1^*V\gamma_1u$, 
$$u=\frac{1}{2}\sum_{j=1}^m(-1)^j\sgn(x-x_j)+ax+b.$$
Now, $u\in L^\infty$ implies $a=0$. Thus, $\gamma_1 u=0$ and hence $c_j=0$ for $j=1,\dots m$. That is, $u=b$ is a constant. $-\Delta -\gamma_1^*V\gamma_1$ has a 1 dimensional $0$-outgoing set of solutions. Thus, it remains to show that $R_V(\lambda)$ always has a simple pole at $\lambda=0$.

\begin{figure}
\centering
\begin{tikzpicture}
\draw[very thin,<->] (-5,0) -- (5,0) ;

\foreach \Point in {(-3,0), (0,0), (2.5,0)}{
    \node at \Point {(};
    
}
\foreach \Point in {(-2,0), (1,0), (3.7,0)}{
    \node at \Point {)};
}
\draw[very thick, ->] (-3,0) -- (-3.5,0) ;
\draw[very thick, ->] (0,0) -- (-0.5,0) ;
\draw[very thick, ->] (2.5,0) -- (2,0) ;
\draw[very thick, ->] (-2,0)-- (-1.5,0) ;
\draw[very thick, ->] (1,0)-- (1.5,0) ;
\draw[very thick, ->] (3.7,0)-- (4.2,0) ;
\node at (-3,-0.5) {$x_1$};
\node at (-2.5,0.5) {$\Omega_1$};
\node at (0,-0.5) {$x_3$};
\node at (.5,0.5) {$\Omega_2$};
\node at (2.5,-0.5) {$x_5$};
\node at (3.1,0.5) {$\Omega_3$};
\node at (-2,-.5) {$x_2$};
\node at (1,-.5) {$x_4$};
\node at (3.7,-.5) {$x_6$};
\node at (5.4,-.5) {$\re$};
\end{tikzpicture}
\caption[Picture of $\Omega\subset \re$ where $\Omega=\bigcup_i\Omega_i$]{Picture of $\Omega\subset \re$ where $\Omega=\bigcup_i\Omega_i$. The arrows show the direction of the normal vector at each point of $\partial\Omega$\label{fig:omega1-d}}
\end{figure}
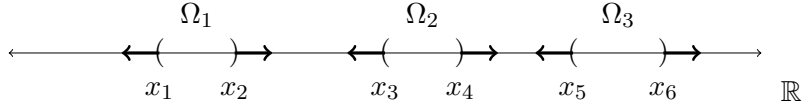

Again, assume $\partial\Omega=\{x_1,\dots x_m\,:\,x_1<x_2<\dots <x_m\}\subset \re$, and $V\gamma_1 u=(c_1,\dots c_m)\in \complex^m$.
Then,
$$(\dDl)_{ij}=\frac{i\lambda}{2}(-1)^{i+j}e^{i\lambda|x_i-x_j|}=\O{}(\lambda).$$
Thus, $(I-V\dDl(\lambda))^{-1}$ has a power series representation at $\lambda=0$ and hence is holomorphic there. In addition, 
\begin{gather*} 
(\gamma_1R_0(\lambda))_j=(-1)^{j+1}\frac{x_j-y}{2|x_j-y|}e^{i\lambda|x_j-y|}\,,\\
 (R_0(\lambda)\gamma_1^*)_j=(-1)^{j}\frac{x_j-y}{2|x_j-y|}e^{i\lambda|x_j-y|}
 \end{gather*}
are both holomorphic at $\lambda=0$. Thus, by \eqref{eqn:resolveForm} $R_V(\lambda)$ has a simple pole at $\lambda=0$ coming from the pole of $R_0(\lambda).$

We now relate the existence of resonances to the solution of the transmission problem \eqref{eqn:mainPrimeTransmit}.
\begin{proof}[{Proof of Theorem} \ref{thm:ResDomainPrime}]
Suppose now that $\Gamma=\partial\Omega$ for a compact domain $\Omega\subset\re^d$ with $C^\infty$ boundary. Then the analysis leading to \eqref{eqn:domainDeltaPrime} shows that a $\lambda$-outgoing solution of \eqref{eqn:mainPrime} with $u\in H^{3/2}_{\Delta,\loc}$  belongs to $\mc E_{\delta',\loc}$ and satisfies the transmission problem \eqref{eqn:mainPrimeTransmit}. Conversely, suppose $u\in\mc E_{\delta'',\loc}$ satisfies \eqref{eqn:mainPrimeTransmit}. For $w\in C_c^\infty(\re^d)$, Green's identities yield
$$
\int_{\re^d} u\,(-\Delta-\lambda^2)w=-\int_{\partial\Omega}(u_1-u_2)\,\partial_\nu w=
\int_{\partial\Omega}(V\partial_\nu u)\, \partial_\nu w\,.
$$
Hence $u$ is a $\lambda$-outgoing $H^{3/2}_{\Delta,\loc}$ distributional solution to 
$(-\Delta-\lambda^2)u+\delta'_{\partial\Omega}\otimes(V\partial_\nu))u=0$, and by the above $\lambda$ is a resonance.
\end{proof}

\begin{remark}
The proof of meromorphic continuation of the resolvent for $-\Deltad{\pO}$ and equivalence of \eqref{eqn:outgoingSoln} and \eqref{eqn:transmitnoPrime} are similar and can be found in \cite{GS}.
\end{remark}


\chapter{Boundary Layer Operators}
\label{ch:layer}
In Chapter \ref{ch:mer}, the existence of resonances for $-\Deltad{\Gamma}$ and $-\Deltap$ was related to a certain equation involving boundary layer operators of the Helmholtz equation. In this chapter we prepare for the analysis of these equations by understanding the boundary layer potentials from a semiclassical point of view.  We first review some of the standard theory of boundary layer potentials. We then proceed to prove (nearly) sharp high frequency estimates on layer potentials using $L^2$ estimates on restrictions of quasimodes and their derivatives to hypersurfaces. We then give a microlocal description of the boundary layer operators for domains with smooth boundary away from glancing. In the process, we give a description of the free resolvent as a semiclassical intersecting Lagrangian distribution. Finally, in the case that the domain is strictly convex, we use the Melrose-Taylor parametrix from Appendix \ref{ch:semiclassicalDirichletParametrices} to produce a microlocal model of the boundary layer operators near glancing. As a consequence of these microlocal models we improve the nearly sharp estimates on these operators to sharp estimates in the case that the domain has smooth, strictly convex boundary.
\section{Classical layer potential theory}
\label{sec:LayerPotential}
We review here some facts about boundary layer potentials in the context of the Helmholtz equation. We start by considering $\Im \lambda >0$. Then, 
$$(-\Delta_x-\lambda^2)R_0(\lambda)(x,y)=\delta_y(x).$$
Moreover, the equality continues analytically through $\Re\lambda\geq 0$ to $\mathbb{C}$ in the case that $d$ is odd and to the logarithmic cover of $\mathbb{C}\setminus\{0\}$ if $d$ is even. 

For $x\notin \pO$, let
\begin{equation*}
\S f(x):=\int_{\partial\Omega}R_0(\lambda,x,y)dS(y)\,,\quad\quad \D f(x):=\int_{\partial\Omega}\partial_{\nu_y}R_0(\lambda,x,y)f(y)dS(y)
\end{equation*}
be respectively the single and double layer potential.
We prove the following lemma similar to \cite[Propositions 11.1, 11.2]{Taylor}
\begin{lemma}
\label{lem:layer}
Let $\Omega\Subset\re^d$ be open with smooth boundary. For $x\in \Omega$, let $v_{+}(x)$ and $v_-(x)$ denote limits respectively from $x\in \Omega$ and $x\in \re^d\setminus \overline{\Omega}.$ Then for $x\in \Omega$,
\begin{gather*} 
(\S f)_{\pm}(x)=Gf(x),\quad \quad\quad\quad
(\D f)_{\pm}(x)=\mp \frac{1}{2}f(x)+\Dl f(x)\\
(\partial_{\nu_x}\S f)_{\pm}(x)=\pm \frac{1}{2}f(x)+\Dl^{\sharp}f(x)
\end{gather*}
where for $x\in \partial\Omega$,
\begin{gather*}
Gf(x):=\int_{\partial\Omega} R_0(\lambda)(x,y)f(y)dS(y)\\
\Dl f(x):=\int_{\partial\Omega}\partial_{\nu_y}R_0(\lambda)(x,y)f(y)dS(y)\\
\Dl^{\sharp}f(x):=\int_{\partial\Omega}\partial_{\nu_x}R_0(\lambda)(x,y)f(y)dS(y)
\end{gather*}
and $\partial_{\nu_x}$ denotes the outward unit normal derivative to $\partial\Omega$ at $x$.
\end{lemma}
\noindent We call $G$ the \emph{single layer operator} and $\Dl$ the \emph{double layer operator}.
\begin{proof}
We start by considering a general pseudodifferential operator $P(x,D)$. Let $S\in \mc{E}'(\re^d)$ denote the surface measure on $\partial\Omega$ and make a local change of coordinates so that $\partial\Omega=\{x_1=0\}$ with $\Omega\cap U=\{x_1<0\}\cap V$. Then, for $f\in \mc{D}'(\partial\Omega)$, letting $x=(x_1,x')$ and $y=(y_1,y')$
\begin{align*}
P(x,D)(fS)&=(2\pi )^{-d}\iint e^{i\la x'-y',\xi'\ra+ix_1\xi_1 }p(x,\xi',\xi_1)f(y')dy'd\xi'd\xi_1\\
&=q(x_1,x',D')f
\end{align*}
where 
\begin{equation}\label{eqn:boundaryPseud} q(x_1,x',\xi')=(2\pi )^{-1}\int e^{ix_1\xi_1}p(x_1,x',\xi)d\xi_1.\end{equation}
Now, suppose that $p\in \Scl{m}{0}$. Then, for $m<-1$, \eqref{eqn:boundaryPseud} is absolutely integrable and hence continuous at $x_1=0$. On the other hand, if $m\geq-1$, we can write
$$p\sim \sum_{j=-\infty}^{m}C^j_{\pm}(x,\xi') \xi_1^j\quad \pm \xi_1\to \infty. $$
Then $q$ is smooth away from $x_1=0$ and, if $C^j_{-}(x,\xi')=(-1)^jC^j_{+}(x,\xi')$ for $j\geq -1$, there is a jump discontinuity at $x_1=0$ (see for example \cite[Chapter 3]{Taylor} or Lemma \ref{lem:diagPiece}). 

Now we apply this to $\S$ and $\D$. Note that the (homogeneous) symbol of $R_0(\lambda)$ is $|\xi|^{-2}$ so we immediately obtain that there is no jump for $\S$. 

On the other hand, let $L$ be a vector field equal to $\partial_\nu$ on $\partial\Omega$. Then, 
$$\D f=R_0(\lambda)L^*(fS),\quad \quad \partial_\nu \S f=LR_0(\lambda)(fS)$$
where $L^*=-L-(\div L)$. So, the symbol of $R_0L^*$ is $-|\xi|^{-2}i\la \nu(x),\xi\ra$ and that of $LR_0$ is $|\xi|^{-2}i\la\nu(x),\xi\ra.$ Then, writing
$$|\xi\pm \tau \nu(x)|^{-2}i\la \nu(x),\xi\pm\tau \nu(x)\ra$$
we see that $\pm C_{\pm}(x,\xi')^{-1}\equiv i$. Computing the integral \eqref{eqn:boundaryPseud} with $p=-|\xi|^2i\xi_1$ gives the constant $\mp \frac{1}{2}$ for $\D$ and, since the symbols are related by multiplication by $-1$, $\pm\frac{1}{2}$ for $\partial_\nu \S$.
\end{proof}

Now, suppose that $\Im \lambda > 0$ and that $u$ solves
\begin{equation}
\label{eqn:LaplaceIn}
(-\Delta-\lambda^2)u(x)=0\quad \quad x\in \Omega.
\end{equation}
Then using Green's formula and the fact that $R_0(\lambda)(x,y)=R_0(\lambda)(x,y)$,
\begin{equation}
\label{eqn:boundaryLayerFormulaIn}
\S \partial_{\nu_i}u|_{\partial\Omega}-\D u|_{\partial\Omega}=\begin{cases}u(x)&x\in \Omega\\
0&x\notin \overline{\Omega}
\end{cases}
\end{equation}
So, taking limits from inside and outside $\Omega$ in \eqref{eqn:boundaryLayerFormulaIn}, we have 
$$G\partial_{\nu_i}u+\frac{1}{2}u-\Dl u=u\quad \quad G\partial_{\nu_i}u-\frac{1}{2}u-\Dl u=0.$$
That is, 
\begin{equation}\label{eqn:sLayerIn} G\partial_{\nu_i}u=\frac{1}{2}u+\Dl u.
\end{equation}
Next, apply $\partial_{\nu_i} $ to \eqref{eqn:boundaryLayerFormulaIn} and take limits from inside and outside $\Omega$ to obtain
$$\frac{1}{2}\partial_{\nu_i}u+\Dl ^\sharp \partial_{\nu_i}u-(\partial_{\nu_i}\mc{D}u)_+=\partial_{\nu_i}u\quad \quad -\frac{1}{2}\partial_{\nu_i}u+\Dl ^\sharp \partial_{\nu_i}u-(\partial_{\nu_i}\mc{D}u)_-=0.$$
That is, 
\begin{equation}
\label{eqn:dDLIn}
(\partial_{\nu_i}\mc{D}u)_\pm =-\frac{1}{2}\partial_{\nu_i}u+\Dl ^\sharp \partial_{\nu_i}u.
\end{equation}
On the other hand, suppose that $u$ solves 
\begin{equation}\label{eqn:LaplaceOut} (-\Delta -\lambda^2)u(x)=0\quad \quad x\notin \overline{\Omega}\quad\quad u\,\text{ is } \lambda\text{-outgoing}.
\end{equation}
Then, using Green's formula and the fact that $R_0(\lambda)(x,y)=R_0(\lambda)(x,y)$,
\begin{equation}
\label{eqn:boundaryLayerFormulaOut}
\mc{S} \partial_{\nu_e}u|_{\partial\Omega}+\mc{D}u|_{\partial\Omega}=\begin{cases}0&x\in \Omega\\
u(x)&x\notin \overline{\Omega}
\end{cases}
\end{equation}
So, taking limits from inside and outside $\Omega$ in \eqref{eqn:boundaryLayerFormulaOut}, we have 
$$G\partial_{\nu_e}u-\frac{1}{2}u+\Dl u=0\quad \quad G\partial_{\nu_e}u+\frac{1}{2}u+\Dl u=u.$$
That is, 
\begin{equation}
\label{eqn:sLayerOut}
G\partial_{\nu_e}u=\frac{1}{2}u-\Dl u.
\end{equation}
Next, apply $\partial_{\nu_i} $ to \eqref{eqn:boundaryLayerFormulaOut} and take limits from inside and outside $\Omega$ to obtain
$$\frac{1}{2}\partial_{\nu_e}u+\Dl ^\sharp \partial_{\nu_e}u+(\partial_{\nu_i}\mc{D}u)_+=0\quad \quad -\frac{1}{2}\partial_{\nu_e}u+\Dl ^\sharp \partial_{\nu_e}u+(\partial_{\nu_i}\mc{D}u)_-=\partial_{\nu_i}u.$$
That is, 
\begin{equation}
\label{eqn:dDLOut}(\partial_{\nu_i}\mc{D}u)_\pm =-\frac{1}{2}\partial_{\nu_e}u-\Dl ^\sharp \partial_{\nu_e}u.
\end{equation}

Now, let $f\in C^\infty(\partial\Omega)$ and $u_i$ be the unique solution to  \eqref{eqn:LaplaceIn} with $u_i|_{\partial\Omega}=f$. Then the \emph{interior Dirichlet to Neumann Map} is given by $\mc{N}_i:f\mapsto \partial_{\nu_i}u_i$. If $u_e$ solves \eqref{eqn:LaplaceOut} with $u_e|_{\partial\Omega}=f$, then the \emph{exterior Dirichlet to Neumann Map} is given by $\mc{N}_e:f\mapsto \partial_{\nu_e}u_e.$ 

Next, suppose that $v_i$ is the unique solution to \eqref{eqn:LaplaceIn} with $\partial_{\nu_i}v_i=f$. Then the \emph{interior Neumann to Dirichlet Map} is given by $\mc{D}_{N_i}:f\mapsto v_i|_{\partial\Omega}$. Finally, suppose that $v_e$ solves \eqref{eqn:LaplaceOut} with $\partial_{\nu_e}v_e=f$. Then, the \emph{exteriror Neumann to Dirichlet Map } is given by $\mc{D}_{N_e}:f\mapsto v_e|_{\partial\Omega}. $ 

Then \eqref{eqn:sLayerIn} \eqref{eqn:dDLIn} \eqref{eqn:sLayerOut} and \eqref{eqn:dDLOut} combined with density of $C^\infty$ in distributions give the following
\begin{lemma}
\label{lem:LayerPotentialInverses}
Let $G$, $\Dl $, and $\Dl ^\sharp$ be as in Lemma \ref{lem:layer}. Then for $\Im \lambda>0$,
$$G\mc{N}_i=\frac{1}{2}I+\Dl \quad \quad G\mc{N}_e=\frac{1}{2}I-\Dl .$$
Moreover, $\partial_{\nu_i}\mc{D}$ has no jump across $\partial\Omega$ and 
$$\dDl =(\partial_{\nu_i}\mc{D})_{\pm}=\left(-\frac{1}{2}I+\Dl ^\sharp\right)\mc{N}_i=\left(-\frac{1}{2}I-\Dl ^\sharp\right)\mc{N}_e $$
where 
$$\dDl(\lambda)f(x)=\int_{\partial\Omega}\partial_{\nu_x}\partial_{\nu_y}R_0(\lambda)(x,y)f(y)dS(y).$$
Finally, 
$$\dDl \mc{D}_{N_i}=-\frac{1}{2}I+\Dl ^\sharp,\quad\quad \dDl\mc{D}_{N_e}=-\frac{1}{2}I-\Dl ^\sharp.$$
\end{lemma}
\noindent We call $\dDl$ the \emph{derivative double layer operator}.

Now let $\Im \lambda>0$ and fix $h\in C^\infty(\partial\Omega)$ and suppose that $u(x)=\mc{S}h$. Then $u|_{\partial\Omega}=Gh$ and hence 
$\partial_{\nu_i}u=\mc{N}_iGh$, $\partial_{\nu_e}u=\mc{N}_eGh.$ On the other hand, taking limits from inside and outside $\Omega$ and using Lemma \ref{lem:layer}, we have 
$$\partial_{\nu_i}u=\left(\frac{1}{2}I+\Dl^\sharp\right)h\quad \partial_{\nu_e}u=\left(\frac{1}{2}I-\Dl^\sharp\right)h.$$

Similarly, if we let $u(x)=\mc{D}h$. Then, $\partial_{\nu_i}u=\dDl h$ and $(u)_+=\mc{D}_{N_i}\dDl h$, $(u)_-=-\mc{D}_{N_e} \dDl h.$ On the other hand, taking limits from inside and outside $\Omega$, and using Lemma \ref{lem:layer}, we have 
$$(u)_+=\left(-\frac{1}{2}I+\Dl\right)h\quad \quad (u)_-=\left(\frac{1}{2}I+\Dl \right)h.$$

Again, using the density of $C^\infty$ in $\mc{D}'$, we have proven
\begin{lemma}
\label{lem:LayerPotentialInverse2}
Let $G$, $\Dl $, and $\Dl ^\sharp$ and $\dDl$ be as in Lemma \ref{lem:LayerPotentialInverses} and $\Im \lambda>0$. Then
$$\mc{N}_iG=\frac{1}{2}I+\Dl^\sharp \quad \quad \mc{N}_eG=\frac{1}{2}I-\Dl^\sharp .$$
Moreover,
$$ \mc{D}_{N_i}\dDl=-\frac{1}{2}I+\Dl ,\quad\quad \mc{D}_{N_e}\dDl=-\frac{1}{2}I-\Dl.$$
\end{lemma}

Now, to see that Lemmas \ref{lem:LayerPotentialInverses} and \ref{lem:LayerPotentialInverse2} hold for $\lambda$ in the domain of $R_0(\lambda)$, observe that computing symbols as in Lemma \ref{lem:layer} (see also Lemma \ref{lem:diagPiece}) for $G$ and $\dDl$, we have that 
$G\in \PsiHom^{-1}$ elliptic and $\dDl\in \PsiHom^{1}$ elliptic. Thus, $G$ and $\dDl$ are meromorphic families of Fredholm operators on the domain of $R_0(\lambda)$. Now, Lemma \ref{lem:LayerPotentialInverses} together with Lemma \ref{lem:LayerPotentialInverse2} imply that $G$ and $\dDl$ are invertible for $\Im \lambda>0$. Thus, the meromorphic Fredholm theorem implies that they have meromorphic inverses. This implies that $\mc{N}_i$, $\mc{N}_e$, $\mc{D}_{N_i}$, and $\mc{D}_{N_e}$ are meromorphic families of operators. Hence, we have 
\begin{prop}
\label{prop:layerInverse}
For $\lambda$ in the domain of meromorphy of $R_0(\lambda)$, 
\begin{gather*} G\mc{N}_i=\frac{1}{2}I+\Dl \quad \quad G\mc{N}_e=\frac{1}{2}I-\Dl \\
\mc{N}_iG=\frac{1}{2}I+\Dl^\sharp \quad \quad \mc{N}_eG=\frac{1}{2}I-\Dl^\sharp \\
.
\end{gather*} 
Moreover, $\partial_{\nu_i}\mc{D}$ has no jump across $\partial\Omega$ and 
$$\dDl =(\partial_{\nu_i}\mc{D})_{\pm}=\left(-\frac{1}{2}I+\Dl ^\sharp\right)\mc{N}_i=\left(-\frac{1}{2}I-\Dl ^\sharp\right)\mc{N}_e. $$
Furthermore
\begin{gather*} \dDl \mc{D}_{N_i}=-\frac{1}{2}I+\Dl ^\sharp,\quad\quad \dDl\mc{D}_{N_e}=-\frac{1}{2}I-\Dl ^\sharp\\
\mc{D}_{N_i}\dDl=-\frac{1}{2}I+\Dl ,\quad\quad \mc{D}_{N_e}\dDl=-\frac{1}{2}I-\Dl .
\end{gather*} 
\end{prop}

\section{High energy estimates on the boundary layer operators}
Next we give semiclassical estimates for the single, double, and derivative double layer operators.  The estimates on single layer operators appear in \cite[Theorem 1.2]{GS}, and those for double layer operators appear in \cite{GalkSLO} but we repeat them below for the convenience of the reader. 

Let $\gamma:H^{1/2+\e}_{\loc} \to L^2(\Gamma)$ denote restriction to $\Gamma$ for a $C^{1,1}$ embedded hypersurface $\Gamma$ and $\gamma^*:L^2(\Gamma)\to H_{\comp}^{-1/2-\e}(\re^d)$ its dual. Then $\gamma^*$ is the inclusion map $f\mapsto f\delta_{\Gamma}$ where $\delta_\Gamma$ is $d-1$ dimensional Hausdorff measure on $\Gamma$. Then when $\Gamma=\partial\Omega$, $G$ can be written 
\begin{equation} \label{eqn:slo}G=\gamma R_0\gamma^*.\end{equation}
Because of this, we redefine the single layer operator to be given by \eqref{eqn:slo}

Similarly, if we assume that $\Gamma=\partial\Omega$ and $L$ is a vectorfield equal to $\partial_\nu$ on $\Gamma$, then 
\begin{gather}
\label{eqn:ddlo}\dDl(\lambda)=\gamma LR_0(\lambda) L^*\gamma^*.
\end{gather}
and we redefine the derivative double layer operator to be given by \eqref{eqn:ddlo}. Here we interpret $\gamma$ as a limit from either inside or outside $\Omega$ as in Lemma \ref{lem:LayerPotentialInverses}. Note that we cannot quite define $\Dl$ by
$$\gamma R_0(\lambda)L^*\gamma^*$$
since there is a jump across $\partial\Omega$. However, since 
$$\Dl=\pm \frac{1}{2}\Id+\gamma^{\pm}R_0L^*\gamma^*,$$
where $\gamma^+$ and $\gamma^-$ denote restrictions from the interior and exterior respectively,
this will not cause problems when obtaining bounds on $\Dl$ that have non-negative powers of $\la \lambda\ra$. 

If $d=1$ then $\delta_\Gamma$ is a finite sum of point measures, and from the formula $R_0(\lambda,x,y)=-(2i\lambda)^{-1}e^{i\lambda|x-y|}$ we see, using the notation of Theorem \ref{thm:optimal} below, that
\begin{equation}\label{eqn:optimal1d}
\begin{aligned} 
\|G(\lambda)\|_{L^2(\Gamma)\rightarrow L^2(\Gamma)}&\le C\,|\lambda|^{-1}\,e^{\LGamma(\Im \lambda)_-}\,,\\
\|\Dl(\lambda)\|_{L^2(\Gamma)\rightarrow L^2(\Gamma)}&\leq C e^{\LGamma(\Im \lambda)_-}\\
\|\dDl(\lambda)\|_{L^2(\Gamma)\rightarrow L^2(\Gamma)}&\leq C|\lambda|e^{\LGamma(\Im\lambda)_-}
\end{aligned}
\end{equation}
In higher dimensions, we establish the following theorem:

\begin{theorem}
\label{thm:optimal}
Let $\Gamma\Subset\re^d$ be a piecewise smooth, Lipschitz hypersurface. Then there exists $\lambda_0>0$ such that for $|\lambda|>\lambda_0$,
\begin{equation}
\label{eqn:optimalFlat}
\begin{aligned}
\|G(\lambda)\|_{L^2(\Gamma)\to L^2(\Gamma)} & \leq \;
C\,\la\lambda\ra^{-\frac 12}\,\log \la \lambda\ra \,e^{\LGamma(\Im \lambda)_-}\\
\|\Dl(\lambda)\|_{L^2(\Gamma)\to L^2(\Gamma)}&\leq C\,\la\lambda\ra^{\frac 14}\,\log \la \lambda\ra \,e^{\LGamma(\Im \lambda)_-}\\
\|\dDl(\lambda)\|_{H^1(\Gamma)\to L^2(\Gamma)}&\leq C\,\la\lambda\ra\,\log \la \lambda\ra \,e^{\LGamma(\Im \lambda)_-}
\end{aligned}
\end{equation}
where $\LGamma$ is the diameter of the set $\Gamma$, and we assume $-\pi\le \arg\lambda\le 2\pi$ if $d$ is even.

If $\Gamma$ can be written as a finite union of compact subsets of strictly convex $C^\infty$ hypersurfaces, then for some $C$ and all $|\lambda|>\lambda_0$ the following stronger estimates hold
\begin{equation}
\label{eqn:optimalConvex}
\begin{aligned}
\|G(\lambda)\|_{L^2(\Gamma)\to L^2(\Gamma)} & \leq \;
C\,\la\lambda\ra^{-\frac 23}\,\log \la \lambda\ra \,e^{\LGamma(\Im \lambda)_-}\\
\|\Dl(\lambda)\|_{L^2(\Gamma)\to L^2(\Gamma)}&\leq C\,\la\lambda\ra^{\frac 16}\,\log \la \lambda\ra \,e^{\LGamma(\Im \lambda)_-}
\end{aligned}
\end{equation}
\end{theorem}

Here we set $\la\lambda\ra=(2+|\lambda|^2)^{\frac 12}$, and $(\Im\lambda)_-=\max(0,-\Im\lambda)\,.$
The powers on $\la \lambda\ra$ in the estimates \eqref{eqn:optimalFlat} and \eqref{eqn:optimalConvex}, respectively, are in general optimal (see \cite[Appendix A]{GalkSLO} for the sharpness of the estimates for $G$ and $\Dl$). The sharpness of the exponent for $\dDl$ follows from an identical argument to that for $G$ i.e. that the corresponding estimate for the restriction of eigenfunctions is optimal.

\subsection{Proof of the Theorem}
We start by proving a conditional result which assumes a certain estimate on restriction of the Fourier transform of surface measures to the sphere of radius $r$.
\begin{lemma}
\label{lem:Q}
Suppose that for $\Gamma\Subset \re^d$ any compact embedded $C^\infty$ hypersurface, and some $\alpha\,,\,\beta>0$, 
\begin{align}
\label{eqn:fourierRestrictionEstimate1}
\int |\widehat{L^*f \delta_\Gamma}|^2(\xi)\delta(|\xi|-r)d\xi&\leq C_{\Gamma}\la r\ra^{2\alpha}\|f\|^2_{L^2(\Gamma)},\\
\label{eqn:fourierRestrictionEstimate2}
\int |\widehat{f \delta_\Gamma}|^2(\xi)\delta(|\xi|-r)d\xi&\leq C_{\Gamma}\la r\ra^{2\beta}\|f\|^2_{L^2(\Gamma)}.
\end{align}
with $2\beta<1$ and $\alpha\geq \beta$.
Let $\Gamma_1,\,\Gamma_2\Subset \re^d$ be compact embedded $C^\infty$ hypersurfaces.
Let $L_i$ be a vector field with $L_i=\partial_{\nu}$ on $\Gamma_i$ for some choice of normal $\nu$ on $\Gamma_1$ and $\psi\in \Cc(\re)$ with $\psi\equiv 1$ in neighborhood of $0$. Then define for $f\in L^2(\Gamma_1)$, $g\in L^2(\Gamma_2)$
\begin{gather*}
\Qs(f,g):=\int R_0(\lambda)(\psi(\lambda^{-1}D)f\delta_{\Gamma_1})\bar{g}\delta_{\Gamma_2}\,,\\
\Qa(f,g):=\int R_0(\lambda)(\psi(\lambda^{-1}D)L_1^*(f\delta_{\Gamma_1}))\bar{g}\delta_{\Gamma_2}\\
\Qdd(f,g):=\int R_0(\lambda)(\psi(\lambda^{-1}D)L_1^*(f\delta_{\Gamma_1})\overline{L_2^*(g\delta_{\Gamma_2})}
\end{gather*}
Then for $\Im\lambda>0$,
\begin{align}
\label{eqn:lowFreqSingle}
|\Qs(f,g)|&\leq C_{\Gamma_1,\Gamma_2}\la \lambda \ra ^{2\beta-1}\log \la \lambda\ra\|f\|_{L^2(\Gamma_1)}\|g\|_{L^2(\Gamma_2)}\\
\label{eqn:lowFreqDouble}
|\Qa(f,g)|&\leq C_{\Gamma_1,\Gamma_2,\psi}\la \lambda \ra ^{\alpha +\beta-1}\log \la \lambda\ra\|f\|_{L^2(\Gamma_1)}\|g\|_{L^2(\Gamma_2)}\\
\label{eqn:lowFreqDerDouble}
|\Qdd(f,g)|&\leq C_{\Gamma_1,\Gamma_2,\psi}\la \lambda \ra ^{2\alpha-1}\log \la \lambda\ra\|f\|_{L^2(\Gamma_1)}\|g\|_{L^2(\Gamma_2)}.
\end{align}
\end{lemma}
\begin{proof}
We follow \cite{GS} \cite{GalkSLO} to prove the lemma. First, observe that due to the compact support of $f\delta_{\Gamma_i}$, \eqref{eqn:fourierRestrictionEstimate1} and \eqref{eqn:fourierRestrictionEstimate2}  imply that for $\Gamma\Subset \re^d$,
\begin{align}
\label{eqn:fourierRestrictionEstimate3}
\int\left|\nabla_{\xi}\,\widehat{L^* f\delta_\Gamma}(\xi)\right|^2\delta(|\xi|-r)&\leq C\,\la r\ra^{2\alpha}\|f\|_{L^2(\Gamma)}^2\,,\\
\label{eqn:fourierRestrictionEstimate4}
\int\left|\nabla_{\xi}\,\widehat{f\delta_\Gamma}(\xi)\right|^2\delta(|\xi|-r)&\leq C\,\la r\ra^{2\beta}\|f\|_{L^2(\Gamma)}^2\,.
\end{align}
Indeed, $\nabla_\xi \widehat{f\delta_\Gamma}=\widehat{xf\delta_{\Gamma}}$
and since $\Gamma$ is compact, 
$$\|xf\|_{L^2(\Gamma)}\leq C\|f\|_{L^2(\Gamma)}.$$
Also, $\nabla_ \xi\widehat{L^*(f\delta_\Gamma)}=\widehat{xL^*(f\delta_{\Gamma})}.$
Then
$$xL^*(f\delta_\Gamma)=L^*(xf\delta_\Gamma)+[x,L^*]f\delta_\Gamma$$
and $[x,L^*]\in C^\infty$. Therefore, using compactness of $\Gamma$,
$$\|xf\|_{L^2(\Gamma)}+\|[x,L^*]f\|_{L^2(\Gamma)}\leq C\|f\|_{L^2(\Gamma)}.$$

Now, $g\delta_{\Gamma_2}\in H^{-\frac 12-\e}(\re^d)$, $L_2^*(g\delta_{\Gamma_2})\in H^{-3/2-\e}(\re^d)$ and 
\begin{equation}
\label{eqn:smoothLeft}
\begin{gathered}R_0(\lambda)(\psi(\lambda^{-1}|D|)L^*(f\delta_{\Gamma_1}))\in C^\infty(\re^d),\\ R_0(\lambda)(\psi(\lambda^{-1}|D|))f\delta_{\Gamma_1})\in C^\infty(\re^d).\end{gathered}
\end{equation}

By Plancherel's theorem,
\begin{gather*}
 \Qs(f,g)=\int \psi(\lambda^{-1}|\xi|)\frac{\widehat{f\delta_{\Gamma_1}}(\xi)\overline{\widehat{g\delta_{\Gamma_2}}}(\xi)}{|\xi|^2-\lambda^2}\\
 \Qa(f,g)=\int \psi(\lambda^{-1}|\xi|)\frac{\widehat{L^* f\delta_{\Gamma_1}}(\xi)\,\overline{\widehat{g\delta_{\Gamma_2}}(\xi)}}{|\xi|^2-\lambda^2},\\
\Qdd(f,g)=\int \psi(\lambda^{-1}|\xi|)\frac{\widehat{L_1^*f\delta_{\Gamma_1}}(\xi)\overline{\widehat{L_2^*g\delta_{\Gamma_2}}}(\xi)}{|\xi|^2-\lambda^2}
\end{gather*}


\noindent Thus, to prove the lemma, we only need estimate 
\begin{equation}
\label{eqn:restrictedDualPlancherel}
\int \psi(\lambda^{-1}|\xi|)\frac{F(\xi)\,G(\xi)}{|\xi|^2-\lambda^2}
\end{equation}
where by \eqref{eqn:fourierRestrictionEstimate1}, \eqref{eqn:fourierRestrictionEstimate2}, \eqref{eqn:fourierRestrictionEstimate3}, and \eqref{eqn:fourierRestrictionEstimate4}
\begin{gather*} \|F\|_{L^2(S_r^{d-1})}+\|\nabla_{\xi}F\|_{L^2(S_r^{d-1})}\leq C\la r\ra^{\delta_1}\|f\|_{L^2(\Gamma)},\\
\|G\|_{L^2(S_r^{d-1})}+\|\nabla_{\xi}G\|_{L^2(S_r^{d-1})}\leq C\la r\ra ^{\delta_2}\|g\|_{L^2(\Gamma)}.\end{gather*}

Consider first the integral in \eqref{eqn:restrictedDualPlancherel} over $\bigl||\xi|-|\lambda|\bigr|\ge 1$. Since $\bigl||\xi|^2-\lambda^2\bigr|\ge \bigl||\xi|^2-|\lambda|^2\bigr|$, by the Schwartz inequality, \eqref{eqn:fourierRestrictionEstimate1}, and \eqref{eqn:fourierRestrictionEstimate2}
this piece of the integral is bounded by
\begin{align}
\int_{\left||\xi|-|\lambda|\right|\geq 1} \left|\psi(\lambda^{-1}|\xi|)\frac{F(\xi)\,G(\xi)}{|\xi|^2-\lambda^2}\right|\!\!\!\!\!\!\!\!\!\!\!\!\!\!\!\!\!\!\!\!\!\!\!\!\!\!\!\!\!\!\!\!\!\!\!\!\!\!\!\!\!\!\!\!\!\!\!\!\!\!\!\!\!\!\!\!\nonumber\\
&\leq \int_{M\lambda \geq \left|r-|\lambda|\right|\geq 1}\frac{1}{r^2-|\lambda|^2}\int_{S_r^{d-1}}F(r\theta)\,G(r\theta)dS(\theta) dr\nonumber\\
&\leq C\|f\|_{L^2(\Gamma)}\|g\|_{L^2(\Gamma)}
\int_{M|\lambda|\ge |r-|\lambda||\ge 1}
\la r\ra^{\delta_1+\delta_1}\,\bigl|\,r^2-|\lambda|^2\,\bigr|^{-1}dr\nonumber\\
&\leq C\|f\|_{L^2(\Gamma)}\|g\|_{L^2(\Gamma)}\lambda^{\delta_1+\delta_2-1}\int_{M|\lambda|\geq \left|r-|\lambda|\right|\geq 1}\left|r-|\lambda|\right|^{-1}dr\nonumber\\
&\leq C\,|\lambda|^{\delta_1+\delta_2-1}\log |\lambda|\,\|f\|_{L^2(\Gamma)}\|g\|_{L^2(\Gamma)}.\label{eqn:logLoss}
\end{align}

\begin{remark} The estimate \eqref{eqn:logLoss} is the only term where the $\log$ appears. 
\end{remark}

Next, if $\Im\lambda\ge 1$, then $\bigl||\xi|^2-\lambda^2\bigr|\ge |\lambda|$, and by 
\eqref{eqn:fourierRestrictionEstimate1}, \eqref{eqn:fourierRestrictionEstimate2} 
$$
\Biggl|\;\int_{||\xi|-|\lambda||\le 1}
\frac{F(\xi)\,G(\xi)}{|\xi|^2-\lambda^2}\,d\xi\;\Biggr|\le C\,|\lambda|^{\delta_1+\delta_2-1}\,\|f\|_{L^2(\Gamma)}\|g\|_{L^2(\Gamma)}.
$$
Thus, we may restrict our attention to $0\le\Im \lambda\leq 1$ and $\bigl||\xi|-|\lambda|\bigr|\le 1$. 

We consider $\Re\lambda\ge 0$, the other case following similarly, and write
$$
\frac{1}{|\xi|^2-\lambda^2}=\frac{1}{|\xi|+\lambda}\;\frac{\xi}{|\xi|}\cdot\nabla_\xi\log(|\xi|-\lambda)\,,
$$
where the logarithm is well defined since $\Im(|\xi|-\lambda)<0$. Let $\chi(r)=1$ for $|r|\le 1$ and vanish for $|r|\ge \frac 32$. We then use integration by parts, together with
\eqref{eqn:fourierRestrictionEstimate1}, \eqref{eqn:fourierRestrictionEstimate2}, \eqref{eqn:fourierRestrictionEstimate3}, and \eqref{eqn:fourierRestrictionEstimate4} to bound
\begin{multline*}
\Biggl|\;\int\chi(|\xi|-|\lambda|)\,\frac{1}{|\xi|+\lambda}\,
F(\xi)\,G(\xi)\,\;\frac{\xi}{|\xi|}\cdot\nabla_\xi\log(|\xi|-\lambda)\,d\xi\;\Biggr|\\
\le C\,|\lambda|^{\delta_1+\delta_2-1}\,\|f\|_{L^2(\Gamma)}\|g\|_{L^2(\Gamma)}.
\end{multline*}
Now, taking $\delta_1=\delta_2=\alpha$ gives \eqref{eqn:lowFreqSingle}, and taking $\delta_1=\alpha$ and $\delta_2=\beta$ gives \eqref{eqn:lowFreqDouble} and taking $\delta_1=\delta_2=\beta$ gives \eqref{eqn:lowFreqDerDouble}.
\end{proof}

\begin{remark}
Note that the estimate on $\Qs$ holds uniformly in $\psi$ and so putting in the cutoff $\psi$ is unnecessary. However, so that the presentation of all of the estimates are similar, we include the cutoff here.
\end{remark}

We now prove the estimates \eqref{eqn:fourierRestrictionEstimate1} and \eqref{eqn:fourierRestrictionEstimate2}.  To do so, we will need the following restriction estimates for quasimodes. 

\begin{lemma}
\label{lem:quasimodeEstimates}
Let $U\Subset \re^d$ be open with $\Gamma \Subset U$ a $C^\infty$ embedded hypersurface. Suppose that $\|u\|_{L^2(U)}=1$ and 
$$(-h^2\Delta -1)u=\O{L^2}(h).$$ 
Then for $0<h<h_0$,
\begin{equation} \label{restrict}
\|u\|_{L^2(\Gamma)}\leq \begin{cases} Ch^{-1/4}\\
Ch^{-1/6}&\Gamma\text{ curved}\end{cases}.
\end{equation}
\end{lemma}
In the setting of smooth Riemannian manifolds with restriction to a submanifold, these estimates along with their $L^p$ generalizations appear in the work of Tataru \cite{Tat} who also notes that the $L^2$ bounds are a corollary of an estimate of Greenleaf and Seeger \cite{GreenSeeg}. Such $L^p$ generalizations were also studied by Burq, G\'erard and Tzvetkov in \cite{BGT}. Semiclassical analogues were proved by Tacy \cite{T} and Hassell-Tacy \cite{HTacy}. 

We also need the corresponding restriction estimates for normal derivatives.
\begin{lemma}
\label{lem:quasimodeNormalEstimates}
Let $U\Subset \re^d$ be open with $\Gamma \Subset U$ a $C^\infty$ embedded hypersurface. Suppose that $\|u\|_{L^2(U)}=1$ and 
$$(-h^2\Delta-1)u=\O{L^2}(h).$$
Then for $0<h<1$
\begin{equation} \label{eqn:restricEigNormal} \|\partial_\nu u\|_{L^2(\Gamma))}\leq Ch^{-1}
\end{equation}
where $\partial_\nu$ is a choice of normal derivative to $\Gamma$. 
\end{lemma}
Estimates of this type first appear in the work of Tataru \cite{Tat} in the form of regularity estimates for restrictions of solutions to hyperbolic equations. Semiclassical analogs of this estimate were proved in  Christianson--Hassell--Toth \cite{christianson2014exterior} and Tacy \cite{T14}.

\begin{lemma}
\label{lem:FourierRestrict}
Let $\Gamma\Subset \re^d$ be a compact $C^{\infty}$ embedded hypersurface. Then estimate \eqref{eqn:fourierRestrictionEstimate2} holds with $\beta=1/4$ and for $L=\partial_{\nu}$ on $\Gamma$, estimate \eqref{eqn:fourierRestrictionEstimate1} holds with $\alpha=1$. Moreover, if $\Gamma$ is curved then \eqref{eqn:fourierRestrictionEstimate2} holds with $\beta =1/6.$  \end{lemma}
\begin{proof}
Let $A:H^s(\re^d)\to H^{s-1}(\re^d)$. To estimate
$$\int |\widehat{A^*(f\delta_{\Gamma})}(\xi)|^2\delta(|\xi|-r),$$ 
write 
\begin{align*} 
\la \widehat{A^*(f\delta_{\Gamma})}(\xi)\delta(|\xi|-r),\phi(\xi)\ra &=\iint A^*(f(x)\delta_{\Gamma})\delta(|\xi|-r)\overline{\phi(\xi)e^{i\la x,\xi\ra}} dxd\xi\\
& =\int_{\Gamma}f AT_r\phi dx
\end{align*}
where 
\begin{equation}
\label{eqn:T}
T_r\phi:=\int \delta(|\xi|-r)\phi(\xi)e^{i\la x,\xi\ra }d\xi.
\end{equation}

For $\chi\in \Cc(\re^d)$, $\chi T_r\phi$ is a quasimode of the Laplacian with eigenvalue $\lambda=r$ in the sense of Lemma \ref{lem:quasimodeEstimates} with $h=r^{-1}$. To see this, observe that 
\begin{equation}
\label{eqn:cutoffQuasi}-\Delta \chi T_r\phi=\chi(-\Delta T_r\phi)+[-\Delta,\chi]T_r\phi=r^2\chi T_r\phi+v
\end{equation}
where $\|v\|_{L^2}\leq Cr\|T_r\phi\|_{L^2}$. Thus, we can use the restriction bounds for eigenfunctions to obtain estimates on $T_r\phi$. 

To prove \eqref{eqn:fourierRestrictionEstimate2}, let $A=I$. Then, by Lemma \ref{lem:quasimodeEstimates}
\begin{equation}
\label{eqn:restrictEig}
\|\chi T_r\phi\|_{L^2(\Gamma)}\leq r^{\frac 14}\|\chi T_r\phi\|_{L^2(\re^d)},
\end{equation}
and if $\Gamma$ is curved then
\begin{equation}
\label{eqn:restrictEigCurved}
\|\chi T_r\phi\|_{L^2(\Gamma)}\leq r^{\frac 16}\|\chi T_r\phi\|_{L^2(\re^d)}.
\end{equation}

Next, we take $A=L$ to obtain \eqref{eqn:fourierRestrictionEstimate1}. Observe that 
$$\chi LT_r\phi =L\chi T_r\phi +[\chi,L]T_r\phi$$
with $[\chi,L]\in \Cc(\re^d)$. Therefore, 
$[\chi,L]T_r\phi$ is a quasimode of the Laplacian with eigenvalue $r$ by \eqref{eqn:cutoffQuasi}. 

Hence, using the fact that $L=\partial_{\nu}$ on $\Gamma$ together with Lemma \ref{lem:quasimodeNormalEstimates}, we can estimate $LT_r\phi$.
\begin{equation}
\label{eqn:restrictEigNormal}\|\chi LT_r\phi\|_{L^2(\Gamma)}\leq  \|L\chi T_r\phi\|_{L^2(\Gamma)}+\|[L,\chi ]T_r\phi\|_{L^2(\Gamma)}\leq Cr\|\chi T_r\phi \|_{L^2(\re^d)}.
\end{equation}

To complete the proof of the Lemma, we estimate $\|\chi T_r\phi\|_{L^2(\re^d)}$ in terms of $\|\phi\|_{L^2(S^{d-1})}.$ To do this, we estimate 
$B:=(\chi T_r)^*\chi T_r:L^2S_r^{d-1}\to L^2S_r^{d-1}.$ This operator has kernel
$$B(\xi,\eta)=\int_{\re^d} \chi^2(y)e^{i\la y,\xi-\eta\ra}dy=\widehat{\chi^2}(\eta-\xi).$$
Now, for $\eta\in S_r^{d-1}$, and any $N>0$,
$$\int_{S_r^{d-1}}|\widehat{\chi^2}(\eta-\xi)|dS(\xi)\leq \int_{B(0,r/2)}\la \xi'\ra^{-N}\left[1-\frac{|\xi'|^2}{r^2}\right]^{-1/2}d\xi'+C\la r\ra^{-N}\leq C.$$
Thus, by Schur's inequality, $B$ is bounded on $L^2S_r^{d-1}$ uniformly in $r$. 
Therefore, 
\begin{align*}
\|\chi T_r\phi\|_{L^2(\re^d)}^2& \leq C \|\phi\|_{L^2(S_r^{d-1})}^2.
\end{align*}
Combining this with \eqref{eqn:restrictEig}, \eqref{eqn:restrictEigCurved} and \eqref{eqn:restrictEigNormal} completes the proof of the Lemma.
\end{proof}

Next, we obtain an estimate on the high frequency component of $\Dl$ and $\dDl$.  We start by analyzing the high frequency components of the free resolvent.
\begin{lemma}
\label{lem:freeHighFreq}
Suppose that $|z|\in [E-\delta,E+\delta]$
Let $\psi\in \Cc(\re)$ with $\psi\equiv 1$ on $[-2E^2,2E^2]$. Then for $\chi\in \Cc(\re^d).$
$$\chi R_0(z/h)\chi(1-\psi(|hD|))= B$$
where $B\in h^2\Ph{-2}{}(\re^d)$ with 
$$\sigma(B)=\frac{\chi^2h^2(1-\psi(|\xi|))}{|\xi|^2-z^2}.$$
If $\Im z>0$, then $\chi$ can be removed from all of the above statements.
\end{lemma}
\begin{proof}
Let $\chi_0=\chi \in \Cc(\re^d)$ and $\chi_n\in \Cc(\re^d)$ have $\chi_n\equiv 1$ on $\supp \chi_{n-1}$ for $n\geq 1$. Let $\psi_0=\psi\in \Cc(\re)$ have $\psi \equiv 1$ on $[-2E^2,2E^2]$, let $\psi_n\in \Cc(\re)$ have $\psi_n\equiv 1$ on $[-3E^2/2,3E^2/2]$ and $\supp \psi_n\subset \{\psi_{n-1}\equiv 1\}$ for $n\geq 1$. Finally, let $\varphi_n=(1-\psi_n).$  
Then, 
\begin{align}
h^{-2}\chi R_0\chi (-h^2\Delta -z^2)&=(\chi^2-\chi h^{-2} \chi_1R_0\chi_1 [\chi,h^2\Delta])\label{eqn:ellipticResolve1}
\end{align}
Now, by Lemma \ref{lem:microlocalElliptic} there exists $A_0\in h^2\Ph{-2}{}(\re^d)$ with $\WFh(A_0)\subset\{\supp \varphi_0\}$, such that 
$$h^{-2}(-h^2\Delta-z)A_0=\varphi(|hD|)+\O{\Ph{-\infty}{}}(h^\infty)$$
and $A_0$ has 
$$\sigma(A_0)=\frac{h^2\varphi(|hD|)}{|\xi|^2-z^2}.$$
Composing \eqref{eqn:ellipticResolve1} on the right with $A_0$, we have 
\begin{align*}
\chi R_0\chi \varphi(|hD|)&=\chi^2 A_0-\chi \chi_1R_0\chi_1\varphi_1(|hD|)[\chi,h^2\Delta]h^{-2}A_0+\O{\Ph{-\infty}{}}(h^\infty)\\
&=\chi^2 A_0-\chi \chi_1R_0\chi_1\varphi_1(|hD|)\O{\Ph{-1}{}}(h)+\O{\Ph{-\infty}{}}(h^\infty)
\end{align*}
Now, applying the same arguments, there exists $A_n\in h^2\Ph{-2}{}(\re^d)$ such that
$$\chi_nR_0\chi_n \varphi_n(|hD|)=\chi_n^2A_n +\chi_{n+1}R_0\chi_{n+1}\varphi_{n+1}(|hD|)\O{\Ph{-1}{}}(h)+\O{\Ph{-\infty}{}}(h^\infty).$$
Hence, by induction
$$\chi R_0\chi \varphi(|hD|)=B\in h^2\Ph{-2}{}(\re^d),$$
with 
$$\sigma(B)=\frac{h^2\chi^2(1-\psi(|\xi|)}{|\xi|^2-z^2}$$
as desired.

Now, if $\Im z>0$, then 
$h^{-2}R_0(-h^2\Delta-z^2)=\Id_{L^2}$
and hence,
$$h^{-2}R_0\varphi(|hD|)=A_0+\O{\Ph{-\infty}{}}(h^\infty).$$
\end{proof}
Now, let $\gamma^{\pm}:H^s(\Omega^{\pm})\to H^{s-1/2}(\partial\Omega)$, $s>1/2$ denote the restriction map where $\Omega^+=\Omega$ and $\Omega^-=\re^d\setminus\overline{\Omega}$. Then we have
\begin{lemma}
\label{lem:outsideSphere}
Let $M>1$ and $\psi\in \Cc(\re)$ with $\psi\equiv 1$ for $|\xi|<M$. Suppose that $\partial\Omega$ is a compact embedded $C^\infty$ hypersurface. Then there exists $\lambda_0>0$ such that for $|\lambda|>\lambda_0$, and $\chi\in \Cc(\re^d)$
\begin{align}
\label{eqn:estHighSingle}
\gamma R_0(\lambda)\chi (1-\psi(|\lambda|^{-1}|D|))\gamma^*&=\O{L^2(\partial\Omega)\to L^2(\partial\Omega)}(|\lambda|^{-1}),\\
\label{eqn:estHighDouble}
\gamma^{\pm} R_0(\lambda)\chi (1-\psi(|\lambda|^{-1}|D|))L^*\gamma^*&=\O{L^2(\partial\Omega)\to L^2(\partial\Omega)}(1),\\
\label{eqn:estHighdDouble}
\gamma^{\pm} L R_0(\lambda)\chi (1-\psi(|\lambda|^{-1}|D|))L^*\gamma^*&=\O{H^1(\partial\Omega)\to L^2(\partial\Omega)}(|\lambda|).
\end{align}
\end{lemma}
\begin{proof}
Let $h^{-1}=|\lambda|$ and $\chi \equiv 1$ on $\Omega$. For $\Im \lambda >0$, we take $\chi \equiv 1$ and for $\arg \lambda\in [-\pi,0]\cup [\pi ,2\pi]$, $\chi \in \Cc(\re^d)$. Then by Lemma \ref{lem:freeHighFreq} 
\begin{equation} \label{eqn:freePsiDO} \chi R_0(\lambda)\chi (1-\psi(hD)) \in h^2\Ph{-2}{}.
\end{equation}
Note that $\gamma$ is a semiclassical FIO and for $s>1/2$,
\begin{equation}\label{eqn:restrictBound}\gamma=\O{H_h^s(\re^d)\to H_h^{s-1/2}(\partial\Omega)}(h^{-1/2}).
\end{equation}
The bound \eqref{eqn:estHighSingle} and the corresponding bound in the lower half plane follow from \eqref{eqn:freePsiDO} and composition with $\gamma$ and $\gamma^*$.

The strategy for obtaining the bounds on $\Dl$ and $\dDl$ on $|\xi|\geq M\lambda $ is to compare them with the corresponding operators for $\lambda=i$.  Note that $\chi R_0(i)\chi (1-\psi(|hD|))\in h^2\Ph{-2}{}$ and for $\Im \lambda>0$, the $\chi$ factors are unnecessary.  We consider
\begin{align*} 
A_{h}&:=\chi(R_0(\lambda)-R_0(i))\chi(1-\psi(|hD|))\\
&=h^{-2}\left(h^2-\frac{\lambda^2}{|\lambda|^2}\right)\chi R_0(\lambda)R_0(i)\chi(1-\psi(|hD|)).
\end{align*} 
Hence, $A_{h}\in h^2\Ph{-4}{}$. Let
$$B_h:= \gamma A_hL^*\gamma^*,\quad \quad C_h:=\gamma L A_{h}L^*\gamma^*.$$
Then, using \eqref{eqn:restrictBound} and the fact that $L,L^*=\O{H_h^s\to H_h^{s-1}}(h^{-1})$, we have that $B_h=\O{L^2\to L^2}(1)$ and $C_h=\O{L^2\to L^2}(h^{-1})$.

Now, by \cite[Section 7.11]{Taylor}
\begin{align*} \gamma^{\pm}  R_0(i)\chi (1-\psi(|hD|))L^*\gamma:& H^s(\partial\Omega)\to H^{s}(\partial\Omega)\\
\gamma^{\pm} LR_0(i)\chi(1-\psi(|hD|)) L^*\gamma:& H^s(\partial\Omega)\to H^{s-1}(\partial\Omega)
\end{align*}
for $\partial\Omega$ a smooth hypersurface. Hence,
\begin{align*} 
\gamma^{\pm} R_0(\lambda)\chi (1-\psi(|hD|))L^*\gamma&=\gamma^{\pm}  R_0(i)\chi (1-\psi(|hD|))L^*\gamma +\gamma L B_hL^*\gamma^*\\
&=\O{L^2\to L^2}(1)\\
\gamma^{\pm} L R_0(\lambda)\chi (1-\psi(|hD|))L^*\gamma&=\gamma^{\pm} L  R_0(i)\chi (1-\psi(|hD|))L^*\gamma \\
&\quad\quad\quad+\gamma^{\pm} L C_hL^*\gamma^*\\
&=\O{H^1\to L^2}(h^{-1}).
\end{align*} 
\end{proof}

\noindent Taking $\partial\Omega=\bigcup_{i}\Gamma_i$ and applying Lemmas \ref{lem:Q} and Lemma \ref{lem:outsideSphere} finishes the proof of Theorem \ref{thm:optimal} for $\Im \lambda\geq 0$.

Our final task is to extend the estimates into the lower half plane. Let $\tilde{Q}_\lambda^1=\Qs$, $\tilde{Q}_\lambda^2=\Qa$, and $\tilde{Q}_\lambda^3=\Qdd$. Then let 
\begin{gather*} 
Q_\lambda^1(f,g):=\la \gamma R_0(\lambda)f\delta_\Gamma,g\ra,\quad\quad Q_\lambda^2(f,g):=\la \gamma^{\pm}R_0(\lambda) L^*(f\delta_\Gamma),g\ra,\\
 Q_\lambda^3(f,g):=\la\gamma^{\pm}LR_0(\lambda) L^*(f\delta_\Gamma),g\ra
\end{gather*}
\begin{lemma}
\label{lem:phragmen}
Suppose that for $|\lambda|\geq \lambda_0$, $\Im \lambda\geq 0$, 
$$|\tilde{Q}^i_\lambda(f,g)|\leq C\la\lambda\ra^\alpha  (\log \la\lambda\ra)^\beta\|f\|_{L^2(\Gamma)}\|g\|_{L^2(\Gamma)}.$$
Then for $|\lambda|\geq \lambda_0$ and $\Im \lambda \leq 0$, if $d$ is odd and for $\arg \lambda \in [-\pi,0]\cup [\pi,2\pi]$ if $d$ is even
$$|Q^i_{\lambda}(f,g)|\leq C\la\lambda\ra ^{\tilde{\alpha}} (\log \la \lambda\ra)^\beta e^{\LOmega \Im \lambda}\|f\|_{\mc{A}_i}\|g\|_{L^2(\Gamma)}$$
where $\LOmega$ is the diameter of $\Omega$, 
$$\tilde{\alpha}:=\begin{cases}\max(-1,\alpha)&i=1\\
\max(0,\alpha)&i=2\\
\max(1,\alpha)&i=3
\end{cases},$$
and $\mc{A}_3=H^1(\Gamma)$, $\mc{A}_i=L^2(\Gamma)$ for $i=1,2$.
\end{lemma}
\begin{proof}
We first consider $d$ odd. Let $\|f\|_{\mc{A_i}}=1$ and $\|g\|_{L^2}=1$. Let $\chi\equiv 1 $ on $\Omega$. Then consider 
$$F(\lambda)=e^{-i\LOmega \lambda}\lambda^{-\tilde{\alpha}}(\log \lambda)^{-\beta}Q_\lambda(f,g)\,,\quad |\lambda|\geq \lambda_0\,, \Im \lambda\geq 0$$
where $\log \lambda$ is defined for $\arg \lambda \in (\pi/2,5\pi/2)$. Lemma \ref{lem:outsideSphere} shows that 
\begin{gather*} \gamma R_0(\lambda)\chi (1-\psi(|\lambda|^{-1}|D|)) \gamma^*=\O{L^2\to L^2}(|\lambda|^{-1})\\
\gamma^{\pm} R_0(\lambda)\chi (1-\psi(|\lambda|^{-1}D)) L^*\gamma^*=\O{L^2\to L^2}(1)
\\\gamma^{\pm} LR_0(\lambda)\chi (1-\psi(|\lambda|^{-1}|D|))  L^*\gamma^*=\O{H^1\to L^2}(|\lambda|)
\end{gather*}
Hence, $|F(\lambda)|\leq C$ on $\re\setminus[-\lambda_0,\lambda_0]$. 

Now, for all $s$
$$\|\chi R_0(\lambda)\chi\|_{H^s\to H^s}\leq C\la \lambda\ra^{-1}e^{D_\chi(\Im \lambda)_-}$$
where $D_\chi=\diam(\supp\chi)$ is the diameter of $\supp \chi$. Moreover, 
$$\psi(|\lambda^{-1}D|):H^s\to H^{s+M}=\O{}(|\lambda|^M).$$
So, there exists $N>0$ such that 
\begin{multline*} \|\gamma^{\pm} L R_0(\lambda)\chi \psi(|\lambda|^{-1}D)L^*\gamma^*\|_{H^1\to L^2}+\|\gamma^{\pm} R_0(\lambda)\chi \psi(|\lambda|^{-1}D)L^*\gamma^*\|_{L^2\to L^2}\\
+\|\gamma R_0(\lambda)\chi \psi(|\lambda|^{-1}D)\gamma^*\|_{L^2\to L^2}\leq C\la \lambda\ra^{N}e^{D_\chi(\Im \lambda)_-}.
\end{multline*}
Letting $\supp \chi \to \Omega$, we see that $|F(\lambda)|$ has at most polynomical growth in the lower half plane. Thus, the Phragm\'en--Lindel\"of theorem shows that $|F(\lambda)|\leq C$. 

When $d$ is even, we note that the assumed bounds hold for $\arg \lambda=2\pi$ and $|\lambda|\geq \lambda_0$. This follows since $R_0(\lambda e^{\pi i})-R_0(\lambda)$ satisfies the same bounds as $R_0(\lambda)$ for $\arg \lambda=0$. Moreover, $R_0(\lambda e^{2\pi i})-R_0(\lambda e^{\pi i})=R_0(\lambda e^{\pi i})-R_0(\lambda).$ Thus, we apply the Phragm\'en--Lindel\"of theorem on the sheet $\pi \leq \arg z\leq 2\pi$. Using a similar argument, we can apply the Phragm\'{e}n--Lindel\"of theorem on $-\pi \leq \arg \lambda \leq 0.$
\end{proof}

Applying Lemma \ref{lem:phragmen} together with Lemmas \ref{lem:Q}, \ref{lem:FourierRestrict} and \ref{lem:outsideSphere} implies Theorem \ref{thm:optimal}.

\section{Microlocal description of the free resolvent}
\label{sec:decompose}
We have already analyzed the high frequency components of the free resolvent in Lemma \ref{lem:freeHighFreq}. In this section, we analyze the remaining kernel of the free resolvent as a semiclassical intersecting Lagrangian distribution (see Appendix \ref{ch:iLagrange}). In particular, we prove 
\begin{theorem}
\label{thm:freeResolve}
Suppose that $a,b>0$, $M>0$, $\gamma<1/2$,  and 
$$z\in [a,b]\times i[-Ch\log h^{-1},Mh^{1-\gamma}]$$
with $\Re z=E+\O{}(h^{1-\gamma})$. Then for $\chi\in \Cc$, the cut-off free resolvent, $\chi R_0(z/h)\chi$, is given by 
$$\chi R_0(z/h)\chi=K_R+K_{\Delta}+\O{\mc{D}'\to \Cc}(h^\infty),$$
where $K_R$ has kernel $K(x,y)\in h^{3/2}e^{\frac{1}{h}(\Im z)_-D_\chi}I_{\gamma}^{\comp}(\re^d; \Lambda_0,\Lambda_1)$
with $D_\chi=\diam(\supp \chi)$, 
$$\Lambda_0=\{(x,\xi,x,-\xi)\in T^*\re^d\times T^*\re^d\}\text{ and }\Lambda_1=\{\exp_t\Lambda_0\cap \{|\xi|=E\}:t\geq 0\}\,,$$
and $K_{\Delta}\in h^2\Psi^{-2}_{\gamma}.$
Moreover, for any $\chi_1\in \Cc$ with $\chi_1(\xi)\equiv 1$ on $|\xi|\leq 2E$ we can take
\begin{multline*}
\sigma\left(e^{\frac{\Im z|x-y|}{h}}\chi_1(hD)K\right)=\left(\frac{\chi_1(\xi)\chi^2(x)h^2}{|\xi|^2-E^2}|dx\wedge d\xi|^{1/2},\right.\\
\left.{}\frac{h^{3/2}e^{\frac{i}{h}(\Re z-E) |x-y|}E^{(d-3)/2}e^{(-d+3)\pi i/4}\pi^{1/2}}{2^{1/2}|x-y|^{(d-1)/2}}\chi(x)\chi(y)|dy\wedge dx|^{1/2}\right)
\end{multline*}
and
$$\sigma(K_\Delta)=(1-\chi_1(\xi))\chi^2(x)h^2(|\xi|^2-E^2)^{-1}.$$
\end{theorem}

\begin{proof}
We now prove Theorem \ref{thm:freeResolve}. Recall that in the context of Fourier integral operator relations we denote a point in $\overline{T}^*M\times \overline{T}^*M'$ by $(x,\xi,y,\eta)$.
 By Lemma \ref{lem:greenFunc}, for $Ch^{1-\gamma}\geq  \Im z\geq 0$, $|\Re z-E|\leq Ch^{1-\gamma}$, and each $M>0$ there exists an operator $U$ that is $z/h$ outgoing with kernel $K(x,y,z/h) \in h^{\frac{3}{2}}I_{\gamma}^{\comp}(\re^d;\Lambda_0,\Lambda_1)+h^{2}\Psi^{-2}_{\gamma}$
where 
\begin{gather*}\Lambda_0:=\{(x,\xi,x,-\xi):x\in \re^d,\,\xi\in T^*_x(\re^d)\}\,,\\
\Lambda_1:=\{(\exp_t(x,\xi), x,-\xi):x\in \re^d,\, |\xi|=E,t\geq 0\}
\end{gather*}
such that for all $\chi,\chi_1\in \Cc(B(0,M))$ with $\chi_1\equiv 1$ on $\supp \chi$,
\begin{gather*}\chi_1(-\Delta-(z/h)^2)U\chi=\chi+\O{\mc{D}'\to \Cc}(h^\infty)\,,\\ 
(-\Delta-(z/h)^2)\chi_1 U\chi =\chi + [-\Delta, \chi_1]U\chi+\O{\mc{D'}\to \Cc}(h^\infty)
\end{gather*}
Thus, 
$$\chi U\chi=\chi R_0(-\Delta-(z/h)^2)\chi_1U\chi=\chi R_0\chi +\chi R_0[-\Delta,\chi_1]U\chi+\O{\mc{D}'\to \Cc}(h^\infty).$$
But, since $\WFh(R_0),\,\WFh(U)\subset \Lambda_0\cup \Lambda_1$, 
$$\chi R_0[-\Delta,\chi_1]U\chi=\O{\mc{D'}\to \Cc}(h^\infty).$$

Hence, for $\Im z\geq 0$ and $\chi\in \Cc$ with $\supp\chi\subset B(0,M)$,
\m \chi U(z/h)\chi =\chi R_0(z/h)\chi+\O{\mc{D}'\to \Cc}(h^\infty).\m 

In order to prove Theorem \ref{thm:freeResolve} we need to compute the symbol of $\chi U\chi $. First, define $P:=-\Delta -z^2/h^2=\weyl(h^{-2}(|\xi|^2-z^2))$. Then, for any $\delta>0$, $P$ has principal symbol
 \m p:=\sigma(P)=h^{-2}(|\xi|^2-E^2)\,\,\m 
and sub-principal symbol
$$\sigma_1(P):=h^{-2}2(E-z):-h^{-2}2\omega_0$$
 as an operator in $\Psi^2_{\delta}$.

Then, by Lemma \ref{lem:greenFunc} we have that
\begin{align*} 
r_0(x,\xi,x,-\xi)&:=\sigma(U)|_{\Lambda_0\cap T^*B(0,R)}\\
&=p^{-1}\sigma(\delta)|dx\wedge d\xi|^{1/2}=h^2(|\xi|^2-E^2)^{-1}|dx\wedge d\xi|^{1/2}.\end{align*} 

\begin{remark} Moreover, we see that in any coordinates each term in the full symbol of $U|_{\Lambda_0}$ has the form
$$a(x,\xi)=\frac{\sum_{|\alpha|=0}^ma_\alpha(x)\xi^\alpha}{|\xi|^2-E^2}$$
where $a_\alpha(x)\in C^\infty$
\end{remark}

\noindent Next, we compute $r_1=\sigma(R_0)|_{\Lambda_1\cap T^*B(0,R)\times T^*B(0,R)}$. Again, by Lemma \ref{lem:greenFunc}, we need to solve 
$$\begin{cases}h H_pr_1+ip_1r_1=0\\
r_1|_{\partial\Lambda_1}=e^{\pi i/4}(2\pi)^{1/2}h^{-1/2}R(r_0)\end{cases}$$
where $H_p$ is the Hamiltonian flow of $p$. Using that $\exp(tH_p)(x,\xi)=\exp_{2th^{-2}}(x,\xi)$, we have 
$$r_1(\exp_{t}(x,\xi),x,-\xi)=e^{i\omega_0 tE/h}e^{\pi i/4}(2\pi )^{1/2}h^{-1/2}R(r_0)(x,\xi,x,-\xi).$$
So, all that remains is to determine $R(r_0)(x,\xi,x,-\xi).$
Taking $g=|\xi|^2-E^2$ and $f=\la x-y,\xi\ra$, gives
\m r_0=h^2(|2\xi_d|^{1/2}g)^{-1}|dx\wedge d\xi' \wedge dg|^{1/2}.\,\,\m
Hence, 
\begin{align*} 
Rr_0&=\frac{h^2}{|2\xi_d|^{1/2}}\{g,f\}^{-1/2}|dx\wedge d\xi' \wedge df|^{1/2}\\
&=\frac{h^2}{2|\xi_d|^{1/2}|\xi|}|dx \wedge d\xi '\wedge (\xi dx-\xi dy+(x-y)d\xi)|^{1/2}.
\end{align*} 

Now, parametrizing of $\Lambda_1$ near $\xi=(0,\ldots,0,E)$ by $(y,\xi',t)$ using the map
$$\Gamma(y,\xi',t)=\left(y+t\xi(\xi'),\xi(\xi'),y,-\xi(\xi')\right)$$
with $\xi(\xi')=\sqrt{E^2-|\xi'|^2}.$
gives 
\m dx\wedge d\xi' \wedge dy_i=\xi_i dy\wedge d\xi' \wedge dt.\,\,\m
Hence, using $\tfrac{E}{\sqrt{E^2-|\xi'|^2}}d\xi'\wedge dt=d\mu_{S_E^{d-1}}(\xi)\wedge dt$,
\begin{align*} R(r_0)(y,\xi',t)&=\frac{h^2}{2(E^2-|\xi'|^2)^{1/4}}|dy\wedge d\xi'\wedge dt|^{1/2}\\
&=\frac{h^2}{2}|E^{-1}dy\wedge d\mu_{S_E^{d-1}}(\xi)\wedge dt|^{1/2}.
\end{align*}

Thus,
$$r_1(y,\theta,t)=\recip{2}e^{\frac{i}{h}\omega_0 tE}e^{\pi i/4}(2\pi)^{1/2}h^{3/2}|E^{-1}dy\wedge d\mu_{S_E^{d-1}}(\xi)\wedge dt|^{1/2}$$
and parametrizing $\Lambda_1$ by $(y,x)$ (instead of $(y,\xi',t)$) for $y\neq x$ gives
$$r_1(x,y)=\frac{E^{(d-3)/2}}{2|x-y|^{(d-1)/2}}e^{\frac{i}{h}\omega_0 |x-y|}e^{(- d+3)\pi i/4}(2\pi)^{1/2}h^{3/2}|dy\wedge dx|^{1/2}.$$
Here, the extra $e^{-\pi (d-2) i/4}$ results from reparametrizing by $x$ instead of $\xi'$, $t$.

Now, taking $\chi\in \Cc(\re^d)$, we have that $R_\chi:=\chi R_0\chi:L^2\to L^2$ continues meromorphically to $\mathbb{C}$ for $d$ odd and to the logarithmic covering space of $\mathbb{C}\setminus \{0\}$ for $d$ even. We show that for $-Ch\log h^{-1}\leq \Im z$, 
\m R_\chi \in h^{3/2}e^{(\Im z)_-D_\chi/h}I^{\comp}_{\gamma}(\re^d;\Lambda_0,\Lambda_1)+h^2\Psi^{-2}+\O{\mc{D}'\to \Cc}(h^\infty).\m 
Moreover, we show that the principal symbol of $R_{\chi}$ is the analytic continuation of that for $\Im z\geq 0$.

To do this, we need the following analog of the three line lemma and semiclassical maximum principle (\cite[Lemma 4.2]{TangZw}, \cite[Lemma 5.1]{Stef}).
\begin{lemma}
\label{lem:semMax}
Suppose that $f(z,h)$ is analytic in 
$$D(h):=[E-5w(h),E+5w(h)]+i[-\alpha(h),\alpha(h)].$$
Let $S(h)=w(h)\alpha^{-1}.$ Assume that $S(h)\to \infty$ and suppose that $|f(z,h)|\leq M_0(h)$ on $$[E-5w(h),E+5w(h)]+i\alpha(h),$$
and $ |f(z,h)|\leq M_1(h)$ on $D(h)$
with $\log \max(M_0(h),M_1(h))=\o{}(S(h)^2).$
Then, 
$|f(z,h)|\leq C M_0^{\frac{\Im z+\alpha}{2\alpha}}M_1^{\frac{\alpha-\Im z}{2\alpha}}$ for $$z\in[E-2w(h),E+2w(h)]+i[-\alpha(h),\alpha(h)]=:\tilde{D}(h).$$
\end{lemma}
\begin{proof}
We follow the proof of \cite[Lemma 4.2]{TangZw}. First, define 
\begin{gather*} 
g(z,h)=(\pi \alpha^2)^{-1/2}\int e^{-\frac{(x-z)^2}{\alpha^2}}\psi_h(x)dx\\
 \psi_h(x)=\begin{cases} 0 &|x-E|\geq 3w(h)\\
1&|x-E|\leq 2w(h)\end{cases}.\end{gather*} 
Then $|g(z,h)|$ is holomorphic in $D(h),$ $|g(z,h)|\geq C$ in $\tilde{D}(h),$ 
$|g(z,h)|\leq C$ in $D(h),$
 and
\m |g(z,h)|\leq Ce^{-CS(h)^2}\text{ on }D(h)\cap |\Re z-E|\geq 4w(h).\,\,\m

Let 
\m F(z,h):=g(z,h)f(z,h)M_0^{-\frac{i(z-i\alpha)}{2\alpha}}M_1^{\frac{i(z+i\alpha)}{2\alpha}}.\,\,\m 
Then $|F(z,h)|\leq 1$ on $ \partial D$ by our assumptions.
By the maximum principle $|F(z,h)|\leq 1$ on $D$. Together with the properties of $g(z,h)$, this gives the result.
\end{proof}

Since for $\Im z\geq 0,$ we have that $\chi( R_0-U)\chi=\O{\mc{D}'\to \Cc}(h^\infty)$, in order to apply Lemma \ref{lem:semMax} to our situation, we need to bound $\chi(R_0-U)\chi$ for $\Im z\leq 0$. In particular, we show that for $\Im z\leq 0$, there exists $N>0$ such that 
\m \chi R_0-U\chi=\O{\mc{D}'\to \Cc}(h^{-N}e^{\frac{1}{h}D_\chi(\Im z)_-}).\m 

Let $\psi\in \Cc(\re)$ have $\psi\equiv 1$ on $|s|<2.$ Then by Lemmas \ref{lem:isecDecomposition} and \ref{lem:freeHighFreq},  
\m \chi (R_0-U)\chi (1-\psi(|hD|))\in h^\infty \Psi_\gamma^{-2} = \O{}(h^\infty)_{\mc{D}'\to \Cc}.\,\,\m 
Now, $\chi \psi(|hD|)U\chi \in h^{3/2}e^{1/hD_\chi (\Im z)_-}I^{\comp}_{\gamma}$. Thus, we see from the definition of an intersecting Lagrangian distribution (Definition \ref{def:isecLagDef}) that there exists $N>0$ such that 
\m \chi \psi(|hD|)U\chi=\O{\mc{D'}\to \Cc}(h^{-N} e^{\frac{1}{h}D_\chi(\Im z)_-}).\m 
This together with standard bounds on the free resolvent (see for example \cite[Theorem 1.2]{Burq}, \cite[Chapter 3]{ZwScat}) gives that 
$$ \chi ( R_0 - U ) \chi = \O{\mc{D'}\to \Cc} ( h^{-N}  e^{\frac{1}{h}D_\chi(\Im z)_-} ).$$

By Lemma \ref{lem:semMax} with 
$$\alpha(h)=Ch\log h^{-1}\,, \,\,\, w(h)=h^{1-\gamma}\,,\,\,\,  M_0=\O{}(h^\infty)\,, \,\,\, M_1=h^{-N}e^{D_\chi(\Im z)_-/h}\,,$$
we have that for $|\Im z|\leq Ch\log h^{-1}$
\begin{equation}
\label{eqn:semMax}\chi (R_0-U)\chi =\O{\mc{D}'\to \Cc}(h^\infty e^{D_\chi(\Im z)_-/h})=\O{\mc{D}'\to \Cc}(h^\infty).
\end{equation}
\end{proof}

\section{Boundary layer operators and potentials away from glancing}
\label{sec:decomposed2}
We now use Theorem \ref{thm:freeResolve} to give microlocal descriptions of the boundary layer operators. We start with the single layer operator. Since the geometry of the situation is the same for all of the boundary layer operators only the computation of symbols will need to be repeated for the $\Dl$ and $\dDl$.
\subsection{Decomposition of $G$}

Recall that $G(z)=\gamma R_0(z/h) \gamma^*$ where $\gamma$ denotes restriction to $\partial\Omega$ and $R_0(\lambda)$ denotes the free outgoing resolvent of $-\Delta -\lambda^2$. We have that $\gamma$ is a semiclassical Fourier integral operator of order $0$ associated to the relation $C\subset T^*\partial\Omega\times T^*\re^d$ given by
$$C=\{(x,\xi',x,\xi):x\in \partial\Omega,\,\xi=\xi_\nu+\xi'\text{ with }\xi_{\nu}\in N_x^*(\partial\Omega),\,\, \xi'\in T^*_x\partial\Omega\}.$$
Thus, $\gamma^*$ is a Fourier integral operator of order $1/4$ associated to the relation $C^{-1}\subset T^*\re^d\times T^*\partial\Omega$.
Then $\gamma\in h^{-1/4}(I^0(C)) $ and $\gamma^*\in h^{-1/4}(I^0(C^{-1}))$ have symbols given by 
$$\sigma(\gamma^*)|_{C^{-1}}=(2\pi h)^{-1/4}|dy\wedge d\xi|^{1/2}\,,\quad \text{ and }\quad \sigma(\gamma)|_C=(2\pi h)^{-1/4}|dx\wedge d\eta|^{1/2}.$$

When $\Omega$ has a smooth boundary, we decompose $G(z)$ into three parts: $G_\Delta$, $G_B$, and $G_g$. $G_\Delta$ is a mildly exotic pseudodifferential operator of order $-1$. $G_B$ is a semiclassical FIO of order $-1$ associated to the billiard ball relation. $G_g$ is an operator $G_g$ microsupported in an $h^\e$ neighborhood of $S^*\partial\Omega\times S^*\partial \Omega$ intersected with the diagonal of $T^*\partial\Omega\times T^*\partial\Omega$.

We now decompose $G$ as claimed above. We begin by showing that the compositions  $C\composed (\Lambda_i)'\composed C^{-1}$ are clean away from the diagonal or away from $S^*\partial\Omega$. First, consider $C\composed (\Lambda_0)'\composed C^{-1}$. We need only work locally, so we assume that $\partial\Omega=\{(x',\Gamma(x'):x'\in U\}.$ Then, 
\begin{align*} C=&\{(x',\xi, (x',\Gamma(x')),(\xi-\tau \nabla\Gamma(x')\, ,\,\nabla \Gamma(x')\cdot \xi+\tau))\,:\, \tau \in \re, x'\in U\}.\\
TC=&\{(\delta_{x'},\delta_\xi, (\delta_{x'}, \nabla\Gamma(x')\cdot \delta_{x'}),\\
&(\delta_\xi -\tau \partial^2\Gamma(x')\delta_{x'}-\delta_{\tau}\nabla\Gamma(x')\,,\,\nabla\Gamma(x')\cdot \delta_\xi +\delta_{x'}\cdot \partial^2\Gamma(x')\xi'+\delta_{\tau})\}
\end{align*}
and $C^{-1}$ and $TC^{-1}$ are obtained by reversing the roles of $T^*\partial\Omega$ and $T^*\re^d$. Then, it is easy to check that $C\composed \Lambda_0$ is clean (indeed, even transverse) and given by 
$C\composed \Lambda_0=C.$ Now, without loss of generality, we can assume that $\nabla \Gamma(y')=\Gamma(y')=0$. so 
\begin{gather*}\begin{aligned}  A&:=(C\times C^{-1})\cap (T^*\partial\Omega \times \Delta(T^*\re^d)\times T^*\partial\Omega)\\
&=\{(y',\eta,(y',0),(\eta,0)+\tau (0,1),y',\eta) \}
\end{aligned}\\
\begin{aligned}TA&=\{(\delta_{y'},\delta_\eta,(\delta_{y'},0), (\delta_\eta-\tau\partial^2\Gamma(x')\delta_{x'},\delta_{y'}\partial^2\Gamma \eta+\delta_{\tau})\}.
\end{aligned}
\end{gather*}

\begin{remark} 
Since we intersect with the diagonal in these formulae, we have suppresed one of the pairs in $T^*\re^d$. 
\end{remark}

\noindent On the other hand $B=TC\times TC^{-1}\cap (T(T^*\partial\Omega\times \Delta(T^*\re^d)\times T^*\partial\Omega))$ at $(y',\eta,(y',0), (\eta,0)+\sigma(0,1), y',\eta)$ is given by
$$B=\{(\delta_{x'},\delta_\xi, (\delta_{y'},0),(\delta_\eta-\sigma\partial^2\Gamma(y')\delta_{y'}\,,\,\delta_{y'}\partial^2\Gamma\eta +\delta_{\sigma}),\delta_{y'},\delta_\eta )\}$$
where
\begin{multline*} 
(\delta_\eta,0)+(0,\delta_{y'}\partial^2\Gamma\eta) +\sigma(-\partial^2\Gamma(y')\delta_{y'},0)+\delta_\sigma(0,1)\\=(\delta_\xi,0)+(0,\delta_{x'}\partial^2\Gamma\eta) +\sigma(-\partial^2\Gamma (y')\delta_{x'},0)+\delta_\tau(0,1)\end{multline*}
and
$\delta_{y'}=\delta_{x'}$. But, since $(0,1)$ is linearly independent from $T_y\partial\Omega $, this implies that $\delta_\eta=\delta_\xi$, $\delta_\tau=\delta_\sigma$ and hence the composition is clean.

Now, recall that 
\begin{gather*}
\Lambda_1=\{(x+t\xi,\xi,x,\xi)\,:\,\xi \in S^{d-1}\,,t\geq 0\}
\end{gather*}
we consider or $\xi\notin T^*\partial\Omega$ Thus,
\begin{gather*}
T\Lambda_1=\{(\delta_x+t\delta_{\xi}+\delta_t\xi,\delta_\xi,\delta_x,\delta_\xi)\,:\delta_{\xi}\in T_\xi S^{d-1}\} 
\end{gather*}
To see that $T(\Lambda_1\times C^{-1})\cap T(T^*\re^d\times \Delta(T^*\re^d)\times T^*\partial\Omega)$ is transverse at 
$$((y',\Gamma(y'))+t\xi, \xi,(y',\Gamma(y')),\xi,y',\eta)$$
 where $\xi-\eta \in N^*_{y'}\partial\Omega$, we choose $\delta_z=\alpha(-\nabla \Gamma(y'),1)$ and $\delta_\zeta=\beta \xi.$ Then for any $v\in \re^d$, $v=\beta \xi+\delta_\xi$ for some $\beta\in \re$ and $\delta_\xi\in T_\xi S^{d-1}$. Moreover, any $w\in \re^d$ can be written $w=\delta_{y'}+\alpha(-\nabla \Gamma(y'),1)$ for some $\delta_{y'}\in T_{y'}\partial\Omega$ and $\alpha\in \re$
Thus, 
$$T(\Lambda_1\times C^{-1})+T(T^*\re^d\times \Delta(T^*\re^d)\times T^*\partial\Omega)=T(T^*\re^d\times T^*\re^d\times T^*\re^d\times T^*\partial\Omega) $$ 
and the composition is transverse. Now
\begin{gather*}
\Lambda_1\composed C^{-1}=\{((y',\Gamma(y'))+t\xi,\xi, y',\eta)\,:\, t\geq 0,\, \xi\in S^{d-1},\, \xi-\eta \in N^*_{y'}\partial\Omega\} \\
T(\Lambda_1\composed C^{-1})=\left\{\begin{gathered}((\delta_{y'},\nabla\Gamma(y'))+\delta_t\xi+t\delta_\xi,\delta_\xi,\delta_{y'},\delta_\eta) \,:\,\\ \delta_\eta=d\pi \delta_\xi\,,\,\delta_{\xi}\in T_\xi S^{d-1} \end{gathered}\right\}
\end{gather*}
Now, if $t>0$, it is clear that any vector $w\in \re^d$ can be written $w=\delta_t\xi+t\delta_{\xi}$. On the other hand, if $t=0$, but $\xi\notin T_y^*\partial\Omega$, then we have that $w$ can be written as 
$$w=(\delta_{y'},\nabla\Gamma(y')\cdot \delta_{y'})+\delta_t\xi.$$
Moreover, parametrizing $\partial\Omega$ near a point $x$ in the intersection with $C$ by $(x',\Gamma_1(x'))$, $w$ can be written
$$w=\delta_{\zeta}+\tau(-\partial^2\Gamma_1\delta_{x'},0)+\delta_\tau(-\nabla \Gamma_1(x'),1)$$
for $\delta_\zeta \in TT^*_y\partial\Omega$. So, an identical analysis to that for the composition on the right by $C^{-1}$ gives that $C\composed \Lambda_1\composed C^{-1}$ is transverse away from the diagonal as well as at the diagonal, but away from $T^*\partial\Omega$.

Since $\partial\Omega\Subset \re^d$, we may take $\chi \equiv 1$ on $\Omega$ in Theorem \ref{thm:freeResolve}. Then by composing relations, using Lemma \ref{lem:fioIsectLagrangian}, and observing that the composition is transverse, we see that for $-3/2<s$,
\begin{equation}
\label{eqn:oneSideRestrict}
 R_\chi\gamma^* \in h^{5/4}I_\gamma^{\comp}( \re^d\times \partial\Omega; \Lambda_0\composed C^{-1}, \Lambda_1\composed C^{-1})+h^{7/4}I_\gamma^{-2}( \Lambda_0\composed C^{-1})+R_1
\end{equation}
where $R_1=\O{\mc{D}'(\partial\Omega)\to C^\infty(\re^d)}(h^\infty)$.
\begin{remark} This implies that the single layer potential has the above decomposition.
\end{remark}

We have that $C$ composes on the left with $\Lambda_1\composed C^{-1}$ transversally. However, $C$ composes on the left with $ \Lambda_0 \composed C^{-1}$ only cleanly. Thus, we cannot apply Lemma \ref{lem:fioIsectLagrangian} in this case to obtain $\gamma R_\chi \gamma^*=\gamma R_0\gamma^*$. Note also that Lemma \ref{lem:FIOcomp} (or rather its homogeneous analog) does not apply directly  to the composition forming $C\composed \Lambda_0\composed C^{-1}$ since  there exist $(x,\xi,y,\eta)\in \Lambda_0\composed C^{-1}$ such that $C(x,\xi)=(x,0)$. Instead we use the following lemma combined with more detailed analysis near fiber infinity.
\begin{lemma}
\label{lem:intersectionPieces}
Suppose that $\partial\Omega$ is smooth and 
$$A\in I^{\comp}(\re^d\times \partial\Omega ; \Lambda_0\composed C^{-1}, \Lambda_1\composed C^{-1}).$$
Then 
$\gamma A=A_1+A_2+R$
where $A_1\in h^{-1/4-\delta/2}I_\delta^{\comp}(C\composed \Lambda_1\composed C^{-1})$, $A_2\in h^{-1/4-\delta/2}I_\delta^{\comp}(C\composed \Lambda_0\composed C^{-1})$ and
 $R$ is microlocalized on an $h^\delta$ neighborhood of the intersection of $S_E^*\partial\Omega\times S_E^*\partial\Omega$ with the diagonal. Moreover, the symbol $A_2$ can be computed using Lemma \ref{lem:FIOcomp} in the sense that 
$$\sigma(A_2)=(2\pi h)^{-3/4}\int \sigma(A\psi(|hD|'))|_{\Lambda_0\composed C^{-1}}|d\la \nu_x,\xi\ra |^{1/2}$$
where $\psi$ is supported $h^\delta$ away from $|\xi'|=E$ and the integral is interpreted as a distributional pairing.
\end{lemma}
\begin{proof}
By Lemma \ref{lem:isecDecomposition}, we need only consider an $h^\delta$ neighborhood of the diagonal intersected with $S^*\re^d|_{\partial\Omega}\times S^*\re^d|_{\partial\Omega}$. 
Let $\chi\in \Cc(\re)$ with $\chi \equiv 1$ near $0$. Let \begin{gather*} A_1(x,y)=\chi(|x-y|/h^\gamma)\chi(|x-y|/h^\delta)A(x,y)\\
A_2(x,y)=(1-\chi(|x-y|/h^\gamma))\chi(|x-y|/h^\delta)A(x,y)
\end{gather*}
where $A(x,y)$ is the kernel of $A$. Then, we can write for $B\in \Ph{}{\delta}(\partial\Omega)$
$$B\gamma A_2(x,y) =\begin{aligned}h^{-M}\int_{\re^{2d-1}} e^{\frac{i}{h}(\la x-w,\eta\ra + E|w-y|)}(1-\chi(|w-y|/h^\gamma))&\\
 &\!\!\!\!\!\!\!\!\!\!\!\!\!\!\!\!\!\!\!\!\!\!\!\!\!\!\!\!\!\!\!\!\!b(x,\eta)a(w,y)dwd\eta.\end{aligned}$$
But, since $d_w|w-y|\to 1$ as $w\to y$, we have that the phase is nonstationary with gradient bounded below by $ch^\delta$ if $b$ is supported $h^\delta$ away from $|\eta|=E$. Hence, integrating by parts we lose at most $h^\gamma$ and gain $h^{1-\delta}$, so when $\gamma<1-\delta$, we obtain a kernel in $\O{C^\infty}(h^\infty).$ Similarly, we have the same result for $A_2B$. 

Next, consider $A_1$. Let $B$ be microlocalized $h^\delta$ away from $|\eta|=E$. That is, since $\eta\in T^*\pO$, $h^\delta$ away from the glancing set.
Then the kernel of $BA_1$ can be written 
$$B\gamma A_1(x,y)=\begin{aligned}(2\pi h)^{-d-1/4}\int_0^\infty \!\!\!\!\int e^{\frac{i}{h}\left(\la x-y,\eta+\xi_\nu\nu_x\ra-\frac{t}{2}(|\eta|^2+\xi_\nu^2-E)\right)}&\\&
\!\!\!\!\!\!\!\!\!\!\!\!\!\!\!\!\!\!\!\!\!\!\!\!\!\!\!\!\!\!\!\!\!b(x,\eta)a(x,\xi)d\xi d\eta  dt.\end{aligned}$$
Then, using Lemma \ref{lem:singSymbol} evaluating the $t$ integral as a distribution, we have
$$\gamma A_1(x,y)=\frac{-2i}{(2\pi h)^{d-3/4}} \int e^{\frac{i}{h}\left(\la x-y,\eta+\xi_\nu\nu_x\ra\right)}\frac{b(x,\eta)a(x,\eta+\xi_\nu\nu_x)}{|\eta|^2+\xi_\nu^2-E-i0}d\xi_\nu d\eta $$
$\eta\in T_x\partial\Omega$ and $\nu_x$ is the unit normal to $\partial\Omega$ at $x$. Note that since $||\eta|-E|\geq ch^\delta$ $(\xi_\nu^2+|\eta|^2-E-i0)^{-1}\in h^{-\delta/2}\mc{S}'$ as a distribution in $\xi_\nu$.  We are working in a small neighborhood of $|\xi|=E$, so we can assume that the integrand is compactly supported in $\xi_\nu$. Now, $\la x-y,\nu_y\ra =\O{}(|x-y|^2)$ and $|x-y|=\O{}(h^\gamma)$ with $\gamma>1/2$. So, we obtain an accurate representation using the Taylor expansion of $e^{\frac{i}{h}\la x-y,\nu_x\xi_\nu\ra}$. Then, a typical term is of the form 
$$\frac{-2i}{(2\pi h)^{d-3/4}} \int e^{\frac{i}{h}\la x-y,\eta\ra}\frac{(\la x-y,\nu_x\ra \xi_\nu)^j}{h^jj!}\frac{b(x,\eta)a(x,\eta+\xi_\nu\nu_x)}{|\eta|^2+\xi_\nu^2-E-i0}d\xi_\nu d\eta.$$
So, integrating by parts $2j$ times in $\eta$, we gain $h^{2j}$. Integrating in $\xi_\nu$ gives the result.
\end{proof}

Now, let $\psi \in \Cc(\re)$ with $\psi\equiv 1$ near $0$ and let $0\leq \e<1/2$. Then, writing $R_{\chi}(x,y)$ for the kernel of $R_\chi$, define
$$\begin{aligned}R_\chi(x,y)=&R_\chi(x,y)(1-\psi(h^{-\e}|x-y|))+R_\chi(x,y)\psi(h^{-\e}(|x-y|))\\
=:&R_1(x,y)+R_2(x,y).\end{aligned}$$
Then, recalling that $G:=\gamma R_0\gamma^*$,
\begin{equation}
\label{eqn:decompExact}
\begin{aligned}G=&\gamma R_1\gamma^*+\gamma R_2\gamma^*(1-\psi(h^{-\e}(|hD'|_g-E)))+\gamma R_2\gamma^*\psi(h^{-\e}(|hD'|_g-E))\\
=:&G_B+G_\Delta+G_g.\end{aligned}
\end{equation}
We will see that in spite of the difficulty at fiber infinity, $G_{\Delta}$ is still a pseudodifferential operator. As in Section \ref{sec:LayerPotential}, to interpret $G_{\Delta}$ appropriately, we must view $\gamma$ as one of two objects, $\gamma^{+}$ for the limit from inside $\Omega$ and $\gamma^{-}$ for that from outside $\Omega$. In Lemma \ref{lem:layer} we saw that $G$ is independent of the choice of $\gamma^{\pm}$, so we choose $\gamma^{+}$. 

\begin{lemma}
\label{lem:diagPiece}
Suppose $A\in \Ph{m}{\delta}(\re^d)$. Choose coordinates so that $\partial\Omega=\{x_d=0\}$ and let $a$ have $A=\opht{0}(a)$ for $a\in \Scl{m}{\delta}(T^*\re^d)$. Suppose that 
\begin{equation*} a(y,\xi)\sim \sum_{j=-\infty}^mC_{j,\pm}(y,\xi')|\xi_d|^j\quad \quad \xi_d\to \pm \infty \end{equation*}
and for $j>-2$, $C_{j,+}=(-1)^jC_{j,-}$. Then, $\gamma^{\pm} A\gamma^*\in h^{-1}\Ph{m+1}{\delta}(\partial\Omega)$.

Moreover, for such operators $A$, the symbol calculus contained in Lemma \ref{lem:FIOcomp} applies in the sense that 
$$\sigma(\gamma^{\pm}A\gamma^*)=(2\pi h)^{-1}\int \sigma(A)(x,\xi'-\nu_x\xi_d)d\xi_d$$
where $\nu_x$ is the outward unit normal to $\partial\Omega$ and the integral is interpreted as $\pm 2\pi i$ times the sum of residues in $\pm \Im \xi_d>0$ if $\sigma(A)$ is not integrable.
\end{lemma}
\begin{proof}
Let 
$$q(y,\xi)=\frac{\sum_{j=-1}^m C_{j,+}(y,\xi')\xi_d^{j+2}}{|\xi|^2+1}$$
and $r=a-q$. 
Then, $A\gamma^*$ has kernel 
$$(2\pi h)^{-d}\iint  e^{\frac{i}{h}(x_d\xi_d+\la x'-y',\xi'\ra)}(q(y',\xi)+r(y',\xi))d\xi_dd\xi'.$$
The integral in involving $r$ in $\xi_d$ is well defined at $x_d=0$ since $|r|\leq C|\xi_d|^{-2}$ for $|\xi_d|$ large. Moreover, since for $n<-1$, 
\begin{align*} 
\int \la \xi\ra^{n}d\xi_d\leq \la \xi'\ra^{n}\int \la \xi_d\la \xi'\ra^{-1}\ra^{n}d\xi_d\leq \la \xi'\ra^{n+1}
\end{align*}
$\gamma^* \opht{0}(r)\gamma \in h^{-1}\Ph{m+1}{\delta}(\partial\Omega)$. 

Now, consider $q$. In this case, we must take a limit as $x_d\to 0$ from above or below since the integral is not apriori well defined.
Consider 
$$u=(2\pi h)^{-1}\int e^{ix_d\xi_d}q(y',\xi')d\xi_d \in \mc{S}'(\re).$$
Let $f_{\pm}\in \Cc(\re_{\pm})$ and write
\begin{align} 
u(f)&=(2\pi h)^{-1}\iint e^{\frac{i}{h}x_d\xi_d}q(y',\xi)f(x_d)dx_d d\xi_d\nonumber\\
&=(2\pi h)^{-1}\iint e^{\frac{i}{h}x_d\xi_d}\la \xi_d\ra^{-2k}\left(1-(h\partial_{x_d})^2\right)^k(q(y',\xi)f(x_d))dx_dd\xi_d\nonumber\\
&=(2\pi h)^{-1}\iint e^{\frac{i}{h}x_d\xi_d}\la \xi_d\ra^{-2k}q(y',\xi)(1-h^2\partial^2_{x_d})^{k}f d\xi_ddx_d.\nonumber\end{align}
So, since $q(y,\xi)$ is rational in $\xi_d$, applying Jordan's lemma, we have
\begin{align*}
u(f)&=\frac{i}{h}\int e^{\mp\frac{x_d}{h}\sqrt{|\xi'|^2+1}}\la \sqrt{|\xi'|^2+1}\ra^{-2k}\\
&\quad\quad\quad\frac{\sum_{j=-1}^mC_{j,+}(y',\xi')(\pm i\sqrt{|\xi'|^2+1})^j}{\pm 2i\sqrt{|\xi'|^2+1}}(1-h^2\partial_{x_d}^2)^{k}fdx_d\label{eqn:Jordan}\\
&=\pm \frac{i}{h}\int e^{\mp\frac{x_d}{h}\sqrt{|\xi'|^2+1}} \frac{\sum_{j=-1}^mC_{j,+}(y',\xi')(\pm i\sqrt{|\xi'|^2+1})^j}{\pm 2i\sqrt{|\xi'|^2+1}}f(x_d)dx_d.\nonumber
\end{align*}
Now, let $f_n\to \delta_0$ $f_n\in \Cc(\re_\pm)$. Then, we have 
$$\lim_{\pm x_d\downarrow 0}u(x_d)=\pm \frac{i}{h}\frac{\sum_{j=-1}^mC_{j,+}(y',\xi')(\pm i\sqrt{|\xi'|^2+1})^j}{\pm 2i\sqrt{|\xi'|^2+1}}\in h^{-1}S^{m+1}_\delta.$$
So, we have that $\gamma^{\pm}\opht{0}(q)\gamma^*\in h^{-1}\Ph{m+1}{\delta}(\partial\Omega)$ as desired.
\end{proof}

Applying Lemmas \ref{lem:intersectionPieces} and \ref{lem:diagPiece}, we have that away from glancing or the diagonal, $\gamma R_0\gamma^*$ is composed of a Fourier integral operator, $G_B$, associated with the relation
\begin{equation*}
C_b:=\left\{\begin{gathered}(\pi(\exp_{t}(x,\xi)),x,\xi'_2)\\
:(x,\xi_2')\in B^*\partial\Omega, \, (x,\xi)\in\pi^{-1}(x,\xi_2),\,
t\geq 0,\, \exp_t(x,\xi)\in S_E^*\re^d|_{\partial\Omega}\end{gathered}\right\}.
\end{equation*}
and a pseudodifferential operator  $G_{\Delta}$.
Here $\pi$ is orthogonal projection $S_E^*\re^d|_{\partial\Omega}\to \overline{B_E^*\partial\Omega}$ and $S_E^*\re^d$ and $B_E^*\partial\Omega$ are respectively the cosphere and coball bundles of radius $E$.

\begin{remark} Note that in the case $\Omega$ is strictly convex, $C_b$ is parametrized by $\beta_E$ where 
\m \beta_E(x,\xi')=(\pi_x\composed \beta(x,\xi'/E),E\,\pi_\xi \composed \beta(x,\xi'/E)).\,\,\m
\end{remark}

\noindent Next, observe that by \eqref{eqn:oneSideRestrict}, when we compose $R_{\chi}\gamma^*$ on the left by $\gamma$, the remainder term is 
$\O{\mc{D}'\to C^\infty}(h^\infty)$ as desired.

Putting this together, we have
\m G(z):=G_{\Delta}(z)+G_B(z)+G_g(z) +\O{\mc{D}'\to C^\infty}(h^\infty)\,\,\m 
where $G_{\Delta}$ is pseudodifferential, $G_B$ is a Fourier integral operator associated with the relation $C_b$, and $G_g$ has $\MS(G_g)\subset U_h\times U_h\cap V_h$ where $U_h$ is an $h^{\e}$ neighborhood of $S_E^*\partial\Omega$ and $V_h$ is an $h^\e$ neighborhood of the diagonal of $\partial\Omega\times \partial\Omega$ lifted to $T^*\partial\Omega$. Moreover, if $\Omega$ is strictly convex, $G_B$ is associated to the billiard ball map.

Next, we compute the symbols of $G_{\Delta}$ and $G_B$. Using Lemmas \ref{lem:intersectionPieces}and \ref{lem:diagPiece} we have 
  \begin{align*} 
  \sigma(G_{\Delta})&=(2\pi h)^{-1/2}\sigma(\gamma)\int r_0 \sigma(\gamma^*)d\xi_d\\
  &=(2\pi h)^{-1}\int h^2\left(|\xi|^2-(E+i0)^2\right)^{-1}d\xi_d\\
  &=\frac{ih}{2}\left(E^2-|\xi'|^2_g\right)^{-1/2}\end{align*}
Here we take the branch of the square root such that $\sqrt{a}$ is positive on $a>0$. This choice is unambiguous since $\text{Arg}(E^2-|\xi'|_g^2)\in\{0,\pi\}.$
Thus,
$$\sigma(G_\Delta)=\sigma(G)|_{C\composed \Lambda_0\composed C^{-1}}=ih\left(2\sqrt{E^2-|\xi'|_g^2}\right)^{-1}.$$

\begin{remark} Note that the symbol of $G_\Delta$ is the same as that if we had naively applied Lemma \ref{lem:FIOcomp}.
\end{remark}

Note also that using the transversality of the intersection $C\composed \Lambda_1\composed C^{-1}$, Lemma \ref{lem:FIOcomp} gives that for $x,\,y\in \pO$,
\begin{equation} 
\label{eqn:symbNonConv}\sigma(e^{\frac{\Im z}{h}|x-y|}G_B)|_{C_b'}=\frac{hE^{(d-3)/2}e^{(-d+3)\pi i/4}e^{\frac{i}{h}\Re\omega_0|x-y|}}{2|x-y|^{(d-1)/2}}|dy\wedge dx|^{1/2}.\end{equation}
Then, assuming that $\Omega$ is strictly convex so that $C_b$ is parametrized by $\beta_E$, we have by composing symbols (see also \cite[Proposition 6.1]{HaZel}) that 
\begin{lemma}
\label{lem:billiardHalfDensityComputation}
Let $q=(y,\eta)\in B_y^*\partial\Omega $. Then $G_B$ has symbol
\begin{multline}\sigma(G_Be^{\frac{\Im z}{h}\oph(l(q,\beta_E(q))}\varphi(h^{-\e}(|hD'|_g-E))|_{C_b'}\\
=\frac{he^{\frac{i}{h}\Re \omega_0l(q,\beta_E(q))}\varphi(h^{-\e}(|\eta(q)|_g-E))}{2(E^2-|\eta( \beta_E (q))|_g^2)^{1/4}(E^2-|\eta(q)|_g^2)^{1/4}}dq^{1/2}.\label{eqn:symBilliard}\end{multline}
 $\chi\in C^\infty(\re)$ has $\chi\equiv 0 $ near 0 and $\varphi=1-\chi$.
 \end{lemma}
 \begin{proof}
 To convert from \eqref{eqn:symbNonConv} to \eqref{eqn:symBilliard}, we reparametrize by $(y,\eta )$. That is, we write $|d\xi'|$ in terms of $|dy|$. Observe that by \eqref{eqn:locBillRelation} $\eta=E d_y|y-x|$ on $C_\beta$. Thus, we compute
 $$|d\eta |=E^{d-1}\det\left(\left. \frac{\partial^2}{\partial s_i\partial t_j}\right|_{\substack{\mathbf{t}=0\\\mathbf{s}=0}}\left|y+\sum_is_ie_i-(x+\sum_it_ie_i')\right|\right)|dx|$$
 where $e_i$ and $e_i'$ ($i=2,\dots,d$) are respectively orthonormal bases for $T_y^*\partial\Omega$ and $T_x^*\partial\Omega$. 
 Without loss of generality, we assume that 
 \begin{equation*}
 \label{eqn:normAssumption}
 \begin{gathered}
 \nu_y=(1,0,0\dots, 0),\quad \quad \nu_x=(\cos \beta, \sin \beta ,0,\dots 0)\\
 y=(0,0,0,\dots ,0),\quad\quad x=(r_1,r_2,r_3,0,\dots 0)=\mathbf{r}
 \end{gathered}
 \end{equation*}
 Then we choose as our orthonormal bases $e_i=\mathbf{e}_i$ $i=2,\dots d$ where $\mathbf{e}_j$ is the standard basis and
 \begin{gather*}
 e_2'=(-\sin \beta, \cos \beta,0,\dots, 0)\,,\quad\quad e_i'=e_i\text{ for }i=3,\dots d
 \end{gather*}
 Next we compute derivatives of $w=w(s_2,\dots, s_{d},t_2,\dots,t_{d})$
 $$w=|(t_2\sin \beta-r_1,s_2-r_2-t_2\cos \beta, s_3-r_3-t_3,s_4-t_4,\dots, s_d-t_d)| $$
 A long but straightforward computation gives
 \begin{multline*}|\mathbf{r}|^3\left.\frac{\partial^2w}{\partial s_i\partial t_j}\right |_{\substack{\mathbf{t}=0\\\mathbf{s}=0}}=
\\
\begin{pmatrix} 
-|\mathbf{r}|^2\cos \beta+r_2(r_2\cos \beta-\sin \beta r_1)&r_3(r_2\cos \beta -\sin \beta r_1)&0\\
r_2r_3&-|\mathbf{r}|^2+r_3^2&0\\
0&0&-|\mathbf{r}|^2I
\end{pmatrix}
.\end{multline*} 
This matrix has 
\begin{align*} \left|\det \left(\left.\frac{\partial^2w}{\partial{s_i}\partial t_j}\right|_{\substack{\mathbf{t}=0\\\mathbf{s}=0}}\right)\right|&=|\mathbf{r}|^{-d+1}\left|\frac{r_1^2\cos \beta+r_1r_2\sin \beta}{|\mathbf{r}|^{2}}\right|\\
&=|\mathbf{r}|^{-d+1}\left |\partial_{\nu_x}|y-x|\right|\left|\partial_{\nu_y}|x-y|\right|\\
&=|x-y|^{-d+1}E^{-2}\sqrt{E^2-|\eta|_g^2}\sqrt{E^2-|\xi'(\beta_E(q))|_g^2}
\end{align*}
 \end{proof}
 
\begin{remark} If $\Omega$ is strictly convex, then the cutoff away from the diagaonal in $R_1$ causes $G_B$ to be microlocalized $h^\e$ away from $|\xi'|_g=E$ and hence $G_B\varphi(h^{-\e}(|hD'|_g-E))=G_B+\O{}(h^\infty)$.
\end{remark}

Now, to understand $\Dl=\gamma^+R_0L^*\gamma^*+\frac{1}{2}\Id$ and $\dDl=\gamma L R_0L^* \gamma^*$ microlocally away from glancing, we only need to compute the symbols of the various pieces since the geometry of the situation is identical to that for $G$. For $\dDl$, it is irrelevant whether we choose $\gamma^+$ or $\gamma^-$ since we have verified that there is no jump at $\partial\Omega$ in Lemma \ref{lem:layer}.
Write $\overline{\Dl}=\gamma^+R_0L^*\gamma^*$. Then for $\overline{\Dl}$, we write 
\begin{align} \overline{\Dl} &=\gamma R_1L^*\gamma^*+\gamma^+ R_2 L^*\gamma^*(1-\psi(h^{-\e}(|hD'|_g-E))\nonumber\\
&\quad\quad\quad+\gamma^+R_2L^*\gamma^*\psi(h^{-\e}|hD'|_g-E)\nonumber\\
\label{eqn:DLCancel}&=:\Dl_B+\Dl_\Delta+\Dl_g+\frac{1}{2}\Id
\end{align}
Also, write
\begin{align*} \dDl &=\gamma LR_1L^*\gamma^*+\gamma^+ L R_2 L^*\gamma^*(1-\psi(h^{-\e}(|hD'|_g-E))\\
&\quad\quad\quad+\gamma^+LR_2L^*\gamma^*\psi(h^{-\e}|hD'|_g-E)\\
&=:\dDl_B+\dDl_\Delta+\dDl_g
\end{align*}
The symbol of $\dDl_B$ is given by
\begin{multline*}\sigma(\Dl_Be^{\frac{\Im z}{h}|x-y|})=\\\frac{E^{(d+1)/2}e^{(-d+3)\pi i/4}e^{\frac{i}{h}\Re \omega_0|x-y|}}{2|x-y|^{(d-1)/2}}d_{\nu_y}|x-y||dy\wedge dx|^{1/2}\end{multline*}
and using the computations from Lemma \ref{lem:billiardHalfDensityComputation}
$$\sigma(\Dl_Be^{\frac{\Im z}{h}\oph(l(q,\beta_E(q)))})=\frac{e^{\frac{i}{h}\Re\omega_0l(q,\beta_E(q))}(E^2-|\xi'(q)|_g^2)^{1/4}}{2(E^2-|\xi'(\beta_E(q))|_g^2)^{-1/4}}dq^{1/2}.
$$
Then, the symbol of $\dDl_B$ is given by
\begin{multline*}\sigma(\dDl_Be^{\frac{\Im z}{h}|x-y|})=\\\frac{h^{-1}E^{(d+1)/2}e^{(-d+3)\pi i/4}e^{\frac{i}{h}\Re \omega_0|x-y|}}{2|x-y|^{(d-1)/2}}d_{\nu_x}|x-y|d_{\nu_y}|x-y||dy\wedge dx|^{1/2}\end{multline*}
and using the computations from Lemma \ref{lem:billiardHalfDensityComputation}
\begin{multline*}\sigma(\dDl_Be^{\frac{\Im z}{h}\oph(l(q,\beta_E(q)))})=\\\frac{h^{-1}e^{\frac{i}{h}\Re\omega_0l(q,\beta_E(q))}(E^2-|\xi'(\beta_E(q))|_g^2)^{1/4}(E^2-|\xi'(q)|_g^2)^{1/4}}{2}dq^{1/2}.
\end{multline*}

To analyze $\Dl_{\Delta}$ and $\dDl_\Delta$,  write 
\begin{multline*} R_2L^*\gamma^*(1-\psi(h^{-\e}(|hD'|-1)))(x,y)\\
=(2\pi h)^{-d}\int e^{\frac{i}{h}\la x-y,\xi\ra} \frac{-i\la \xi,\nu_y\ra +p_1(x,y,\xi)}{|\xi|^2-E^2-i0}(1-\psi(h^{-\e}(|\xi'|_g-E)))d\xi
\end{multline*}
and 
\begin{multline*} LR_2L^*\gamma^*(1-\psi(h^{-\e}(|hD'|-1)))(x,y)\\
=(2\pi h)^{-d}\int e^{\frac{i}{h}\la x-y,\xi\ra} \frac{\la\xi,\nu_x\ra \la \xi,\nu_y\ra +p_2(x,y,\xi)}{|\xi|^2-E^2-i0}(1-\psi(h^{-\e}(|\xi'|_g-E)))d\xi
\end{multline*}
where the $p_i$ are polynomial in $\xi$.
Then, in appropriate coordinates
\begin{multline*} R_2L^*\gamma^*(1-\psi(h^{-\e}(|hD'|-1)))(x,y)\\
=(2\pi h)^{-d}\int e^{\frac{i}{h}\la x-y,\xi\ra} \frac{-i\xi_d +p(x,y,\xi)}{|\xi|^2-E^2-i0}(1-\psi(h^{-\e}(|\xi'|_g-E)))d\xi
\end{multline*}
and
\begin{multline*} LR_2L^*\gamma^*(1-\psi(h^{-\e}(|hD'|-1)))(x,y)\\
=(2\pi h)^{-d}\int e^{\frac{i}{h}\la x-y,\xi\ra} \frac{\xi_d^2 +p(x,y,\xi)}{|\xi|^2-E^2-i0}(1-\psi(h^{-\e}(|\xi'|_g-E)))d\xi
\end{multline*}
Hence, the relevant parts of $R_2L^*\gamma^*$ and $LR_2L^*\gamma^*$ satisfy the requirements of Lemma \ref{lem:diagPiece}. When we compute the symbol of $\gamma^+R_2L^*\gamma^*$, we obtain $\frac{1}{2}$ which is exactly the $\frac{1}{2}\Id$ appearing in \eqref{eqn:DLCancel}. Hence, we can compute symbols to obtain:

For the case that $\Omega$ is strictly convex, we summarize the result of this decomposition in the following Lemma
\begin{lemma}
\label{lem:decompose}
\label{lem:decomposeDl}
\label{lem:decomposePrime}
Let $\Omega\subset \re^d$ be strictly convex with $\partial\Omega\in C^\infty$. Then for all $1/2>\e,\gamma >0$, and $z=E+\O{}(h^{1-\gamma})$ with $\Im z\geq -Ch\log h^{-1}$. Then\begin{gather*} 
G(z/h):=G_{\Delta}(z)+G_B(z)+G_g(z)+\O{\mc{D}'\to C^\infty}(h^\infty)\\
\Dl(z/h):=\Dl_{\Delta}(z)+\Dl_B(z)+\Dl_g(z)+\O{\mc{D}'\to C^\infty}(h^\infty)\\
\dDl(z/h):=\dDl_{\Delta}(z)+\dDl_B(z)+\dDl_g(z)+\O{\mc{D}'\to C^\infty}(h^\infty)
\end{gather*}
where $G_{\Delta}\in h^{1-\frac{\e}{2}}\Psi_{\e}^{-1}$, $\Dl_{\Delta}\in h^{1-2\e}\Ph{-1}{\e}$, $\dDl_{\Delta}\in h^{-1}\Ph{1}{\e}$, and  $G_B\in h^{1-\frac{\e}{2}}e^{(\Im z)_-d_{\Omega}/h}I^{comp}_{\delta}(C_b)$, $\Dl_B\in e^{(\Im z)_-d_{\Omega}/h}I^{\comp}_{\delta}(C_b)$,  and $\dDl_B\in h^{-1}e^{(\Im z)_-d_{\Omega}/h}I^{comp}_{\delta}(C_b)$ are FIOs associated to $\beta_E$ where $\delta=\max(\e,\gamma)$.
Moreover, 
\begin{gather*}{\MS}'((\cdot)_B)\subset \left\{\begin{gathered}(q,p)\in B_E^*\partial\Omega\times B_E^*\partial\Omega:\\\min(E-|\xi'(q)|_g,\,E-|\xi'(q)|_g\,,\,l(q,p))>ch^\e\end{gathered}\right\}\\
{\MS}'((\cdot)_g)\subset \left\{\begin{gathered}(q,p)\in T^*\partial\Omega\times T^*\partial\Omega:\\\max(|E-|\xi'(q)|_g|,\,|E-|\xi'(p)|_g|, l(q,p))<ch^\e\end{gathered}\right\}
\end{gather*}
\begin{gather*}
\sigma(G_{\Delta})=\frac{ih}{2\sqrt{E^2-|\xi'|_g^2}}\,,\quad \quad \sigma(\dDl_{\Delta})=\frac{ih^{-1}\sqrt{E^2-|\xi'|^2_g}}{2},\\
\sigma(G_Be^{\frac{\Im z}{h}\oph(l(q,\beta_E(q)))})=\frac{he^{\frac{i}{h}\Re\omega_0l(q,\beta_E(q))}}{2(E^2-|\xi'(\beta_E(q))|_g^2)^{1/4}(E^2-|\xi'(q)|_g^2)^{1/4}}dq^{1/2},\\
\sigma(\Dl_Be^{\frac{\Im z}{h}\oph(l(q,\beta_E(q)))})=\frac{e^{\frac{i}{h}\Re\omega_0l(q,\beta_E(q))}(E^2-|\xi'(q)|_g^2)^{1/4}}{2(E^2-|\xi'(\beta_E(q))|_g^2)^{1/4}}dq^{1/2},\\
\begin{aligned}\sigma(\dDl_Be^{\frac{\Im z}{h}\oph(l(q,\beta_E(q)))})=&\\
&\!\!\!\!\!\!\!\!\!\!\!\!\!\!\!\!\!\!\!\!\!\!\!\!\!\!\!\!\!\!\!\!\!\!\!\!\!\!\!\!\!\!\!\!\!\!\!\!\!\!\!\!\!\!\!\frac{h^{-1}e^{\frac{i}{h}\Re\omega_0l(q,\beta_E(q))}(E^2-|\xi'(\beta_E(q))|_g^2)^{1/4}(E^2-|\xi'(q)|_g^2)^{1/4}}{2}dq^{1/2}.\end{aligned}
\end{gather*}
where we take $\sqrt{z}=\sqrt{|z|}e^{\frac{1}{2}\Arg(z)}$ for $-\pi/2<\Arg(z)<3\pi /2$.
\end{lemma}
\begin{remarks} 
\item The change in this Lemma when $\Omega$ is only assumed to be convex is that we lose restriction on $|\xi'|_g$ and $|\eta'|_g$ in ${\MS}'(G_B)$ and thus must use \eqref{eqn:symbNonConv} for the symbol of $G_B$ near glancing points and away from the diagonal.
\item The microsupports of the various components of $G$ are shown graphically in Figure \ref{fig:waveFront}.
\end{remarks}

\begin{figure}
\centering
 \begin{tikzpicture}[scale=.35]
 \draw [thin](-4.05,1.2)to (-4.25,1.2)to  (-4.25,0) node[left] {$\mc{H}$}to (-4.25,-1.2)to (-4.05,-1.2);
 \draw [thin](-4.25,4)to  (-4.25,3.5)node[left]{$\mc{E}$}to (-4.25,1.8)to(-4.05,1.8);
 \draw [thin](-4.05,-1.8) to ( -4.25,-1.8)to (-4.25,-3.5) node[left]{$\mc{E}$}to (-4.25,-4);
 \draw [thin](-4.05,1.8)to (-4.25,1.8) to (-4.25,1.5)node[left]{$\mc{G}$}to(-4.25,1.2)to(-4.05,1.2);
  \draw [thin](-4.05,-1.8)to (-4.25,-1.8) to (-4.25,-1.5)node[left]{$\mc{G}$}to(-4.25,-1.2)to(-4.05,-1.2);
 \draw [thick] (-4,0) to [out=-15, in=-165](4,0) node[right]{$\partial\Omega$} ;
 \draw  (-4,4) to [out=-15, in=-165](4,4)node[right]{$T^*\partial\Omega$} ;
 \draw  (-4,-4) to [out=-15, in=-165](4,-4) ;
 \draw [ultra thin] (-4,1.5) to [out=-15, in=-165](4,1.5)node[right]{$S_E^*\partial\Omega$}  ;
 \draw [ultra thin] (-4,-1.5)  to [out=-15, in=-165](4,-1.5)node[right]{$S_E^*\partial\Omega$} ;
 \draw [dashed] (-4,1.8) to [out=-15, in=-165](4,1.8) ;
 \draw [dashed] (-4,-1.8) to [out=-15, in=-165](4,-1.8) ;
 \draw [dashed] (-4,1.2) to [out=-15, in=-165](4,1.2) ;
 \draw [dashed] (-4,-1.2) to [out=-15, in=-165](4,-1.2) ;
 \draw (-4,4) to (-4,-4) ;
 \draw (4,4)node[right]{$T^*\partial\Omega$} to (4,-4);
 \draw (0,-4.8)to (0,3.6)node[above]{$T^*_y\partial\Omega$};
 
 \draw [ultra thick, ->] (7,0)to(12,0);
 
 \begin{scope}[shift={(18,0)}]
  \draw [thin](-4.05,1.2)to (-4.25,1.2)to  (-4.25,0) node[left] {$\mc{H}$}to (-4.25,-1.2)to (-4.05,-1.2);
 \draw [thin](-4.25,4)to  (-4.25,3.5)node[left]{$\mc{E}$}to (-4.25,1.8)to(-4.05,1.8);
 \draw [thin](-4.05,-1.8) to ( -4.25,-1.8)to (-4.25,-3.5) node[left]{$\mc{E}$}to (-4.25,-4);
 \draw [thin](-4.05,1.8)to (-4.25,1.8) to (-4.25,1.5)node[left]{$\mc{G}$}to(-4.25,1.2)to(-4.05,1.2);
  \draw [thin](-4.05,-1.8)to (-4.25,-1.8) to (-4.25,-1.5)node[left]{$\mc{G}$}to(-4.25,-1.2)to(-4.05,-1.2);
 \draw [thick] (-4,0) to [out=-15, in=-165](4,0) node[right]{$\partial\Omega$} ;
 \draw  (-4,4) to [out=-15, in=-165](4,4)node[right]{$T^*\partial\Omega$} ;
 \draw  (-4,-4) to [out=-15, in=-165](4,-4) ;
 \draw [ultra thin] (-4,1.5) to [out=-15, in=-165](4,1.5)node[right]{$S_E^*\partial\Omega$}  ;
 \draw [ultra thin] (-4,-1.5)  to [out=-15, in=-165](4,-1.5)node[right]{$S_E^*\partial\Omega$} ;
 \draw [dashed] (-4,1.8) to [out=-15, in=-165](4,1.8) ;
 \draw [dashed] (-4,-1.8) to [out=-15, in=-165](4,-1.8) ;
 \draw [dashed] (-4,1.2) to [out=-15, in=-165](4,1.2) ;
 \draw [dashed] (-4,-1.2) to [out=-15, in=-165](4,-1.2) ;
 \draw (-4,4) to (-4,-4) ;
 \draw (4,4)node[right]{$T^*\partial\Omega$} to (4,-4);
 \draw (0,-4.8)to (0,3.6)node[above]{$T^*_y\partial\Omega$};
 \draw[ultra thin] (0.3,1.2)to [out=-179, in =-1](-.3,1.2)to(-.3,.6) to[out=-1,in = -179] (.3,.6)to (0.3,1.2);
\draw[ultra thin] (0.3,-1.8)to [out=-179, in =-1](-.3,-1.8)to(-.3,-2.4) to[out=-1,in = -179] (.3,-2.4)to (0.3,-1.8);
\draw [line width=1.2mm, red, opacity=1] (0.3,.6) [out=-35, in =125] to (2,0.3) to [out =-55, in =-180] (4,.2);
\draw [line width=1.2mm, red, opacity=1] (-0.3,-1.8) [out=135, in =15] to (-2,0.3) to [out =195, in =-35] (-4,-.2);
 \end{scope}
 \draw [line width=1.2mm, red, opacity=1] (0,-1.8) to (0,.6);
 \draw[ line width=1.2mm,red,opacity=1] (6.75,-5.25) to (7.25,-5.25);
 \draw (7.25,-5.25)node[right]{${\WFh}'(G_B)$};
 \end{tikzpicture}
 
  \begin{tikzpicture}[scale=.35]
 \draw [thin](-4.05,1.2)to (-4.25,1.2)to  (-4.25,0) node[left] {$\mc{H}$}to (-4.25,-1.2)to (-4.05,-1.2);
 \draw [thin](-4.25,4)to  (-4.25,3.5)node[left]{$\mc{E}$}to (-4.25,1.8)to(-4.05,1.8);
 \draw [thin](-4.05,-1.8) to ( -4.25,-1.8)to (-4.25,-3.5) node[left]{$\mc{E}$}to (-4.25,-4);
 \draw [thin](-4.05,1.8)to (-4.25,1.8) to (-4.25,1.5)node[left]{$\mc{G}$}to(-4.25,1.2)to(-4.05,1.2);
  \draw [thin](-4.05,-1.8)to (-4.25,-1.8) to (-4.25,-1.5)node[left]{$\mc{G}$}to(-4.25,-1.2)to(-4.05,-1.2);
 \draw [thick] (-4,0) to [out=-15, in=-165](4,0) node[right]{$\partial\Omega$} ;
 \draw  (-4,4) to [out=-15, in=-165](4,4)node[right]{$T^*\partial\Omega$} ;
 \draw  (-4,-4) to [out=-15, in=-165](4,-4) ;
 \draw [ultra thin] (-4,1.5) to [out=-15, in=-165](4,1.5)node[right]{$S_E^*\partial\Omega$}  ;
 \draw [ultra thin] (-4,-1.5)  to [out=-15, in=-165](4,-1.5)node[right]{$S_E^*\partial\Omega$} ;
 \draw [dashed] (-4,1.8) to [out=-15, in=-165](4,1.8) ;
 \draw [dashed] (-4,-1.8) to [out=-15, in=-165](4,-1.8) ;
 \draw [dashed] (-4,1.2) to [out=-15, in=-165](4,1.2) ;
 \draw [dashed] (-4,-1.2) to [out=-15, in=-165](4,-1.2) ;
 \draw (-4,4) to (-4,-4) ;
 \draw (4,4)node[right]{$T^*\partial\Omega$} to (4,-4);
 \draw (0,-4.8)to (0,3.6)node[above]{$T^*_y\partial\Omega$};
 
 \draw [ultra thick, ->] (7,0)to(12,0);
 
 \begin{scope}[shift={(18,0)}]
  \draw [thin](-4.05,1.2)to (-4.25,1.2)to  (-4.25,0) node[left] {$\mc{H}$}to (-4.25,-1.2)to (-4.05,-1.2);
 \draw [thin](-4.25,4)to  (-4.25,3.5)node[left]{$\mc{E}$}to (-4.25,1.8)to(-4.05,1.8);
 \draw [thin](-4.05,-1.8) to ( -4.25,-1.8)to (-4.25,-3.5) node[left]{$\mc{E}$}to (-4.25,-4);
 \draw [thin](-4.05,1.8)to (-4.25,1.8) to (-4.25,1.5)node[left]{$\mc{G}$}to(-4.25,1.2)to(-4.05,1.2);
  \draw [thin](-4.05,-1.8)to (-4.25,-1.8) to (-4.25,-1.5)node[left]{$\mc{G}$}to(-4.25,-1.2)to(-4.05,-1.2);
 \draw [thick] (-4,0) to [out=-15, in=-165](4,0) node[right]{$\partial\Omega$} ;
 \draw  (-4,4) to [out=-15, in=-165](4,4)node[right]{$T^*\partial\Omega$} ;
 \draw  (-4,-4) to [out=-15, in=-165](4,-4) ;
 \draw [ultra thin] (-4,1.5) to [out=-15, in=-165](4,1.5)node[right]{$S_E^*\partial\Omega$}  ;
 \draw [ultra thin] (-4,-1.5)  to [out=-15, in=-165](4,-1.5)node[right]{$S_E^*\partial\Omega$} ;
 \draw [dashed] (-4,1.8) to [out=-15, in=-165](4,1.8) ;
 \draw [dashed] (-4,-1.8) to [out=-15, in=-165](4,-1.8) ;
 \draw [dashed] (-4,1.2) to [out=-15, in=-165](4,1.2) ;
 \draw [dashed] (-4,-1.2) to [out=-15, in=-165](4,-1.2) ;
 \draw (-4,4) to (-4,-4) ;
 \draw (4,4)node[right]{$T^*\partial\Omega$} to (4,-4);
 \draw (0,-4.8)to (0,3.6)node[above]{$T^*_y\partial\Omega$};
 \filldraw[fill=red,opacity=1] (0.3,1.2)to [out=-179, in =-1](-.3,1.2)to(-.3,.6) to[out=-1,in = -179] (.3,.6)to (0.3,1.2);
 \filldraw[fill=red,opacity=1] (0.3,-1.8)to [out=-179, in =-1](-.3,-1.8)to(-.3,-2.4) to[out=-1,in = -179] (.3,-2.4)to (0.3,-1.8);
 \end{scope}
\draw[line width=1.2mm,red,opacity=1] (0,1.2)to (0,.6) ;
 \draw[line width=1.2mm, red,opacity=1] (0,-1.8)to (0,-2.4);
\draw[line width=1.2mm,red,opacity=1] (6.75,-5.25) to (7.25,-5.25);
 \draw (7.25,-5.25)node[right]{${\WFh}'(G_g)$};
 \end{tikzpicture}
 
  \begin{tikzpicture}[scale=.35]
 \draw [thin](-4.05,1.2)to (-4.25,1.2)to  (-4.25,0) node[left] {$\mc{H}$}to (-4.25,-1.2)to (-4.05,-1.2);
 \draw [thin](-4.25,4)to  (-4.25,3.5)node[left]{$\mc{E}$}to (-4.25,1.8)to(-4.05,1.8);
 \draw [thin](-4.05,-1.8) to ( -4.25,-1.8)to (-4.25,-3.5) node[left]{$\mc{E}$}to (-4.25,-4);
 \draw [thin](-4.05,1.8)to (-4.25,1.8) to (-4.25,1.5)node[left]{$\mc{G}$}to(-4.25,1.2)to(-4.05,1.2);
  \draw [thin](-4.05,-1.8)to (-4.25,-1.8) to (-4.25,-1.5)node[left]{$\mc{G}$}to(-4.25,-1.2)to(-4.05,-1.2);
 \draw [thick] (-4,0) to [out=-15, in=-165](4,0) node[right]{$\partial\Omega$} ;
 \draw  (-4,4) to [out=-15, in=-165](4,4)node[right]{$T^*\partial\Omega$} ;
 \draw  (-4,-4) to [out=-15, in=-165](4,-4) ;
 \draw [ultra thin] (-4,1.5) to [out=-15, in=-165](4,1.5)node[right]{$S_E^*\partial\Omega$}  ;
 \draw [ultra thin] (-4,-1.5)  to [out=-15, in=-165](4,-1.5)node[right]{$S_E^*\partial\Omega$} ;
 \draw [dashed] (-4,1.8) to [out=-15, in=-165](4,1.8) ;
 \draw [dashed] (-4,-1.8) to [out=-15, in=-165](4,-1.8) ;
 \draw [dashed] (-4,1.2) to [out=-15, in=-165](4,1.2) ;
 \draw [dashed] (-4,-1.2) to [out=-15, in=-165](4,-1.2) ;
 \draw (-4,4) to (-4,-4) ;
 \draw (4,4)node[right]{$T^*\partial\Omega$} to (4,-4);
 \draw (0,-4.8)to (0,3.6)node[above]{$T^*_y\partial\Omega$};
 
 \draw [ultra thick, ->] (7,0)to(12,0);
 
 \begin{scope}[shift={(18,0)}]
  \draw [thin](-4.05,1.2)to (-4.25,1.2)to  (-4.25,0) node[left] {$\mc{H}$}to (-4.25,-1.2)to (-4.05,-1.2);
 \draw [thin](-4.25,4)to  (-4.25,3.5)node[left]{$\mc{E}$}to (-4.25,1.8)to(-4.05,1.8);
 \draw [thin](-4.05,-1.8) to ( -4.25,-1.8)to (-4.25,-3.5) node[left]{$\mc{E}$}to (-4.25,-4);
 \draw [thin](-4.05,1.8)to (-4.25,1.8) to (-4.25,1.5)node[left]{$\mc{G}$}to(-4.25,1.2)to(-4.05,1.2);
  \draw [thin](-4.05,-1.8)to (-4.25,-1.8) to (-4.25,-1.5)node[left]{$\mc{G}$}to(-4.25,-1.2)to(-4.05,-1.2);
 \draw [thick] (-4,0) to [out=-15, in=-165](4,0) node[right]{$\partial\Omega$} ;
 \draw  (-4,4) to [out=-15, in=-165](4,4)node[right]{$T^*\partial\Omega$} ;
 \draw  (-4,-4) to [out=-15, in=-165](4,-4) ;
 \draw [ultra thin] (-4,1.5) to [out=-15, in=-165](4,1.5)node[right]{$S_E^*\partial\Omega$}  ;
 \draw [ultra thin] (-4,-1.5)  to [out=-15, in=-165](4,-1.5)node[right]{$S_E^*\partial\Omega$} ;
 \draw [dashed] (-4,1.8) to [out=-15, in=-165](4,1.8) ;
 \draw [dashed] (-4,-1.8) to [out=-15, in=-165](4,-1.8) ;
 \draw [dashed] (-4,1.2) to [out=-15, in=-165](4,1.2) ;
 \draw [dashed] (-4,-1.2) to [out=-15, in=-165](4,-1.2) ;
 \draw (-4,4) to (-4,-4) ;
 \draw (4,4)node[right]{$T^*\partial\Omega$} to (4,-4);
 \draw (0,-4.8)to (0,3.6)node[above]{$T^*_y\partial\Omega$};
 \draw [line width=1.2mm, red] (0,-4.6)to (0,-2.4) (0,-1.8) to (0,.6) (0,1.2) to (0,3.4);
 \end{scope}
 \draw [line width=1.2mm, red] (0,-4.6)to (0,-2.4) (0,-1.8) to (0,.6) (0,1.2) to (0,3.4);
 \draw [line width=1.2mm, red](6.75,-5.25)to (7.25,-5.25);
 \draw (7.25,-5.25)node[right]{${\WFh}'(G_\Delta)$};
 \end{tikzpicture}
 
\caption[Wavefront sets for the decomposition of boundary layer operators]{\label{fig:waveFront}We show the wavefront relation for each of the pieces in the decomposition of $G$ ($\Dl$ or $\dDl$). The formulae for these wavefront sets are contained in Lemmas \ref{lem:decompose}, \ref{lem:decomposeDl}, and \ref{lem:decomposePrime}. We label the elliptic, glancing, and hyperbolic regions by $\mc{E}$, $\mc{G}$, and $\mc{H}$ respectively. The top, middle, and bottom pictures correspond to $G_B$, $G_g$ and $G_{\Delta}$ respectively. In the left copy of $T^*\partial\Omega$, we show the wavefront set of each operator in the fiber over $y\in \partial\Omega$. The right copy of $T^*\partial\Omega$ shows how each operator maps the wavefront set in the fiber over $y$. Note that the curve shown in the right copy of $T^*\partial\Omega$ for ${\WFh}'(G_B)$ continues outside of the portion of $T^*\partial\Omega$ shown. }
\end{figure}

\section{Boundary layer operators and potentials near glancing}
\label{sec:BLONearGlance}
In this section, we complete the microlocal descriptions of the boundary layer operators using the Melrose--Taylor parametrix constructed in Appendx \ref{ch:semiclassicalDirichletParametrices}. 
\subsection{Estimates for a simple transmission problem}
We start by proving estimates for the following transmission problem. Let $\Omega_1=\Omega$, $\Omega_2=\re^d\setminus\overline{\Omega}$, and $u=u_11_{\Omega_1}\oplus u_21_{\Omega_2}.$ 
Suppose that $\chi \in \Cc(\re^d)$ with $\chi \equiv 1$ on $\Omega_1$ and
\begin{equation}
\label{eqn:simpleTransmit}
\begin{aligned}
(-h^2\Delta-z^2)u_i&=h^2\chi f_i&\text{in } \Omega_i\\
u_1-u_2&=g_1&\text{ on }\pO\\
\partial_{\nu_1}u_1+\partial_{\nu_2}u_2&=g_2&\text{ on }\pO\\
u_2&\text{ is $z/h$ outgoing}
\end{aligned}
\end{equation}
Then, it is easy to check that as a distribution, 
$$(-h^2\Delta-z^2)u=h^2(f+L^*\delta_{\pO} \otimes g_1+\delta_{\pO}\otimes g_2 )$$
where $f=1_{\Omega_1}f_1\oplus1_{\Omega_2}f_2$ and $L$ is a vector field with $L|_{\pO}=\partial_{\nu_1}$. Thus, applying $h^{-2}R_0(z/h)$ to this equation shows that for $z/h$ in the domain of $R_0(z/h)$, \eqref{eqn:simpleTransmit} has a unique solution given by
$$u=R_0\chi f+\S g_2+\D g_1.$$
Hence
\begin{equation}
\label{eqn:simpleTransmitSolnBV}
\begin{aligned} 
u_1|_{\pO}&=\gamma R_0 f+Gg_2-\frac{1}{2}g_1+\Dl g_1\\
u_2|_{\pO}&=\gamma R_0 f+Gg_2+\frac{1}{2}g_1+\Dl g_1\\
\partial_{\nu_1}u_1|_{\pO}&=\gamma \partial_{\nu_1} R_0f+\frac{1}{2}g_2+\Dl^\sharp g_2+\dDl g_1\\
\partial_{\nu_2}u_2|_{\pO}&=\gamma \partial_{\nu_2}R_0f+\frac{1}{2}g_2-\Dl^\sharp g_2-\dDl g_1
\end{aligned}
\end{equation}
To obtain an $L^2$ estimate on $u$, we simply apply standard resolvent estimates (see for example \cite[Chapter 3]{ZwScat}), 
\begin{equation}
\label{eqn:resolveStd}
\|\chi R_0(z/h)\chi\|_{H_h^s\to H_h^{s+2}}\leq Che^{\Lchi (\Im z)_-/h}.
\end{equation}
So
\begin{align*} 
\|\chi u\|_{L^2(\re^d)}&\leq Ce^{\Lchi(\Im z)_-}(h\|\chi f\|_{L^2(\re^d)}+h^{1/2}\|g_1\|_{L^2(\pO)}+h^{1/2}\|g_2\|_{L^2(\pO)}).
\end{align*} 

To upgrade this to estimates on $u_i$ in $H^k(\Omega_i)$, we observe that for $\chi_1\in \Cc(\re^d)$ with $\chi_1\equiv 1$ on $\chi$, and $\chi_2\in \Cc(\re^d)$ with $\chi_2\equiv 1$ on $\supp \chi_1$,
\begin{align*} 
(-h^2\Delta -z^2)\chi_1 u&=[\chi_1,h^2\Delta]u+h^2(\chi f+L^*\delta_{\pO} \otimes g_1+\delta_{\pO}\otimes g_2 )\\
\chi_1 u&=\chi_2 R_0(0)(h^{-2}([\chi_1,h^2\Delta] +z^2\chi_1 ) u+\chi f)\\
&\quad\quad\quad+\chi_2 \D(0)g_1+\chi_2 \S(0)g_2
\end{align*} 
and for $k\geq -1$, $\D(0):H^{k+3/2}(\pO)\to H^{k+2}(\Omega_1)\oplus H^{k+2}(\Omega_2)$ and $\S(0):H^{k+1/2}(\pO)\to H^{k+2}(\Omega_1)\oplus H^{k+2}(\Omega_2)$, $\chi R_0(0)\chi:H^k(\Omega_1)\oplus H^k(\Omega_2)\to H^{k+2}(\Omega_1)\oplus H^{k+2}(\Omega)_2$. (See \cite[Theorems 9, 10]{Epstein})
So,
\begin{multline*}
\|u_1\|_{H^{k+2}(\Omega_1)}+\|\chi_1 u_2\|_{H^{k+2}(\Omega_2)}\\
\leq h^{-2}((\|u_1\|_{H^{k}(\Omega_1)}+h\| u_1\|_{H^{k+1}(\Omega_1)})+(\|\chi_1 u_2\|_{H^k(\Omega)} +h\|\chi_2u_2\|_{H^{k+1}(\Omega_2)}))\\
+\|\chi f\|_{H^{k}(\re^d)}+\|g_1\|_{H^{k+1/2}(\pO)}+\|g_2\|_{H^{k+3/2}(\pO)}
\end{multline*} 

Using the description $G,\,\Dl,\,$ and $\dDl$ at high energy in Lemmas \ref{lem:decompose}, \ref{lem:decomposeDl} and \ref{lem:decomposePrime} as psuedodifferential operators, we have for $\psi\in \Cc(\re)$ with $\psi\equiv 1$ on $[-2E,2E]$,
\begin{align*} 
\|Gu\|_{H_h^k}&\leq\|(1-\psi(|hD|))Gu\|_{H_h^k}+\|\psi(|hD|)Gu\|_{H_h^k}\\
&\leq h\|u\|_{H_h^{k-1}}+\|Gu\|_{L^2}\\
\|\Dl u\|_{H_h^k}&\leq \|(1-\psi(|hD|))\Dl u\|_{H_h^k}+\|\psi(|hD|)\Dl u\|_{H_h^k}\\
&\leq \|u\|_{H_h^{k}}+\|\Dl u\|_{L^2}\\
\|\Dl^\sharp u\|_{H_h^k}&\leq \|(1-\psi(|hD|))\dDl^\sharp  u\|_{H_h^k}+\|\psi(|hD|)\dDl^\sharp  u\|_{H_h^k}\\
&\leq \|u\|_{H_h^{k}}+\|\Dl ^\sharp u\|_{L^2}\\
\|\dDl u\|_{H_h^k}&\leq \|(1-\psi(|hD|))\dDl u\|_{H_h^k}+\|\psi(|hD|)\dDl u\|_{H_h^k}\\
&\leq h^{-1}\|u\|_{H_h^{k+1}}+\|\dDl u\|_{L^2}.
\end{align*} 
 Together with \eqref{eqn:resolveStd} and Theorem \ref{thm:optimal}, this implies the estimates
\begin{lemma}
\label{lem:transmitEst}
Suppose that $z/h$ is in the domain of $R_0$, $\chi\in \Cc(\re^d)$ with $\chi \equiv 1 $ on $\Omega_1$ and $u\in L^2_{\loc}(\re^d)$ is the solution to \eqref{eqn:simpleTransmit}. Then 
$$u=R_0\chi f+\S g_2+\D g_1,$$
\eqref{eqn:simpleTransmitSolnBV} holds 
and for any $\e>0$, $k\geq -1/2$, $m\geq 0$, there exists $h_0>0,$ $C,\,N_k>0$ such that for $0<h<h_0$,
\begin{multline*}
\left(\begin{gathered}
\|u_1\|_{H_h^{k+2}(\Omega_1)}+\|\chi u_2\|_{H_h^{k+2}(\Omega_2)}+\|u_1\|_{H_h^{k+\frac{3}{2}}(\pO)}+\|u_2\|_{H_h^{k+\frac{3}{2}}(\pO)}\\
+\|\partial_{\nu_1}u_1\|_{H^{k+\frac{1}{2}}(\pO)}+\|\partial_{\nu_2}u_2\|_{H^{k+\frac{1}{2}}(\pO)}
\end{gathered}\right)\\
\leq Ch^{-N_k}e^{\frac{\Lchi(\Im z)_-}{h}}(\|\chi f\|_{H_h^{k}(\re^d)}+\|g_2\|_{H_h^{k+\frac{1}{2}}(\pO)}+\|g_1\|_{H_h^{k+\frac{3}{2}}(\pO)})
\end{multline*}
\end{lemma}

\subsection{Microlocal Description of $G$ and $\S$ near glancing}

Now, let $u$ solve \eqref{eqn:simpleTransmit} with $f_i\equiv 0$ and $g_1=0$ and $g_2=g$ microlocalized sufficiently close to a glancing point $(x',\xi')$ so that the parametrices from Appendix \ref{ch:semiclassicalDirichletParametrices} can be constructed. 

In particular, let $(y_0,\eta_0)\in S^*\pO$ and 
\begin{gather}
\label{eqn:psiCond}
\psi\equiv 1 \text{ on }\{|y-y_0|<\delta,\,|\eta-\eta_0|<\delta_1,\,\left||\eta|_g-1\right|<\gamma h^2 \e(h)^{-2}.\\
\supp \psi\subset\{|y-y_0|<2\delta,\,|\eta-\eta_0|<2\delta_1,\,\left||\eta|_g-1\right|<2\gamma h^2 \e(h)^{-2}
\end{gather}
and suppose that $g=\oph(\psi)g+\O{\Ph{-\infty}{}}(h^\infty)g$.

Recall that by Lemmas \ref{lem:extPara} and \eqref{eqn:DtoNMicrolocal} a microlocal description of the exterior Dirichlet to Neumann map, $\mc{N}_2$, is given by
\begin{equation} 
\label{eqn:dtongreens} \mc{N}_2g=J(h^{-2/3}C\Phi_-+B)J^{-1}g+\O{\Ph{-\infty}{}}(h^\infty)g\,\,\end{equation}
where $C\in \Ph{}{}$ is elliptic, $B\in \Ph{}{}$, $\Phi_-$ is the Fourier multiplier 
\begin{equation}
\label{eqn:NeumannFAIO}\Phi_-(u)=(2\pi h)^{-d+1}\int \frac{A_-'(h^{-2/3}\alpha)}{A_-(h^{-2/3}\alpha)}e^{\frac{i}{h}\la x-y,\xi'\ra}ud\xi'.
\end{equation}
where,
$$\alpha(\xi')=\xi_1+i\e(h)\quad \quad \text{ with }ch\leq \e(h)=\O{}(h\log h^{-1}).$$
Let $\mc{A}i\mc{A}_-$, $\mc{A}i'\mc{A}_-$, $\mc{A}i\mc{A}_-'$, and $\mc{A}i'\mc{A}_-'$ be the Fourier multiplies obtained by replacing $\frac{A_-'}{A_-}$ in \eqref{eqn:NeumannFAIO} by $AiA_-$, $Ai'A_-$, $AiA_-'$, and $Ai'A_-'$ respectively. 

Let $q_1=h^{2/3}\beta^{-1}JC^{-1}J^{-1}g$ and $q_2=h^{2/3}\beta^{-1}J\mc{A}i\mc{A}_-C^{-1}J^{-1}g$ where $\beta=\frac{e^{-\pi i/6}}{2\pi}.$
Then, let $w_1=A_{1,g}q_1$ where $A_1$ is as in Lemma \ref{lem:glideBLOPara} and $w_2=\mc{H}_dq_2$ where $\mc{H}_d$ is the solution operator to 
$$\begin{cases}(-h^2\Delta-z^2)\mc{H}_dq_2=0&\text{in } \re^d\setminus\overline{\Omega}\\
\mc{H}_dq_2|_{\pO}=q_2\\
\mc{H}_dq_2\text{ $z/h$ outgoing}
\end{cases}.$$
Then, by Lemma \ref{lem:glideBLOPara} and \eqref{eqn:dtongreens}, 
\begin{equation}\label{eqn:SLOMode} 
w_1|_{\pO}=w_2|_{\pO}=h^{2/3}\beta^{-1}J\mc{A}i\mc{A}_-C^{-1}J^{-1}g+\O{\Ph{-\infty}{}}(h^\infty).
\end{equation} 
and 
\begin{align} \partial_\nu w_1+\partial_{\nu_2} w_2&\\
&\!\!\!\!\!\!\!\!\!\!\!\!\!\!\!\!\!\!\!\!\!\!\!\!\!\!\!\!\!\!\!\!\!\!\!\!\!\!\!=\beta^{-1}J\left(C\mc{A}i\mc{A}_-'C^{-1}+B\mc{A}i\mc{A}_-C^{-1}-C\mc{A}i'\mc{A}_-C^{-1}-B\mc{A}i\mc{A}_-C^{-1}\right)J^{-1} g\nonumber\\
&+\O{\Ph{-\infty}{}}(h^\infty)g\nonumber\\
&=\beta^{-1}J(C(\mc{A}i'\mc{A}_--\mc{A}i\mc{A}_-')C^{-1})J^{-1}g+\O{\Ph{-\infty}{}}(h^\infty)g\label{eqn:wronskUse}\\
&=JCC^{-1}J^{-1}g+\O{\Ph{-\infty}{}}(h^\infty)g=g+\O{\Ph{-\infty}{}}(h^\infty)g\nonumber
\end{align}
where we have use the Wronskian for the Airy equation to reduce \eqref{eqn:wronskUse}.

Thus, $(u_1-w_1,u_2-w_2)$ solves \eqref{eqn:simpleTransmit} with 
$$\|\chi f\|_{H_h^N(\re^d)}+\|g_1\|_{H_h^N(\pO)}+\|g_2\|_{H_h^N(\pO)}=\O{}(h^\infty)\| g\|_{H^{-N}}$$
for any $N$. 
Hence, $u=w+\O{C^\infty{\loc}}(h^\infty)$
and we have that 
\begin{equation} 
\label{eqn:Gglance}Gg=J\beta^{-1} h^{2/3} \mc{A}i\mc{A}_-C^{-1}J^{-1}g+\O{\Ph{-\infty}{}}(h^\infty)g
\end{equation}
for any $\Im z=\O{}(h\log h^{-1})$. Moreover, 
\begin{equation}
\label{eqn:SLglance}
\S g|_{\Omega}=h^{2/3}\beta^{-1}A_{1,g}JC^{-1}J^{-1}g+\O{\mc{D}'(\pO)\to C^\infty(\Omega)}(h^\infty)g
\end{equation}

\begin{lemma}
\label{lem:nearGlanceEst}
Suppose that $\varphi\in L^2(\partial\Omega)$ and there exists $\e>0$ such that $\MS(\varphi)\subset \{|1-|\xi'|_g|\leq h^\e\},$ and $\Im z\geq -Mh \log h^{-1}$. Then,
$$\|G\varphi\|_{L^2}\leq Ch^{2/3}\|\varphi\|_{L^2}.$$
\end{lemma}
\begin{proof}
Let $\chi_\e\in S_\e(T^*\partial\Omega)$ have $\chi_\e\equiv$ 1 on $\{|1-|\xi'|_g|\leq h^\e\}$ with $\supp\chi_\e\subset \{|1-|\xi'|_g|\leq 2h^\e\}$ and $X=\oph(\chi_\e).$ Then
$X\varphi=\varphi +\O{}(h^\infty)\varphi.$

By Lemma \ref{lem:decompose}, there exists $1/2>\delta>0$ such that for any $x_0\in \pO$ and $\beta>0$, if $\zeta_1\in S_\delta\cap \Cc(\pO)$ has $\supp \zeta_1\subset \{|x-x_0|\leq \beta h^\delta\}$ and $\zeta_2\in S_\delta\cap \Cc(\pO)$ has $\zeta_2\equiv 1$ on $\{|x-x_0|\leq 2\beta h^\delta\}$ then
\begin{gather*} 
\zeta_2G\zeta_1X\varphi=G\zeta_1X\varphi+\O{}(h^\infty)\varphi\\
\zeta_1G\zeta_2X\varphi=\zeta_1GX\varphi +\O{}(h^\infty)\varphi
\end{gather*} 
Now, by \eqref{eqn:Gglance}, 
\begin{gather*}
\label{eqn:cutoffi1}\zeta_i G\zeta_jX\varphi= \zeta_iJ\beta^{-1}h^{2/3}\mc{A}i\mc{A}_-C^{-1}J^{-1}\zeta_jX\varphi +\O{}(h^\infty)\varphi 
\end{gather*}
So, since $\mc{A}i\mc{A}_-=\O{L^2\to L^2}(1)$, and the  $\zeta J$ terms are elliptic semiclassical FIO's, with symbols in $h^{-\alpha}S_\delta$ for some $\alpha>0$, we have 
$$\|\zeta_iG\zeta_jX\varphi\|_{L^2}\leq C_0h^{2/3}\|\zeta_jX\varphi\|_{L^2}$$
where $C_0$ is a constant depending only on $\Omega$.

Let ${x_i}_{i=1}^{R(\e)}$ have $\pO\subset \bigcup_{i=1}^{R(\e)}B(x_i,\e)$ be such that for all $0<\e<1$, 
$$\sup_{x\in \pO}\#\{i\,:\,x\in B(x_i,10\e)\}\leq M_{\Omega}.$$
To see that this is possible, see for example \cite[Lemma 2]{Minicozzi}. Then, $R(\e)\leq c\e^{-d+1}$. Now, let $\{\zeta_{i,\beta}\}_{i=1}^{R(\beta h^{\delta})}$ be a partition of unity with $\supp \zeta_{i,\beta} \subset B(x_i,2\beta h^\delta)$ and $\zeta_{i,\beta}\equiv 1$ on $B(x_i,\beta h^\delta)$.

\begin{gather*} 
\sum_{j=1}^{R(\beta h^\delta)}\|G\zeta_{i,\beta}XX^*\zeta_{j,\beta}G^*\|^{1/2}\leq CM_{\Omega}h^{2/3}+\O{}(h^\infty)\leq C_{\Omega}h^{2/3}\\
\begin{aligned}\sum_{j=1}^{R(\beta h^{\delta})}\|X^*\zeta_{i,\beta}G^*G\zeta_{j,\beta}X\|^{1/2}&=\sum_{i=1}^{Ch^{-\e}}\|X^*\zeta_{i,\beta}G^*\zeta_{i,4\beta}\zeta_{j,4\beta}G\zeta_jX\|^{1/2}+\O{}(h^\infty)\\
&\leq C_{\Omega}h^{2/3}\end{aligned}
\end{gather*}
Hence, by the Cotlar-Stein Lemma (see for example \cite[Theorem C.5]{EZB}),
$$\|GX\|_{L^2\to L^2}=\left\|\sum_jG\zeta_{j,\beta}X\right\|_{L^2\to L^2}\leq C_{\Omega}h^{2/3}.$$
\end{proof}

Combining Lemma \ref{lem:nearGlanceEst} with Lemma \ref{lem:decompose}, the $L^2$ boundedness of semiclassical FIOs associated to canonical graphs gives, and Lemma \ref{lem:phragmen} gives the following improvement of Theorem \ref{thm:optimal} in the case that $\Omega$ is strictly convex with smooth boundary
\begin{theorem}
\label{thm:GEstImproved}
Let $\Omega\subset \re^d$ be strictly convex with smooth boundary. Then there exists $\lambda_0>0$ such that for some $C$ and all $|\lambda|>\lambda_0$ the following estimate holds
\begin{equation}
\nonumber
\|G(\lambda)\|_{L^2(\partial\Omega)\to L^2(\partial\Omega)} \; \leq \; 
C\,\la\lambda\ra^{-\frac 23}\,e^{\LO(\Im \lambda)_-}.
\end{equation}
\end{theorem}
\begin{remark}
The improvement from Theorem \ref{thm:optimal} is that we have removed the $\log\lambda$ from the right hand side of \eqref{eqn:optimalConvex}
\end{remark}


\subsection{Microlocal description of $\Dl$ near glancing}
To obtain a microlocal description of $\Dl$ near glancing, we combine Proposition \ref{prop:layerInverse} with the microlocal decomposition of $G$ and the microlocal parametrix for $\mc{N}_e$ constructed in Appendix \ref{ch:semiclassicalDirichletParametrices}. In particular, for $g$ microlocalized near glancing point $(y_0,\eta_0)$,
\begin{gather*} 
Gg=J\beta^{-1}h^{2/3}\mc{A}i\mc{A}_-C^{-1}J^{-1}+\O{}(h^\infty)g\label{eqn:modelG}\\
\mc{N}_2g=J(h^{-2/3}C\Phi_{-}+B)J^{-1}g+\O{}(h^\infty)g
\end{gather*}
where $\mc{N}_2$ as denotes the exterior Dirichlet to Neumann map.
Now, $\mc{N}_2$ has microsupport contained in an $h^\e$ neighborhood of the diagonal and hence $\mc{N}_2g$ remains microlocalized near glancing and we can use the microlocal model \eqref{eqn:modelG} in the composition $G\mc{N}_2$.
Proposition \ref{prop:layerInverse} implies that 
\begin{align*} \Dl g&=\frac{1}{2}g-G\mc{N}_2g\\
&=\frac{1}{2}g-\beta^{-1}J(\mc{A}i\mc{A}_-'+h^{2/3}\mc{A}i\mc{A}_-C^{-1}B)J^{-1}g+\O{\Ph{-\infty}{}}(h^\infty)g\\
&=-\beta^{-1}J\left(\frac{1}{2}\left(\mc{A}i'\mc{A}_-+\mc{A}i\mc{A}_-'\right)\right.\\
&\quad\quad\quad+h^{2/3}\mc{A}_-\mc{A}iC^{-1}B)J^{-1}g+\O{\Ph{-\infty}{}}(h^\infty)g\\
&=-\beta^{-1}J(\mc{A}i'\mc{A}_-+h^{2/3}\mc{A}i\mc{A}_-C^{-1}B)J^{-1}g-\frac{1}{2}g+\O{\Ph{-\infty}{}}(h^\infty)g
\end{align*}
Hence, for $g$ microlocalized near glancing
\begin{equation}
\label{eqn:dlglance}
\Dl g=\frac{1}{2}g-\beta^{-1}J(\mc{A}i\mc{A}_-'+h^{2/3}\mc{A}i\mc{A}_-C^{-1}B)J^{-1}g+\O{\Ph{-\infty}{}}(h^\infty)g
\end{equation}

So, by analogous arguments to those in Lemma \ref{lem:nearGlanceEst}, we have 
\begin{lemma}
\label{lem:nearGlanceEstDl}
Suppose that $\varphi\in L^2(\partial\Omega)$, and there exists $\e>0$ such that $\MS(\varphi)\subset \{|1-|\xi'|_g|\leq h^\e\},$ and $\Im z\geq -Mh \log h^{-1}$. Then,
$$\|\Dl\varphi\|_{L^2}\leq C\|\varphi\|_{L^2}.$$
\end{lemma}
Combining Lemma \ref{lem:nearGlanceEstDl} with Lemma \ref{lem:decomposeDl}, the $L^2$ boundedness of semiclassical FIOs associated to canonical graphs, and Lemma \ref{lem:phragmen} gives the following improvement of Theorem \ref{thm:optimal} in the case that $\Omega$ is strictly convex with smooth boundary
\begin{theorem}
Let $\Omega \subset \re^d$ be strictly convex with smooth boundary. Then there exists $\lambda_0>0$ such that for some $C$ and all $|\lambda|>\lambda_0$ the following estimate holds
$$\|\Dl \|_{L^2(\pO)\to L^2(\pO)}\leq Ce^{\LO(\Im \lambda)_-}.$$
\end{theorem}
\begin{remark} 
This theorem improves the estimate for $\Dl$ in Theorem \ref{thm:optimal} by removing the factor $\la \lambda\ra^{\frac{1}{6}}\log \la \lambda\ra$. The improved estimate is sharp in the case of a strictly convex domain as can be seen by taking Neumann eigenfunctions on the ball.
\end{remark}

\subsubsection{Microlocal description of $\dDl $ and $\D$ near glancing}

Now, let $u$ solve \eqref{eqn:simpleTransmit} with $f_i\equiv 0$ and $g_2=0$ and $g_1=g$ microlocalized sufficiently close to a glancing point $(y_0,\eta_0)$ so that the parametrices from Appendix \ref{ch:semiclassicalDirichletParametrices} can be constructed. In particular, let $\psi$ be as in \eqref{eqn:psiCond} and assume the $\oph(\psi)g=g+\O{\Ph{-\infty}{}}(h^\infty)g$.  

We know that $u=\D g$ and so by Lemma \ref{lem:layer}  $u_1|_{\pO}=-\frac{1}{2}g+\Dl g$, $u_2|_{\pO}=\frac{1}{2}g+\Dl g$. 
Motivated by this and \eqref{eqn:dlglance}, let 
\begin{equation*} 
\begin{aligned} 
w_1&=-\beta^{-1}A_{2,g}g-\beta^{-1}h^{2/3}A_{1,g}J C^{-1}BJ^{-1}g\\
w_2&=-\beta^{-1}\mc{H}_dJ(\mc{A}i'\mc{A}_-+h^{2/3}\mc{A}_-\mc{A}iC^{-1}BJ^{-1}g
\end{aligned}
\end{equation*}
where $A_{i,g}$ are as in Lemma \ref{lem:glideBLOPara}. 
Then,
\begin{equation}
\label{eqn:prelimBVs}
\begin{aligned}
w_1|_{\pO}&=-\beta^{-1}J(\mc{A}i\mc{A}_-'+h^{2/3}\mc{A}i\mc{A}_-C^{-1}B)J^{-1}g+\O{\Ph{-\infty}{}}(h^\infty)g\\
w_2|_{\pO}&=-\beta^{-1}J(\mc{A}i'\mc{A}_-+h^{2/3}\mc{A}i\mc{A}_-C^{-1}B)J^{-1}g+\O{\Ph{-\infty}{}}(h^\infty)g\\
\partial_{\nu_1}w_1|_{\pO}&=\beta^{-1}J(h^{-2/3}C\mc{A}i'\mc{A}_-'+B\mc{A}i\mc{A}_-'+C\mc{A}i'\mc{A}_-C^{-1}BJ^{-1}g\\
&\quad\quad+h^{2/3}B\mc{A}i\mc{A}_-C^{-1}B)J^{-1}g+\O{\Ph{-\infty}{}}(h^\infty)g\\
\partial_{\nu_2}w_2|_{\pO}&
= -\beta^{-1}J(h^{-2/3}C\mc{A}i'\mc{A}_-'+C\mc{A}i\mc{A}_-'C^{-1}B+B\mc{A}i'\mc{A}_-)J^{-1}g\\
&\quad\quad-h^{2/3}\beta^{-1}JB\mc{A}i\mc{A}_-C^{-1}B)J^{-1}g+\O{\Ph{-\infty}{}}(h^\infty)g
\end{aligned}
\end{equation}
Thus, 
\begin{equation*} 
\begin{aligned} 
\partial_{\nu_1}w_1|_{\pO}+\partial_{\nu_2}w_2|_{\pO}&=\O{\Ph{-\infty}{}}(h^\infty)g\\
w_1-w_2&=g+\O{\Ph{-\infty}{}}(h^\infty)g
\end{aligned}
\end{equation*}
where we have used the Wronksian for the Airy equation in simplifying the expressions in \eqref{eqn:prelimBVs}.

Thus, $(u_1-w_1,u_2-w_2)$ solves \eqref{eqn:simpleTransmit} with 
$$\|\chi f\|_{H_h^N(\re^d)}+\|g_1\|_{H_h^N(\pO)}+\|g_2\|_{H_h^N(\pO)}=\O{}(h^\infty)\|g\|_{H^{-N}}$$
for any $N>0$.
This gives
$$u_i=w_i+\O{\Ph{-\infty}{}}(h^\infty)g.$$
So we have that 
\begin{equation} 
\label{eqn:GglancePrime}
\begin{aligned} \dDl g&=\beta^{-1}J(h^{-2/3}C\mc{A}i'\mc{A}_-'+B\mc{A}i\mc{A}_-'+C\mc{A}i'\mc{A}_-C^{-1}B)Jg\\
&\quad\quad+h^{2/3}\beta^{-1}JB\mc{A}i\mc{A}_-C^{-1}B)J^{-1}g+\O{\Ph{-\infty}{}}(h^\infty)g\\
&=\beta^{-1}h^{-2/3}JC\mc{A}i'\mc{A}_-'J^{-1}g+J\O{H_h^s\to H_h^s}(1)J^{-1}g
\end{aligned}
\end{equation}
and 
\begin{equation}
\label{eqn:dLIn}
\D g|_{\Omega_1}=-\beta^{-1}A_{2,g}g-\beta^{-1}h^{2/3}A_{1,g}JC^{-1}BJ^{-1}g+\O{\mc{D}'(\pO)\to C^\infty(\Omega_1)}(h^\infty)
\end{equation}
for any $\Im z=\O{}(h\log h^{-1}).$

\begin{lemma}
\label{lem:nearGlanceEstPrime}
Suppose that $\varphi\in L^2(\partial\Omega)$ and there exists $\e>0$ such that $\MS(\varphi)\subset \{|1-|\xi'|_g|\leq h^\e\},$ and $\Im z\geq -Mh \log h^{-1}$. Then,
$$\|\dDl\varphi\|_{L^2}\leq Ch^{1-\e/2}\|\varphi\|_{L^2}.$$
\end{lemma}
\begin{proof}
Let $\chi_\e\in S_\e(T^*\partial\Omega)$ have $\chi_\e\equiv$ 1 on $\{|1-|\xi'|_g|\leq h^\e\}$ with $\supp\chi_\e\subset \{|1-|\xi'|_g|\leq 2h^\e\}$ and $X=\oph(\chi_\e).$ Then
$X\varphi=\varphi +\O{}(h^\infty)\varphi.$
Fix $0<\e_2<\e_1=\e$. Then let $x_0\in \partial\Omega$ and $\zeta_1, \zeta_2\in \Cc(\partial\Omega)\cap S_\e$ such that $\zeta_i\equiv 1$ on $\{|x-x_0|<Ch^{\e_i}\}$ and $\supp \zeta_1\subset \{|x-x_0|<2Ch^{\e_i}\}.$ Then by Lemma \ref{lem:decomposePrime}, 
$$\zeta_2 \dDl \zeta_1X\varphi =\dDl\zeta_1X\varphi +\O{}(h^\infty)\varphi.$$
Now, by \eqref{eqn:Gglance},
\begin{multline*}
\zeta_2 \dDl\zeta_1X\varphi= -\zeta_2 J\beta^{-1}h^{-2/3}\mc{A}'_-\mc{A}i'CJ^{-1}J(1+\O{L^2\to L^2}(h^{2/3}))J^{-1}\zeta_1X\varphi \\
+\O{L^2\to L^2}(h^\infty)\varphi .\end{multline*}
Next, observe that on for $|\Im z|\leq |\Re z|^{-1/2}$,
$$|Ai'(z)A_-'(z)|\leq C\la z\ra ^{1/2}$$
and $\zeta_iJ$ are elliptic semiclassical FIO's, with symbol in $h^{-\alpha}S_\delta$ for some $\alpha>0$. Therefore,
$$\|\zeta_2 Jh^{-2/3}\mc{A}_-'\mc{A}i'CJ^{-1}\zeta_1X\varphi\|_{L^2}\leq C_0 h^{-1+\e/2}\|\zeta_1X\varphi\|$$
where $C_0$ is a constant depending only on $\Omega$. Taking a partitions of unity as in Lemma \ref{lem:nearGlanceEst} completes the proof.
\end{proof}

Combining Lemma \ref{lem:nearGlanceEstPrime} with Lemma \ref{lem:decomposePrime}, the $L^2$ boundedness of semiclassical FIOs associated to canonical graphs gives, and Lemma \ref{lem:phragmen} gives the following improvement of Theorem \ref{thm:optimal} 
\begin{theorem}
\label{thm:dDlEstImproved}
Let $\Omega\subset \re^d$ be strictly convex with smooth boundary. Then there exists $\lambda_0>0$ such that for some $C$ and all $|\lambda|>\lambda_0$ the following estimate holds
\begin{equation}\nonumber
\|\dDl(\lambda)\|_{L^2(\partial\Omega)\to L^2(\partial\Omega)} \; \leq \; 
C\,\la\lambda\ra\,e^{\LO(\Im \lambda)_-}.
\end{equation}
\end{theorem}


\chapter{Dynamical Resonance Free Regions}
\label{ch:resFree}
 In the early 1900s, Sabine \cite{Sabine} postulated that the decay rate of acoustic waves in a region with leaky walls is determined by the average decay over billiards trajectories. Such a Sabine type law incorporates the detailed properties of both the potential and the domain and has been suggested as a way to study resonances in quantum corrals \cite{Heller} and to study propagation of cellular signals in indoor environments \cite{Cellular}. Our main theorem will give a Sabine type law for the size of the resonance free region for the operators $-\Deltad{\pO}$ and $-\Deltap$ when $\partial\Omega\in C^\infty$ is strictly convex and $V$ is a pseudodifferential operator.
 
 \subsection{Results for $-\Deltad{\pO}$}

Denote the set of rescaled resonances and the set of rescaled resonances that are logarithmically close to the real axis by
\begin{equation}
\label{def:lambda}
\Lambda(h):=\{z\in \complex: z/h \text{ is a resonance of } -\Deltad{\partial\Omega}\}
\end{equation}
and \m\Lambda_{\log}(h):=\{z\in \Lambda(h) :z\in [1-Ch,1+Ch]+i[-Mh\log h^{-1},0]\}\,\,\m respectively.

\begin{remark}
Notice that by rescaling $h\to hE$, we can replace $1$ in the definition of $\Lambda_{\log}$ by $E>0$. Therefore, we restrict our attention to $\Re z$ near 1
\end{remark}
 
The following theorem is a consequence of the much finer Theorem \ref{thm:resFree}
\begin{theorem}
\label{thm:resFreeCompute}
Let $\Omega\subset \re^d$ be a strictly convex domain with $C^\infty$ boundary, $V\in \Ph{}{}(\partial\Omega)$ with $|\sigma(V)|>c>0$. Suppose that $z\in\Lambda_{\log}$. Then for every $\e>0$ there is an $h_0>0$ such that for $0<h<h_0$ 
$$-\frac{\Im z}{h}\geq \frac{1}{d_{\Omega}}\left[\log h^{-1}-\frac{1}{2}\sup_{(a,b)\in\mc{A}}\log\left(\frac{|\sigma(V)(a,0)\sigma(V)(b,0)|}{4}\right)\right]-\e$$
where $d_{\Omega}$ is the diameter of $\Omega$ and 
\m\mc{A}=\{(x,y)\in \partial\Omega\times \partial\Omega: |x-y|=d_{\Omega}\}.\m
\end{theorem}

Let $\mc{D}$ be the domain of $-\Delta_{V,\partial\Omega}$ (see Section \ref{sec:formalDefinition}). Then, as discussed in Chapter \ref{ch:mer},
$\lambda$ is a resonance of the system if and only if there is a nontrivial $\lambda$-outgoing solution $u\in\Dloc$ to the equation
\begin{equation}
\label{eqn:mainDelta}(-\Delta -\lambda^2+V\otimes\delta_{\partial\Omega})u=0\,.
\end{equation}
If $V:H^{1/2}(\partial\Omega)\to H^{1/2}(\partial\Omega),$ the author and Smith showed in \cite[Section 5]{GS} that this is equivalent to solving the following transmission problem 
\begin{equation}
\label{eqn:main}
\begin{cases}(-\Delta -\lambda^2)u_1=0&\text{ in }\Omega\\
(-\Delta -\lambda^2)u_2=0&\text{ in }\re^d\setminus \overline{\Omega}\\
u_1=u_2\,,\;\partial_\nu u_1+\partial_{\nu '}u_2+Vu_1=0&\text{ on }\partial \Omega\\
u_2\;\lambda\text{-outgoing}
\end{cases}
\end{equation}
where we set $u|_{\Omega}=u_1$ and $u|_{\re^d\setminus\overline{\Omega}}=u_2$. Here, we say that $u_2$ is $\lambda$-\emph{outgoing} if there exists $R<\infty$ and 
$\varphi \in L^2_{\comp}(\re^d)$ such that 
$u_2(x)=\bigl(R_0(\lambda)\varphi\bigr)(x)$ for $|x|\ge R\,,$
where $R_0(\lambda)$ is the analytic continuation of the free resolvent $(-\Delta-\lambda^2)^{-1}$, defined initially for $\Im\lambda>0$. In odd dimensions, we take $\lambda\in\complex$ for the above meromorphic continuation, but for even dimensions, we need to consider $\lambda$ as an element of the logarithmic covering of $\complex\setminus\{0\}$.

We now introduce the dynamical and microlocal objects for the finer version of Theorem \ref{thm:resFreeCompute}. Let $\pi:T^*\re^d\to \re^d$ denote projection to the base, $B^*\partial \Omega$ be the coball bundle of the boundary, and $g$ be the induced metric on $\partial\Omega$. Define also 
$$B^*\partial\Omega_r:=\{q\in T^*\partial\Omega\,:\,|\xi'(q)|_g< r\}$$
so that $B^*\partial\Omega=B^*\partial\Omega_1$. 
 Then we denote the billiard ball map (see Section \ref{sec:billiard}) by $\beta:B^*\partial\Omega\to \overline{B^*\partial \Omega}.$
We also denote for $A\subset B^*\partial\Omega$,  $\beta_{-N}(A)=\bigcap_{i=1}^N\beta^{-i}(A).$

Let $l:T^*\partial\Omega\times T^*\partial\Omega\to \re$ be given by $l(q,q'):=|\pi(q)-\pi(q')|$ and write $l_N:B^*\partial\Omega\to \re$ where
\begin{equation}
\label{def:averageLength}
l_N(q):=\frac{1}{N}\sum\limits_{n=0}^{N-1}l(\beta^n(q),\beta^{n+1}(q))
\end{equation}
is the average length between the first $N$ iterates of the billiard ball map originating at $q$.

Recall that for $z\in \Lambda_{\log}$, the single layer operator, $G$, has
\begin{equation}
\label{eqn:G}
G(z/h)=G_{\Delta}(z;h)+G_B(z;h)+G_g(z;h)+\O{L^2\to C^\infty}(h^\infty)
\end{equation}
where $G_{\Delta}$ is a pseudodifferential operator and $G_B$ is a semiclassical Fourier integral operator associated to $\beta$ and $G_g$ is microlocalized near $|\xi'|_g=1$ and the diagonal (see Lemma \ref{lem:decompose}).

Next, let $\chi\in C^\infty(\re)$ with $\chi\equiv 1$ for $x>2C$ and $\chi \equiv 0$ for $x<C$. Then fix $\e>0$ and let 
\begin{equation}
\label{eqn:reflectionOperator}
R_{\delta}(z):=-(I+G_{\Delta}^{1/2}VG_{\Delta}^{1/2})^{-1}G_{\Delta}^{1/2}VG_{\Delta}^{1/2}\chi\left(\frac{1-|hD'|_g}{h^\e}\right)\in h^{1-\frac{\e}{2}}\Psi_{\e}^{-1}.
\end{equation}
The order of $R_{\delta}(z)$ in $h$ may vary from point to point in $B^*\Omega$. In Section \ref{sec:shymbol}, we developed the notion of the shymbol of a pseudodifferential operator, a notion of symbol which is sensitive to local changes of order. Using this idea, we have that the compressed shymbol of $R_{\delta}$ ),
$$\tilde\sigma(R_{\delta}(z))=\frac{h\tilde\sigma(V)}{2i\sqrt{1-|\xi'|_g^2}-h\tilde{\sigma}(V)}\chi\left(\frac{1-|\xi'|_g}{h^\e}\right)$$ 
is the reflection coefficient at the point $(x',\xi')\in B^*(\partial\Omega)$. We call $R_{\delta}$ the reflection operator. 

\begin{remark} The compressed shymbol of the reflection operator agrees, up to lower order terms, with the reflection coefficient found when a plane wave with tangential frequency $\xi'$ interacts with a delta function potential of constant amplitude $\tilde{\sigma}(V)$ on a hyperplane (See \eqref{eqn:reflect}).
\end{remark}

Let $T(z):=G_{\Delta}^{-1/2}(z)G_B(z)G_{\Delta}^{-1/2}(z)$ where $G_B$ is the Fourier integral operator component of  $G(z)$. Then define $r_N(z): B^*(\partial\Omega)\to \re$, the logarithmic average of the reflectivity at successive iterates of the billiard map, by
\begin{equation}
\label{eqn:defineAverageReflection}
r_N(z,q):=\frac{\Im z}{h}l_N(q)+\frac{1}{2N}\log\tilde{\sigma}(((R_{\delta}T(z))^*)^N(R_{\delta}T(z))^N)(q).
\end{equation}
The term $\tfrac{\Im z}{h}l_N$ in \eqref{eqn:defineAverageReflection} serves to cancel the growth of $T(z)$ in the right hand term. In fact,
for $0<N$ independent of $h$ we have
\begin{equation}
\label{def:averageReflection}
r_N(z,q)=\frac{1}{2N}\sum_{n=1}^N\log\left|\left(\tilde{\sigma}(R_{\delta})\composed \beta^n(q) +\O{}(h^{I_{R_\delta}(q)+1-2\e})\right)\right|^2
\end{equation}
where $I_{R_\delta}(q)$ is the local order of $R_{\delta}$ at $q$ (see Section \ref{sec:shymbol}). The expression \eqref{def:averageReflection} illustrates that $r_N$ is the logarithmic average reflectivity over $N$ iterations of the billiard ball map.  Moreover, $r_N(z,q)$ is independent of $z\in \Lambda_{\log}$, so we suppress the dependence on $z$.

 \noindent Note that, if for any $1\leq i \leq N$, $\beta^{i}(q)\notin \WFh(V)$, then for all $M>0$ there exists $h_0$ such that for $0<h<h_0$, 
\begin{equation}\label{eqn:leaveWF}r_N(q)\leq -M\log h^{-1}.\end{equation}

Using Lemma \ref{lem:dynStrictlyConvex} and \ref{lem:iteratedStability}, we have for $h$ small enough, $\e< 1/2$, and $V(h)\in h^{-2/3}\Psi(\partial\Omega),$
\begin{align*}
\inf_{1-\delta\leq |\xi'|_g\leq 1-h^\e}\frac{\log |\tilde{\sigma}(R_\delta)(\beta(q))|^2}{2l(q,\beta(q))}&\\
&\!\!\!\!\!\!\!\!\!\!\!\!\!\!\!\!\!\!\!\!\!\!\!\!\!\!\!\!\!\!\!\!\!\!\!\!\!\!\!\!\leq \inf_{Ch^{\e/2}\leq r\leq \delta^{1/2}}\frac{1+\O{}(r)}{2Cr} \log\left(\frac{\O{}(h^{2/3})}{4r^2+\O{}(r^3)+\O{}(h^{2/3})}\right)\\
&\leq -C\delta^{-{1/2}}\log h^{-1}\end{align*}
where $\tilde{\sigma}(R_\delta )$ denotes the shymbol of $R_\delta$ (see Section \ref{sec:shymbol}).

\noindent Thus, we see that for all strictly convex domains $\Omega$, $0<\e<1/2$, $N_1>0$, and $V\in h^{-2/3}\Ph{}{}$ there exists $\delta_1>0$ such that
\begin{equation}
\label{eqn:awayFromGlancing}
\sup_{N<N_1}\inf_{B^*\partial\Omega_{1-ch^\e}}-l_N^{-1}r_N=\sup_{N<N_1}\inf_{B^*\partial\Omega_{1-\delta_1}}-l_N^{-1}r_N\end{equation}
for some $\delta_1>0$ small enough. That is, the slowest decay rates are those at least a fixed distance away from the glancing region.

With these definitions in hand, we state our main result. 
\begin{theorem}
\label{thm:resFree}
Let $\Omega\subset \re^d$ be a strictly convex domain with $C^\infty$ boundary. Then there exists $\e_{\Omega}>0$ such that for all $V\in h^{-2/3}\Psi(\partial\Omega)$ with $\|\sigma(V)\|_{L^\infty}< \e_{\Omega}h^{-2/3}$, the following holds. For $z\in\Lambda_{\log}$ there exists $\delta_1>0$ such that for every $\e>0$ and $N_1>0$, there is an $h_0>0$ such that for $0<h<h_0$ 
\begin{equation}
-\frac{\Im z}{h}\geq \sup_{N<N_1}\inf_{B^*\pO_{1-\delta_1}}\,-l_N^{-1}(q)r_N(q)-\e.\label{eqn:resFree}
\end{equation}
\end{theorem}
\begin{remarks}

\item The proof of Theorem \ref{thm:resFree} also shows that for each $0<\e$ small enough, and $\e_1>0$,
\begin{align*}-\frac{\Im z}{h}&\leq \inf_{N<N_1}\sup_{B^*\pO_{1-ch^\e}}-l_N^{-1}(q)r_N(q)+\e_1.\end{align*}
However, in strictly convex domains with $V\in h^{-2/3}\Psi(\partial\Omega)$, the quantity on the left goes to infinity for $|\xi'|_g\sim 1-ch^\e$ for $\e>0$ small enough. 

\item Theorem \ref{thm:resFree} is sharp in the case of the unit disk in two dimensions with potential $V\equiv h^{-\alpha}$ (See Appendix \ref{ch:model}). 

\item In typical physical systems, the strength of the interaction between a wave and a potential is a function of the frequency of the waves. This corresponds to considering $h$-dependent $V$. The requirement $\|\sigma(V)\|_{L^\infty}<\e_{\Omega}h^{-2/3}$ comes from the construction of a parametrix for \eqref{eqn:main} near glancing in Section \ref{sec:glancingPoint}. 

However, this is not the natural bound for there to be glancing effects. In fact, the scaling of the problem near glancing dictates that the closest particles can concentrate to glancing is $h^{2/3}$ (i.e.  $|\xi'|_g-1\sim h^{2/3}$). Under this restriction and naively intepreting $r_N$ as the expression \eqref{def:averageReflection}, $|\sigma(V)|=C h^{-5/6}$ coincides with the first time that $|r_N/l_N|\geq c$. Hence, when $|\sigma(V)|\geq ch^{-5/6}$, we expect nontrivial effects from glancing points. In \cite{GalkCircle}, we verify this for  $\Omega$ the unit disk in $\re^2$ and $V\equiv h^{-\alpha}$.
\end{remarks}

To see that the characterization of the resonance free region \eqref{eqn:resFree} can be thought of as a time-averaged Sabine type law, observe that if a wave packet intersecting the boundary for the first time at $q\in B^*\partial\Omega$ starts with energy $E$, then the energy remaining in $\Omega$ after $N$ reflections is given by $$\prod\limits_{n=1}^N|\sigma(R_{\delta})(\beta^n(q))|^2E=\exp\left(2Nr_N(q)\right)E.$$
Thus $-r_N$ is the average exponential rate of decay of the $L^2$ norm over $N$ reflections. Moreover, during the $N$ reflections it takes an average of time $l_N(q)$ to undergo each reflection. Hence, the $L^2$ time rate of decay is $-l_N^{-1}r_N$. Together with the resonance expansions from \cite{GS}, this characterization is a step towards mathematically justifying the use of Sabine laws in the analysis of quantum corrals \cite{Heller}, as well as propagation of cellular signals in indoor environments \cite{Cellular}.

When there is no potential at a point in $B^*\partial\Omega$, the Sabine Law also predicts that wave packets will leave without reflection. Hence, there will be an arbitrarily large exponential rate of decay if every trajectory eventually intersects a point outside of the potential's support.  Theorem \ref{thm:resFree} combined with \eqref{eqn:leaveWF} shows that any trajectory which leaves $\WFh(V)$ has 
$$-l_N^{-1}(q)r_N(q)\geq M\log h^{-1}$$ 
for all $M>0$. Hence, the infimum in Theorem \ref{thm:resFree} effectively excludes trajectories that leave $\WFh(V(\Re zh))$. Moreover, if every trajectory at least $\delta_1$ from glancing eventually leaves $\WFh(V(\Re zh))$, then there is an arbitrarily large logarithmic resonance free strip, verifying the predictions of the Sabine Law.

Theorem \ref{thm:resFree} immediately gives us the following corollary:

\begin{corol}
\label{cor:resFree}
Let $\Omega\subset \re^d$ be a strictly convex domain with $C^\infty$ boundary, $V\in h^{-\alpha}\Psi(\partial\Omega)$ with $\alpha <2/3$. Suppose that $z\in\Lambda_{\log}$. Then, for every $\e>0$, there is an $h_0>0$ such that for $0<h<h_0$ 
\begin{equation}
\label{eqn:resFree2}
-\frac{\Im z}{h\log h^{-1}}\geq \sup_{N>0}\,\,\,\inf_{|\xi'|_g<1}\left\{l_N^{-1}(q)(1-\alpha)-\e\,:\, q\in \beta_{-N}(\WFh (V))\right\}.
\end{equation}
\end{corol}

\begin{remark} Unlike Theorem \ref{thm:resFree}, Corollary \ref{cor:resFree} does not provide information about $C_2$ in $-\Im z/h\geq - C_1\log h^{-1}+C_2$. However, the dynamical quantities are easier to compute than those in Theorem \ref{thm:resFree}.
\end{remark}

\begin{figure}
\centering
\includegraphics[width=4.7in]{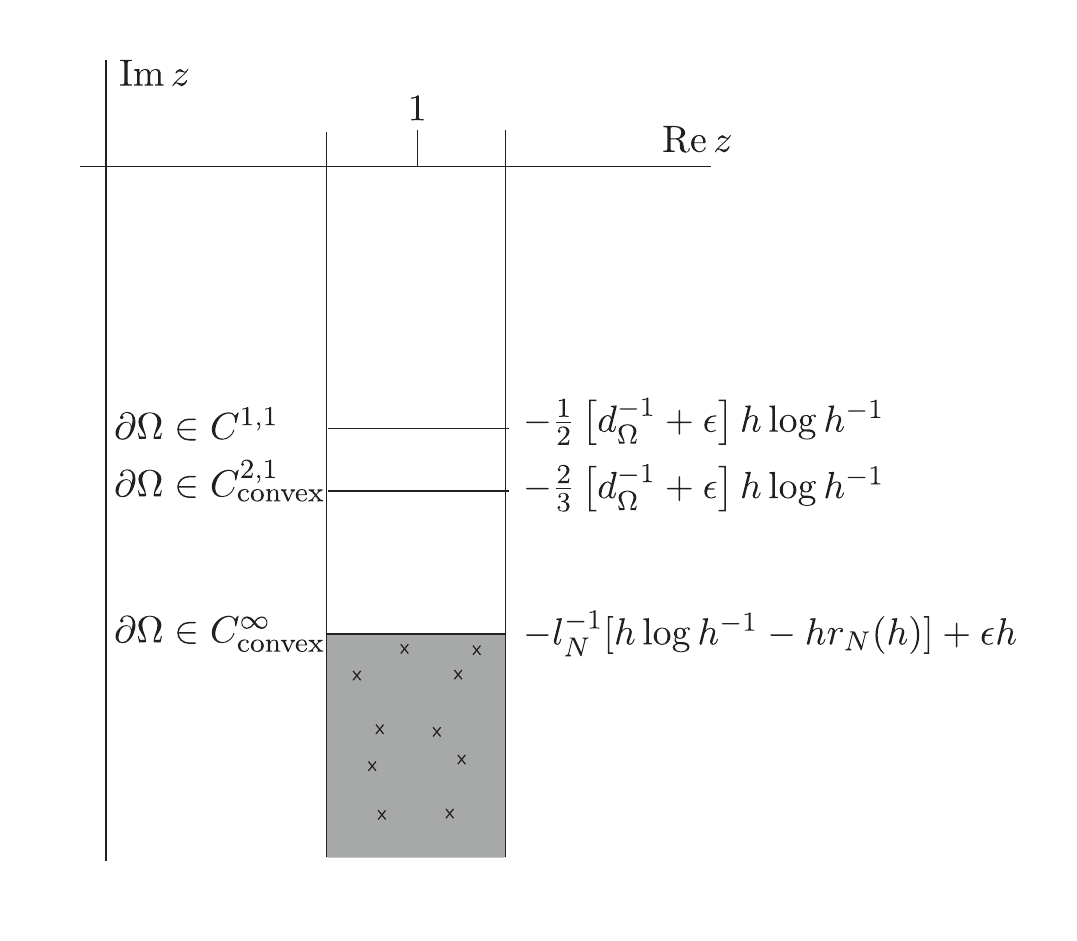}
\hspace{-.5cm}
\vspace{-1cm}
\caption[Resonance free regions for $-\Deltad{\Omega}$]{\label{f:resFree}
We show the various resonance free regions known for $-\Deltad{\Omega}$ when $V\in \Ph{}{}(\partial\Omega)$. The top two lines show the bounds from Chapter \cite{GS}. The lowest is the resonance free region bound from Theorem \ref{thm:resFree}. Since $l_N^{-1}\geq d_{\Omega}^{-1}$, the gap between the bounds from \cite{GS} and Theorem \ref{thm:resFree} is at least $\tfrac{1}{3}d_{\Omega}^{-1}h\log h^{-1}$. If $\sigma(V)=0$ at some points in $B^*\partial\Omega$, then the resonance free region given by Theorem \ref{thm:resFree} can be much larger, while those from \cite{GS} will not change. }
\end{figure}

\subsection{Results for $-\Deltap$}

Here we consider resonances for the operator $-\Deltap$. Recall that resonances are defined as poles of the meromorphic continuation of the resolvent 
$$
R_V(\lambda)=(-\Deltap-\lambda^2)^{-1}\,,\quad\quad \Im \lambda \gg 1,
$$
 and $-\Deltap$ is the unbounded operator
$$
-\Deltap:=-\Delta +\delta'_{\partial\Omega}(V(\lambda)\partial_\nu|_{\partial\Omega}).
\,\,$$

\noindent (See Section \ref{sec:formalDefinition} for the formal definition of $-\Deltap$.) 

We now assume that $\Omega$ is strictly convex and $\partial\Omega\subset \re^d$ is a smooth hypersurface and take $V=V(h)\in h^\alpha \Ph{}{}(\partial\Omega)$ an elliptic (semiclassical) pseudodifferential operator with semiclassical parameter $h=\Re \lambda^{-1}$ and $\alpha>5/6$.  

Denote the set of rescaled resonances and the set of rescaled resonances that are $h\log h^{-1}$ close to the real axis by
\begin{equation}
\label{def:lambdaPrime}
\Lambda^{\delta'}(h):=\{z\in \complex: z/h \text{ is a resonance of } -\Deltap\}
\end{equation}
and \m\Lambda^{\delta'}_{\log }(h):=\{z\in \Lambda^{\delta'
}(h) :z\in [1-Ch,1+Ch]+i[-Mh\log h^{-1},0]\}\,\,\m respectively.
\begin{remark}
As for the $\delta$ potential, it is possible to rescale in $h$ to obtain $\Re z\sim E$ for any $E>0$
\end{remark}

Then the following theorem is a consequence of the much finer Theorem \ref{thm:resFreePrime}
\begin{theorem}
\label{thm:resFreeComputePrime}
Let $\Omega\subset \re^d$ be a strictly convex domain with $C^\infty$ boundary, $V\in h^\alpha C^\infty(\partial\Omega)$ with $V>ch^\alpha>0$, and $\alpha>5/6$. Then there exists a constant $C_{V,\Omega}$ such that for every $\e>0$ there exists $h_0>0$ such that for $0<h<h_0$ and $z\in\Lambda^{\delta'}_{\log}(h)$
$$-\Im z\geq  (C_{V,\Omega}-\e)\begin{cases} h^{3-2\alpha}&5/6<\alpha\leq 1\\
h\log h^{-1}&\alpha>1\end{cases}.$$
\end{theorem}

\begin{remark}
The power $5/6$ is not optimal as can be seen, for example in \cite[Chapter 2]{thesis} in the case of the circle. However, the arguments we rely on in the hyperbolic region require that $\alpha>5/6$. 
\end{remark}

As for the $\delta$ potential, we are able to rescale to the case that $\Re z=1$. We now introduce the dynamical and microlocal objects for the finer version of Theorem \ref{thm:resFreeCompute}. Let $\pi$, $B^*\partial\Omega$, $\beta$, $l$, and $l_N$ be as above. Furthermore, recall that
\begin{equation}
\label{eqn:GPrime}
\dDl (z/h):=\dDl _{\Delta}(z;h)+\dDl _B(z;h)+\dDl _g(z;h)+\O{L^2\to C^\infty}(h^\infty)
\end{equation}
where $\dDl _{\Delta}$ is a pseudodifferential operator, $\dDl _B$ is a semiclassical Fourier integral operator associated to $\beta$ and $D_g$ is microlocalized near $|\xi'|_g=1$ and the diagonal (see Lemma \ref{lem:decompose}). 

We now suppose $V\in h^\alpha \Ph{}{}(\partial\Omega)$ for $\alpha>5/6$ and is self-adjoint with $\sigma(V)>ch^\alpha>0$. Let $\chi\in C^\infty(\re)$ with $\chi\equiv 1$ for $x>2C$ and $\chi \equiv 0$ for $x<C$. Then fix $\e>0$ and let 
\begin{multline}
\label{eqn:reflectionOperatorPrime}
R_{\delta'}(z):=\\(I-\dDl _{\Delta}^{1/2}V\dDl _{\Delta}^{1/2})^{-1}\dDl _{\Delta}^{1/2}V\dDl _{\Delta}^{1/2}\chi\left(\frac{1-|hD'|_g}{h^\e}\right)\in \Ph{1}{\e},
\end{multline}
Then
$$\sigma(R_{\delta'}(z))=\frac{i\sigma(V)\sqrt{1-|\xi'|_g^2}}{i\sigma(V)\sqrt{1-|\xi'|_g^2}-2h}\chi\left(\frac{1-|\xi'|_g}{h^\e }\right)$$ 
is the reflection coefficient at the point $(x',\xi')\in B^*(\partial\Omega)$. We call $R_{\delta'}$ the reflection operator.

\begin{remark} The symbol of the reflection operator agrees up to lower order terms, with the reflection coefficient found when a plane wave with tangential frequency $\xi'$ interacts with a derivative delta function potential of constant amplitude $\sigma(V)$ on a hyperplane. 
\end{remark}

Let $T_{\delta'}(z):=\dDl _{\Delta}^{-1/2}(z)\dDl _B(z)\dDl_{\Delta}^{-1/2}(z)$ where $\dDl _B$ is the Fourier integral operator component of  $\dDl(z)$. Then, using the notion of shymbol defined in Section \ref{sec:shymbol}, we define $r_N^{\delta'}(z): B^*(\partial\Omega)\to \re$, the logarithmic average of the reflectivity at successive iterates of the billiard map, by
\begin{equation}
\label{eqn:defineAverageReflectionPrime}
r^{\delta'}_N(z,q):=\frac{\Im z}{h}l_N(q)+\frac{1}{2N}\log\tilde{\sigma}(((R_{\delta'}T_{\delta'}(z))^*)^N(R_{\delta'}T_{\delta'}(z))^N)(q).
\end{equation}
The term $\tfrac{\Im z}{h}l_N$ in \eqref{eqn:defineAverageReflectionPrime} serves to cancel the growth of $T_{\delta'}(z)$ in the right hand term.  In fact,
for $0<N$ independent of $h$ we have
\begin{equation}
\label{def:averageReflectionPrime}
r^{\delta'}_N(z,q)=\frac{1}{2N}\sum_{n=1}^N\log\left|\left(\sigma(R_{\delta'})\composed \beta^n(q) +\O{}(h^{I_{R_{\delta'}}(\beta^n(q))+1-2\e}))\right)\right|^2.
\end{equation}
The expression \eqref{def:averageReflectionPrime} illustrates that $r^{\delta'}_N$ is the logarithmic average reflectivity over $N$ iterations of the billiard ball map. Moreover, $r^{\delta'}_N(z,q)$ is independent of $z\in \Lambda_{\log}^{\delta'}$, so we suppress the dependence on $z$.

Using Lemma \ref{lem:dynStrictlyConvex} and \ref{lem:iteratedStability}, we have for $h$ small enough and $\e< 1/2$ with $V\in h^{\alpha}\Psi(\partial\Omega),$
\begin{align*}
\inf_{1-\delta\leq |\xi'|_g\leq 1-h^\e}-\frac{\log |\sigma(R_{\delta'})(\beta(q))|^2}{2l(q,\beta(q))}&\\
&\!\!\!\!\!\!\!\!\!\!\!\!\!\!\!\!\!\!\!\!\!\!\!\!\!\!\!\!\!\!\!\!\!\!\!\!\!\!\!\!=\inf_{Ch^{\e/2}\leq r\leq \delta^{1/2}}-\recip{2Cr+\mc O(r^2)} \log\left(\frac{\O{}(h^{-2(\alpha-1)}r^2)}{4+\O{}(h^{-2(\alpha-1)}r^2})\right)\\
&\geq C\delta^{-{1/2}}\min(\log h^{-1}, h^{3-2\alpha}).\end{align*}

\noindent Thus, we see that for all strictly convex domains $\Omega$, $0<\e<1/2$, $N_1>0$, and $V=\O{}(h^\alpha)$
\begin{equation}
\label{eqn:awayFromGlancingPrime}
\sup_{N<N_1}\inf_{|\xi'|_g\leq 1-ch^\e}-l_N^{-1}r^{\delta'}_N=\sup_{N<N_1}\inf_{|\xi'|_g\leq 1-\delta_1}-l_N^{-1}r_N^{\delta'}\end{equation}
for some $\delta_1>0$ small enough. That is, the slowest decay rates are those at least a fixed distance away from the glancing region.

With these definitions in hand, we state our main result. 
\begin{theorem}
\label{thm:resFreePrime}
Let $\Omega\subset \re^d$ be a strictly convex domain with $C^\infty$ boundary, $\alpha>5/6$. Then for all $V\in h^\alpha\Ph{}{}(\partial\Omega)$ self-adjoint and elliptic the following holds. There exists $\delta_1>0$ such that for every $\e>0$ and $N_1>0$, there is an $h_0>0$ such that for $z\in\Lambda^{\delta'}_{\log}(h)$ and  $0<h<h_0$ 
\begin{equation}
\begin{gathered}-\frac{\Im z}{h}\geq \sup_{N<N_1}\inf_{q\in B^*\partial\Omega_{1-\delta_1}}-l_N^{-1}(q)r^{\delta'}_N(q)(1-\e)\label{eqn:resFreePrime}.
\end{gathered}
\end{equation}
\end{theorem}

\section{Conjectures and numerical computation of resonances}
\label{sec:numerical}

We conjecture that the conclusions of Theorem \ref{thm:resFree} hold for much more general domains $\Omega$. In particular, we conjecture that the results hold for convex domains $\Omega$ with piecewise smooth, $C^{1,1}$ boundary.

\begin{figure}[htbp]
\begin{subfigure}[t]{\textwidth}
\centering
\begin{tabular}{ll}
\raisebox{-.075cm}{$\Im \lambda$}\\
\includegraphics[width=4.25in]{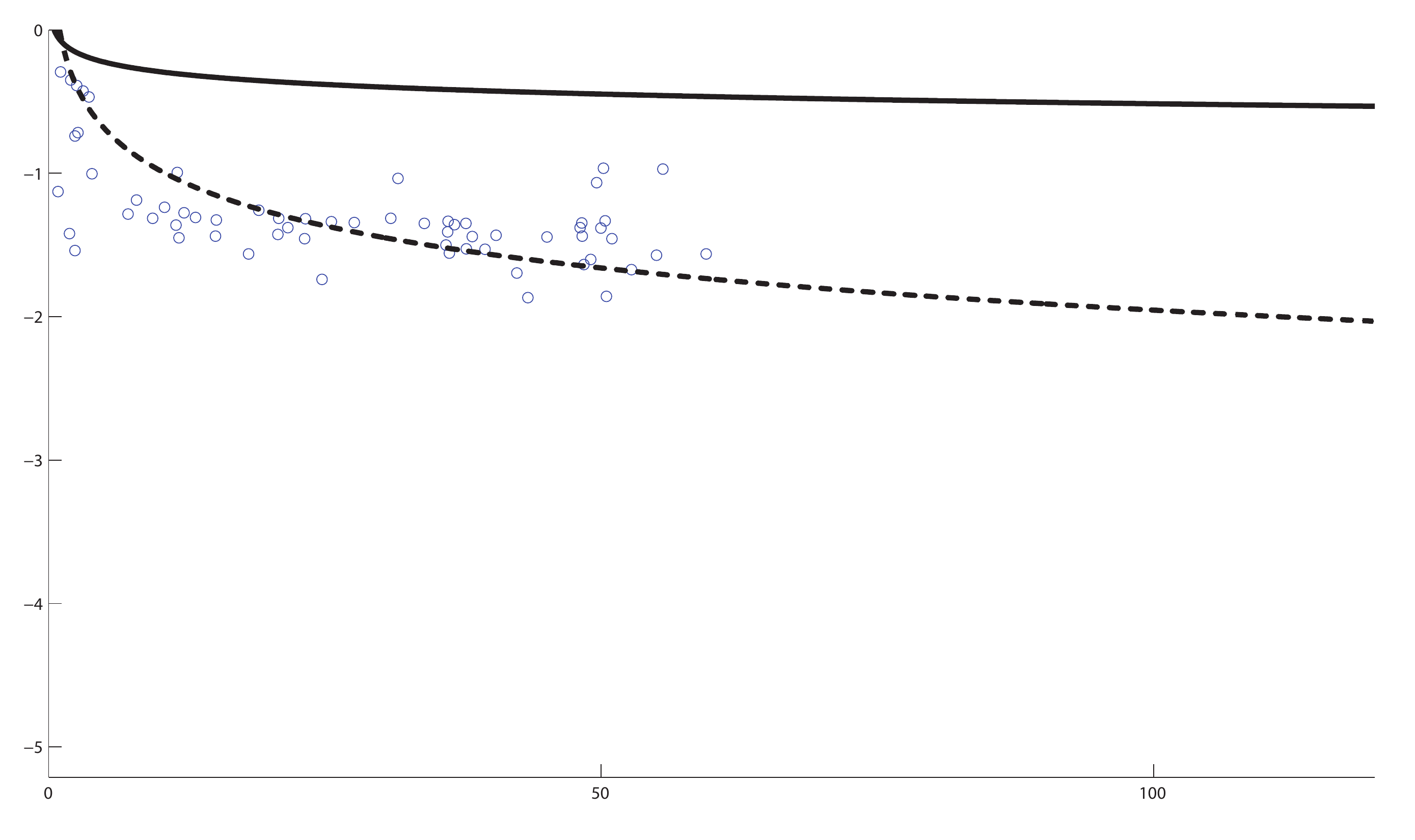}&\raisebox{.45cm}{$\Re \lambda$}
\end{tabular}
\end{subfigure}
\begin{subfigure}[b]{\textwidth}
\centering
\begin{tabular}{ll}
\raisebox{-.25cm}{$\Im \lambda$}\\
$\!\!\!$\includegraphics[width=4.4in]{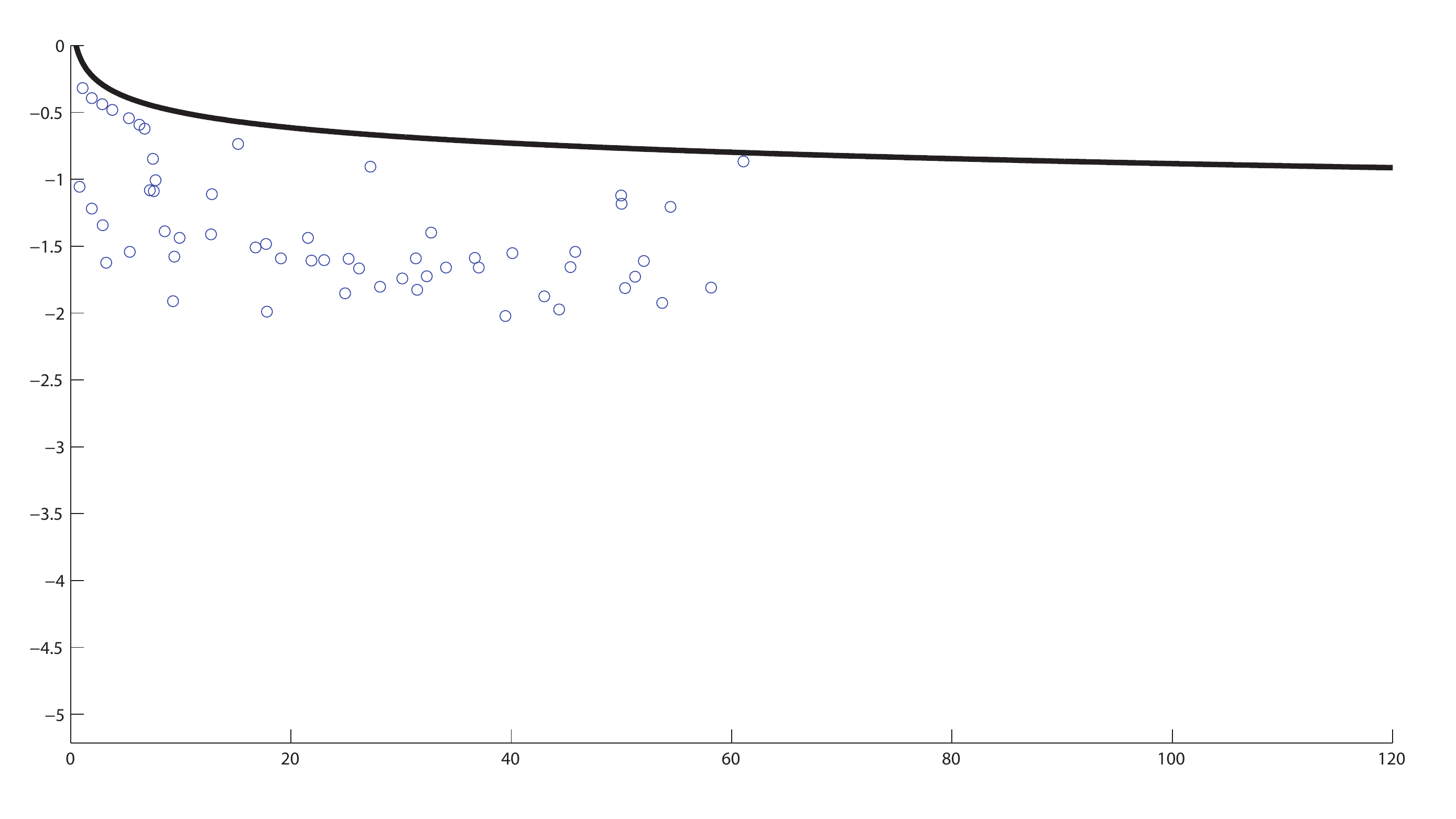}&\raisebox{.6cm}{$\!\!\!\Re \lambda$}
\end{tabular}
\end{subfigure}
\caption[Numerically computed resonance free regions for the ellipse and Bunimovich stadium]{The figure shows resonances for $-\Deltad{\Omega}$ with $\Omega$ the Bunimovich and ellipse with $V=I$ on the top and bottom respectively. The solid lines show the bound of Theorem \ref{thm:resFree}. On the left, the dashed line is that of Conjecture \ref{conj:lowDensity}. As predicted by the conjecture, the resonances appear to cluster around the dashed line for the Bunimovich stadium. On the right, observe that there are many resonances close to the solid line. This gives evidence that the conclusions of Theorem \ref{thm:lowerBoundNumRes} are valid for the ellipse.}
\label{fig:numRes}
\end{figure}

Moreover, we conjecture that 
\begin{conjecture}
\nonumber
\label{conj:lowDensity}
Let $\Omega\subset \re^d$ be a convex domain with piecewise smooth boundary, $V\in h^{-\alpha}\Psi(\partial\Omega)$ for $\alpha<2/3$. Then, for every $\e>0$,
\begin{multline*}
\label{eqn:conj1}
\#\left\{z\in \Lambda_{\log}:-\tfrac{\Im z}{h\log h^{-1}}\leq \essinf_{B^*_1\pO\cap\beta_{-N}(\WFh(V))}\limsup_{N>0}l_N^{-1}(q)(1-\alpha)-\e\right\}\\
=\o{}(h^{-d}) .
\end{multline*}
If, moreover $|\sigma(V)|>ch^{-\alpha}$, then
\begin{equation*}
\#\left\{z\in \Lambda_{\log}:-\tfrac{\Im z}{h\log h^{-1}}\geq \esssup_{|\xi'|_g<1}\liminf_{N>0}l_N^{-1}(1-\alpha)+\e\right\}=\o{}(h^{-d}).
\end{equation*}
\end{conjecture}

This conjecture is a mathematical statement of the space-averaged Sabine law - for ergodic billiards, the exponential decay rate of waves is given by the reciprocal of the average chord length of billiards trajectories. In Figure \ref{fig:numRes}, we can see that the resonances cluster around the line given by the space-averaged Sabine law. The authors of \cite{Heller} numerically compute resonances for scattering by quantum corrals on various domains with ergodic billiard flow. They observe that the resulting resonances cluster around the logarithmic line given by the space-averaged Sabine law. 

In order to compute the resonances of \eqref{eqn:main} in some example domains $\Omega\subset \re^2$, we consider the boundary problem \eqref{eqn:boundaryPrelim}. We discretize $\partial\Omega$ in steps of equal length. After this process, \eqref{eqn:boundaryPrelim} reduces to a matrix equation. We then use a maximum searching algorithm to maximize the condition number of the resulting matrix.

\section{Outline of the proof and organization of the chapter}
\label{sec:outline}

The starting point for the proofs of Theorem \ref{thm:resFree} is the reduction of the solution of \eqref{eqn:main} to the solution of the boundary problems
\begin{equation}\label{eqn:boundaryPrelim}(N_1+N_2+V)\psi=0 \quad\quad\Leftrightarrow
\quad \quad  G(N_1+N_2+V)\psi =(I+GV)\psi=0\end{equation}
where $G$ is as in \eqref{eqn:G}
and $N_1$ and $N_2$ are the Dirichlet to Neumann maps on $\Omega$ and $\re^d\setminus\overline{ \Omega}$ respectively (see Chapter \ref{ch:mer}). The second equality above follows from Section \ref{sec:LayerPotential} or \cite[Section 7.11]{Taylor}.

The strategy for proving Theorem \ref{thm:resFree} is to microlocally decompose the boundary and treat each region separately. The hyperbolic, glancing, and elliptic regions ($\mathcal{H}$, $\mathcal{G}$, and $\mathcal{E}$ respectively) have the property that, letting $U'$ denote a slightly enlarged version of $U$, 
\begin{equation} \label{eqn:stableMicrolocal}(I-\chi_{\mc{H}'})(I+GV)\chi_{\mc{H}}=(I-\chi_{\mc{G}'})(I+GV)\chi_{\mc{G}}=(I-\chi_{\mc{E}'})(I+GV)\chi_{\mc{E}}\equiv 0
\end{equation}
microlocally. Thus, the invertibility of $I+GV$ can be treated separately on each region. 

The first step (see Section \ref{sec:decompose}) in the proof is to use Theorem \ref{thm:freeResolve} to microlocally decompose $G$ into a Fourier integral operator associated with the billiard ball map, a pseudodifferential operator, and an operator microsupported in an $h^\e$ small neighborhood of the diagonal of $S^*\partial\Omega\times S^*\partial\Omega$ (that is, in a small neighborhood of glancing). We denote these operators by $G_B$, $G_\Delta,$ and $G_g$ respectively. 

Section \ref{sec:appDynamics} examines the hyperbolic region, $\mc{H}$. Let $\psi=u|_{\partial\Omega}$ where $u$ is a solution to \eqref{eqn:main}. After some algebraic manipulation of \eqref{eqn:boundaryPrelim}, we arrive at the equation 
\begin{equation}
\label{eqn:reflectionCondition}
(I-(R_{\delta}T)^N)G_{\Delta}^{1/2}V\psi =0,\quad \text{microlocally in } \mathcal{H}
\end{equation}
where $T= G_{\Delta}^{-1/2}G_BG_{\Delta}^{-1/2}$ and $R_{\delta}$ is the reflection operator described in \eqref{eqn:reflectionOperator}. The restrictions on resonances in Theorem \ref{thm:resFree} appear as a consequence of \eqref{eqn:reflectionCondition}. The crucial fact that leads to a logarithmic resonance free region is that $R_{\delta}$ has semiclassical order $<0$. 

In Section \ref{sec:elliptic}, we use the fact that $G_\Delta=\O{L^2\to L^2}(h^{1-\e/2})$ to show that $\MS(\psi)\cap \mathcal{E}=\emptyset$. Finally, we show that $\MS(\psi)\cap \mc{G}=\emptyset$ and hence that $\MS(\psi)=\emptyset$. To do this, we use the microlocal description of $G$ near glancing given in Section \ref{sec:glancingPoint}.

\begin{remark} If one assumes that $V\in h^{-\alpha}\Psi(\partial\Omega)$ for $\alpha<2/3$, then one can avoid the use of the Melrose--Taylor parametrix. We outline this proof in Section \ref{sec:altProof}. 
\end{remark}

The proof of Theorem \ref{thm:resFreePrime} is similar to that of Theorem \ref{thm:resFree}. In particular in Theorem \eqref{thm:ResDomainPrime} we saw that $\lambda$ is a resonance of $-\Deltap$ if and only if there exists a nonzero solution $\psi\in L^2(\pO)$ to
\begin{gather}
 \label{eqn:boundaryReduced} (I-\dDl V)\partial_\nu \psi=0
 \end{gather}
 In particular, \eqref{eqn:bveqn} shows that if $u$ is a $\lambda$-outgoing solution to $(-\Deltap-\lambda^2)u=0$, then
 \begin{gather*} 
 u=R_0(\lambda)\gamma_1^*V\partial_{\nu }u\\
 (I-\dDl(\lambda)V)\partial_\nu u=0
 \end{gather*}
Because of this, the analysis reduces to an analysis of $I-\dDl V$

We decompose $(I-\dDl V)$ microlocally into the hyperbolic, elliptic, and glancing regions as above. These regions have the property given in \eqref{eqn:stableMicrolocal} with $G$ replaced by $\dDl$. Thus, the invertibility of $I-\dDl V$ can be treated separately on each region. The analysis of $\mc{H}$ and $\mc{E}$ is nearly identical to that of $I+GV$, however unlike $G$, the operator $\dDl$ has a smaller operator norm when microlocally restricted near glancing. This allows us to complete the proof of Theorem \ref{thm:resFreePrime}. 

\begin{remark}
Note that the proof of Theorem \ref{thm:resFreePrime} will use the microlocal structure of $\dDl$ near glancing in an essential way for $\alpha\leq 1$.
\end{remark}

\section{Resonance free regions -- delta Potential}
\label{sec:resFree}
We let $z=1+i\omega_0$ with and  $\omega_0\in[-Ch\log h^{-1}, Ch\log h^{-1}]$.

\subsection{Hyperbolic Region: Appearance of the Dynamics}
\label{sec:appDynamics}
Recall from Lemma \ref{lem:decompose} that \m G=G_\Delta+G_B+G_g+\O{L^2\to C^\infty}(h^\infty).\m 

In order to obtain the dynamical restriction on $\Im z$, we localize away from an $h^\e$ neighborhood of $S^*\partial \Omega$. For $k=1,2$, let $\chi_k\in S_{\e}$ with $\chi_k\equiv 1$ on $\{|\xi'|_g\leq 1-(2k+1)Ch^{\e}\}$ and $\supp \chi_k\subset \{|\xi'|_g\leq 1-2kCh^{\e}\}.$ Let $X_k=\oph(\chi_k).$ 
Then, suppose that 
\m (I+GV)X_1\psi=f\,\,\m 
and let $G_{\Delta}^{-1/2}$ be a microlocal inverse for $G_{\Delta}^{1/2}$ on 
$$\mc{H}:=\{|\xi'|_g\leq 1-r_{\mc{H}}h^{\e}\}.$$
Then
\begin{align*}(I+GV)X_1\psi&=(I+(G_\Delta+G_B)V)X_1\psi +\O{}(h^\infty)\psi \\
&=(I+G_\Delta^{1/2}(I+G_{\Delta}^{-1/2}G_BG_{\Delta}^{-1/2})G_{\Delta}^{1/2}V)X_1\psi+\O{}(h^\infty)\psi=f.
\end{align*}
Thus, $f$ is microlocalized on $\mc{H}$ and, following the formal algebra in \cite[Section 2]{Zaletel} multiplying by $G_\Delta^{1/2}V$ and writing $\varphi=G_{\Delta}^{1/2}VX_1\psi$, $T=G_{\Delta}^{-1/2}G_BG_{\Delta}^{-1/2}$, we have 
$$\varphi=-G_\Delta^{1/2}VG_\Delta^{1/2}(I+T)\varphi+\O{}(h^\infty)\psi+G_{\Delta}^{1/2}Vf.$$

\begin{remark} By Lemma \ref{lem:iteratedStability}, a microlocal inverse on $\mc{H}$ will be a microlocal inverse on $\MS(G_BX_1)$.
\end{remark}
Hence, letting 
\m R_{\delta}:=-(I+G_{\Delta}^{1/2}VG_{\Delta}^{1/2})^{-1}G_{\Delta}^{1/2}VG_{\Delta}^{1/2},\,\,\m 
we have 
\m \varphi=R_{\delta}T\varphi+\O{}(h^\infty)\psi-R_{\delta}G_{\Delta}^{-1/2}f.\,\,\m 
Here, $T$ is an FIO associated to the billiard map such that
\begin{multline*}\sigma(\exp\left(\frac{\Im z}{h}\oph (l(q,\beta_E(q)))\right)T)(\beta(q),q)\\
=e^{-i \pi/4} dq^{1/2}\in S\end{multline*}
and $R_{\delta}\in h^{1-\frac{\e}{2}}\Psi_{\e}$ is as in \eqref{eqn:reflectionOperator}.

Thus by standard composition formulae for FIOs, we have for $0<N$ independent of $h$,
\begin{equation}\label{eqn:nonGlance}
(I-(R_{\delta}T)^N)\varphi=\O{}(h^\infty)\psi-\sum_{m=0}^{N-1}(R_{\delta}T)^mR_{\delta}G_{\Delta}^{-1/2}f.\end{equation}

We have that 
\begin{equation}
\label{eqn:appEgorov}
(R_{\delta}T)_N:=((R_{\delta}T)^*)^N(R_{\delta}T)^N=\oph(a_N)+\O{\Ph{-\infty}{}}(h^{\infty})
\end{equation}
where $a_N\in S_{\e}$ and, moreover, for $u$ with $\MS(u)\subset\mc{H}$, by the  Sharp G$\mathring{\text{a}}$rding inequality and \cite[Theorem 13.13]{EZB},
\begin{gather*} 
 \inf_{\mc{H}} \left(|\tilde{\sigma}((R_{\delta}T)_N)(q)|+\O{}(h^{I_{(R_{\delta}T)_N}(q)+1-2\e})\right)\|u\|_{L^2}\leq \|(R_\delta T)^Nu\|_{L^2}^2\\
\|(R_{\delta}T)^Nu\|^2\leq \sup_{\mc{H}} \left(|\tilde{\sigma}((R_{\delta}T)_N)(q)|+\O{}(h^{I_{(R_{\delta}T)_N}(q)+1-2\e})\right)\|u\|_{L^2}.
\end{gather*} 
Let 
$$\beta_1:=1-\sqrt{\sup_{\mc{H}}\tilde\sigma((R_\delta T)_N)}\quad\quad \beta_2:=\sqrt{\inf_{\mc{H}}\tilde\sigma((R_\delta T)_N)}-1.$$
Finally, let $\beta=\max(\beta_1,\beta_2)$. Then, we have 

\begin{lemma}
\label{lem:parametrixMS}
Suppose that $\beta>h^{\gamma_1}$ where $\gamma_1<\min(\e/2,1/2-\e).$ Let $c>r_{\mc{H}}$ and $g\in L^2$ have $\MS(g)\subset \{1-Ch^{\e}\leq |\xi'|_g\leq 1-ch^{\e}\}. $ If 
$$(I-(R_{\delta}T)^N)u=g,$$
then for any $\delta>0$,
 $$\MS(u)\subset \{1-(C+\delta)h^{\e}\leq |\xi'|_g\leq 1-(c-\delta)h^{\e}\}.$$ 
 In particular, there exists an operator $A$ with $\|A\|_{L^2\to L^2}\leq 2\beta^{-1}$, 
$$A(I-(R_\delta T)^N)=I\text{ microlocally on }\mc{H}$$
and if $\MS(g)\subset \{1-Ch^{\e}\leq |\xi'|_g\leq 1-ch^{\e}\}$, then 
$$\MS(Ag)\subset \{1-(C+\delta)h^{\e}\leq |\xi'|_g\leq 1-(c-\delta)h^{\e}\}.$$
\end{lemma}

\begin{proof}
In the case that $\beta_2>h^{\gamma_1}$, we write 
$$(I-(R_\delta T)^N)=-(R_\delta T)^N(I-(R_\delta T)^{-N})$$
microlocally on $\mc{H}$ and invert by Neumann series to see that for any $g$, $(I-(RT)^N)u=g$ has a unique solution modulo $h^\infty$ with $\|u\|\leq \beta^{-1}\|g\|$. On the other hand, if $\beta_1>h^{\gamma_1}$, $\|(R_\delta T)^N\|\leq 1-\beta_1$, and we have that for any $g$, $(I-(R_\delta T)^N)u=g$ has a unique solution with $\|u\|\leq \beta_1^{-1}\|g\|.$

We consider the case of $\beta_1>h^{\gamma_1}$, the case of $\beta_2<h^{\gamma_1}$ being similar with $(R_\delta T)^N$ replace by $(R_\delta T)^{-N}$. Suppose
$(I-(R_\delta T)^N)u_1=g.$ Then, $\|u_1\|\leq \beta ^{-1} \|g\|.$ 

For $k\geq1$, let $\chi_k=\chi_k(|\xi'|_g)$ with $\chi_{k+1}\equiv 1$ on $\supp \chi_k$ and $\chi_1\equiv 1$ on $\MS(g)$ so that 
$$\supp \chi_k\subset \{1-(C+\delta)h^\e\leq |\xi'|_g\leq 1-(c-\delta)h^\e\}.$$ 
Let $X_k=\oph(\chi_k)$. Finally, let $\chi_\infty\in S_\e$ with $\chi_\infty\equiv 1$ on $\bigcup\limits_k\supp \chi_k$ and 
$$ \supp \chi_\infty \subset\{1-(C+2\delta)h^\e\leq |\xi'|_g\leq 1-(c-2\delta)h^\e\}. $$
 Then, 
$$(I-(R_\delta T)^N)X_1u_1=g+\O{}(h^\infty)g+[X_1,(R_\delta T)^N]X_\infty u_1=:g+g_1.$$
 Then by Lemma \ref{lem:iteratedStability} together with the fact that $\chi_1$ depends only on $|\xi'|_g$,
$$[X_1,T]=T(T^{-1}X_1T-X_1)=T(h^{\e}A+h^{1-2\e}B)$$
with $A, B\in \Ph{}{\e}$. 
In fact, 
\begin{equation}
\label{eqn:commuteT} 
T^{-1}X_1T=\oph(\chi_1(\beta(q))+\O{\Ph{}{\delta}}(h^{1-2\e})
\end{equation}
and 
\begin{multline*} 
\chi_1(\beta(q))-\chi_1(q)\\
=\int_0^1\chi_1'((1-t)|\xi'(q)|_g+t|\xi'|_g(\beta(q)))(|\xi'(\beta(q))|_g-|\xi'(q)|_g)dt\in h^{\e}S_\e.
\end{multline*} 
Hence, since $X_\infty u$ is microlocalized $h^\e$ close to glancing, 
$$\MS([X_1,(R_\delta T)^N]X_\infty u_1) \subset \{\chi_2\equiv 1\}$$
and 
$g_1:=[X_1,(R_\delta T)^N]X_\infty u_1$ has 
$$\|g_1\|\leq C(h^{\e}+h^{1-2\e})\beta^{-1}\|g\|_{L^2}.$$
 Then, there exists $u_2$ such that 
 \begin{gather*} (I-(R_\delta T)^N)u_2=-g_1\\
 \|u_2\|\leq \beta^{-1}\|g\|_1\leq C(h^{\e}+h^{1-2\e})\beta^{-2}\|g\|
 \end{gather*}
 So,
$$(I-(R_\delta T)^N)(X_1u+u_2)=g+\O{}(h^\infty)g.$$
Continuing in this way, let 
$$(I-(R_\delta T)^N)u_k=-g_{k-1}\,,\quad g_{k-1}=[X_{k-1},(RT)^N]X_\infty u_{k-1}.$$
Then, 
$$\|u_k\|\leq  \beta^{-2k}(h^{k\e}+h^{k(1-2\e)})\|g\|_{L^2}.$$
Moreover, letting $\tilde{u}\sim \sum_k X_ku_k$, we have $X_\infty \tilde{u}=\tilde{u}+\O{}(h^\infty)\tilde{u}$ and  
$$(I-(R_\delta T)^N)\tilde{u}=g+\O{}(h^\infty)g$$
which implies $\tilde{u}-u=\O{}(h^\infty).$
\end{proof}

Now, assume that $\psi$ solves \eqref{eqn:boundaryPrelim}. Then
$$(I+GV)X_1 \psi =-[X_1,GV]\psi=:f$$ and by Lemmas \ref{lem:iteratedStability} and \ref{lem:decompose}
\m \MS(f)\subset \mc{H}\cap\{|\xi'|_g\geq 1-\tfrac{3}{2}Ch^{\e}\}.\,\,\m
Hence, using \eqref{eqn:nonGlance} and Lemmas \ref{lem:iteratedStability}, and \ref{lem:parametrixMS}, provided that $\beta>h^{\gamma_1}$ for some $\gamma_1<\min(\e/2,1/2-\e)$, 
\begin{equation}
\label{eqn:nonGlanceWithBoundaryEqn}
X_2\varphi =\O{}(h^\infty)\psi.
\end{equation}

We now examine when $\beta\ll h^{\gamma_1}$. For this to occur,
$$\liminf_{h\to 0} \frac{\inf \left||\tilde\sigma((R_{\delta}T)_N)(q)| -1\right|}{h^{\gamma_1}}=0.\,\,\m 
So, let 
$$|\tilde{\sigma}(R_\delta T)_N(q)|=e^{e(q)}.$$
Taking logs and renormalizing we have
$$\frac{2\Im z}{h}Nl_N(q)-\frac{2\Im z}{h}Nl_N(q) +\log |\tilde{\sigma}((R_{\delta}T)_N)(q)|=e(q).$$
So,
\begin{align*}
-\frac{\Im z}{h}&=-l_N^{-1}(q)\left[\frac{\Im z}{h}l_N(q)+\frac{1}{2N}\log |\tilde{\sigma}((R_{\delta}T)_N)(q)|+e(q)\right]\\
&=-l_N^{-1}(q)(r^\delta_N(q)+e(q)).
\end{align*}
where $r_N^{\delta}$ as in \eqref{eqn:defineAverageReflection}.
Thus, if $X_2\varphi\neq \O{L^2}(h^\infty)$, for any $c>0$,
\begin{equation}
\nonumber
\inf_{\mc{H}}-l_N^{-1}(r_N^\delta+ch^{\gamma_1})\leq -\frac{\Im z}{h}\leq \sup_{\mc{H}}-l_N^{-1}(r_N^\delta-ch^{\gamma_1}).
\end{equation}

Now, writing 
\begin{multline*}R_{\delta}T=\left[R_{\delta}\exp\left(-\frac{\Im z}{h}\oph (l(q),\beta(q))\right)\right]\\\left[\exp\left(\frac{\Im z}{h}\oph (l(q),\beta(q))\right)T\right]\end{multline*}
and applying Lemma \ref{lem:EgorovSheaf} shows that
\begin{multline*}\tilde{\sigma}((R_{\delta}T)_N)(q)
=\exp\left(-\frac{2\Im z}{h}\sum\limits_{n=0}^{N-1}l(\beta^n(q),\beta^{n+1}(q))\right)\\\prod\limits_{i=1}^N\left(|\tilde{\sigma}(R_{\delta})(\beta^{i}(q))|^2+\O{}(h^{I_{R_\delta}(\beta^i(q))+1-2\e})\right).\end{multline*}
Since we have assumed $z\in \Lambda_{\log}$, this implies that if $\beta^i(q)\notin \WFh(V)$ for some $0<i\leq N$ then $r_N^{\delta}(q)\leq -M\log h^{-1}$ for all $M$. Hence,
\begin{equation*}
\inf_{\mc{H}\cap \beta_{-N}(\WFh V)}-l_N^{-1}(r^\delta_N+ch^{\gamma_1})\leq -\frac{\Im z}{h}\leq \sup_{\mc{H}}-l_N^{-1}(r^\delta_N-ch^{\gamma_1}).
\end{equation*}

 Now, suppose that $X_2\varphi=\O{L^2}(h^\infty).$ We have by Theorem \ref{thm:optimal} that 
 $$\|G\|_{L^2\to L^2}\leq Ch^{2/3}\log h^{-1}e^{C\frac{(\Im z)_-}{h}}.$$
 \begin{remark} We can also use Theorem \ref{thm:GEstImproved} to remove the $\log h^{-1}$ from the above expression.
 \end{remark} 
 Therefore, 
\m X_3(I +GV)X_1 \psi =X_3\psi +\O{}(h^\infty)\psi=\O{}(h^\infty)\psi.\,\m
We summarize the discussion above in the following lemma.
\begin{lemma}
\label{lem:dynamicalRestriction}
Let $0<\e<1/2$. If for some $\gamma_1<\min(\e/2,1/2-\e)$, and $c>0$
\begin{equation*}
-\frac{\Im z}{h}<\inf_{\mc{H}\cap\beta_{-N}(\WFh(V))}-l_N^{-1}(r_N+ch^{\gamma_1})
\text{  or  }
-\frac{\Im z}{h}>\sup_{\mc{H}}-l_N^{-1}(r_N-ch^{\gamma_1})),
\end{equation*}
where $l_N$ and $r_N$ are as in \eqref{def:averageLength} and \eqref{eqn:defineAverageReflection} respectively, then
\begin{equation*}
\MS (\psi)\subset \{|\xi'|_g\geq 1- ch^\e\}.
\end{equation*}
\end{lemma}

\subsection{Elliptic Region}
\label{sec:elliptic}
Next, we show that solutions to \eqref{eqn:boundaryPrelim} cannot concentrate in the elliptic region $\mc{E}:=\{|\xi'|_g\geq 1+ch^\e\}$ for some $\e>0$.

Fix $\e<1/2$. Let $\chi_1\in S_{\e}$ have $\chi_1 \equiv 1 $ on $|\xi'|_g\geq 1+2Ch^{\e}$ and $\supp \chi_1\subset |\xi'|_g\geq 1+Ch^{\e}$. Also, let $\chi_2\in S_\e$ have $\supp \chi_2\subset |\xi'|_g\geq 1+3Ch^{\e}$ and $ \chi_2\equiv 1$ on $|\xi'|_g\geq 1+4Ch^{\e}$. Let $X_1=\oph(\chi_1)$ and $X_2=\oph(\chi_2).$

Let $\psi$ solve \eqref{eqn:boundaryPrelim}. Then, we have 
\m X_2(I+GV)X_1\psi= -X_2(I+GV)(1-X_1)\psi.\,\,\m 
Now, by Lemma \ref{lem:decompose}
\m GV(1-X_1)=(G_B +G_{\Delta}+G_g)V(1-X_1)+\O{L^2\to L^2}(h^\infty).\,\,\m
But, $X_2(G_B+G_{\Delta}+G_g)V(1-X_1)=\O{L^2\to L^2}(h^\infty)$ since 
\begin{gather*} 
{\MS}'(G_g)\composed \supp(1-\chi_1)\subset \{||\xi'|_g-1|\leq ch^\e\},\\
{\MS(G_B)}'\composed \supp(1-\chi_1)\subset \{|\xi'|_g<1\}.
\end{gather*}
By similar arguments
\m X_2 GVX_1=X_2G_\Delta VX_1+\O{L^2\to L^2}(h^\infty).\,\,\m
Thus, 
\m X_2(I+G_{\Delta}V)X_1\psi =O(h^\infty)\psi.\,\, \m 
Since $\sigma(X_2G_\Delta)=\O{S^{-1}_\e}(h^{1-\frac{\e}{2}})$ and $V\in h^{-2/3}\Psi(\partial\Omega)$, $|\sigma(I+G_{\Delta}V)|>c>0$, this implies
\m \MS (\psi)\cap \{|\xi'|_g\geq 1+2h^{\e}\}=\emptyset.\,\,\m

We also need an elliptic estimate. Let $\supp \chi_3\subset \{|\xi'|_g\geq E+c\}$ and $X_3=\oph(\chi_3)$. Using the fact  
\m \oph(q):=C(1-X_3)+X_2(I+G_{\Delta}V)X_1\,\,\m 
is elliptic and has 
\m \oph(q)\psi=C(1-X_3)\psi +\O{}(h^\infty)\psi,\,\,\m
we have
\m
\|X_1\psi \|_{L^2}\leq C\|(1-X_3)\psi\|_{L^2}+\O{}(h^\infty).\m

Summarizing,
\begin{lemma}
\label{lem:onlyGlancing}
Fix $1/2>\e>0$. Suppose that $|\Im z|\leq Ch\log h^{-1}$ and $\psi=u|_{\partial\Omega}$ where $u$ solves \eqref{eqn:main}. Then
\m \MS(\psi)\cap \{|\xi'|_g\geq 1+h^\e\}=\emptyset.\,\,\m
Moreover, for $\chi\in S$ with $\supp \chi \subset\{ |\xi'|_g\geq 1+c\}$, 
\begin{equation}
\label{eqn:ellpticEstimate}
\|X_1\psi \|_{L^2}\leq C\|(1-\oph(\chi))\psi\|_{L^2}+\O{}(h^\infty).\end{equation}
If, in addition, the hypotheses of Lemma \ref{lem:dynamicalRestriction} hold, then
\m \MS(\psi)\subset \{(x',\xi ')\in T^*\partial\Omega:||\xi '|_g-1|\leq ch^{\e}\}.\m
\end{lemma}

\subsection{Glancing Points}
\label{sec:glancingPoint}

Now, we consider $I+GV$ microlocally near a glancing point.  We use the estimate from Lemma \ref{lem:nearGlanceEst}.

Fix $\e<1/2$. Let $\chi\in S_{\e}$ have $\chi\equiv 1$ on $\{|1-|\xi'|_g|\leq h^\e\}$ and $\supp \chi\subset \{|1-|\xi'|_g|\leq 2h^\e\}$ with $X=\oph(\chi)$. Then suppose that $\psi$ solves \eqref{eqn:boundaryPrelim} and the hypotheses of Lemma \ref{lem:dynamicalRestriction} hold. Then by Lemma \ref{lem:onlyGlancing} and \eqref{eqn:ellpticEstimate},
$X\psi =\psi +\O{}(h^\infty)\psi.$
Therefore, 
\m (I+GV)X \psi =\O{}(h^\infty)\psi.\m
Now, by Lemma \ref{lem:nearGlanceEst},
\begin{align*} 
\|GVX\psi\|_{L^2}&\leq C_{\Omega}h^{2/3}\|VX\psi\|_{L^2}\\
&\leq C_{\Omega}h^{2/3}\left(\sup_{1-Ch^\e\leq |\xi'|_g\leq 1+Ch^\e} |\sigma(V)|+\O{}(h^{1-\alpha-2\e})\right)\|X\psi\|_{L^2}.
\end{align*}
So, provided that $|\sigma(V)|\leq \frac{1}{2C_{\Omega}}h^{-2/3}$, 
$X\psi =\O{L^2}(h^\infty)$
and hence $\psi=\O{L^2}(h^\infty)$, a contradiction. 

\subsection{Sketch of an Alternate Proof Near Glancing}
\label{sec:altProof}
For $V \in h^{-\alpha}\Psi(\partial\Omega)$ with $\alpha<2/3$ one can give an alternate proof avoiding the use of the Melrose--Taylor parametrix and instead use the estimates from Theorem \ref{thm:optimal} on $G$ and that, by Lemma \ref{lem:decompose}, $G_g$ is microlocalized near the diagonal. In particular, note that for $\psi$ microlocalized to a $\delta$ neighborhood of glancing, and $\chi \in \Cc(\partial\Omega)$ with $\chi_1 \equiv 1$ near $x_0$, $\supp \chi_1\subset B(x_0,\delta),$ and $\chi\in \Cc(\partial\Omega)$ with $\chi \equiv 1 $ on $\supp\chi_1$, $\chi_1(I+GV)\chi \psi =\O{}(h^\infty)\psi$ and hence $\chi_1\psi=\o{}(1)\psi$ by Theorem \ref{thm:optimal} together with the fact that $\supp \chi \subset B(x_0,\delta).$

The improvement given by the Melrose-Taylor parametrix comes from the fact that the microlocal model for $G$ gives estimates $\|G\|\leq Ch^{-2/3}$ in $h^\e$ neighborhoods of the diagonal, while those in Theorem \ref{thm:optimal} are of the form $\|G\|\leq Ch^{-2/3}\log h^{-1}$. We also expect that a more detailed analysis of the microlocal model for $G$ near glancing will allow the analysis of potentials $V\in h^{-\alpha}\Psi(\partial\Omega)$ for $\alpha>2/3$.

\subsection{Proof of Corollary \ref{cor:resFree}}
We deduce Corollary \ref{cor:resFree} from Theorem \ref{thm:resFree}.
\begin{proof}
Suppose $V\in h^{-\alpha}\Psi(\partial\Omega)$ and fix $\e>0$. Let 
\m \mc{H}_{\delta_1}:=\{|\xi'|_g< 1-\delta_1\}\cap \beta_{-N}(\WFh V).\,\,\m
First observe that if for some $0<i\leq N$, $\beta^i(q)\notin \WFh(V)$, then for all $M>0$, there exists $h_0$ such that for $0<h<h_0$, $r_N\leq -M\log h^{-1}$. Together with the fact that we assume $\Im z=\O{}(h\log h^{-1})$, this shows that the infimum in
 $$-\frac{\Im z}{h}\geq \inf_{|\xi'|_g<1-\delta_1}-l_N^{-1}(q)r_N(q)$$
excludes trajectories leaving $\beta_{-N}(\WFh V)$ (see also \eqref{eqn:resFree}). Hence, we may reduce to taking an infimum over $\mc{H}_{\delta_1}.$
Observe that 
\begin{align*}
\inf_{\mc{H}_{\delta_1}}-\frac{r_N}{l_N}
&\geq \inf_{\mc{H}_{\delta_1}}l_N^{-1}\left(\log h^{\alpha -1}+\inf_{\mc{H}_{\delta_1}}-(r_N+\log h^{\alpha-1})\right)\\
&\geq \inf_{\mc{H}_{\delta_1}}l_N^{-1}(1-\alpha)\log h^{-1} -C\geq \inf_{\mc{H}_0}l_N^{-1}(1-\alpha)\log h^{-1} -C
\end{align*}
\noindent since $h^{2\alpha}|\tilde\sigma(h^{-1}R_{\delta})|^2\leq C$ on $|\xi'|_g\leq 1-\delta_1$. 

Now, fix $N_1>0$ such that 
\m
\sup_{N>0}\inf_{\mc{H}_0}l_N^{-1}-\tfrac{\e}{2}\leq \inf_{\mc{H}_0}l_{N_1}^{-1}.
\,\,\m
Then, apply Theorem \ref{thm:resFree} and observe that for $h$ small enough, we can absorb $C$ into the first term increasing the factor of $\log h^{-1}$ by at most $\e/2$.
\end{proof}

\section{Resonance free regions -- delta prime potential}
\label{sec:resFreePrime}
Let $V\in h^\alpha \Ph{}{}(\partial\Omega)$, $\alpha>5/6$ be a positive definite self-adjoint operator that has $\sigma(V)\geq Ch^\alpha>0$. In this case, $V$ is invertible for $h$ small enough and $-\Deltap$ can be defined as in Section \ref{sec:defDeltap}. Next, let $z=1+i\omega_0$ with
$$\omega_0\in[-Mh\log h^{-1}, Ch].$$

Recall also that $z/h$ ($z\neq 0$ or $d\neq 1$) is a resonance if and only if there is a nontrivial solution $\psi$ to 
\begin{equation}
\label{eqn:boundaryPrelimPrime}
(I-\dDl V)\psi=0.
\end{equation}
\subsection{Hyperbolic Region: Appearance of the Dynamics}
\label{sec:appDynamicsPrime}
Recall from Lemma \ref{lem:decompose} that 
$$\dDl(z)=\dDl_\Delta+\dDl_B+\dDl_g+\O{L^2\to C^\infty}(h^\infty).$$ 

Let $0<\e<1/2$. Then, suppose that $\chi \in S_\e$ has $\supp \chi \subset \{|\xi'|\leq 1-2h^\e\}$ and let $X=\oph(\chi)$. Finally, suppose that
\m (I-\dDl V)X\psi=f\,\,\m 
and let $\dDl _{\Delta}^{-1/2}$ be a microlocal inverse for $\dDl _{\Delta}^{1/2}$ on 
$$\mc{H}:=\{|\xi'|_g\leq 1-r_{\mc{H}}h^{\e}\}.$$
 Then, following the same process used in section \ref{sec:resFree} for $-\Deltad{\pO}$, we have,
writing $\varphi=\dDl _{\Delta}^{1/2}VX \psi$ and $T=\dDl _{\Delta}^{-1/2}\dDl _B\dDl _{\Delta}^{-1/2}$,
$$(I-\dDl _{\Delta}^{1/2}V\dDl _{\Delta}^{1/2})\varphi=\dDl _\Delta^{1/2}V\dDl _{\Delta}^{1/2}T\varphi+\O{}(h^\infty)\psi+\dDl _{\Delta}^{1/2}Vf.$$

\begin{lemma}
\label{lem:inversePseudo}
The operator $I-\dDl_{\Delta}^{1/2}V\dDl_{\Delta}^{1/2}$ has a microlocal inverse, 
$$(I-\dDl_{\Delta}^{1/2}V\dDl_{\Delta}^{1/2})^{-1}\in \min(h^{1-\alpha-\e/2},1)\Ph{}{\e}(\pO)$$
on $\mc{H}$. 
\end{lemma}
\begin{proof}
Let $B:=\dDl_{\Delta}^{1/2}V\dDl_{\Delta}^{1/2}$. We show that $I-B$ is microlocally invertible. To see this, let $\e_1=2-2\alpha<\frac{1}{2}$. Then for $\tilde{\chi}$ supported on $|\xi'|_g\leq 1-Ch^{\e_1}$, $B^{-1}\tilde{\chi} \in \Ph{}{\e}$ and on $|\tilde{\chi}|>c>0$,
$|\sigma((I-B^{-1})\oph(\tilde{\chi}))|\geq c>0$. Therefore, we can write, microlocally on $|\xi'|_g\leq 1-Ch^{\e_1}$, 
$$(I-B)^{-1}=-(B^{-1}(I-B^{-1}))^{-1}=-(I-B^{-1})^{-1}B.$$
On the other hand, for $\tilde{\chi}$ supported on  $1-2Ch^{\e_1}\leq |\xi'|_g\leq 1-r_{\mc{H}}h^\e$, $B\tilde{\chi}\in \Ph{}{\e}$ and on $|\tilde{\chi}|>c>0$, 
$|\sigma((I-B)\oph(\tilde{\chi}))|>c>0$. Therefore, $(I-B)^{-1}$ exists microlocally on $1-2Ch^{\e_1}\leq |\xi'|_g\leq 1-r_{\mc{H}}h^\e.$ 

Combining these two statements, we see that $(I-\dDl_{\Delta}^{1/2}V\dDl_{\Delta}^{1/2})^{-1}$ exists microlocally on $\mc{H}$ and has the required property.
\end{proof}

\noindent Letting 
\begin{align*} R_{\delta'}&:=(I-\dDl _{\Delta}^{1/2}V\dDl _{\Delta}^{1/2})^{-1}\dDl _{\Delta}^{1/2}V\dDl _{\Delta}^{1/2}\\
&=-I+(I-\dDl _{\Delta}^{1/2}V\dDl _{\Delta}^{1/2})^{-1}\in \min(1,h^{\alpha -1+\e/2})\Ph{}{\e}(\partial\Omega)
\end{align*}
we have 
\m \varphi=R_{\delta'}T\varphi+\O{}(h^\infty)\psi-R_{\delta'}\dDl _{\Delta}^{-1/2}f.\,\,\m 
Here, $T$ is an FIO associated to the billiard map such that
$$\sigma(e^{\frac{\Im z}{h}\oph (l(q,\beta_E(q)))}T)(\beta_E(q),q)=e^{-i \pi/4} dq^{1/2}\in S$$
and $R_{\delta'}$ is as in \eqref{eqn:reflectionOperatorPrime}.

Thus by standard composition formulae for FIOs, we have for $0<N$ independent of $h$,
\begin{equation}\label{eqn:nonGlancePrime}
(I-(R_{\delta'}T)^N)\varphi=\O{}(h^\infty)\psi-\sum_{m=0}^{N-1}(R_{\delta'}T)^mR_{\delta'}\dDl_{\Delta}^{-1/2}f.\end{equation}

We also have that 
\begin{equation}
\label{eqn:appEgorovPrime}
(R_{\delta'}T)_N:=((R_{\delta'}T)^*)^N(R_{\delta'}T)^N=\oph(a_N)+\O{\Ph{-\infty}{}}(h^{\infty})
\end{equation}
where $a_N\in \min (1,h^{N(\alpha-1+\e)})S_{\e}(T^*\partial\Omega)$.

We  now analyze the case that $\psi$ solves \eqref{eqn:boundaryPrelimPrime}.  We start by showing that under a dynamical condition on $\Im z$, there is an $1/2>\e>0$ so that if $\chi_0=\chi_0(|\xi'|_g) \in S_{\e}$ with $\supp \chi_0\subset \{1-Ch^{\e}<|\xi'|_g<1-2h^{\e}\}$
\begin{equation} 
\label{eqn:nearGlanceGoal} \|\oph(\chi_0)\psi\|=\O{}(h^\infty)\psi.
\end{equation}
We then let $\chi_1=\chi_1(|\xi'|_g)\in S_{\e}$ with  $\chi_1\equiv 1$ on $\{|\xi'|_g\leq 1-2h^{\e}\}$ and $\supp \chi_1\subset \{|\xi'|_g\leq E-h^{\e}\}$ and show that there exists $\e>0$ such that both
$$\|\oph(\chi_1)\psi\|\leq (\|\oph(\chi_{0})\psi\|)+\O{}(h^\infty)\|\psi\|$$
and \eqref{eqn:nearGlanceGoal} hold.

To simplify notation, let $X_1=\oph(\chi_1).$ For $i=1,2$, let  $\chi^{(i)}_0=\chi^{(i)}_0(|\xi'|_g)\in S_{\e}$ with $\chi^{(i)}_0\equiv 1$ on $\supp \chi^{(i-1)}_0$ and $\supp \chi_0^{(i)}\subset \{1-(i+1)Ch^{\e}<|\xi'|_g<1-(2-\frac{i}{2})h^{\e}\}.$ Here, $\chi^{(0)}_0=\chi_0$. Finally, let $X^{(i)}_0=\oph(\chi^{(i)}_0).$
We have that 
$$(I-\dDl V)X^{(1)}_0\psi=[X^{(1)}_0,\dDl V]\psi.$$
So, by \eqref{eqn:nonGlancePrime}
$$
(I-(R_{\delta'}T)^N)\varphi=\O{}(h^\infty)\psi-\sum_{m=0}^{N-1}(R_{\delta'}T)^mR_{\delta'}\dDl_{\Delta}^{-1/2}[X_0,\dDl V]\psi.
$$
with $\varphi=\dDl_{\Delta}^{1/2}VX^{(1)}_0\psi.$
Moreover, since $\chi^{(2)}_0\equiv 1$ on $\supp \chi^{(1)}_0$, 
\begin{equation}
\label{eqn:reflectEqn0}
(I-(R_{\delta'}T)^NX^{(2)}_0)\varphi=\O{}(h^\infty)\psi-\sum_{m=0}^{N-1}(R_{\delta'}T)^mR_{\delta'}\dDl_{\Delta}^{-1/2}[X_0,\dDl V]\psi
\end{equation}
Now, let 
$$\mc{N}\mc{G}:=\{1-2Ch^{\e}\leq|\xi'|_g\leq 1-\frac{1}{100}h^{\e}\}.$$
Then 
\begin{multline*}\|(R_{\delta'}T)^NX^{(2)}_0 u\|^2\leq \\\sup_{\mc{NG}} \left(|\tilde{\sigma}((R_{\delta'}T)_N)(q)|^2 +\O{}(\min(1,h^{N(2\alpha-2+\e)})h^{1-2\e}\right)\|u\|^2_{L^2}.\end{multline*}
Let 
$$\beta_0:=1-\sqrt{\sup_{\mc{NG}}\sigma((R_{\delta'}T)_N)}$$

Then the proof of the following lemma is nearly identical to that of Lemma \ref{lem:parametrixMS}.  
\begin{lemma}
\label{lem:parametrixMSprime}
Suppose that $\beta_0>h^{\gamma_1}$ where $\gamma_1<\min(\e/2,1/2-\e).$ Let $c>r_{\mc{H}}$ and $g\in L^2$ have $\MS(g)\subset \{1-Ch^{\e}\leq |\xi'|_g\leq 1-ch^{\e}\}. $ Then if 
$$(I-(R_{\delta'}T)^NX^{(2)}_0)u=g,$$
for any $\delta>0$,
 $$\MS(u)\subset \{1-(C+\delta)h^{\e}\leq |\xi'|_g\leq 1-(c-\delta)h^{\e}\}.$$ In particular, there exists an operator $A$ with $\|A\|_{L^2\to L^2}\leq 2\beta_0^{-1}$, 
$$A(I-(R_\delta T)^N)=I\text{ microlocally on }\mc{NG}$$
and if $\MS(g)\subset \{1-Ch^{\e}\leq |\xi'|_g\leq 1-ch^{\e}\}$, then 
$$\MS(Ag)\subset \{1-(C+\delta)h^{\e}\leq |\xi'|_g\leq 1-(c-\delta)h^{\e}\}.$$
\end{lemma}

Writing 
\m R_{\delta'}T=(R_{\delta'}e^{-\frac{\Im z}{h}\oph (l(q),\beta_E(q))})(e^{\frac{\Im z}{h}\oph (l(q),\beta_E(q))}T)\,\,\m
and applying Lemma \ref{lem:EgorovSheaf} shows that
\begin{multline*}\tilde{\sigma}((R_{\delta'}T)_N)(q)=\exp\left(-\frac{2\Im z}{h}\sum\limits_{n=0}^{N-1}l(\beta^n(q),\beta^{n+1}(q))\right)\\\prod\limits_{i=1}^N\left(|\tilde{\sigma}(R_{\delta'})(\beta^{i}(q))|^2+\O{}(h^{I_{R_{\delta'}}(\beta^i(q))+1-2\e})\right).\end{multline*}
Now, on $\mc{NG}$,
$$|\sigma(R_{\delta'})|^2\leq1-Ch^{2-2\alpha-\e}$$
so, using Lemma \ref{lem:iteratedStability}, we have that on $\mc{NG}$
$$ |\sigma(R_{\delta'}T)_N|\leq 1-ch^{2-2\alpha-\e}+C\frac{\Im z}{h}h^{\frac{\e}{2}}.$$ 
Hence, for $\Im z\geq -Mh^{3-2\alpha-\frac{\e}{4}}$, 
$$ |\sigma(R_{\delta'}T)_N|\leq 1-ch^{2-2\alpha-\e}\quad\quad \imply\quad\quad \beta>h^{2-2\alpha-\e}.$$
So, using that $V$ is invertible and applying $X_0 V^{-1}\dDl^{-1/2}A$ to \eqref{eqn:reflectEqn0} gives
\begin{lemma}
\label{lem:nearGlance}
Fix $M>0$ and suppose that 
$$\Im z\geq -M\min(h^{3-2\alpha-\frac{\e}{4}},h\log h^{-1})$$
and $2-2\alpha-\e<\min(\frac{\e}{2},\frac{1}{2}-\e)$. Then 
$$\|X_0\psi\|=\O{}(h^\infty)\|\psi\|.$$
In particular, the estimate holds when $\frac{2}{3}(2-2\alpha)<\e<\frac{1}{2}.$ 
\end{lemma}

Now, we obtain an estimates on $X_1\psi$. Following the same argument used to get \eqref{eqn:reflectEqn0}, we have
$$(I-(R_{\delta'}T)^N)\varphi_1=\O{}(h^\infty)\psi-\sum_{m=0}^{N-1}(R_{\delta'}T)^mR_{\delta'}\dDl_{\Delta}^{-1/2}[X_1,\dDl V]\psi$$
where $\varphi_1=\dDl_{\Delta}^{1/2}VX_1\psi.$ 

Next, by \cite[Theorem 13.13]{EZB}
\begin{gather*}
\|(R_{\delta'}T)^N\varphi\|^2\leq \sup_{\mc{H}} \left(|\tilde{\sigma}((R_{\delta'}T)_N)(q)|^2 +\O{}(h^{I_{(R_{\delta'}T)_N}(q)+1-2\e})\right)\|\varphi\|^2_{L^2}.\end{gather*}
Define
$$\beta_1:=\min\left(\frac{1}{2},1-\sqrt{\sup_{\mc{H}}\sigma((R_{\delta'}T)_N)}\right).$$
Now, if $\beta_1\geq 0$, then $I_{(R_{\delta'}T)_N}\geq 0$ on $\mc{H}$ and hence
\begin{equation}
\label{eqn:estimateAway}
\begin{aligned} 
\left(\beta_1 -Ch^{1-2\e}\right)\|\varphi_1\|_{L^2}&\leq \|(I-(R_{\delta'}T)^N)\varphi_1 \|_{L^2}\\
&=\left\|\sum_{m=0}^{N-1}(R_{\delta'}T)^mR_{\delta'}\dDl_{\Delta}^{-1/2}[X_1,\dDl V]\psi\right\|\\
&\quad\quad\quad+\O{}(h^\infty)\|\psi\|
\end{aligned} 
\end{equation}

But, by Lemma \ref{lem:nearGlance}, if $\Im z\geq -\min(h^{3\alpha-2-\frac{\e}{4}},h\log h^{-1})$, then 
$[X_1,\dDl V]\psi=\O{}(h^\infty)\psi.$ So, provided that $\beta_1\gg h^{1-2\e}$, 
$$\|\varphi_1\|=\O{}(h^\infty)\|\psi\|$$
and hence, since $X_1\psi=V^{-1}\dDl^{-1/2}\varphi$, 
$$\|X_1\psi\|=\O{}(h^\infty)\|\psi\|.$$

Thus, in order for \eqref{eqn:nonGlancePrime} to hold with $\MS(\psi)\cap \mc{H}\neq \emptyset$, and $z\in \Lambda_{\log}$, for any $\gamma_1<1-2\e$,
\begin{equation}
\label{eqn:symbolRequirementPrime}
\limsup_{h\to 0}\frac{\sup \tilde\sigma((R_{\delta'}T)_N)(q)-1}{h^{\gamma_1}} \geq 0.
\end{equation}
Let
$$|\tilde{\sigma((R_{\delta'}T)_N)(q)}|=e^{e(q)}.$$

Taking logs and renormalizing in \eqref{eqn:symbolRequirementPrime}, we have
$$\frac{2\Im z}{h}Nl_N(q)-\frac{2\Im z}{h}Nl_N(q) +\log |\tilde{\sigma}((R_{\delta'}T)_N)(q)|=e(q)$$
and hence 
\begin{align*}
-\frac{\Im z}{h}&= l_N^{-1}(q)\left[-\left(\frac{\Im z}{h}l_N(q)+\frac{1}{2N}\log |\tilde{\sigma}((R_{\delta'}T)_N)(q)|\right)+e(q)\right]\\
&=l_N^{-1}(q)\left[-r_N(q)+e(q)\right].
\end{align*}
where $r_N$ as in \eqref{eqn:defineAverageReflectionPrime}.
Thus, if $\MS(\psi)\cap \mc{H}\neq \emptyset$, for any $c>0$,
\begin{equation}
\label{eqn:restrict1Prime}
\inf_{\mc{H}}-l_N^{-1}\left[r_N+ch^{\gamma_1}\right]\leq -\frac{\Im z}{h}.
\end{equation}
Notice that when $\Im z=0$ and $|\xi'|_g<1-c$ for some $c>0$, 
$$-r_N\sim \min(h^{2-2\alpha},h\log h^{-1}).$$ This implies that \eqref{eqn:restrict1Prime} provides information about $\Im z$ when $2-2\alpha<1-2\e$. However, by Lemma \ref{lem:nearGlance}, we also need that $\frac{2}{3}(2-2\alpha)<\e$.  Since we have assumed that $\gamma_1<\min(2\alpha-\frac{3}{2},\frac{1}{2})$, we can choose such an $\e$ when $\alpha>11/14$.

Summarizing, we have the following lemma 
\begin{lemma}
\label{lem:dynamicalRestrictionPrime}
Fix $c>0$ and $\frac{2}{3}(2-2\alpha)<\e<\min(\frac{1}{2},\alpha-\frac{1}{2})$. Let $\gamma_1<1-2\e$. If 
\begin{gather}
\label{eqn:prelimResFreePrime}
-\frac{\Im z}{h}<\inf_{\{|\xi'|_g<1-Ch^\e\}\cap\beta_{-N}(\WFh(V))}-l_N^{-1}\left[r_N+ch^{\gamma_1}\right]
\end{gather}
where $l_N$ and $r_N$ are as in \eqref{def:averageLength} and \eqref{eqn:defineAverageReflectionPrime} respectively, and $\psi$ solves \eqref{eqn:boundaryPrelimPrime} then
\begin{equation}
\label{eqn:dynamicsRestrictPrime}
\MS (\psi)\subset \{|\xi'|_g\geq 1- Ch^{\e}\}.
\end{equation}
\end{lemma}

\subsection{Elliptic Region}
\label{sec:ellipticPrime}
Next, we show that solutions to \eqref{eqn:boundaryPrelimPrime} cannot concentrate in the elliptic region $\mc{E}:=\{|\xi'|_g\geq 1+ch^\e\}$ for any $\e<\frac{1}{2}$.

Fix $\e<\frac{1}{2}$. Let $\chi_1\in S_\e$ have $\chi_1 \equiv 1 $ on $|\xi'|_g\geq 1+2Ch^\e$ and $\supp \chi_1\subset |\xi'|_g\geq E+Ch^\e$. Also, let $\chi_2\in S_\e$ have $\supp \chi_2\subset |\xi'|_g\geq 1+3Ch^\e$ and $ \chi_2\equiv 1$ on $|\xi'|_g\geq 1+4Ch^\e$.  Finally, define $X_i:=\oph(\chi_i)$ $i=1,2$. 

Let $\psi$ solve \eqref{eqn:boundaryPrelimPrime}. Then, we have 
\m (I-\dDl V)X_1\psi =[X_1,\dDl V]\psi\,\,\m 
and by Lemma \ref{lem:decompose}
\begin{gather*}  \dDl VX_1=\dDl _{\Delta}VX_1+\O{L^2\to L^2}(h^\infty),\\
X_1\dDl V=X_1\dDl _{\Delta}V+\O{L^2\to L^2}(h^\infty)
\end{gather*}
Observe that the ellipticity of $V$, $\sigma(V)\geq 0$, $\sigma(\dDl_\Delta)\leq 0$ and arguments similar to those giving Lemma \ref{lem:inversePseudo}  show that microlocally on $|\xi'|_g\geq 1+Ch^\e$, 
$$(I-\dDl_\Delta V)^{-1}\in \min(h^{1-\alpha-\e/2},1)\Ph{-1}{\e}(\partial\Omega).$$
Hence
$$X_2\psi=X_2(I-\dDl_\Delta V)^{-1}[X_1,\dDl V]\psi+\O{}(h^\infty)\psi=\O{}(h^\infty)\psi$$
which implies
\m \MS (\psi)\cap \{|\xi'|_g\geq 1+2h^\e\}=\emptyset.\,\,\m

We also need an elliptic estimate. Let $\supp \chi_3\subset \{|\xi'|_g\geq E+c\}$ and $X_3=\oph(\chi_3)$. Using the fact  
\m \oph(q):=C(1-X_3)+X_2(I-\dDl_{\Delta}V)X_1\,\,\m 
is elliptic and has 
\m \oph(q)\psi=C(1-X_3)\psi +\O{}(h^\infty)\psi,\,\,\m
we have
\m
\|X_1\psi \|_{L^2}\leq C\|(1-X_3)\psi\|_{L^2}+\O{}(h^\infty).\m

Summarizing,
\begin{lemma}
\label{lem:onlyGlancingPrime}
Suppose that $|\Im z|\leq Ch\log h^{-1}$ and $\psi=u|_{\partial\Omega}$ where $u$ solves \eqref{eqn:mainPrime} and
$$\frac{2}{3}(2-2\alpha)<\e<\min(\frac{1}{2},\alpha-\frac{1}{2}).$$ Then
\m \MS(\psi)\cap \{|\xi'|_g\geq 1+h^\e\}=\emptyset.\,\,\m
Moreover, for $\chi\in S$ with $\supp \chi \subset\{ |\xi'|_g\geq 1+c\}$, 
\begin{equation}
\label{eqn:ellpticEstimatePrime}
\|X_1\psi \|_{L^2}\leq C\|(1-\oph(\chi))\psi\|_{L^2}+\O{}(h^\infty).\end{equation}
If, in addition, the hypotheses of Lemma \ref{lem:dynamicalRestriction} hold, then for $\e<\min(2\alpha-1,1/2)$,
\m \MS(\psi)\subset \{(x',\xi ')\in T^*\partial\Omega:||\xi '|_g-E|\leq ch^\e\}.\m
\end{lemma}

\subsection{Glancing Points}
\label{sec:glancingPointPrime}

Now, we consider $I-\dDl V$ microlocally near a glancing point.  We use the estimate from Lemma \ref{lem:nearGlanceEstPrime}.

Suppose that $\varphi$ solves \eqref{eqn:boundaryPrelimPrime}, then by Lemma \ref{lem:onlyGlancing}, if $\Im z$ satisfies \eqref{eqn:prelimResFreePrime}, 
\begin{equation} \MS{\varphi}\subset \{|1-|\xi'|_g|\leq \delta h^\e\}\,,\quad \frac{2}{3}(2-2\alpha)<\e<\min(\frac{1}{2},\alpha-\frac{1}{2}).
\label{eqn:microsupportNearGlancingPrime}
\end{equation}
So, let $\chi \in S_\e$ have $\chi\equiv 1$ on $\{|1-|\xi'|_g|\leq h^\e\}$ and $\supp \chi \subset \{|1-|\xi'|_g|\leq 2h^\e\}$ with $X=\oph(\chi).$ Then $X\varphi=\varphi+\O{}(h^\infty) \varphi$. Therefore, 
$$(I-\dDl V)X\varphi =\O{}(h^\infty)\varphi.$$
Then, by Lemma \ref{lem:nearGlanceEstPrime} 
\begin{align*} \|\dDl VX\varphi\|_{L^2}&\leq C_{\Omega} h^{-1+\e/2}\|VX\varphi\|_{L^2}\\
&\leq C_{\Omega,V} h^{-1+\e/2+\alpha}\|X\varphi\|_{L^2}.
\end{align*}
Since $\alpha>5/6$, we can take $2-2\alpha <\e<\min(\alpha-\frac{1}{2},\frac{1}{2})$, and we obtain 
$X\varphi=\O{}(h^\infty)\varphi$ and hence $\varphi =\O{L^2}(h^\infty)$, a contradiction.

\chapter{Existence Resonances for the Delta Potential}
\label{ch:lowerBound}

In this chapter we show that the resonance free region given by Corollary \ref{cor:resFree} is generically optimal for $V\in C^\infty(\partial\Omega)$. In particular, for every periodic billiards trajectory with $M$ reflections whose intersection with $T^*\partial\Omega$ does not leave $\{V\neq 0\}$, there are infinitely many resonances with 
\m-\Im z\leq (l_M^{-1}(q)+\e)h\log h^{-1}\m 
where $q$ is a point in the billiards trajectory.
\begin{theorem}
\label{thm:lowerBoundNumRes}
There exists an open dense collection 
$$\mc{A}\subset \{\Omega\subset \re^d:\partial\Omega\in C^\infty\text{ and } \Omega\text{ is strictly convex}\}$$
such that for all $\Omega \in \mc{A}$ the following statement holds. Suppose that there exists $q\in B^*\partial\Omega$, $M\in \ints^+$ such that $\beta^M(q)=q$. Then for $V\in C^\infty(\partial\Omega)$, if $V(\pi \composed \beta^i(q))\neq 0$ for $0\leq i<M$, we have that for all $\delta>0$ and $\rho>l_M^{-1}(q)$ there exists $h_0>0$ such that for $0<h<h_0$,
$$\#\left\{z\in \Lambda:|z|\leq 1,\,\Im z>-\rho h\log \left(\Re zh^{-1}\right)\right\}\geq ch^{-1+\delta}.$$
\end{theorem}
\begin{remark} The bound on $\rho$ in Theorem \ref{thm:lowerBoundNumRes} matches that in Corollary \ref{cor:resFree}. Hence, the bounds from Corollary \ref{cor:resFree} are generically sharp up to $o(h\log h^{-1})$ corrections. However, one must note that in Theorem \ref{thm:lowerBoundNumRes}, $V\in C^\infty(\partial\Omega)$ is a multiplication operator.
\end{remark}

\subsection{Outline of the Proofs}

Theorem \ref{thm:lowerBoundNumRes} is proved in Section \ref{sec:lowerBound}. The main component of the proof is to describe the singularities of the wave trace, $\sigma(t)$, at times $t>0$ for the problem 
\begin{equation}
\label{eqn:waveProblem}
(\partial_t^2-\Delta +V(x)\otimes \delta_{\partial \Omega})u=0.
\end{equation}
As in \cite{Chaz} and \cite{DuiGui}, we first show that the singularities occur at times $T$ such that $T$ is the length of a closed billiard trajectory. To examine contributions from non-glancing trajectories, we follow \cite{DuiGui}, using the parametrix for \eqref{eqn:waveProblem} constructed in \cite{Saf1} (see also \cite{SafRaySplit}) along with a finer analysis near the boundary.  In particular, we show that 
\m
|\widehat{\psi_{\e,T}\sigma_{\beta}}(\tau)|\geq c\tau^{-N}
\m
where $\sigma_{\beta}(t)$ denotes the wave trace microlocalized near a periodic trajectory of length $T$, $\psi_{\e,T}\in \Cc(\re)$ is a cutoff function near $T$, and $N$ is the number of times the trajectory intersects the boundary. 
Finally, we use the Melrose Taylor parametrix \cite{MelTayl} (see Appendix \ref{sec:waveParametrices}) to show that contributions from trajectories sufficiently close to glancing can be neglected. Moreover, we show that, generically, the wave trace is smooth at accumulation points of the length spectrum. In particular, we have the following consequence of \eqref{eqn:traceEstGlancing}
\begin{prop}
\label{prop:accumulation}
For a generic strictly convex domain $\Omega$ we have that for any closed geodesic $\gamma\in \partial\Omega$, there exists a neighborhood, $U_M\ni T_\gamma$, such that $\sigma(t)$ is $C^M$ on $U_M$. In particular, $\sigma(t)$ is smooth at $T_{\gamma}$.
\end{prop}
\begin{remark}
There has been interest in the singularities of wave traces near accumulation points in the length spectrum. In \cite{CDV}, the authors show that the wave trace is smooth at such points for the Dirichlet Laplacian inside the unit disk in $\re^2$. Proposition \ref{prop:accumulation} gives an analog of such a result in our setting. Generically, the only accumulation points in $L_{\Omega}$ are the lengths of closed geodesics, $\gamma\in \partial\Omega$, \cite[Section 7.4]{Pet}. 
\end{remark}

Next, the Poisson formula of \cite{ZwPoisson} shows that the wave trace $\sigma$ is a distribution of the form
\m\sigma(t)=\sum_{\lambda\in \Lambda}e^{-it\lambda}.\,\,\m
Hence, we are able to use the estimate on the singularities of the wave trace along with \cite[Theorem 1]{SjoZw} to obtain Theorem \ref{thm:lowerBoundNumRes}.

\section{Existence for generic domains and potentials}
\label{sec:lowerBound}
We will use \cite{SjoZw} to establish a lower bound on the number of resonances in a logarithmic region. In particular, letting $\Lambda(h)$ be as in \eqref{def:lambda}, we prove
\begin{lemma}
\label{lem:resLowerBound}
Let $\Omega\subset \re^d$ be strictly convex with smooth boundary and $V\in C^\infty (\partial \Omega)$. Suppose that the length spectrum of the billiard trajectories, $L_{\Omega}$, is simple, that the length of all periodic billiards trajectories are isolated inside $L_{\Omega}$, and that all periodic billiards trajectories are clean. Finally, suppose that there exists $q\in B^*\partial\Omega$ and $M\in \ints^+$ such that $\beta^M(q)=q$ and $V(\pi\composed \beta^i(q))\neq 0$ for all $0\leq i<M$. Then, for all $\rho>l^{-1}_M(q)$ and $\delta >0$, there exist $h_0>0$ and $c>0$ such that for $0<h<h_0$, 
$$\#\left\{z\in \Lambda(h):|z|\leq 1,\, \Im z\geq -\rho h \log\left( \Re zh^{-1}\right)\right\}\geq ch^{-1+\delta}.$$
\end{lemma}
\begin{remark} Our convention is not to include closed geodesics in the boundary in the set of periodic billiards trajectories. We do, however, include the length of such geodesics in $L_{\Omega}.$
\end{remark}

To do this, we describe the singularities of the wave trace for our problem. We then apply the Poisson formula of \cite{SjZwJFA,ZwPoisson} to see that for $t>0$ and andy $k>0$, the wave trace is of the form 
\m \sum_{\lambda\in \Lambda_\gamma} m(\lambda)e^{-i\lambda|t|}+\O{C^\infty}(1)\,,\,\,\m
 where 
 $$
\Lambda_{\gamma}=\{\lambda\text{ a resonance for }-\Delta_{V,\partial\Omega}:\Im \lambda\geq -\gamma | \lambda|\}.$$
Last, the results of \cite{SjoZw} can be applied to yield the lemma.

Let $U_0$ denote the forward free wave propagator. Then for all $T>0$, there exist $M>0$ such that, for $t\leq T$ and $|x|>M$, $U(t,x,x)=U_0(t,x,x)$. Hence, letting $\chi\in \Cc$, $\chi\equiv 1$ on $B(0,M)$, we have 
\begin{align*}\sigma(t)&=\int U(t,x,x)-U_0(t,x,x)dx\\
&=\int \chi U(t,x,x)-\chi U_0(t,x,x)dx=:\sigma_1(t)+\sigma_2(t).
\end{align*}
But, the singularities of $\sigma_2(t)$ occur only at times $t$ for which there exist a periodic geodesic on $\re^d$ with period $t$ \cite{Chaz}. Thus, $\sigma_2\in C^\infty((0,T])$ and we only need to consider singularities of $\sigma_1(t)$ as a distribution in $t$. We denote
$$\sigma_{1,\alpha}(t)=\int \chi(x) (U\composed \alpha)(t,x,x) dx,$$
where $\alpha$ is a microlocal cutoff. 

\subsection{A non-glancing parametrix} 

Safarov \cite[Section 3]{Saf1} (see also \cite[Appendix B]{SafRaySplit}) constructs a local parametrix for the wave transmission problem associated to \eqref{eqn:waveProblem}. We recall the results of the construction in Lemmas \ref{lem:interiorWaveParametrix} and \ref{lem:propSymbols}. 

Let $x=(x_1,x')$ be coordinates near $\partial\Omega$ where $x_1$ is the signed distance from the point to $\partial \Omega$ and $x'$ are coordinates on $\partial \Omega$. (Here $x_1>0$ in $\Omega$ and $x_1<0$ in $\re^d\setminus\overline{\Omega}).$ Then, let $\{g^{ij}\}$ be the inverse metric tensor and $a(x,\xi)$ the Riemannian quadratic form. Finally, let $g'$ and $a'$ be the restrictions of $g$ and $a$ to $T^*\partial\Omega$. In $(x_1,x')$ coordinates,
\begin{gather*}g'(x')=g(0,x'),\quad a(x_1,x',\xi_1,\xi')=\xi_1^2+\tilde{a}(x_1,x',\xi'),\\ a'(x',\xi')=\tilde{a}(0,x',\xi').\end{gather*}

Let $\alpha$ be a pseudodifferential operator. We seek operators $U_\alpha(t)$ with kernel $U_\alpha(t,x,y)$ such that, writing $|_{x_1^+=0}$ for restriction from $\Omega$ and $|_{x_1^-=0}$ for restriction from $\re^d{\setminus \overline{\Omega}}$, we have  
\begin{equation}
\label{eqn:wave}
\begin{cases}
(\partial_t^2-\Delta_x)U_\alpha=0\,,&\text{ for }x\in \re^d\setminus\partial \Omega\\
U_{\alpha}|_{x_1^+=0}=U_{\alpha}|_{x_1^-=0}\,,\\
\partial_{x_1} U_{\alpha}|_{x_1^-=0}-\partial_{x_1}U_{\alpha}|_{x_1^+=0}+V(x)U_{\alpha}=0&\text{ on }\partial\Omega
\\
U_{\alpha}|_{t=0}=\alpha(x,y)\,,\\
 U_\alpha(t)\in C^\infty \,, &\text{ for }t\ll 0.
\end{cases}
\end{equation}
Recall that $G_k^t\,$, the billiards flow of type $k$, is defined as in \eqref{eqn:billiardFlow}, \eqref{eqn:billiardFlow2}, and $\mc{O}_T$; the glancing set is defined as in \eqref{eqn:glancingSet}.

When $t$ is small enough so that no geodesics starting in $\supp \alpha$ hit the boundary, then $U_\alpha$ is, modulo $C^\infty$, the solution to the free wave equation on $\re^d$. Hence it is a homogeneous FIO associated to $G_0^{t}$ \cite[Section 3]{Saf1}. Also, if $\alpha_0$ is a pseudodifferential operator such that $\alpha_0=1$ in a neighborhood of $G_0^{t_1}(\supp \alpha)$, then 
\m U_{\alpha_0}(t)U_{\alpha}(t_1)=U_{\alpha}(t+t_1)\,\,\m
modulo a smoothing operator. 

Thus, to construct $U_\alpha$ we only need to consider the case when geodesics from $\supp \alpha$ hit the boundary in short times. We do this using the ansatz 
\begin{equation}
\label{eqn:boundaryPropForm}
U_{\alpha}(x,y,t)=\sum_{j=1}^3 \int e^{i\varphi_j(x,y,t,\theta)}b_j(x,y,t,\theta)d\theta\end{equation}
with terms in the sum corresponding to the incident ($j=1$), reflected ($j=2$) and transmitted ($j=3$) components. (Here $x_1\leq 0$ for $j=1,2$ and $x_1\geq 0$ for $j=3$.) The phase functions $\varphi_j$ coincide when $x\in \partial \Omega$ and satisfy the eikonal equations 
\begin{equation}
\label{eqn:eikonalWave}
\begin{cases}
\partial_{x_1}\varphi_1+\left[(\partial_t\varphi_1)^2-\tilde{a}(x_1,x',\nabla_{x'}\varphi_1)\right]^{1/2}=0,\\
\partial_{x_1}\varphi_3+\left[(\partial_t\varphi_3)^2-\tilde{a}(x_1,x',\nabla_{x'}\varphi_3)\right]^{1/2}=0,\\
\partial_{x_1}\varphi_2-\left[(\partial_t\varphi_2)^2-\tilde{a}(x_1,x',\nabla_{x'}\varphi_2)\right]^{1/2}=0,\\
\varphi_1|_{x_1^+=0}=\varphi_2|_{x_1^+=0}=\varphi_3|_{x_1^-=0}=0.\end{cases}
\end{equation}

Let the amplitudes $b_j\sim\sum_{n=0}^\infty b_j^n$ where $b_j^n$ is homogeneous in $\theta$ of degree $-n$. The functions $b_j^n$ can be found using the transport equations
\begin{equation}
\label{eqn:transportWave}
2i(\partial_tb_k^j\partial_t\varphi_k-\nabla b_k^j\cdot \nabla\varphi_k)+i(\partial_t^2\varphi_k-\Delta\varphi_k)b_k^j=(\partial_t^2-\Delta)b_k^{j-1}
\end{equation}
once boundary conditions are imposed. These boundary conditions follow from \eqref{eqn:wave} and are given by the equations
\begin{equation}
\label{eqn:FIOSymbolsBoundary1}
ib_1^j\partial_{x_1}\varphi_1+ib_2^j\partial_{x_1}\varphi_2-ib_3^j\partial_{x_1}\varphi_3+\partial_{x_1}(b_1^{j-1}+b_2^{j-1}-b_3^{j-1})
+V(x')b_3^{j-1}=0
\end{equation}

\begin{remark} Equation \eqref{eqn:FIOSymbolsBoundary1} is the only place where we require that $V\in C^\infty$ rather than $V\in \PsiHom(\partial\Omega)$. This is due to the fact that the zero frequency would otherwise appear in the symbol of $V$.
\end{remark}
\begin{equation}
\label{eqn:FIOSymbolsBoundary2}
b_1^j+b_2^j=b_3^j
\end{equation}
at $x_1=0$.  We use the convention that $b_k^{-1}\equiv 0.$

\begin{remark} Notice that $\varphi_1$ and $\varphi_3$ solve the same eikonal equation and hence there is a $C^\infty$ function $\varphi'$ that has $\varphi'|_{x_1\leq 0}=\varphi_3$ and $\varphi'|_{x_1>0}=\varphi_1.$
\end{remark}

Now, combining \cite[Lemma 1.3.17]{SafVas} with \cite[Propositions 3.2 and 3.3]{Saf1} gives that in the case that $\Omega$ is strictly convex

\begin{lemma}

\label{lem:interiorWaveParametrix}
For $t_0<T$, $\nu\in S^*\re^d\setminus (S^*\re^d|_{\partial\Omega}\cup\mc{O}_T)$ there is a conical neighborhood, $W$ of $\nu$ such that for every $\alpha$ with $\WF(\alpha)\subset W$, and $G_k^{t_0}\nu \notin S^*\re^d|_{\partial\Omega}$ for all $k\in K$ 
\m U_{\alpha}(t_0)=\sum_{k}U_{\alpha,k}(t_0)\,\,\m
where $U_{\alpha,k}(t_0)$ is an FIO associated to the canonical relation $G_k^{t_0}$. 
\end{lemma}

In order to describe the singularities of the wave trace, we need to determine the symbols of the $U_{\alpha,k}$. We again follow \cite{Saf1} to do this in our special case. 

\begin{remark} We note that although our case falls into the framework of \cite{Saf1}, unlike in the cases explicitly considered there, the FIO associated to a reflected geodesic will decrease in order and hence become smoother with increasing numbers of reflections. 
\end{remark}

By conjugating by $(\det g^{ij})^{1/2}$, we can associate operators $\alpha'=:\gamma$, $U_\alpha'=V_\gamma$, and $U_{\alpha,k}=:V_{\gamma,k}$, to $\alpha$, $U_\alpha$, and $U_{\alpha,k}$ that act on half densities instead of functions. Let $C_k\subset \re_+\times S^*\re^d\times S^*\re^d$ be the graph of $G_k^t$ with the points $(t+0,G_k^{t+0}\nu,\nu)$ and $(t-0,G_k^{t-0}\nu,\nu)$ sewn together when $G_k^{t}\nu\in S^*\re^d|_{\partial 
\Omega}.$ Then, one can use $(t,y,\eta)$ as coordinates on $C_k$.

Now, we compute the half density component of the symbol of $V_{\gamma,k}(t)$, as in \cite{Saf1}. Define the section $E_k$, by 
\begin{itemize} 
\item[-] $E_k(t,y_0,\eta_0)$ is right continuous for fixed $(y_0,\eta_0)$
\item[-] $E_k|_{t=0}=\exp(i\pi d/4)$
\item[-] $E_k(t,y_0,\eta_0)$ is locally constant for $(y_0,\eta_0)$ fixed  and has discontinuities at the points $t_n$ where the geodesic of type $k$ starting at $(y_0,\eta_0)$ hits the boundary. 
\item[-] At discontinuity points, $E_k(t_n+0,y_0,\eta_0)=F_nE_k(t_n-0,y_0,\eta_0).$
\end{itemize}

\begin{lemma}
\label{lem:propSymbols}
Let $\gamma,\,\gamma_0\in\PsiHom^0(\partial\Omega;\Omega^{1/2})$ be pseudodifferential operators with 
$$ (\WF(\gamma)\cup\WF(\gamma_0))\cap T^*\partial\Omega=\emptyset\,,\quad \text{ and }\quad \WF(\gamma)\cap\mc{O}_T=\emptyset$$
acting on half densities. Then, for any $t_0<T$, $\gamma_0^0E_k(t_0)\gamma^0$ is the symbol of $\gamma_0^0 V_{\gamma,k}$.  Furthermore, the order of the FIO $V_{\gamma,k}$ decreases by 1 after each reflection. 
\end{lemma}

This follows from \cite[Proposition 3.3]{Saf1} and \cite[Appendix B]{SafRaySplit}. We compute the $F_n$ for use below. In particular, assume that $b^j=0$ for $j<N.$ Then we have 
$F_N=1$
if the geodesic is transmitted through $\partial \Omega$ 
and 
$$F_N=0\,,\quad\quad F_{N+1}=-\frac{V+\partial_{x_1}(b_1^N+b_2^N-b_3^N)}{2i\xi_1}=-\frac{V}{2i\xi_1}$$
if it is reflected at $\partial\Omega$. For the second equality, we use the transport equations at $x_1=0$. To determine $F_n$ when the geodesic is reflected, we notice that the principal term is 0 so the computation above follows from \eqref{eqn:FIOSymbolsBoundary1} and \eqref{eqn:FIOSymbolsBoundary2} with $j=N+1$.

\subsection{Analysis for Non-Glancing Trajectories Away from the Boundary}
\label{sec:nonGlancNonBound}
Let $\partial\Omega\Subset U$ and $\chi\in \Cc(\re^d)$ have $\chi\equiv 1 $ on $U$. Let $\chi_1=1-\chi$. We now consider $\sigma_{1,\chi_1}(t)$, that is, the contributions from points away from the boundary. Fix $T>0$ and let $\alpha$ be a pseudodifferential operator with $\WF(\alpha)\cap\mc{O}_T=\emptyset$. 

Using Lemma \ref{lem:interiorWaveParametrix} and an analysis similar to that in \cite{Chaz} and \cite{DuiGui}, singularities of $\sigma_{1,\chi_1\alpha}$ can only occur at $T_j$ where $T_j$ is the period of some billiard flow. That is, $G_k^{T_j}(x,\xi)=(x,\xi).$ We assume that $\Omega$ is strictly convex. Thus, the only periodic trajectories are of type $k=0$ and are trapped inside $S^*\Omega$. 

Since we have assumed that the length spectrum of $\Omega$ is simple and discrete, the fixed point set for the billiards trajectory at a time $T$ is always a submanifold of $S^*M$ of dimension 1 with boundary. Moreover, we have assumed that it is a clean submanifold in the sense of \cite{DuiGui} away from the boundary of $\Omega$.

We now follow \cite{DuiGui} to compute the symbol $\sigma_{1,\chi_1\alpha}$ as a Lagrangian distribution. Let $\rho:\re \times \re^d\to \re \times \re^d\times \re^d$ be the diagonal map. Then, given a half density, $u$ on $\re\times \re^d\times \re^d$, we can pull it back to the diagonal and multiply the two half density factors in $\re^d$ to get, $\rho^*u$, a density in $\re^d$ times a half density in $\re$. Then, we integrate over $\re^d$ to get a half density on $\re$. We denote this by $\pi_*\rho^* u$ where $\pi:\re\times X \to \re$ is the projection map. 

Then, letting $\alpha_i$ be a partition of unity for $S^*B(0,M)$, up to smooth terms, $\pi_*\rho^*\chi  U_{\chi_1\alpha}=\sum_{\alpha_i} \chi \sigma_{1,\chi_1\alpha\alpha_i}$.  Fix $T$, the length of a periodic billiard trajectory. For $\alpha_i$ supported away from the periodic trajectory, Lemma \ref{lem:propSymbols} together with analysis similar to that in \cite{Chaz} shows that $\sigma_{1,\chi_1\alpha\alpha_i}\in C^\infty$. Therefore, we assume without loss that $\alpha_i$ is supported near a periodic trajectory with $N$ reflections and period $T$.

Then, 
\m \sigma_{1,\chi_1\alpha\alpha_i}\in I^{\tfrac{1}{2}-\tfrac{1}{4}-N}(\Lambda_T)\,\,\m with $\Lambda_T:=\{(T,\tau):\tau\in \re\setminus\{ 0\}\}.$ 
Thus, $\sigma_{1,\chi_1\alpha\alpha_i}(t)$ is its symbol times
$$\frac{1}{2\pi}\int s^{-N}e^{-is(t-T)}ds$$
plus lower order terms. The symbol of $\sigma_{1,\chi_1\alpha\alpha_i}(t)$ is given, modulo Maslov factors, by

$$\sqrt{ds}\int_Z \chi_1\alpha E_0(T,\cdot,\cdot)\alpha_i d\mu_Z$$
where $Z$ is the fixed point set of the billiard trajectories of length $T$. 

Now, $Z$ consists of a single billiard trajectory and hence $E_0$ is constant on the fixed point set. Thus the symbol is nonzero as long as $V\neq 0 $ on $\pi(Z\cap S^*B(0,M))|_{\partial \Omega}$. Hence, we have that, summing over $\alpha_i$ supported near the periodic trajectory, 
\begin{equation}
\label{eqn:traceEstInterior}|\hat{\sigma}_{1,\chi_1\alpha}(\tau)|=|\sum_{i}\hat{\sigma}_{1,\chi_1\alpha\alpha_i}(\tau)|\geq c\tau^{-N}.\end{equation}

\subsection{Analysis for Non-Glancing Trajectories Near the Boundary}
\label{sec:nonGlancBound}
We now analyze $\sigma_{1,\chi\alpha}$.  That is, we analyze the wave propagator near boundaries. To do this, we assume without loss of generality that the geodesic starting at $\nu\in S^*\re^d$ intersects $\partial\Omega$ for the first time at $0\leq t_1$ and that the geodesic is traveling in the $-x_1$ direction. 
Let $\alpha$ have support in a small conic neighborhood of $\nu$. By Lemma \ref{lem:propSymbols}, and formula \eqref{eqn:boundaryPropForm}, we see that for $t$ sufficiently close to $t_1$, 
\begin{equation}
\label{eqn:nearBoundary}
U_{\alpha}(t)=A_{\alpha}+H(x_1)R_{\alpha}+H(-x_1)T_{\alpha}
\end{equation}
where $H$ is the Heaviside function, $A_{\alpha}$, and $T_{\alpha}$ are classical FIOs associated to $G_{1/3}^t$ and $R$ is an FIO associated with $\tilde{G}_{0}^t$ where 
\m \tilde{G}_0^t=\exp_{t-t_1}\composed G_0^{t_1}.\,\,\m
That is, $A$ and $T$ are associated to trajectories that are transmitted through the boundary while $R$ is associated to a reflected trajectory. Also, $R$ and $T$ are one order lower than $A$. 

We now check that the multiplication of $R_{\alpha}$ and $T_{\alpha}$ by the Heaviside function is well defined as a distribution and compute its wavefront set. 

We have that
\begin{gather*}
\begin{aligned}\WF(H(x_1))&=\WF(H(-x_1))\\
&=\{(t,(0,x'),y,0,(\xi_1,0),0),\,\xi_1\neq 0\}=N^*(\{x_1=0\})\,,\end{aligned}\\
\begin{aligned}\WF(A)&=\WF(T)\\
&\!\!\!\!\!=\{(t,x,y,\tau,-\xi,\eta):\tau^2=|\eta|^2,\, G_{1/3}^t((y,\eta))=(x,\xi),\,(\tau,-\xi,\eta)\neq 0\}.\end{aligned}\\
\WF(R)=\{(t,x,y,\tau,-\xi,\eta):\tau^2=|\eta|^2,\,\tilde{G}_{0}^t((y,\eta))=(x,\xi),\,(\tau,\xi,\eta)\neq 0\}.
\end{gather*}
Therefore, it is easy to check that 
$$\WF(R)\cap -\WF(H)=\WF(T)\cap-\WF(H)=\emptyset$$
where 
\m -B=\{(t,x,y,-\tau,-\xi,-\eta):(t,x,y,\tau,\xi,\eta)\in B)\}.\,\,\m
Thus, the multiplication is a well defined distribution. Moreover, $\WF(HT)\subset \WF(H)\cup \WF(T)\cup A$ where $$
A:=\left\{\begin{gathered}(t,(0,x'),y,\tau,(\xi_1,\xi'),\eta):\\\exists\, \xi_\nu,\, (t,(0,x'),y,\tau,(\xi_\nu+\xi_1,\xi'),\eta)\in \WF(T)\end{gathered}\right\}.$$
The computation for $\WF(HR)$ follows similarly.

Putting this together with Lemma \ref{lem:interiorWaveParametrix}, we have the following description of $U_\alpha$ away from glancing

\begin{lemma}

\label{lem:waveParametrix}
For $t_0<T$, $\nu\in S^*\re^d\setminus \mc{O}_T$ there is a conical neighborhood, $U$, of $\nu$ such that for every $\alpha$ with support in $U$,
\m U_{\alpha}(t_0)=\sum_{k}U_{\alpha,k}(t_0)\,\,\m
where, if $G_k^{t_0}\nu\notin S^*\re^d|_{\partial\Omega}$, $U_{\alpha,k}(t_0)$ is an FIO associated to the canonical relation $G_k^{t_0}$ and if $G_k^{t_0}\nu\in S^*\re^d|_{\partial\Omega}$, $U_{\alpha,k}(t_0)$ is of the form \eqref{eqn:nearBoundary}.
\end{lemma}
At this point, we can prove a result analogous to that of Chazarain \cite{Chaz}. 
\begin{lemma}
\label{lem:ChazLem}
Fix $T<\infty$ and $\nu \in S^*\re^d\setminus\mc{O}_T$. Then there exists $W$, a neighborhood of $\nu$, such that for $\alpha$ with $\WF(\alpha)\subset W$
\begin{multline*}
\WF(\sigma_{\alpha})\subset\left\{(t,\tau):\exists \nu\in S^*\re^d\cap\supp\alpha,\,k,\text{ such that }\right.\\
\left.G^t_k\nu=\nu \text{ or }(x,\xi)\in S^*\re^d|_{\partial\Omega}\text{ and } G^t_k(x,\xi)=(x,(-\xi_1,\xi'))\right\}
\end{multline*}
\end{lemma}
\begin{proof}

First, we compute the wave front set of $\pi_*\rho^*HT$. The computation for $HR$ will follow similarly. 

We need to check that the pull back $\rho^*(HT)$ is well defined, that is, that $\WF(HT)\cap N^*(\{x=y\})=\emptyset.$ Here 
\m N^*(\{x=y\})=\{(t,x,x,0,\xi,-\xi):\xi\neq 0\}.\,\,\m
Thus, $\WF(H)\cap N^*(\{x=y\})=\emptyset.$ To see that $\WF(T)\cap N^*(\{x=y\})$ is empty, observe that if the intersection were nonempty, then there would be a point with $\tau=0$ in $\WF(T)$. But this implies that $\xi=\eta=0$ and thus $(\tau,\xi,\eta)=0$. Last, we need to check that the third piece of $\WF(HT)$ does not intersect $N^*(\{x=y\}).$ This follows from the same arguments as those for the intersection with $\WF(T)$. 

Finally, we compute 
$$\WF(\pi_*\rho^*(HT))\subset\{(t,\tau):\exists x,\,\eta\text{ such that } (t,x,x,\tau,-\eta,\eta)\in \WF(HT)\}.$$

\noindent Thus, since $\xi_1\neq 0$ and $\eta_1=0$ in $\WF(H)$, the contribution from $\WF(H)$ is empty. As usual, the contribution from $\WF(T)$ is 
\m \{(t,\tau):\exists (x,\eta),\, G^t_{1/3}((x,\eta))=(x,\eta)\}.\,\,\m
Finally, the contribution from 
$$\{(t,(0,x'),y,\tau,(\xi_1,\xi'),\eta):\exists\, \xi_{\nu},\, (t,(0,x'),y,\tau,(\xi_{\nu},\xi'),\eta)\in \WF(T)\}$$
is given by 
$$\{(t,\tau):G^t_{1/3}((0,x'),(\xi_1,\xi'))\in\{((0,x'),(\xi_1,\xi')),((0,x'),(-\xi_1,\xi'))\}\}.$$

\noindent Putting this together with Lemma \ref{lem:waveParametrix}, we obtain the result.
\end{proof}

Now, we need to estimate the size of the singularities of tr$(A)$, tr$(H R)$ and tr$(HT)$ after a given number of reflections. Suppose that there is a periodic trajectory of length $T$ containing $N$ reflections starting at $\nu\in S^*\re^d|_{\partial\Omega}\setminus \mc{O}_{2T}$. 

The terms in \eqref{eqn:nearBoundary} of the form $A$ are classical FIO's and can be analyzed using the methods from the previous section. However, we must determine the size of the singularities. We have that $G^T_0\nu=\nu$. There are two cases.  First, suppose that $\nu$ is inward pointing. Then the relevant term of the form $A$ has no singularities in its trace since it is associated to $G^{T}_{0\dots 01}\nu\neq \nu$ where there are $N-1$ 0's. Now, if $\nu$ is outward pointing, then at $t=0$, $U_\alpha(t)$ is of the form \eqref{eqn:nearBoundary} and the term associated to $G^t_0$ is of the form $HR$ and hence has order -1. Therefore, since at time $T$ the term of the form $A$ has undergone $N-1$ additional reflections, it can be treated as in section \ref{sec:nonGlancNonBound} and cut off in an arbitrarily small neighborhood of the boundary to obtain
\begin{equation}
\label{eqn:mainTerm}
\hat{\sigma}_A(\tau)=o(\tau^{-N}).
\end{equation}

 To handle tr$(HR)$ and tr$(HT)$, observe that, for any $\chi\in \Cc(\re)$ with $\chi(0)=1$, 
$$\chi(x_1)H(x_1)=(2\pi)^{-1}\int (1-\chi_1(\xi_\nu))(\xi_\nu^{-1}+\mc O(\xi_\nu^{-\infty}))e^{ix_1\xi_\nu}d\xi_\nu$$
for some $\chi_1\in \Cc$ with $\chi_1(0)=1$.
By Lemma \ref{lem:ChazLem} (and the fact that $\Omega$ is convex) we only need to consider times $t$ for which there are periodic billiards trajectories. Suppose that there is a billiard trajectory with period $T$ undergoing $N$ reflections.
Then, for $t$ near $T$,
$$\chi H(x_1)R=\sum_j\int e^{ix_1\xi_\nu}e^{i\varphi(t,x,y,\theta)}a_j(t,x,y,\theta,\xi_\nu)d\xi_\nu d\theta$$
where $a_j$ is homogeneous of degree $-j$ in $\theta$ and $-1$ in $\xi_\nu$. Also, since there are $N$ reflections, $a_j\equiv 0$ for $j<N$.

Now, let $\psi\in \Cc$, $\psi(0)=1$, and $\psi_{\e,T}(t)=\psi(\e^{-1}(t-T))$. Let $\chi\in \Cc$ with $\chi\equiv 1$ in a neighborhood of $x_0\in \partial\Omega$ and be  supported in a neighborhood of the boundary with volume $\delta^d$, we are interested in the decay rate of 
\begin{multline*}\hat{\sigma}_{HR}(\tau):=\mc{F}_{t\to \tau}(\text{tr}(\psi_{\e,T}(t)\chi(x)HR))(\tau)=\\
\sum_j\int \psi_{\e,T}(t)e^{ix_1\xi_\nu}e^{i\varphi(t,x,x,\theta)}e^{-it\tau}(\chi(x)a_j(t,x,x,\theta,\xi_\nu))d\xi_\nu dx d\theta dt\\+\mc O(\tau^{-\infty})
\end{multline*}
for sufficiently small $\e$. We rescale $\xi_\nu$ and $\theta$ to $\xi_\nu /\tau$ and $\theta/\tau$ to obtain
\begin{multline*}
\hat{\sigma}_{HR}(\tau)=\\
\sum_j\tau^{d-j}\int \psi_{\e,T}(t)\chi(x)a_j(t,x,x,\theta,\xi_\nu)e^{i\tau(x_1\xi_\nu+\varphi(t,x,x,\theta)-t)}d\xi_\nu dx d\theta dt\\
+\mc O(\tau^{-\infty}).\end{multline*}

\noindent Next, observe that the phase is stationary in $x$, $t$, and $\theta$ on 
$$d_\theta\varphi=d_t\varphi-1=d_{x'}\varphi+d_{y'}\varphi=\xi_\nu+d_{x_1}\varphi +d_{y_1}\varphi=0.$$

Now, since $R$ is associated to $\tilde{G}_0^t$, these equations are equivalent to 
$$|(\xi_1,\xi')|=1,\quad x=\pi \tilde{G}_0^t(x,\xi),\quad \xi'=\pi_{\xi'}\tilde{G}_0^t(x,\xi),\quad \xi_1=\pi_{\xi_1}\tilde{G}_0^t(x,\xi)+\xi_\nu.$$
Hence, since the length spectrum is simple and discrete, the critical points form a submanifold of dimension 1 with volume less than $c\delta$. Now, by assumption the fixed point set of the billiards trajectory is clean and hence $\varphi$ is clean with excess 1. Let $\psi =x_1\xi_\nu +\varphi-t$ then since as functions of $(x,\theta,t)$,
$\partial^2\psi=\partial^2\varphi$
 these stationary points are clean with excess 1.

Thus, after cutting off to a compact set in $\xi_{\nu}$, we may apply the clean intersection theory of Duistermaat and Guillemin \cite{DuiGui}. To handle $|\xi_\nu| \geq C$, observe that there are no critical points for $|\xi_\nu|\geq 3$ and hence that this piece of the integral can be handled using the principle of nonstationary phase.

Now, letting $\delta\to 0$, and observing that we are integrating in $2d+1$ variables and have a $1$ dimensional submanifold of critical points, we obtain
\begin{equation}
\label{eqn:boundarySingEstimate}\hat{\sigma}_{HR}(\tau)=o(\tau^{-N+d-\frac{2d+1}{2}+\frac{1}{2}})=o(\tau^{-N}).\end{equation}

\begin{remark} Observe that this computation relies on the fact that, after reflection, the order of the FIO decreases by 1.
\end{remark}

Now, putting \eqref{eqn:boundarySingEstimate} together with \eqref{eqn:traceEstInterior}, and \eqref{eqn:mainTerm}, 
\begin{equation}
\label{eqn:traceEstNonglancing}
|\hat{\sigma}_{1,\alpha}(\tau)|\geq c\tau^{-N}.
\end{equation}

\subsection{Glancing Trajectories}
\label{sec:glanceTraj}
We are interested in tr($U(t)-U_0(t))$ as a distribution. Let
$$(u_1,u_2)(t)=(U(t)-U_0(t))u_0.$$
Then
\begin{equation}
\label{eqn:trace}
\begin{cases}
(\partial_t^2-\Delta)u_i=0&\text{in }\Omega_i\\
u_1-u_2=0\,,\text{ and }\partial_{\nu}u_1+\partial_{\nu'}u_2+Vu_1=-VU_0(t)u_0&\text{on }\partial\Omega
\end{cases}
\end{equation}
By the arguments in Sections \ref{sec:nonGlancNonBound} and \ref{sec:nonGlancBound}, we can assume that $u_0$ has wave front set in a neighborhood of $\nu\in \mc{O}_T$. Then, by arguments identical to those in Section \ref{sec:glancingPoint}, we can use the operators constructed in Appendix \ref{sec:waveParametrices} (see also \cite[Section 11.3]{MelTayl}) together with the ideas in Section \ref{sec:BLONearGlance} to find a microlocal description of the solution to this problem restricted to the boundary of the form 
$$(I+J\beta^{-1}\mc{A}i\mc{A}_-J^{-1}B)u|_{\pO}=-VU_0(t)u_0|_{\pO}$$
where $B\in \PsiHom^{-2/3}$.
Because $B\in \PsiHom^{-2/3}$, we can invert the operator on the left hand side by Neumann series
\begin{equation}\label{eqn:AirySum}(I+\beta^{-1}J\mc{A}i\mc{A}_-J^{-1}B)^{-1}=\sum_{k=0}^\infty (-\beta ^{-1}J\mc{A}_-\mc{A}iJ^{-1}B)^{k}.\end{equation}
 Hence, truncating the sum \eqref{eqn:AirySum} at a sufficiently high $K>0$ gives a $C^M$ parametrix and the remainder contributes a term of size $O(\tau^{-M})$ to the trace.

To analyze the singularities near glancing, we need the following lemma (see \cite[Sections 5.4]{MelTayl} and Lemma \ref{lem:WFhAiry}).

\begin{lemma}
\label{lem:waveFrontGlance}
$\WF{}'(\mc{A}_-\mc{A}i)\subset \graph(Id)\cup \graph(\beta)=:C_b$
\end{lemma}
\noindent Thus, up to glancing, the wavefront set of the solution $u$ is contained in the billiard trajectories through $\nu$. 
\begin{proof}
The proof follows that of Lemma \ref{lem:WFhAiry} by letting $h=\tau^{-1}$, $\e(h)=h$, and rescaling $\xi\to \tau\xi.$ Using this transformation and the fact that in the construction of the Melrose Taylor parametrix, we have that $\alpha=\alpha_0+i$, we replace $\alpha_h$ with $\zeta=\tau^{-1/3}(\xi_1+i)$. 
Then in $\xi_1<0$, up to a lower order term $\O{}(\zeta')$, we obtain the phase function
$$\tau=\la t-s,\tau\ra+\la x-y,\xi\ra +\frac{4}{3}(-\xi_1)^{3/2}\tau^{-1/2}$$
which parametrizes $\beta$. The $i$ term is a symbolic perturbation hence the wavefront set is given by the billiard relation.

The only thing that remains is to show that 
\m {\WF} '(\mc{A}_-\mc{A}i)\cap \{\xi_1=0\}\subset\{t=s\}.\,\,\m
To do this, let
\m V_1=\partial_{\tau}+\frac{1}{3}\tau^{-1}(\xi_1+i)\partial_{\xi_1}.\,\,\m 
Then, $V_1\zeta=0$ and hence 
\m V_1(A_-Ai(\zeta))=0.\,\,\m
The symbol of $V_1$ is given by $it+\tfrac{1}{3}\tau^{-1}(\xi_1+i)x_1$. Thus, it is elliptic on $\xi_1=0$ away from $t=0$, and we have
\m {\WF}'(\mc{A}_-\mc{A}i)\cap \{\xi_1=0\}\subset \{t=0\}.\,\,\m
Combining this with the results of Lemma \ref{lem:WFhAiry} completes the proof
\end{proof}

By Lemma \ref{lem:waveFrontGlance} together with the arguments used to prove Lemma \ref{lem:ChazLem}, we have that the singularities of $\tr\,\sigma_{\alpha}(t)$ for $\alpha$ supported near a glancing trajectory are contributed by closed billiards trajectories. Observe that since $\Omega$ is strictly convex, as a trajectory approaches glancing, the number of reflections in a closed billiard trajectory increases without bound. Hence we see that in a small enough neighborhood of glancing the first $K$ terms in \eqref{eqn:AirySum} contribute no singularities to the trace. Thus for all $M>0$, by choosing a small enough neighborhood of glancing, we have
 \begin{equation}
 \label{eqn:traceEstGlancing}
 |\hat{\sigma}_{1,1-\alpha}(\tau)|=\mc O(\tau^{-M}).
 \end{equation}
 
One does not need the precise estimate \eqref{eqn:traceEstGlancing} to prove Theorem \ref{thm:lowerBoundNumRes}. One only needs that the singularities up to glancing are contained in the length of periodic billiards trajectories. That is, Lemma \ref{lem:waveFrontGlance} or another propagation of singularities result is enough. However, the precise estimate \eqref{eqn:traceEstGlancing} proves Proposition \ref{prop:accumulation}.

 \subsection{Completion of the proof of Lemma \ref{lem:resLowerBound}}
 \label{sec:completeLowBound}
Let $T$ be the primitive length of the periodic billiard trajectory with $M$ reflection points contained in $\sigma(V)\neq 0$. Then, let $\varphi\in \Cc$, $\varphi(0)\equiv 1$ and $\supp \varphi\subset (-1,1).$ Define $\varphi_{T,\e}(t)=\varphi(\e^{-1}(t-T)).$ Then, Lemma \ref{lem:waveFrontGlance}, the fact that $T$ is isolated in $L_{\Omega}$, and \eqref{eqn:traceEstNonglancing} show that for $\e>0$ small enough and $\tau$ large enough,
\m \left|\widehat{\varphi_{T,\e}\sigma_1}(\tau)\right|\geq c \tau^{-M}.\,\,\m

Now, using \cite[Lemma 7.1]{GS} we see that 
\m \mu_j(\chi R_V(1+iM))\leq Cj^{-1/d}\,\,\m 
and hence by \cite{SjoDist}, \cite{Vod1}, \cite{Vod2}, and \cite{Vod3},
\m \#\{z_j:|z_j|\leq 1, \Im z_j\geq -\gamma/h\}=\mc O(h^{-d}).\,\,\m
Thus, by \cite[Theorem 1]{SjoZw} for all $\delta>0$ and $\rho>\frac{d+M}{T}$, there exists $c>0$ such that
\begin{equation}
\label{eqn:resLowerBound}
\#\left\{z_j:|z_j|\leq 1,\, \Im z_j\geq-\rho h\log\left( \Re z_jh^{-1}\right)\right\}\geq ch^{-1+\delta}
\end{equation}
for $h$ small enough. But, if there is a periodic geodesic with $M$ reflections of length $T$, then there is also one with $nM$ reflections of length $nT$. Therefore, taking $n$ large enough, we have \eqref{eqn:resLowerBound} for all $\rho >\frac{M}{T}$. This completes the proof of Lemma \ref{lem:resLowerBound}.

Theorem \ref{thm:lowerBoundNumRes} follows from Lemma \ref{lem:resLowerBound}, together with the fact that strictly convex domains in $\re^d$ generically have simple length spectrum such that the only accumulation points are the lengths of closed geodesics in $\partial\Omega$ and have periodic billiards trajectory which are clean submanifolds (\cite[Chapter 3, Section 7.4]{Pet}).

\appendix


\chapter{Model Cases}
\label{ch:model}

In the present appendix, we show that the resonance free regions found in Chapter \ref{ch:resFree} are sharp. In particular, we consider $-\Deltad{\partial\Omega}$, $-\Deltap$, when $\Omega=B(0,1)\subset \re^2$ and $V$ is a constant depending on $h$. In this case, we are able to separate variables and avoid most of the microlocal analysis involved in obtaining the more general results. Separating variables reduces the existence of resonances to the existence of a solution to one of an infinite family of transcendental equations. The symbols of the operators involved in the general analyses appear as asymptotic limits of the Bessel and Airy functions in these equations. 

\subsection{Statement of results for the $\delta$ potential}
For the purposes of this section, we define the resonances of $-\Deltad{\partial\Omega}$ as follows: We say that $z/h$ is a resonance for $-\Deltad{\partial\Omega}$ if there exists a nonzero $z/h$-outgoing solution, $(u_1,u_2)\in H^2(\Omega)\oplus H^2_{\loc}(\re^d\setminus\Omega)$ to 
\begin{equation}
\label{eqn:mainCircle}
\begin{cases}(-h^2\Delta -z^2)u_1=0&\text{ in }\,\Omega\\
(-h^2\Delta -z^2)u_2=0 & \text{ in }\,\re^d\setminus \overline{\Omega}\\
u_1=u_2&\text{ on }\,\partial \Omega\\
\partial_\nu u_1+\partial_{\nu '}u_2+V\gamma u_1=0&\text{ on }\,\partial \Omega
\end{cases}
\end{equation}
where, $\partial_\nu$ and $\partial_{\nu'}$ are respectively the interior and exterior normal derivatives of $u$ at $\partial\Omega\,.$
In Chapter \ref{ch:mer}, we show that having such a solution corresponds to having a pole $R_{-\Deltad{\partial\Omega}}$ and hence that these are indeed the resonances for $-\Deltad{\partial\Omega}$. 

Denote the set of rescaled resonances for $-\Deltad{\partial\Omega}$ by
\begin{equation}
\label{def:lambdaCircle}
\begin{aligned} \resd&:=\{z\in B(h): z/h\text{ is a resonance of } -\Deltad{\partial\Omega}\}\\
B(h)&:=[1-ch^{3/4},1+ch^{3/4}]+i[-Mh\log h^{-1},0]
\end{aligned}.
\end{equation}

\begin{remark} The power $3/4$ can be taken to be any power $>0$.\end{remark}

We assume throughout that $V\equiv h^{-\alpha} V_0$ for $\alpha <1$, and $V_0> 0$ a constant independent of $h$. 
The next theorem shows that the resonance free regions in Chapter \ref{ch:resFree} are sharp for $-\Deltad{\pO}$.
\begin{theorem}
\label{thm:existenceCircle}
For all $N>0$, there exists $h_0>0$ such that for $h<h_0$, there exist $z(h)\in \Lambda$ with 
$$-\Im z(h)= \frac{1-\alpha}{2}h\log h^{-1}-\frac{h}{2}\log \frac{V_0}{2}+\O{}(h^{7/4})
$$
\end{theorem}

\subsection{Statement of results for the $\delta'$ potential}
As for the case of the $\delta$ potential, we make a preliminary definition of resonances in order to present simple arguments in the case of the disk. In particular, we say that $z/h$ is a  resonance of $-\Deltap$ if there exists a nonzero $z/h$-outgoing solution $(u_1,u_2)\in H^2(\Omega)\oplus H^2_{\loc}(\re^d\setminus \Omega)$ to 
\begin{equation}
\label{eqn:mainCirclePrime}
\begin{cases}(-h^2\Delta -z^2)u_1=0&\text{ in }\,\Omega\\
(-h^2\Delta -z^2)u_2=0 & \text{ in }\,\re^d\setminus \overline{\Omega}\\
\partial_{\nu_1}u_1=-\partial_{\nu_2}u_2&\text{ on }\,\partial \Omega\\
u_1-u_2+V \partial_{\nu_1}u_1=0&\text{ on }\,\partial \Omega\\
\end{cases}
\end{equation}
As for the $\delta$ potential, we show in Chapter \ref{ch:mer} that having such a solution at $z_0$ corresponds to $R_{-\Deltap}$ having a pole at $z_0$. 

Denote the set of rescaled resonances by 
\begin{equation*}
\begin{aligned}\resp&:=\{z\in B(h)\,:\, z/h\text{ is a resonance of }-\Deltap\}\\
B(h)&:=[1-ch^{3/4},1+ch^{3/4}]+i[-Mh\log h^{-1},0]
\end{aligned}
\end{equation*}

We assume throughout the analysis of the $\delta'$ potential that $V\equiv h^{\alpha} V_0$ for $0\leq \alpha $, and $V_0> 0$ a constant independent of $h$.
The next theorem shows that the resonance free regions in Chapter \ref{ch:resFree} are sharp for $-\Deltap$.
\begin{theorem}
\label{thm:existenceCirclePrime}
Let $\Omega$ and $V\equiv h^\alpha $. Then there exists $h_0>0$ such that for $h<h_0$, there exist $z(h)\in \resp$ with 
$$-\Im z(h)=
\left(V_0^{-2}+\o{}(1)\right)h^{3-2\alpha}\quad 1/2<\alpha<1.
$$
\end{theorem}

\subsection{Asymptotics of Bessel Functions}
We collect here some properties of the Bessel functions that are used in the analysis of case case of the unit disk. These formulae can be found in, for example \cite[Chapter 9,10]{NIST}.

Recall that the Bessel of order $n$ functions are solutions to 
$$z^2y''+zy'+(z^2-n^2)y=0.$$
We consider the two independent solutions $H_n^{(1)}(z)$ and $J_n(z)$. The Wronskian of the two solutions is given by 
\begin{equation} 
\label{eqn:wrongskBessel}W(J_n,H_n^{(1)})=J_n{H_n^{(1)}}'(z)-J_n'H_n^{(1)}(z)=\frac{2i}{\pi z}\end{equation}  

We now record some asymptotic properties of Bessel functions. Consider $n$ fixed and $z\to \infty$ 
\begin{align}
J_n(z)&=\left(\frac{1}{2\pi z}\right)^{1/2}\left( e^{i(z-\frac{n}{2}\pi -\frac{1}{4}\pi)}+e^{-i(z-\frac{n}{2}\pi -\frac{1}{4}\pi )}+\O{}(|z|^{-1}e^{|\Im z|})\right)\nonumber\\
H_n^{(1)}(z)&=\left(\frac{2}{\pi z}\right)^{1/2}\left(e^{i(z-\frac{n}{2}\pi -\frac{1}{4}\pi )}+\O{}(|z|^{-1}e^{|\Im z|})\right)\nonumber\\
J_n'(z)&=i\left(\frac{1}{2\pi z}\right)^{1/2}\left( e^{i(z-\frac{n}{2}\pi -\frac{1}{4}\pi)}-e^{-i(z-\frac{n}{2}\pi -\frac{1}{4}\pi )}+\O{}(|z|^{-1}e^{|\Im z|})\right)\nonumber\\
{H_n^{(1)}}'(z)&=i\left(\frac{2}{\pi z}\right)^{1/2}\left(e^{i(z-\frac{n}{2}\pi -\frac{1}{4}\pi )}+\O{}(|z|^{-1}e^{|\Im z|})\right)\nonumber\\
J_n(z)H_n(z)&=\frac{1}{\pi z}\left(e^{i(2z-n\pi -\frac{1}{2}\pi)}+1+\O{}(|z|^{-1}e^{2|\Im z|})\right)\label{eqn:normalAsymptotics}\\
J_n'(z)H_n'(z)&=-\frac{1}{\pi z}\left(e^{i(2z-n\pi -\frac{1}{2}\pi)}-1+\O{}(|z|^{-1}e^{2|\Im z|})\right)\label{eqn:normalAsymptoticsPrime}
\end{align}

\section{The delta potential}
\label{sec:modelDelta}
This section is organized as follows. In section \ref{sec:preliminaries} we reduce the problem of the existence of resonances to finding solutions of a transcendental equation. Then in Section \ref{sec:existence}, we show the existence of the resonances in Theorem \ref{thm:existenceCircle}.
\subsection{Reduction to Transcendental Equations on the Circle}
\label{sec:preliminaries}
We now consider \eqref{eqn:mainCircle} with $\Omega=B(0,1)\subset \re^2$ and $V\equiv h^{-\alpha}  V_0$ on $\partial \Omega$. Then for $i=1,2$,
\begin{equation}
\label{eqn:circleMain}
\begin{cases} \left(-h^2\partial_r^2-\frac{h^2}{r}\partial_r-\frac{h^2}{r^2}\partial_{\theta}^2-z^2\right)u_{1}=0&\text{in } B(0,1)\\
\left(-h^2\partial_r^2-\frac{h^2}{r}\partial_r-\frac{h^2}{r^2}\partial_{\theta}^2-z^2\right)u_{2}=0&\text{in } \re^2\setminus B(0,1)\\
u_1(1,\theta)=u_2(1,\theta)\\
\partial_ru_1(1,\theta)-\partial_ru_2(1,\theta)+Vu_1(1,\theta)=0\\
u_2\text{ is }z \text{ outgoing}
\end{cases}.
\end{equation}
Expanding in Fourier series, write $u_{i}(r,\theta):=\sum_{n}u_{i,n}(r)e^{in\theta}.$
Then, $u_{i,n}$ solves
$$\left(-h^2\partial_r^2-h^2\recip{r}\partial_r+h^2\frac{n^2}{r^2}-z^2\right)u_{i,n}(r)=0.$$
Multiplying by $r^2$ and rescaling by $x=zh^{-1}r$, we see that $u_{i,n}(r)$ solves the Bessel equation with parameter $n$ in the $x$ variables. Then, using that $u_2$ is outgoing and $u_1\in L^2$, we obtain that $$u_{1,n}(r)=K_nJ_n(zh^{-1} r)\quad\text{ and }\quad  u_{2,n}(r)=C_nH_n^{(1)}(zh^{-1} r)$$
where $J_n$ is the $n^{\text{th}}$ Bessel function of the first kind, and $H_n^{(1)}$ is the $n^{\text{th}}$ Hankel function of the first kind.

To solve \eqref{eqn:circleMain} and hence find a resonance, we only need to find $z$ such that the boundary conditions hold. 
Using the boundary condition $u_1(1,\theta)=u_2(1,\theta)$, we have $K_nJ_n(zh^{-1})=C_nH_n^{(1)}(zh^{-1})$. Hence, 
$$C_n=\frac{K_nJ_n(zh^{-1})}{H_n^{(1)}(zh^{-1})}.$$ 
Next, we rewrite the second boundary condition in \eqref{eqn:circleMain} and use that $V\equiv h^{-\alpha}V_0$ to get
$$\sum_n(K_nzh^{-1} J_n'(zh^{-1})-C_nzh^{-1} H_n^{(1)'} (zh^{-1})+h^{-\alpha}V_0K_n J_n(zh^{-1}))e^{in\theta}=0.$$
Then, since $e^{in\theta}$ are $L^2$ orthogonal, we have for $n\in \mathbb{Z}$
$$K_n\left(zh^{-1} J_n'(zh^{-1})-zh^{-1}\frac{J_n(zh^{-1})}{H_n^{(1)}(zh^{-1})}H_n^{(1)'}(zh^{-1})+h^{-\alpha}V_0J_n(zh^{-1})\right)=0.$$
Thus
\begin{equation*}
\label{eqn:preWronskian}
K_nh^{-\alpha}V_0= K_nzh^{-1}\left(\frac{H_n^{(1)'}(zh^{-1})}{H_n^{(1)}(zh^{-1})}-\frac{J_n'(zh^{-1})}{J_n(zh^{-1})}\right).
\end{equation*}
which can be written
\begin{equation}
\label{eqn:diskResonances}
h^{-\alpha}V_0K_n=K_nzh^{-1}\frac{W(J_n,H_n^{(1)})}{J_n(zh^{-1})H_n^{(1)}(zh^{-1})}=\frac{2 iK_n}{\pi J_n(zh^{-1})H_n^{(1)}(zh^{-1})}
\end{equation}
where $W(f,g)$ is the Wronskian of $f$ and $g$.

Then, without loss, we assume $K_n=1$ or $K_n=0$. 
Hence, we seek solutions $z(h,n)$ to 
\begin{equation} 
\label{eqn:toSolveExact}1-\frac{\pi h^{-\alpha}V_0}{2i}J_n(h^{-1}z(h,n))H_n^{(1)}(h^{-1}z(h,n))=0.
\end{equation}

The quantity $nh^{-1}$ is the tangential frequency of the mode $u_{i,n}e^{in\theta}.$ In particular, the \emph{wave front set,} denoted $\WFh$ (see Chapter \ref{sec:semiclassicalPreliminaries} or \cite[Chapter 4]{EZB}), of $e^{in\theta}$ has 
$$\WFh(e^{in\theta})\subset \{\xi'=nh\mod \o{}(1)\}.$$
 Thus, $|n|<(1-\e)h^{-1}$ corresponds to modes concentrating near directions transverse to the boundary,  $|n|\sim h^{-1}$ are the glancing frequencies, that is directions tangent to the boundary, and $|n|>(1+\e)h^{-1}$ corresponds to elliptic frequencies.

\subsection{Construction of Resonances}
\label{sec:existence}
In this section, we demonstrate the existence of resonances. That is, we prove Theorem \ref{thm:existenceCircle}.  We first prove the following analog of Newton's method:
\begin{lemma}
\label{lem:Newton}
Suppose that $z_0\in\mathbb{C}$. Let 
$\Omega:=\{z\in\complex: |z-z_0|\leq \e(h)\}$
and suppose $f:\Omega\to \mathbb{C}$ is analytic. Suppose that 
$$|f(z_0)|\leq a(h)\,,\quad |\partial_zf(z_0)|\geq b(h)\,,\quad \sup_{z\in \Omega}|\partial_z^2f(z)|\leq d(h).$$
Then if 
\begin{equation}
\label{eqn:condition}
a(h)+d(h)\e(h)^2<\e(h)b(h)<c<1
\end{equation}
there is a unique solution $z(h)$ to $f(z(h))=0$ in $\Omega$.
\end{lemma}
\begin{proof}
Let 
$$g(z):=z-\frac{f(z)}{\partial_zf(z_0)}.$$
Then, 
$$|\partial_zg(z)|=\left|1-\frac{\partial_zf(z)}{\partial_zf(z_0)}\right|\leq \frac{d(h)\e(h)}{b(h)}$$
and 
$$|g(z)-z_0|\leq |g(z_0)-z_0|+\sup_{\Omega}|\partial_zg(z)||z-z_0|\leq \left|\frac{a(h)}{b(h)}\right|+\left|\frac{d(h)\e(h)^2}{b(h)}\right|.$$
Thus under the condition \eqref{eqn:condition}, $g:\Omega\to \Omega$ and 
$$|g(z)-g(z')|<\sup_{w\in\Omega}|\partial_zg(w)||z-z'|<c|z-z'|.$$
Hence, $g$ is a contraction mapping and by the contraction mapping theorem, there is a unique fixed point of $g$ in $\Omega$ and hence a zero of $f(z)$ in $\Omega$. 
\end{proof}

We consider $n$ fixed relative to $h$. That is, we consider modes that concentrate normal to $\partial B(0,1)$. 

Using asymptotics \eqref{eqn:normalAsymptotics} in \eqref{eqn:toSolveExact}, we have
\begin{equation}
\label{eqn:asympToSolve}
1-\frac{h^{1-\alpha}V_0}{2i z(h,n)}\left(e^{2iz(h,n)/h-(n+\tfrac{1}{2})\pi i}(1+\O{}(hz(h,n)^{-1}))+1\right)=0.
\end{equation}
Let 
$$F(\e,k,n,h)=1-\frac{2h^{1-\alpha}V_0}{i\pi h(4k+2n+1)}\left(e^{2i\e/h}+1\right).$$

Then, 
$$\e_0(k,n,h)=\frac{-ih}{2}\log\left[h^{\alpha-1}\frac{i\pi h(4k+2n+1)}{2V_0}-1\right]$$
has 
$$F(\e_0(k,n,h),k,n,h)=0,\quad |\partial_\e F(\e_0(k,n,h),k,n,h)|\geq ch^{-1}.$$

Now, for $0<c$ and $ch^{-1}<k<Ch^{-1}$ by \eqref{eqn:asympToSolve}, $z(h,k,n)$ can be defined by a solution
$z(h,k,n)=\frac{\pi h}{4}(4k+2n+1)+\e(k,n,h)$ where 
$$F(\e,k,n,h)=\O{}(e^{2i\e/h}h/z+\e).$$
So, by the implicit function theorem there is a solution $\e$ satisfying
\begin{align*} 
\e(k,n,h)&=\e_0(k,n,h)+(\partial_\e F(\e_0(k,n,h),k,n,h))^{-1}\O{}(h^{1-\alpha}e^{2i\e_0/h}(h/z+\e_0))\\
&=\e_0(k,n,h)+\O{}(h^{2}).
\end{align*}

Thus, for all $\e>0$ and $0<h<h_\e$, there exist $z(h)\in \Lambda$ with
\begin{equation}\label{eqn:normal}
\frac{\Im z}{h}\sim\begin{cases} -\frac{(1-\alpha)}{2}\log h^{-1}+\frac{1}{2}\log\left(\frac{2}{V_0}\right)+\O{}(h^{3/4})&\alpha<1\\
-\frac{1}{4}\log \left(1+\frac{4}{V_0^2}\right)+\O{}(h^{3/4})&\alpha=1\end{cases}
\end{equation}
\begin{remark} Note that the size of the error terms in \eqref{eqn:normal} comes from the fact that we allow $\Re z \in [1-Ch^{3/4},1+Ch^{3/4}]$. 
\end{remark}

\section{The delta prime potential}

\label{sec:modelDeltaPrime}

\subsection{Reduction to Transcendental Equations on the Circle $\delta'$}
We now consider \eqref{eqn:mainCirclePrime} with $\Omega=B(0,1)\subset \re^2$ and $V\equiv h^{\alpha}  V_0$ with $V_0>0$ and $0\leq \alpha.$ Then for $i=1,2$,
\begin{equation}
\label{eqn:circleMainPrime}
\begin{cases} \left(-h^2\partial_r^2-\frac{h^2}{r}\partial_r-\frac{h^2}{r^2}\partial_{\theta}^2-z^2\right)u_{1}=0&\text{in } B(0,1)\\
\left(-h^2\partial_r^2-\frac{h^2}{r}\partial_r-\frac{h^2}{r^2}\partial_{\theta}^2-z^2\right)u_{2}=0&\text{in } \re^2\setminus B(0,1)\\
\partial_r u_1(1,\theta)=\partial_r u_2(1,\theta)\\
u_1(1,\theta)-u_2(1,\theta)+V\partial_r u_1(1,\theta)=0\\
u_2\text{ is }z \text{ outgoing}
\end{cases}.
\end{equation}
Expanding in Fourier series, write $u_{i}(r,\theta):=\sum_{n}u_{i,n}(r)e^{in\theta}.$
Then, $u_{i,n}$ solves
$$\left(-h^2\partial_r^2-h^2\recip{r}\partial_r+h^2\frac{n^2}{r^2}-z^2\right)u_{i,n}(r)=0.$$
Multiplying by $r^2$ and rescaling by $x=zh^{-1}r$, we see that $u_{i,n}(r)$ solves the Bessel equation with parameter $n$ in the $x$ variables. Then, using that $u_2$ is outgoing and $u_1\in L^2$, we obtain that $u_{1,n}(r)=K_nJ_n(zh^{-1} r)$ and $u_{2,n}(r)=C_nH_n^{(1)}(zh^{-1} r)$ where $J_n$ is the $n^{\text{th}}$ Bessel function of the first kind, and $H_n^{(1)}$ is the $n^{\text{th}}$ Hankel function of the first kind.

To solve \eqref{eqn:circleMainPrime} and hence find a resonance, we only need to find $z$ such that the boundary conditions hold. 
Using the boundary condition $\partial_r u_1(1,\theta)=\partial_r u_2(1,\theta)$, we have $zh^{-1}K_nJ_n'(zh^{-1})=zh^{-1}C_nH_n'^{(1)}(zh^{-1})$. Hence, 
$$C_n=\frac{K_nJ_n'(zh^{-1})}{H_n'^{(1)}(zh^{-1})}.$$ 
Next, we rewrite the second boundary condition in \eqref{eqn:circleMain} and use that $V\equiv h^{\alpha}V_0$ to get
$$\sum_n(K_nJ_n(zh^{-1})-C_n H_n'^{(1)} (zh^{-1})+h^{\alpha}V_0K_nzh^{-1} J_n'(zh^{-1}))e^{in\theta}=0.$$
Then, since $e^{in\theta}$ are $L^2$ orthogonal, we have for $n\in \mathbb{Z}$
$$K_n\left(J_n(zh^{-1})-\frac{J_n'(zh^{-1})}{H_n'^{(1)}(zh^{-1})}H_n^{(1)}(zh^{-1})+h^{\alpha }V_0zh^{-1}J_n'(zh^{-1})\right)=0.$$
Thus
\begin{equation*}
\label{eqn:preWronskianPrime}
-K_nzh^{-1+\alpha}V_0= K_n\left(\frac{J_n(zh^{-1})}{J_n'(zh^{-1})}-\frac{H_n^{(1)}(zh^{-1})}{H_n'^{(1)}(zh^{-1})}\right).
\end{equation*}
which can be written
\begin{equation}
\label{eqn:diskResonancesPrime}
-h^{-1+\alpha}zV_0K_n=K_n\frac{W(J_n,H_n^{(1)})(zh^{-1})}{J_n'(zh^{-1})H_n'^{(1)}(zh^{-1})}=\frac{2 iK_n}{\pi zh^{-1}J_n'(zh^{-1})H_n'^{(1)}(zh^{-1})}
\end{equation}
where $W(f,g)$ is the Wronskian of $f$ and $g$.

Then, without loss, we assume $K_n=1$ or $K_n=0$. 
Hence, we seek solutions $z(h,n)$ to 
\begin{equation} 
\label{eqn:toSolveExactPrime}1+\frac{\pi z^2(h,n)h^{-2+\alpha}V_0}{2i}J_n'(h^{-1}z(h,n))H_n'^{(1)}(h^{-1}z(h,n))=0.
\end{equation}

\subsection{Resonances normal to the boundary}
As for the $\delta$ potential, we consider $n$ fixed relative to $h$. That is, we consider modes that concentrate normal to $\partial B(0,1)$. 

Using asymptotics \eqref{eqn:normalAsymptoticsPrime} in \eqref{eqn:toSolveExactPrime}, we have
\begin{equation}
\label{eqn:asympToSolvePrime}
1+\frac{ izh^{-1+\alpha}V_0}{2}\left(e^{i(2zh^{-1}-n\pi -\pi/2)}-1+\O{}(|z|^{-1}he^{2|\Im z|h^{-1}})\right)=0.
\end{equation}
Let $z=\frac{h}{4}(\pi(2n+4k+1)+4\e h^{-1})$ with $\pi k=h^{-1}+\O{}(1)$. Substituting  this in to \eqref{eqn:asympToSolvePrime} and ignoring the error term, as well as higher order terms in $\e$, we obtain  
$$F(\e,k,n,h)=1+i\frac{\pi(2n+4k+1)h^\alpha V_0}{8}\left(e^{2i\e/h}-1\right).$$

Then, 
\begin{align*}
\e_0(k,n,h)&=-\frac{ih}{2}\log\left[1+i\frac{8h^{-\alpha}}{\pi(2n+4k+1)V_0}\right]\\
=\frac{-ih}{2}\log\left[1+i2h^{1-\alpha}V_0^{-1}(1+\O{}(h))\right]
\end{align*} 
has 
$$F(\e_0(k,n,h),k,n,h)=0,\quad |\partial_\e F(\e_0(k,n,h),k,n,h)|\geq ch^{-1}.$$

Now, for $0<c$ and $ch^{-1}<k<Ch^{-1}$ by \eqref{eqn:asympToSolvePrime}, $z(h,k,n)$ can be defined by a solution
$z(h,k,n)=\frac{\pi h}{4}(4k+2n+1)+\e(k,n,h)$ where 
$$F(\e,k,n,h)=\O{}(e^{2i\e/h}h^\alpha|z|^{-1}+\e h^{-1+\alpha}).$$
So, by the implicit function theorem there is a solution $\e_1$ satisfying
\begin{align*} 
\e(k,n,h)&=\e_0(k,n,h)\\
&\quad\quad+(\partial_\e F(\e_0(k,n,h),k,n,h))^{-1}\O{}(h^{-1+\alpha}e^{2i\e_0/h}(h|z|^{-1}+\e_0))\\
&=\e_0(k,n,h)+\O{}(h^{2}+\min(h^2\log h^{-1},h^{3-\alpha})=\e_0+\o{}(\Im \e_0)
\end{align*}
where the last equality follows from the fact that $\alpha>1/2$. 

Thus, for all $\alpha>1/2$ $\e>0$ and $0<h<h_\e$, there exist $z(h)\in \Lambda$ with
\begin{equation}\label{eqn:normalPrime}
\Im z=\begin{cases} -(1+\o{}(1))V_0^{-2}h^{3-2\alpha}&1/2<\alpha<1\\
-(1+\o{}(1))\frac{h}{4}\log \left(1+4h^{2-2\alpha}V_0^{-2}\right)+\O{}(h^{3/4})&\alpha\geq 1\end{cases}
\end{equation}


\chapter{Semiclassical Intersecting Lagrangian Distributions}
\label{ch:iLagrange}
We follow \cite{UhlMel} to construct intersecting Lagrangian distributions in the semiclassical regime.
\subsection{Definitions}
A pair $(\Lambda_0,\Lambda_1)$ where $\Lambda_0\subset T^*X$ is a Lagrangian manifold and $\Lambda_1\subset T^*X$ is a Lagrangian manifold with boundary, is said to be an \emph{intersecting pair} of Lagrangian manifolds if $\Lambda_0\cap \Lambda_1=\partial \Lambda_1$ and the intersection is clean:
\m T_\lambda(\Lambda_0)\cap T_\lambda(\Lambda_1)=T_\lambda(\partial \Lambda_1)\text{ for all }\lambda\in \partial\Lambda_1.\,\,\m
Two such pairs $(\Lambda_0,\Lambda_1)$ and $(\Lambda_0',\Lambda_1')$, with given base points $\lambda\in \partial\Lambda_1$ and $\lambda'\in \partial\Lambda_1'$ are said to be locally equivalent if there is a neighborhood $V$ of $\lambda$ and a symplectic transformation $\chi:V\to T^*X$ such that $\chi(\lambda)=\lambda'$, $\chi(\Lambda_0\cap V)\subset \Lambda_0'$ and $\chi(\Lambda_1\cap V)\subset \Lambda_1'$. Then, we have the following lemma \cite[Theorem 21.2.10 and remark thereafter]{HOV3}.
\begin{lemma}
\label{lem:equivalence}
If $\Lambda_1,$ $\Lambda_2\subset M$, and $\bar{\Lambda}_1$, $\bar{\Lambda}_2\subset \bar{M}$ are two pairs of intersecting Lagrangians with $\dim M=\dim \bar{M}$ and $\dim \Lambda_1\cap \Lambda_2=\dim \bar{\Lambda}_1\cap \bar{\Lambda}_2$ then $(\Lambda_1,\Lambda_2)$ is locally equivalent to $(\bar{\Lambda}_1,\bar{\Lambda}_2).$
\end{lemma}

We associate spaces of distributions to the pair $(\tilde{\Lambda}_0,\tilde{\Lambda}_1)$ of intersecting Lagrangian manifolds, where
$ \tilde{\Lambda}_0=T_0^*\re^d$ $\tilde{\Lambda}_1=\{((x_1,x'),\xi)\in T^*\re^d:x'=0, \xi_1=0, x_1\geq 0\}.$ 

\begin{remark} One can also associate distributions to intersecting Lagrangians with intersections of various dimensions as in \cite{GuiUhl}, but we do not pursue that here.
\end{remark}

\begin{defin}
\label{def:prototype}
For $\delta\in[0,1/2)$, denote by $I^{m}_{\delta}(\re^d;\tilde{\Lambda}_0,\tilde{\Lambda}_1)$ the subspace of $\Cc(\re^d)$ consisting of functions $u$ which can be written in the form $u=u_1+u_2$ with $u_2\in h^{1/2}I_{\delta}^{m-1/2}(\tilde{\Lambda}_0)$ and
\begin{equation}\label{eqn:normalForm}u_1(x)=(2\pi h)^{-(3d+2)/4}\int_0^\infty\int_{\re^{d}}e^{\frac{i}{h}((x_1-s)\xi_1+\la x',\xi'\ra)}a(s,x,\xi)d\xi ds=:J(a),\end{equation}
where $a\in S_\delta^{m+\recip{2}-\frac{d}{4}}$ has compact support in $x$. 
\end{defin}
\begin{remarks} 
\item \eqref{eqn:normalForm} is well defined as an oscillatory integral and as such depends continuously on $a$ in the topology of $S_\delta^{m'}$, for any $m'>m+\recip{2}-\recip{4}d$.
\item We show in Lemma \ref{lem:isecDecomposition} that functions of the form \eqref{eqn:normalForm} are microlocalized on $\tilde{\Lambda}_0\cup \tilde{\Lambda}_1.$
\end{remarks}

\begin{lemma}
\label{lem:isecDecomposition}
If $u\in I^{m}_{\delta}(\re^d;\tilde{\Lambda}_0,\tilde{\Lambda}_1)$, then 
\begin{equation}\label{eqn:waveNormForm}\WFh(u)\subset \tilde{\Lambda}_0\cup \tilde{\Lambda}_1.\end{equation}
Suppose $\gamma\leq \delta$ and $B\in S_\gamma$ is a zeroth order pseudo-differential operator with $\MS(B)\cap \tilde{\Lambda}_0=\emptyset$  then $Bu\in I_{\delta}^{m}(\re^d;\tilde{\Lambda}_1)$. If $\MS(B)\cap \tilde{\Lambda}_1=\emptyset$ then  $Bu\in  h^{1/2-\gamma}I_{\delta}^{m-1/2}(\re^d; \tilde{\Lambda}_0)\,)$.
\end{lemma}
\begin{proof}
Let $\pi:\re\times \re^d\to \re^d$ be the projection off the first factor, then $u=\pi_*(H(s)\tilde{u})$ where $H$ is the Heaviside function and 
$$\tilde{u}(s,x)=(2\pi h)^{-(3d+2)/4}\int e^{\frac{i}{h}((x_1-s)\xi_1+\la x',\xi'\ra )}a(s,x,\xi)d\xi.$$
We now use the standard bounds on wavefront sets for pullbacks, tensors, and pushforwards (see Lemmas \ref{lem:tensor}, \ref{lem:pushforward}, and \ref{lem:pullback}) to obtain \eqref{eqn:waveNormForm}. 

Now, suppose $B\in \psi_{\gamma}^{0} $. Then, 
\m B(e^{\frac{i}{h}\la x,\xi\ra}a(x,\xi))=e^{\frac{i}{h}\la x,\xi\ra}(\mc{B}a)\,\,\m
defines a continuous linear map $\mc{B}:S_\delta^{m}\to S_{\delta}^{m}$. In particular, $B$ can be applied under the integral sign in \eqref{eqn:normalForm}. This shows that $Bu_1$ is of the same form with $a$ replaced by $\mc{B}a$. 

Observe that since $u_2\in h^{1/2}I_\delta^{m-1/2}(\tilde{\Lambda}_0)$, $Bu_2=\O{\Cc}(h^\infty)$. Then, if $\MS(B)\cap \tilde{\Lambda}_0=\emptyset$, we can assume, by disregarding an $\O{\Cc}(h^\infty)$ term, that for some $\e>0$, $\mc{B}(a)=0$ in $|x|<\e h^\gamma$.  Choose $\mu\in C^\infty(\re)$ with $\mu(s)=1$ in $s\geq \recip{2}\e$, $\mu(s)=0$ in $s<\recip{4}\e$. From the definition of semiclassical Lagrangian distributions (see Section \ref{sec:preliminaries})
$$v_1=(2\pi h)^{-(3d+2)/4}\iint e^{\frac{i}{h}((x_1-s)\xi_1+\la x',\xi '\ra)}\mu(h^{-\delta}s)(\mc{B}a)(s,x,\xi)d\xi ds$$
is an element of $I_{\delta}^{m}(\re^d;\tilde{\Lambda}_1)$. To show that $Bu$ is also in this space, we need to verify that 
\begin{equation}\label{eqn:residual}\begin{aligned}Bu-v_1&=(2\pi h)^{-(3d+2)/4}\int_0^\infty\!\!\!\int e^{\frac{i}{h}((x_1-s)\xi_1+\la x',\xi'\ra}(1-\mu(h^{-\delta}s))\mc{B}(a)d\xi ds\\
& =\mc O_{C^\infty}(h^\infty).\end{aligned}\end{equation}
The operator 
\m L=((x_1-s)^2+|x'|^2)^{-1}[(x_1-s)hD_{\xi_1}+x'hD_{\xi'}]\,\,\m 
satisfies $L\exp(\frac{i}{h}((x_1-s)\xi_1+\la x',\xi'\ra))=\exp(\frac{i}{h}((x_1-s)\xi_1+\la x',\xi'\ra))$ and has coefficients in $h^{-\gamma}S_\gamma$ on the support of $(1-\mu)\mc{B}a$. Then, \eqref{eqn:residual} follows from integration by parts. Thus, $Bu\in I_{\delta}^{\comp}(\re^d;\tilde{\Lambda}_1)$. 

Now, suppose that $\MS(B)\cap \tilde{\Lambda}_1=\emptyset$. Then we can assume, with $a$ replaced by $\mc{B}a$ that $a=0$ if $|x'|^2+\xi_1^2<\e^2h^{2\gamma}$ and $x_1>-\e h^\gamma$. Thus, the operator 
\m M=(|x'|^2+\xi_1^2)^{-1}(x'hD_{\xi'}-\xi_1hD_s)\,\,\m
has coefficients in $h^{-\gamma}S_{\gamma}$ on supp $a$ provided $x_1>-\e h^\gamma$. Since $\exp(\frac{i}{h}((x_1-s)\xi_1+\la x',\xi'\ra))$ is an eigenfunction of $M$ with eigenvalue 1, integration by parts gives 

\begin{multline}
\label{eqn:WFVanishBoundary}
Bu=(2\pi h)^{-(3d+2)/4}\int_0^\infty\int e^{\frac{i}{h}((x_1-s)\xi_1+\la x',\xi '\ra)}M^t(\mc{B}a)d\xi ds\\ 
+(2\pi h)^{1-(3d+2)/4}\int e^{\frac{i}{h}\la x,\xi\ra} \frac{-i\xi_1\mc{B}a(0,x,\xi)}{(|\xi_1|^2+|x'|^2)}d\xi. 
\end{multline}
The second term in \eqref{eqn:WFVanishBoundary} is a distribution in $h^{1/2-\gamma}I_{\delta}^{m-1/2}(\re^d;\tilde{\Lambda}_0)$. Then, iterating this process, we have for any $k\in \mathbb{N}$, 
$$Bu-(2\pi h)^{-(3d+2)/4}\int_0^\infty\int e^{\frac{i}{h}((x_1-s)\xi_1+\la x',\xi'\ra)}(M^t)^k\mc{B}ad\xi ds$$
lies in $h^{\tfrac{1}{2}-\gamma}I_{\delta}^{m-1/2}(\re^d;\tilde{\Lambda}_0)$. Since $(M^t)^k\mc{B}a\in h^{k(1-\delta)} S_{\delta}^{m-k}$, we conclude that 
$$Bu\in h^{\tfrac{1}{2}-\gamma}I_{\delta}^{m-1/2}(\re^d;\tilde{\Lambda}_0). $$
\end{proof}

Next, we show that $a$ need not be allowed to depend on $s$.
\begin{lemma}
Suppose $u=J(a)$ for $a\in S^m_\delta$. Then there exists $b_j=b(x,\eta)\in S^{m-j}_\delta$ such that
$$u-\sum_{j=0}^{N-1}J(b_j)\in h^{N(1-2\delta)}I^{m-N}_{\delta}(\re^d;\tilde{\Lambda}_0,\tilde{\Lambda_1}).$$
\end{lemma}
\begin{proof}
By Taylor's theorem at $y_1=s$, there exists $b_0$ such that 
$$\left|a(y_1,y',s,\eta)-b_0(y,\eta)\right|=\O{}(h^{-\delta }(y_1-s)).$$
Then, integrating by parts with respect to $\eta_1$ in the formula for $J(a-b_0)$ gives that 
$$J(a-b_0)= h^{1-2\delta}J(c)$$
with $c\in S_\delta^{m-1+1/2-d/4}.$
So, repeating this process gives the Lemma. 
\end{proof}

Finally, we show that an element of $I^{m}_\delta(\re^d;\tilde{\Lambda}_0,\tilde{\Lambda}_1)$ can be written as a Lagrangian distribution with singular symbol.
\begin{lemma}
\label{lem:singSymbol}
Suppose that $u=J(a)$ where $a=a(y,\eta)\in S^m.$ Then,
$$u=(2\pi h)^{-(3d-2)/4}\int_{\re^d}e^{\frac{i}{h}\la x,\xi\ra} \frac{-ia(y,\eta)}{\eta -i0}d\eta+\O{\Cc}(h^\infty). $$
\end{lemma}
\begin{proof}
Observe that by the Paley-Wiener theorem, 
$$f(\eta_1)=\int_0^\infty e^{-\frac{i}{h}s\eta_1}ds$$
is holomorphic in $\Im \eta_1<0$. So, we can take limits from $\eta_1$ in the lower half plane to obtain
$$f(\eta_1)=\frac{h}{i(\eta_1-i0)}.$$
This gives the result.
\end{proof}

\subsection{General Lagrangians}
Suppose that $(\Lambda_0,\Lambda_1)$ is an intersecting pair of Lagrangian manifolds in a $C^\infty$ manifold $X$ with dim $X=d$ and $\dim \Lambda_0\cap \Lambda_1=d-1$ and $\Lambda_0\cap \Lambda_1\Subset T^*X$. Given $\lambda\in \Lambda_0\cap \Lambda_1$, by Lemma \ref{lem:equivalence} we can find a local parametrization of the the intersecting pair. Therefore, we define
\begin{defin}
\label{def:isecLagDef}
$I_\delta^{m}(X;\Lambda_0,\Lambda_1)$ consists of those $C^\infty$ $\recip{2}$ densities, $u$ on $X$ which are modelled microlocally on Definition \ref{def:prototype}. We say that $u\in I_{\delta}^{m}(X;\Lambda_0,\Lambda_1)$ if there exist distributions $u_0\in h^{1/2}I_{\delta}^{m}(\Lambda_0)$, $u_1\in I_\delta^{\comp}(\Lambda_1\setminus \partial\Lambda_1)$, a finite set of parametrizations $\chi_j:V_j\to T^*\re^d$ reducing $(\Lambda_0,\Lambda_1)$ locally to normal form, zeroth order Fourier integral operators $F_j$ associated to $\chi_j^{-1}$ and distributions $v_j\in I^{\comp}_{\delta}(\re^d;\tilde{\Lambda}_0,\tilde{\Lambda}_1)$ such that 
\m u-u_0-u_1-\sum_{j}F_jv_j=\mc O_{\mc{S}}(h^\infty).\m
\end{defin}

\begin{remark} Recall that for open $\Lambda$, all $u\in I^{\comp}_{\delta}(\Lambda)$ are compactly microlocalized inside $\Lambda$. Thus, $I^{\comp}_{\delta}(\Lambda_1\setminus\partial\Lambda_1)$ consists of distributions which are compactly microlocalized away from $\partial\Lambda_1$.
\end{remark}

To show that these distributions are well defined, we need to show that if $\chi$ is a canonical transformation on $\re^d$ which leaves both $\tilde{\Lambda}_0$ and $\tilde{\Lambda}_1$ invariant and $F$ is a properly supported zeroth order Fourier integral operator associated to $\chi$, then $Fu\in I_{\delta}^{\comp}(\re^d;\tilde{\Lambda}_0,\tilde{\Lambda}_1)$ provided $u$ is in this space. We will actually prove something stronger. Let $\tilde{\Lambda}_i^{d'}\subset T^*\re^{d'}$, $\tilde{\Lambda}_i^d\subset T^*\re^d$, $i=0,1$. 

\begin{lemma}
\label{lem:fioIsectLagrangian}
Suppose that $d,d'\geq 2$ and $\Gamma$ is a canonical relation such that $\Gamma\composed \tilde{\Lambda}_0^{d'}\subset \tilde{\Lambda}_0^{d}$, $\Gamma\composed \tilde{\Lambda}_1^{d'}\subset \tilde{\Lambda}_1^d$ and the compositions are transversal. Let $F\in I^{\comp}(\Gamma)$. Then 
\begin{equation}
\label{eqn:FIOmapProperty}
F:I^{\comp}_{\delta}(\re^{d'};\tilde{\Lambda}_0^{d'},\tilde{\Lambda}_1^{d'})\to I^{\comp}_{\delta}(\re^d;\tilde{\Lambda}_0^d,\tilde{\Lambda}_1^d).
\end{equation}
\end{lemma}
\begin{proof}
We can always decompose $F$ by using a microlocal partition of unity and so assume that $\Gamma\composed \tilde{\Lambda}_0^{d'}=\Lambda_0^d$ and $\Gamma\composed \tilde{\Lambda}_1^{d'}=\tilde{\Lambda}_1^d$ in the region of interest. Suppose that 
\begin{gather*} 
u=(2\pi h)^{-(3d'+2)/4}\int_0^\infty\int e^{\frac{i}{h}((y_1-s)\eta_1+\la y', \eta'\ra)}a(s,y,\eta)d\eta ds,\\
Fv=(2\pi h)^{-(d+d'+2L)/4}\int e^{\frac{i}{h}\phi(x,y,\theta)}b(x,y,\theta)v(y)dyd\theta,\end{gather*}
where $\phi$ non-degenerate phase function defining $\Gamma$. Then, 
\begin{equation} \label{eqn:FIOIsectComp}Fu=(2\pi h)^{-(d+4d'+2L+2)/4}\int_0^\infty\int \left[\int e^{\frac{i}{h}\psi(x,y,s,\theta,\eta)}badyd\eta\right]d\theta ds\end{equation}
with $\psi =\phi(x,y,\theta)+(y_1-s)\eta_1+\la y',\eta'\ra.$ Now, note that $d_\eta \psi=0$ if and only if $y_1=s,y_2=...=y_{d'}=0$, $d_y\psi=0$ if and only if $\eta=-d_y\phi$ and
$$\partial_{y\eta}^2\psi=\begin{pmatrix}\partial^2_y\phi &I\\I&0\end{pmatrix}$$
which has determinant 1.

Thus, by stationary phase,
$$Fu=(2\pi h)^{-(d+2L+2)/4}\int_0^\infty\int e^{\frac{i}{h}\phi(x,(s,0),\theta)}c(x,s,\theta)d\theta ds.$$
Notice that $\Gamma\composed \tilde{\Lambda}_i^{(d')}=\Lambda_i^{(d)}$ implies that 
\begin{align*}
d_\theta \phi=0\,\imply\,  y=0&\Leftrightarrow x=0\\
d_\theta \phi=0,\,\imply\, \phi_{y_1}'=0&\imply \phi_{x_1}'=0,\,x'=0,\,x_1 \geq 0 \\
d_\theta \phi=0,\, \phi_{x_1}'=0&\imply \phi_{y_1}'=0,\, y'=0.\, y_1\geq 0.
\end{align*}
Since we have assumed that the compositions $\Gamma\composed \tilde{\Lambda}_1^{d'}$ is transversal $\varphi(x,s,\theta)=\phi(x,(s,0),\theta)$ is non-degenerate and since $\Gamma \composed \tilde{\Lambda}_i^{d'}=\tilde{\Lambda}_1^d$, 
\begin{gather}
\label{eqn:phaseForm1}
d_\theta\varphi=0, s=0\Leftrightarrow x=0, d_\theta\varphi=0\\
\label{eqn:phaseForm2}
d_\theta\varphi=0,\, d_s\varphi=0 \Leftrightarrow x'=0,\, d_{x_1}\varphi=0,\,x_1\geq 0,\, d_\theta\phi=0
\end{gather}
Since away from $s=0$, $u\in I^{\comp}(\Lambda_1)$, we may work in a small neighborhood of $s=0$. Suppose that there exists $\{(s_i,x_i,\theta_i)\}_{i=1}^\infty$ such that $s_i\to 0$, $s_i,x_i\neq 0$, $d_\theta\varphi(x_i,x_i,\theta_i)=0$. Then, since $c$ has compact support, we may assume that $(x_i,\theta_i)\to (x,\theta).$ But, $\varphi\in C^\infty$. Therefore, $d_\theta\varphi(x,0,\theta)=0$ and hence $x=0$ by \eqref{eqn:phaseForm1} and we may also work in a neighborhood of $x=0$. 

Suppose that $\partial^2\varphi/\partial\theta\partial\theta(0,0,\theta)\neq 0$. Then there exist $i,\, j$ such that $\partial^2\varphi/\partial\theta_i\partial\theta_j(0,0,\theta)\neq 0$. Suppose $i=j$. Then $\partial^2\varphi/\partial\theta_i^2(0,0,\theta)\neq 0$ and we can use stationary phase to eliminate the $\theta_i$ variable. Therefore, we may assume that $\partial^2\varphi/\partial\theta_i^2(0,0,\theta)=0$ for all $i$ and $\theta$ in $d_{\theta}\varphi=0$. Suppose that $i\neq j$. Then, since $\partial^2\varphi/\partial^2\theta_i(0,0,\theta)= 0$ for all $i$, we may use stationary phase in the $\theta_i$ and $\theta_j$ variables. Now, observe that if $\partial^2\varphi/\partial\theta\partial\theta(0,0,\theta)\neq 0$ then the same is true in a neighborhood of $s=0$, $x=0$. 

Hence, reducing the size of the neighborhood of $(0,0)$ if necessary and using stationary phase we can reduce the number of $\theta$ variables, $L$, until $\partial^2\varphi/\partial\theta\partial\theta =0$ at $(0,0,\tilde{\theta})$. Then, by \eqref{eqn:phaseForm1} and the fact that $\Gamma\composed \tilde \Lambda_0^{d'}$ is transverse $L=d$  and $\text{det}(\partial^2\varphi/\partial x\partial \theta)\neq 0.$ Therefore, 
\begin{gather*} 
Fu=(2\pi h)^{-(3d+2)/4}\int_0^\infty\int e^{\frac{i}{h}\phi(x,(s,0),\theta)}c(x,s,\theta)d\theta ds,\\
\frac{\partial\varphi}{\partial \theta_j}=\sum_i C_{ji}(x_i-s\alpha_i(x,s,\theta)),\end{gather*}
where $C$ is invertible. Now, we want to show that there is a change of  variables $\theta =\Theta(x,s,\theta)$, $s=sT(x,s,\theta)$ where $T>0$ such that $Fu$ is of the form \eqref{eqn:normalForm}.

First, replace $\theta_i$ by $\sum_jC_{ji}\theta_j$ to reduce $\varphi$ to 
\m \varphi(x,s,\theta)=\theta \cdot x-s \alpha(x,s,\theta).\,\,\m
Now, write $\alpha=\alpha(0,0,\theta)+x\cdot \beta(x,s,\theta)+s\gamma(x,s,\theta)$ and let $\theta_i=\theta_i-s\beta_i$. Then, 
$$\varphi(x,s,\theta)=\theta\cdot x -s\alpha(\theta)+s^2\gamma(x,s,\theta).$$
Now, \eqref{eqn:phaseForm2} reads
\begin{equation}\label{eqn:phase0}
x_i-s\frac{\partial\alpha}{\partial\theta_i}+s^2\frac{\partial\gamma}{\partial\theta_i}=0,\quad -\alpha(\theta)+\frac{\partial}{\partial s}(s^2\gamma(x,s,\theta))=0\end{equation}
if and only if $x'=0$, $\partial_{x_1}\varphi=0$, and $x_1\geq 0$. But, using \eqref{eqn:phaseForm1}, we have that $s\neq 0$ implies $x_1\neq 0$ and hence $x_1>0$. Thus, $\partial \alpha/\partial \theta_1>0$ on the surface $S$, defined by \eqref{eqn:phase0}. Hence, $s$ and $\theta'=(\theta_2,...,\theta_d)$ can be taken as coordinates. Moreover, for $i\geq 2$, $x_i\equiv 0$ on $S$ so that differentiating with respect to $s$ in the first equation of \eqref{eqn:phase0} and setting $s=0$ gives $\partial\alpha/\partial\theta'=0$ on $\alpha=0$. But, $\partial\alpha/\partial\theta_1\neq 0$, so 
\m \alpha(\theta)=(\theta_1-\rho(\theta'))\beta(\theta).\,\,\m
Now,
$$0<\frac{\partial \alpha}{\partial \theta_1}=\beta(\theta)+\frac{\partial \beta}{\partial\theta_1}(\theta_1-\rho(\theta')).$$
But, on $\alpha=0$, $\theta_1-\rho(\theta')=0$ and hence $\beta>0$. Then, since $\beta=0$ implies $\alpha=0$, $\beta>0$. We also have that $\partial \rho /\partial \theta '=0$ since 
$$0=\frac{\partial\alpha}{\partial\theta'}=\frac{\partial \beta}{\partial\theta '}(\theta_1-\rho(\theta'))+\beta(\theta)\frac{\partial\rho}{\partial \theta '}.$$
and $\theta_1-\rho(\theta ')=0$ and on $\alpha=0$. Therefore, $\partial\rho/\partial\theta'\equiv 0$ and $\rho\equiv C_1$. Hence, by relabeling $\theta_1=\theta_1-C_1$, and $s=s\beta$, we have 
\m \varphi=\theta\cdot x +x_1C_1-s\theta_1+s^2\gamma(x,s,\theta).\,\,\m
and, using \eqref{eqn:phaseForm2}, and setting $s=0$, $\theta_1=0$, we have $\partial \varphi /\partial x_1=\theta_1+C_1=0$. Therefore, $C_1=0$ and we have 
\m \varphi=\theta\cdot x-s\theta_1+s^2\gamma(x,s,\theta).\,\,\m

Relabeling $\gamma=\alpha$ and repeating the argument gives for any fixed $k$ that 
\m \varphi(x,s,\theta)=x\cdot \theta -s\theta_1+s^k\gamma(x,s,\theta).\,\,\m

Now, we apply the method used by H\"{o}rmander \cite{Ho} to show that $\varphi$ is equivalent to $\phi(x,s,\theta)=x\cdot \theta-s \theta_1$ under a change of phase variables preserving $s=0$ and $s>0$.

The map 
$$\chi:(x,s,\theta)\mapsto \left(x,\frac{\partial \varphi}{\partial x},\frac{\partial \varphi}{\partial \theta}, \frac{\partial \varphi}{\partial s}\right)$$
is injective and, $\chi(x,s, \theta)=(x,\theta,(x_1-s,x'),-\theta_1)+\mc O(s^\infty)$. Hence $\chi$ has a left inverse $\Psi(x,\xi,\eta,\sigma)$ such that, on the surface $\eta_1=x_1$, $\Psi(x,\theta, (x_1,\eta'),\sigma)=(x,0,\theta)$ to high order. Let $\kappa^{-1}(x,s,\theta)=\Psi(x,\theta,(x_1-s,x'),-\theta_1)$ and put $\psi=\kappa^*\varphi$. Then $\kappa:\tilde{S}\to S$ and $\kappa$ is equal to the identity to high order at $s=0$.

Now, write 
\m \kappa(x,s,\theta)=(x,t(x,s,\theta),\eta(x,s,\theta)).\,\,\m
Then, 
\begin{gather*}\left.{}\frac{\partial\varphi}{\partial x}\right|_{(x,s,\theta)=(x,t,\eta)}=\theta,\quad\quad \left.{}\frac{\partial\varphi}{\partial \theta}\right|_{(x,s,\theta)=(x,t,\eta)}=(x_1-s,x'),\\
 \left.{}\frac{\partial\varphi}{\partial s}\right|_{(x,s,\theta)=(x,t,\eta)}=-\theta_1.\end{gather*}
Hence, by \eqref{eqn:phaseForm2} on $\tilde{S}$
\m x_1-s=0,\quad x'=0,\quad \theta_1=0.\,\,\m
Therefore, on $\theta_1=0$, $\partial_s\psi=0$ and hence $\partial_{x_1}\psi=0$. Thus, $(\psi-\phi)(x_1,x',x_1,0,\theta')=0$. That is, $\psi-\phi$ vanishes on $\tilde{S}$. 
Note also that we have on $\tilde{S}$ that
$$\partial \psi =\partial_{x'}\varphi \frac{\partial x'}{\partial x'}=\partial _{x'}\varphi =\theta'=\partial \phi.$$
and we have $\partial (\psi-\phi) =0.$
Hence $\psi-\phi$ vanishes to second order on $\tilde{S}$. 

Thus, 
\m \psi(x,s,\theta)-\phi(x,s,\theta)=Z\cdot A\cdot Z,\,\,\m
where $Z=(x_1-s,x',-\theta_1)=(\partial \phi/\partial \theta_1,\partial \phi/\partial \theta',\partial \phi/\partial s)$ and $A$ vanishes at $s=0$. We need to find a coordinate change $(\tilde{s},\tilde{\theta})=(s,\theta)+B(x,s,\theta)\cdot Z$ such that $\varphi(x,s,\theta)=\phi(x,\tilde{s},\tilde{\theta})$ and $B=0$ at $s=0$. Since 
$$\phi(x,\tilde{s},\tilde{\theta})=\phi(x,s,\theta)+Z\cdot B\cdot Z+Z\cdot B\cdot G\cdot B\cdot Z$$
where $G$ is a matrix depending smoothly on $x$, $\theta$, $s$ and $B$, it suffices to choose $B$ as the unique small solution  of $B+BGB=A$. Then we have that $B=0$ at $s=0$ since $A=0$ there. Thus the phase functions $\phi$ and $\varphi$ are equivalent.

\end{proof}

Now, the symbol calculus follows from \cite{UhlMel}. We include the relevant results in the semiclassical setting.

First, suppose $\lambda_0\in \partial\Lambda_1$ and choose $h_1,...,h_{d-1}$ functions whose differentials are linearly independent on $\partial\Lambda_1$ near $\lambda_0$. Choose also $f,g$ such that $f=0$ on $\Lambda_0$, $f>0$ on $\Lambda_1\setminus\partial\Lambda_1$, $df(\lambda_0)\neq 0$, $g=0$ on $\Lambda_1$, $dg(\lambda_0)\neq 0$, and $\{f,g\}(\lambda_0)< 0$. Let $a\in C^\infty(\Lambda_0\setminus \partial\Lambda_1;\Omega^{1/2})$ such that if $g\in C^\infty(\Lambda_0)$ vanishes on $\partial\Lambda_1$ then $ga\in C^\infty(\Lambda_0)$. Then write 
\m a=g^{-1}r|dh_1\wedge \dots \wedge dh_{d-1}\wedge dg|^{1/2}\,\,\m 
and define 
$$Ra:=r|dh_1\wedge \dots \wedge dh_{d-1}\wedge df|^{1/2}\{g,f\}^{-1/2}.$$
Then \cite[Section 4]{UhlMel} shows that $R$ is independent of the choice of $h_i$, $g$, and $f$ as above.

\begin{defin}
\label{def:iLagrangeSymbolClass}
We define the symbol class 
$$S_\delta^{\comp}(\Lambda_0\cup\Lambda_1)\subset h^{1/2}S_\delta^{\comp}(\Lambda_0\setminus \partial\Lambda_1;\Omega^{1/2})\times S_\delta^{\comp}(\Lambda_1;\Omega^{1/2})$$
as the subspace consisting of those sections $(a,b)$ such that for all $g$ vanishing on $\partial\Lambda_1$, $ga\in C^\infty(\Lambda_0)$ and $b|_{\partial\Lambda_1}=e^{\pi i/4}(2\pi)^{1/2}h^{-1/2}R(a)$.
\end{defin}

Then we have the following \cite[Theorem 4.13]{UhlMel}
\begin{lemma}
The following sequence is exact: 
$$0\hookrightarrow h^{1-2\delta}I_\delta^{\comp}(\Lambda_0,\Lambda_1)\hookrightarrow I_\delta^{\comp}(\Lambda_0,\Lambda_1) {\overset{\sigma}{  \rightarrow}} S_\delta^{\comp}(\Lambda_0\cup\Lambda_1)\rightarrow 0.$$ 
\end{lemma}
\begin{remark} Here $\sigma$ is the usual symbol map for Lagrangian distributions applied to each component $\Lambda_0\setminus\partial\Lambda_1$ and $\Lambda_1$ separately.
\end{remark}

We need the analog of \cite[Propositions 5.4 and 5.5]{UhlMel} in the semiclassical setting. First, we characterize the appearance of transport equations. The following lemma follows from Lemma \ref{l:lieDerivativeFIO}.
\begin{lemma}
Let $P\in \Psi_{\delta}^m(X)$ be a properly supported pseudodifferential operator such that $p:=\sigma(P)$ vanishes on the part $\Lambda_1$ of an intersecting pair $(\Lambda_0,\Lambda_1)$ of Lagrangians. Then for $u\in I^{m'}_{\delta}(X;\Lambda_0,\Lambda_1)$, $Pu=f+g$, $f\in h^{1/2}I_{\delta}^{m+m'-1/2}(X,\Lambda_0)$, $g\in h^{1-2\delta}I_{\delta}^{m+m'-1}(X;\Lambda_0,\Lambda_1)$ and 
$$\sigma(g)|_{\Lambda_1}=(-ih\mc{L}_{H_p}+p_1)\sigma(u)|_{\Lambda_1}$$
where $\mc{L}_{H_p}$ is the Lie action of the Hamilton vector field $H_p$ and $p_1$ is the subprincipal symbol of $P.$
\end{lemma}

Second, we need the asymptotic summability of the spaces $I^{m}_{\delta}(X;\Lambda_0,\Lambda_1).$ 

\begin{lemma}
Assume that $u_j\in h^{j(1-2\delta)}I_{\delta}^{m-j}(X;\Lambda_0,\Lambda_1)$ for $j=0,1,\dots$ then there exists $u\in I_{\delta}^{m}(X;\Lambda_0,\Lambda_1)$ such that for every $N$ there exists $N'>0$ large enough such that
$$u-\sum_{j=0}^{N'}u_j\in h^NC^N(X).$$
\end{lemma}

Finally, we need the following analog of \cite[Proposition 6.6]{UhlMel}. Define the characteristic set of $P$,
\m \Sigma(P)=\{\nu\in T^*X:\sigma(P)(\nu)=0\}.\,\,\m
We say that $P\in \Psi^m_{\delta}$ is of \emph{real principal type} if, letting $p:=\sigma(P)$ and $p_1:=\sigma_1(P)$, the subprincipal symbol, $p$ is real,
\m \partial p(q)\neq 0\text{ for }q\in \Sigma(P)\,\,\m
and $ \Im p_1\geq 0.$ We say that $P\in \Psi^m_{\delta}$ is elliptic if there exists $M>0$ such that for $|\xi|\geq M$, $|\sigma(P)|\geq C|\xi|^m.$

\begin{lemma}
\label{lem:greenFunc}
Let $P\in \Psi_{\delta}^m(X)$ be elliptic and of real principal type. Then let
\m \Lambda_0=\{(x,\xi,x,-\xi)\in T^*X\times T^*X\},\,\,\m
and $\Lambda_1^e$ be the $H_p$ flow out of $\Lambda_0\cap \Sigma(P)$ with orientation $e$. Assume that $\exp(tH_p)$ is non-trapping on $\Sigma(P)$. Then there exists $ u\in h^{-1/2}I_{\delta}^{-m}(X\times X;\Lambda_0,\Lambda_{1}^e)$, such that for each $K\Subset M$,
$$P(u+v)=\delta(y,y')+\O{\mc{D}'\to C^\infty}(h^\infty)\text{ for }(y,y')\in K\times K.$$
In particular, we have take
$$ \sigma(u)=(\sigma(P)^{-1},r)\in h^{-1/2}S_\delta^{-m}(\Lambda_0\cup\Lambda_1;\Omega^{1/2}) $$
where $r$ solves 
\begin{equation}
\label{eqn:transportILagrange}
h\mc{L}_{H_p}r+ip_1r=0,\quad r|_{\partial\Lambda_1}=e^{\pi i/4}(2\pi)^{1/2}h^{-1/2}R(\chi\sigma(P)^{-1})
\end{equation}
and where $p_1$ is the subprincipal symbol of $P$.
\end{lemma}

\begin{remark} Lemma \ref{lem:greenFunc} gives us the kernel of a right parametrix for $Pu=f$.
\end{remark}

\begin{proof}
First, let $\chi \in \Cc(T^*M\times T^*M)$ have $\chi \equiv 1$ on $\Lambda_1^e$ and $\chi_1\in \Cc(M)$ have $\chi_1\equiv 1$ on $K$. Then, since $\WFh(\oph(1-\chi))\cap \Lambda_1^e=\emptyset$, there exists $v\in I^{-m}_{\delta}(\Lambda_0)$ such that 
$$Pv=(1-\oph(\chi))\delta(y,y')\chi_1(y)\chi_1(y')+\O{\mc{D}'\to C^\infty}(h^\infty).$$
In particular, $v$ is the kernel of a pseudodifferential operator $V\in \Psi^{-m}_{\delta}(X).$

We now solve $Pu=\oph(\chi)\delta(y,y')\chi_1(y)\chi_1(y')+\O{\mc{D}'\to C^\infty}(h^\infty)$. To do so, we proceed symbolically. Suppose that $u_0\in h^{-1/2}I_{\delta}^{\comp}(X\times X;\Lambda_0,\Lambda_1^e).$ Then we have 
$$Pu_0=f_0+g_0,\quad f_0\in I_{\delta}^{\comp}(\Lambda_0),\quad g_0\in h^{1/2-2\delta}I_{\delta}^{\comp}(X\times X;\Lambda_0,\Lambda_1^e).$$ 
Now,
\m \sigma(f_0)=\sigma(P)\sigma(u_0)|_{\Lambda_0}.\,\,\m 
Thus, writing $p=\sigma(P)$, we have 
$$\sigma(u_0)|_{\Lambda_0\setminus \partial\Lambda_1^e}=p^{-1}\sigma(\oph(\chi)\delta(y,y')\chi_1(y)\chi_1(y'))\in S^{\comp}_{\delta}(\Lambda_0\setminus \partial\Lambda_1^e).$$
Thus, using the fact that $\chi \equiv 1$ on $\Lambda_1^e$,
$$\sigma(u_0)|_{\partial\Lambda_1}=e^{\pi i/4}(2\pi )^{1/2}h^{-1/2}R(p^{-1}\sigma(\delta(y,y')\chi_1(y)\chi_1(y')))$$
and hence
\m \sigma(g_0)|_{\Lambda_1}=(-ih\mc{L}_{H_p}+p_1)\sigma(u_0)\,\,\m
where $p_1$ is the subprincipal symbol of $P$. Thus, $\sigma(g_0)=0$ on $\Lambda_1$ yields the transport equation
\m h\mc{L}_{H_p}\sigma(u_0)+ip_1\sigma(u_0)=0\text{ on }\Lambda_1.\,\,\m
Under our assumptions, \cite[Section 6.4]{FIO2} gives that this equation has a unique solution. Then since $\Im p_1\geq 0$ and $\sigma(u_0)|_{\Lambda_0}\in S^{\comp}_{\delta}$, we have that for $q\in \partial\Lambda_1^e$, $u_0|_{\Lambda_1^e\cap T^*K}\in h^{-1/2}S^{\comp}_{\delta}$.
Thus, for $(y,y')\in K$,
\begin{align*} Pu_0-\delta(y,y')\chi_1(y)\chi_1(y')&=f_1+g_1\\
&\in h^{1-2\delta}I_{\delta}^{\comp}(\Lambda_0)+h^{3/2-4\delta}I_{\delta}^{\comp}(X;\Lambda_0,\Lambda_1^e).\end{align*}

Finally, let $\chi_2\in \Cc(M)$ have $\chi_2\equiv 1$ on $\supp \chi_1$. Then relabel $u_0=\chi_2(y)\chi_2(y')u_0$.
Now, we proceed iteratively to find $u_j\in h^{j(1-2\delta)-1/2}I_{\delta}^{\comp}(X;\Lambda_0,\Lambda_1^e)$, given 
$f_j\in h^{j(1-2\delta)}I_{\delta}^{\comp}(\Lambda_0),$ and $ g_j\in h^{
j(1-2\delta)+1/2-2\delta}I_{\delta}^{\comp}(X\times X;\Lambda_0,\Lambda_1^e),$
such that 
$\sigma(u_j)|_{\Lambda_0}=p^{-1}\sigma(f_j)|_{\Lambda_0},$ \m h\mc{L}_{H_p}\sigma(u_j)+ip_1\sigma(u_j)=i\sigma(g_j)\,\text{ on }\Lambda_1^e.\,\,\m
As above, the transport equation has a unique solution  satisfying the initial condition 
\m \sigma(u_j)|_{\partial\Lambda_1}=e^{\pi i/4}(2\pi)^{1/2}h^{-1/2}R(p^{-1}\sigma(f_j)).\,\,\m
Then, letting
\m u\sim \sum u_j\,\,\m
gives that for $(y,y')\in K\times K$
$$ P(u+v)=\delta(y,y')+\O{\mc{D}'\to C^\infty}(h^\infty)$$
as desired. Then simply relabel $u=u+v$.
\end{proof}

\chapter{The Semiclassical Melrose--Taylor Parametrix}
\label{ch:semiclassicalDirichletParametrices}

Let $\Omega\subset \re^d$ be strictly convex with smooth boundary. We need to analyze
\begin{equation}\label{eqn:dirichlet}(-h^2\Delta-z^2)u=0\text{ in }\Omega_i\,,\quad\quad
u|_{\partial \Omega}=f
\end{equation}
where $\Omega_1=\Omega$ and $\Omega_2=\re^d\setminus \overline{\Omega}$, and $f$ is microlocalized near glancing.

We give a construction similar to that in \cite[Appendix A.II.3]{Gerard} and \cite[Chapter 11]{TaylorPseud} in $\Omega_2$ and adapt the results there to the case of $\Omega_1$ using methods similar to those in \cite[Chapter 7]{MelTayl}. Throughout, we assume $z=1+i\Im z$, $-Ch\log h^{-1}\leq \Im z\leq Ch\log h^{-1}$.

\begin{remark} To obtain $\Re z\neq 1$, we simply rescale $h$ in the resulting parametrices. 
\end{remark}

Define $\e(h)$ and $\mu(h)$ by
$$h\leq \e(h):= \max(h,|\Im z|)=\O{}(h\log h^{-1})\quad \quad \mu(h):=\Im z.$$
We construct parametrices in a neighborhood of glancing where the size of the neighborhood will depend on $\e(h)$.

In particular, if $\chi\in C_c^\infty(\re)$, $\chi\equiv 1$ in a neighborhood of $0$. Then, let $x_0\in \partial \Omega$, $\delta>0$ and define 
$$\chi^{\delta,\gamma}_h(x,\xi):=\chi\left(\frac{|\xi'|_g-1}{\gamma [h(\e(h))^{-1}]^2}\right)\chi(\delta^{-1}|x-x_0|)$$
where $|\cdot|_g$ denotes the norm induced on the $T^*\partial\Omega$ by the euclidean metric restricted to the boundary.
Then $\chi_h^{\delta,\gamma}$  localizes microlocally near a glancing point $(x_0,\xi_0)$. We construct an operator $H$ such that for $U_i\subset \Omega_i$ a neighborhood of $x_0$
\begin{equation}
\label{eqn:parametrixCondEuc}
\begin{cases}
(-h^2\Delta -z^2)H f=\mc O_{C^\infty}(h^\infty) & \text{in }U_i\\
Hf=\oph (\chi_h^{\delta,\gamma})f+\mc O_{C^\infty}(h^\infty)&\text{ in a neighborhood of }x_0\in \partial\Omega \\
Hf\text{ is outgoing if }\Omega_i=\Omega_2
\end{cases}
\end{equation}
In fact, we need to construct two such operators $H_{g}$ for gliding points and  $H_{d}$ for diffractive points corresponding to $\Omega_1$ and $\Omega_2$ respectively. Because of the fact that we are only looking for properties of $H$ in a neighborhood of $x_0$, we are able to construct operators $H_d$, and $H_g$ with the property \eqref{eqn:parametrixCondEuc} for $|\Im z|\leq Ch\log h^{-1}$.

With these operators in hand, using the arguments in \cite[Appendix A.5]{StefVod}, we show (see Section \ref{sec:relationExact}) that in a neighborhood of $\pO$, $H_d$ is $\O{C^\infty}(h^\infty)$ close to the true solution operator for \eqref{eqn:dirichlet}. However, because of the local nature of our construction of $H_g$, we are not able to do this for the gliding case. Instead, we must restrict our attention to the construction of the operators that were used in Section \ref{sec:BLONearGlance} to produce microlocal descriptions of the single and double layer potentials and operators for $|\Im z|\leq Ch\log h^{-1}$ (see Section \ref{sec:glide}). 

\begin{remark}
Note that in \cite[Chapter 7]{MelTayl}, the parametrix that the authors construct for the wave equation in the interior of $\Omega$ is $C^\infty$ close to the true solution. However, it is valid only for $t\in[0,T)$ for any $T>0$. Showing that $H_d$ is close to the true solution of \eqref{eqn:dirichlet} would be equivalent to a global in time $C^\infty$ parametrix for the wave problem.
\end{remark}

\section{Construction of the operators $H_g$ and $H_d$}
Following \cite[Chapter 7]{MelTayl} and \cite{MelroseParametrix}, the ansatze for our constructions will be Fourier-Airy integral operators \cite{MelroseAiry} of the form:
\begin{gather}
\label{eqn:FAIO}\begin{aligned}B_1F:=(2\pi h)^{-d+1} &\int_{\re^{d-1}} [g_0A_-(h^{-2/3}\rho)+ih^{1/3}g_1A_-'(h^{-2/3}\rho)]\\
&\quad\quad\quad\quad A_-(h^{-2/3}\alpha)^{-1}e^{i\theta/h}\mc{F}_h{F}(\xi)d\xi,\end{aligned}\\
\label{eqn:FAIOIn}\begin{aligned}B_2F:=(2\pi h)^{-d+1} &\int_{\re^{d-1}} [g_0Ai(h^{-2/3}\rho)+ih^{1/3}g_1Ai'(h^{-2/3}\rho)]\\
&\quad\quad\quad\quad Ai(h^{-2/3}\alpha)^{-1}e^{i\theta/h}\mc{F}_h{F}(\xi)d\xi.\end{aligned}\end{gather}

\noindent where $F$ and $f$ will be related below, $\rho|_{\partial \Omega}=\alpha$, $\rho\,,\,\theta \in C^\infty$ solve certain eikonal equations, $g_0\,,\,g_1$ solve transport equations, and $Ai$ is the solution to $-A''(s)+sA=0$ given by 
$$Ai(s)=\frac{1}{2\pi}\int_{-\infty}^\infty e^{i(st+t^3/3)}dt$$
for $s$ real, and $A_-(z)=Ai(e^{2\pi i/3}z).$ Finally, $\theta|_{\partial\Omega}$ will parametrize the canonical transformation reducing the billiard ball map for glancing pair $\{x\in \partial\Omega\}$ and $\{|\xi|^2-1=0\}$ to that for the Friedlander normal form
\begin{equation} \label{eqn:normalFormGlance}Q_{\Fried}:=\{x_d=0\}\subset T^*\re^d\quad\quad \text{ and }\quad \quad P_{\Fried}:=\{ \xi_d^2-x_d+\xi_1=0\}\subset T^*\re^d\end{equation}
The Hamiltonian flow for this system is shown in Figure \ref{fig:FriedlanderModel}.

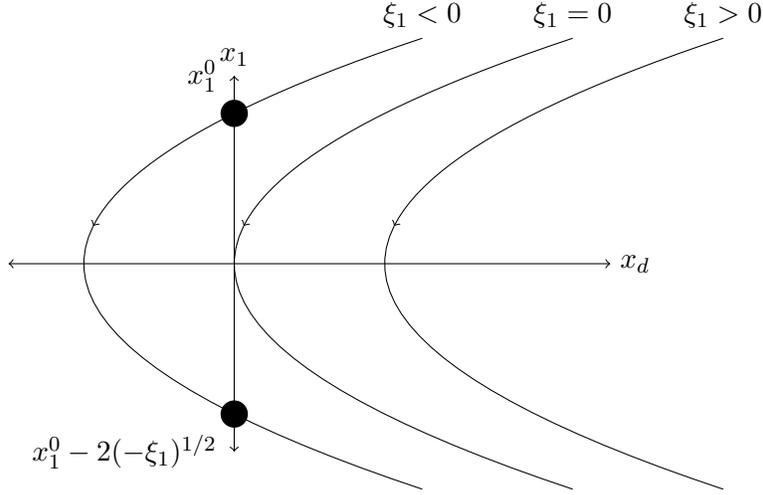
\begin{figure}
\centering
\begin{tikzpicture}
\draw[<->](-3,0)to (5,0)node[right]{$x_d$};
\draw[<->](0,-2.5)to (0,2.5)node[above]{$x_1$};
\foreach \y in {-2,0,2}{
\draw [domain=3:.5,->] plot ({.5*\x*\x-\y},\x);
\draw [domain=-3:.5] plot ({.5*\x*\x-\y},\x);}
\draw (-.1,2.5)node[left]{$x_1^0$};
\draw (-.1,-2.5)node[left]{$x_1^0-2(-\xi_1)^{1/2}$};
\draw (.5*3*3-2,3)node[above]{$\xi_1<0$};
\draw (.5*3*3,3)node[above]{$\xi_1=0$};
\draw (.5*3*3+2,3)node[above]{$\xi_1>0$};
\draw (0,2)node[circle,fill]{};
\draw (0,-2)node[circle,fill]{};
\end{tikzpicture}
\caption[Trajectories for the Friedlander Mode]{\label{fig:FriedlanderModel} The figure shows several trajectories of the Hamiltonian flow for the Friedlander model. When $\xi_1=0$, the trajectory is tangent to the boundary $x_d=0$ and hence is glancing. When, $\xi_1<0$, the billiard ball map takes the upper intersection with $x_d=0$ to the lower. This corresponds to the hyperbolic region. Finally, when $\xi_1>0$, the trajectory does not intersect the boundary and hence this corresponds to the elliptic region.}
\end{figure}

\subsection{The Friedlander Model}
As a first, step, we consider the Friedlander model. This toy example guides us when we consider the general case.
 The Friedlander model is given by 
$$P=(hD_{x_d})^2-x_d+hD_{x_1}\quad \quad \partial\Omega=\{x_d=0\}.$$
Suppose that 
\begin{equation}
\label{eqn:BVPModel}
(P-i\mu)u=0\quad\quad
u|_{\partial\Omega}=f
\end{equation}
Then, taking the semiclassical Fourier transform in the $x'$ variables gives
$$(-h^2\partial_{x_d}^2-x_d+\xi_1-i\mu)\mc{F}_{h,x'}u(x_d,\xi')=0\quad \quad \mc{F}_{h,x'}u(0,\xi')=\mc{F}_h(f)(\xi').$$
The solution to this problem for $\mu=0$ is 
$$u=(2\pi h)^{-d+1}\int \frac{A(h^{-2/3}(-x_d+\xi_1))}{A(h^{-2/3}\xi_1)}e^{\frac{i}{h}\la x',\xi'\ra}\mc{F}_h(f)(\xi')d\xi'$$
where $A$ is a solution to the Airy equation. Let $\rho_0:=-x_d+\xi_1$ and $\theta_0=\la x',\xi'\ra $. Now, suppose that $\mu=\O{}(h\log h^{-1})\neq 0$. We could simply replace $\rho_0$ by $-x_d+\xi_1-i\mu$, however, because the function $Ai$ has zeros on the real axis, it is more convenient when we consider the general case to make a perturbation of $\theta_0$ and $\rho_0$ so that uniformly in $\mu$, $\rho_0|_{x_d=0}$ has nonzero imaginary part. To do this, we compute 
\begin{align*} 
P(A(h^{-2/3}\rho)e^{\frac{i}{h}\theta})&= \left[(\partial_{x_d}\theta)^2-\rho (\partial_{x_d}\rho)^2+\partial_{x_1}\theta-x_d\right]A(h^{-2/3}\rho)e^{\frac{i}{h}\theta}\\
&\quad -ih^{1/3}\left[ 2\partial_{x_d}\rho\partial_{x_d}\theta +\partial_{x_1}\rho \right]A'(h^{-2/3}\rho)e^{\frac{i}{h}\theta}
\end{align*}
So, we seek to find $\theta$ and $\rho$ solving the model eikonal equations
$$\begin{cases} (\partial_{x_d}\theta)^2-\rho(\partial_{x_d}\rho)^2+\partial_{x_1}\theta -x_d=i\mu \\
2\partial_{x_d}\rho \partial_{x_d}\theta +\partial_{x_1}\rho=0
\end{cases}.
$$

We find $\rho \sim \sum_{n\geq 0} \rho_{n}\e(h)^n $ and $\theta \sim \sum_{n\geq0} \theta_{n} \e(h)^n$ where $\rho_0$ and $\theta_0$ are as above, $\theta_n=\theta_n(x,\xi',\mu),$ and $\rho_n=\rho_n(x,\xi',\mu)$. Then, we solve for $\rho_{n}$, $\theta_{n}$ successively by solving transport equations of the form 
$$
\begin{cases}
2\partial_{x_d}\theta_0\partial_{x_d}\theta_{n} -2\rho_0\partial_{x_d}\rho_0\partial_{x_d}\rho_{n}-\rho_{n}(\partial_{x_d}\rho_0)^2+\partial_{x_1}\theta_{n}=F_1\\
2\partial_{x_d}\rho_0\partial_{x_d}\theta_{n}+2\partial_{x_d}\theta_0\partial_{x_d}\rho_{n}+\partial_{x_1}\rho_{n}=F_2
\end{cases}$$
where $F_1$ and $F_2$ depend on $\theta_{n}$ and $\rho_{n}$ for $m<n$.

In the next section, we construct solutions to these equations with $\rho_{1}(x_1,0,\xi')=i$.
With these solutions in hand $u$ will solve \eqref{eqn:BVPModel} up to $\O{}(h^\infty).$

\subsection{Eikonal and Transport Equations}
First, we consider a general differential operator 
$$P(x,hD)=\sum a_{jk}(x)hD_jhD_k+\sum b_j(x)hD_j+c(x)$$
with $a_{jk}=a_{kj}$ applied to \eqref{eqn:FAIO} and \eqref{eqn:FAIOIn}.
Then, for $A$ an Airy function, we have, letting $f_j$ denote $\partial_j f$, and $\rho_h=h^{-2/3}\rho$
\begin{align*}
hD_j \left(g A(\rho_h)e^{\frac{i}{h}\theta}\right)&=\theta_j gA(\rho_h)e^{\frac{i}{h}\theta}-ihg_jA(\rho_h)e^{\frac{i}{h}\theta}-ih^{1/3}\rho_jgA'(\rho_h)e^{\frac{i}{h}\theta}\\
hD_khD_j\left(gA(\rho_h)e^{\frac{i}{h}\theta}\right)&=\\
&\!\!\!\!\!\!\!\!\!\!\!\!\!\!\!\!\!\!\!\!\!\!\!\!\!\left[ (\theta_k\theta_j -\rho_j\rho_k\rho)g-ih(\theta_k g_j+\theta_jg_k+\theta_{jk}g)-h^2g_{jk}\right]A(\rho_h)e^{\frac{i}{h}\theta}\\
&\!\!\!\!\!\!\!\!\!\!\!\!\!\!\!\!\!\!\!\!\!\!\!\!\! -ih^{1/3}\left[ (\theta_j\rho_k+\rho_j\theta_k)g-ih(g_j\rho_k+\rho_jg_k+\rho_{jk}g)\right]A'(\rho_h)e^{\frac{i}{h}\theta}\nonumber\\
hD_j \left(g A'(\rho_h)e^{\frac{i}{h}\theta}\right)&=\\
&\!\!\!\!\!\!\!\!\!\!\!\!\!\!\!\!\!\!\!\!\!\!\!\!\!\theta_j gA'(\rho_h)e^{\frac{i}{h}\theta}-ihg_jA'(\rho_h)e^{\frac{i}{h}\theta}-ih^{-1/3}\rho_j\rho gA(\rho_h)e^{\frac{i}{h}\theta}\\
hD_khD_j\left(gA'(\rho_h)e^{\frac{i}{h}\theta}\right)&=-ih^{-1/3}\left[ (\theta_j\rho_k+\theta_k\rho_j)\rho g\right.\\
&\!\!\!\!\!\!\!\!\!\!\!\!\!\!\!\!\!\!\!\!\!\!\!\!\! \left. -ih(g_j\rho_k\rho +g_k\rho_j\rho  +\rho_{jk}\rho g+\rho_j\rho_k g )\right]A(\rho_h)e^{\frac{i}{h}\theta}\nonumber\\
&\!\!\!\!\!\!\!\!\!\!\!\!\!\!\!\!\!\!\!\!\!\!\!\!\!\!\!\!\!\!\!\!\!  +\left[ (\theta_j\theta_k-\rho_j\rho_k\rho )g-ih(\theta_{kj}g+\theta_jg_k+\theta_kg_j)-h^2g_{jk}\right]A'(\rho_h)e^{\frac{i}{h}\theta}\nonumber
\end{align*}

So, 
\begin{align*} P(g_0A(\rho_h)e^{\frac{i}{h}\theta})&\\
&\!\!\!\!\!\!\!\! \!\!\!\!\!\!\!\! \!\!\!\!\!\!\!\! \!\!\!\!\!\!\!\! =\left[
\begin{aligned} (\la ad\theta,d\theta\ra -\rho\la ad\rho,d\rho\ra+\la b,d \theta\ra +c)g_0 \\
 -ih(2\la a d\theta,dg_0\ra  -P_2\theta g_0+\la b, dg_0\ra )+h^2P_2g_0
 \end{aligned}
 \right]A(\rho_h)e^{\frac{i}{h}\theta}\nonumber \\
&\quad -ih^{1/3}\left[\begin{aligned} (2\la ad\theta,d\rho\ra +\la b,d\rho\ra) g_0\\
-ih(2\la ad\rho,dg_0\ra -(P_2\rho)g_0)
\end{aligned}\right]A'(\rho_h)e^{\frac{i}{h}\theta}\nonumber\\
P(ih^{1/3}g_1A'(\rho_h)e^{\frac{i}{h}\theta})&=\\
&\!\!\!\!\!\!\!\! \!\!\!\!\!\!\!\! \!\!\!\!\!\!\!\! \!\!\!\!\!\!\!\!\left[\begin{aligned} 
\rho( 2\la ad\theta,d\rho\ra + \la b,d\rho\ra )g_1\\
-ih(2\rho\la ad\rho,dg_1\ra +\la ad\rho,d\rho\ra g_1-\rho (P_2\rho) g_1)\end{aligned}
\right]A(\rho_h)e^{\frac{i}{h}\theta}\nonumber \\
&\!\!\!\!\!\!\!\! \!\!\!\!\!\!\!\! \!\!\!\!\!\!\!\! \!\!\!\!\!\!\!\!  \!\!\!\!\!\!\!\!+ih^{1/3}\left[\begin{aligned}  (\la a d\theta,d\theta\ra -\rho\la ad\rho,d\rho\ra+\la b,d\theta\ra +c) g_1\\
-ih(2\la a d\theta,dg_1\ra -(P_2\theta)g_1+\la b,dg_1\ra)+h^2P_2g_1\end{aligned}\right]A'(\rho_h)e^{\frac{i}{h}\theta}\nonumber
\end{align*}
where $a_{jk}=a_{jk}(x)$, $P_2=h^{-2}(P-\la b,hD\ra -c(x)) $ and $\la\cdot, \cdot\ra$ denotes the euclidean inner product.

Now, applying $P$ under the integral in \eqref{eqn:FAIO} and \eqref{eqn:FAIOIn}
gives the eikonal equations 
\begin{equation}
\label{eqn:eikEuc}
\begin{cases}
\la a d \theta, d \theta \ra - \rho \la a d \rho , d \rho \ra +\la b,d\theta\ra +c=0\\
2\la ad \theta, d \rho \ra +\la b,d\rho\ra =0
\end{cases}. \end{equation}
Writing 
\begin{equation}\label{eqn:phipm}\phi^{\pm}=\theta\pm \frac{2}{3}(-\rho)^{3/2},\end{equation}
the eikonal equations are equivalent to $p(x,d\phi^{\pm})=0$. 
Now, suppose that $\rho$ has the form $\sum_{n\geq 0}\rho_{n}\e(h)^n$ and $\theta$ has the form $\sum_{n\geq 0}\theta_n \e(h)^n$ and 
$$g_i\sim \sum_{n} g_{i}^{[n]}(x,\xi',\mu)h^n.$$
Then the transport equations have the form 
\begin{equation}
\label{eqn:transEuc}
\left\{\begin{aligned}
\begin{gathered}2\la a d\theta_0,dg_0^{[n]}\ra+2\rho_0\la ad\rho_0,dg_1^{[n]}\ra  +\la b, dg_0^{[n]}\ra \\
+\la ad\rho_0,d\rho_0\ra g_1^{[n]}-P_2\theta_0 g_0^{[n]}-\rho_0 (P_2\rho_0) g_1^{[n]}\end{gathered}
=F_1^{[n]}(\theta,\rho,g_{i}^{[m]<[n]},\mu)\\
{}\\
\begin{gathered}2\la ad\rho_0,dg_0^{[n]}\ra-2\la a d\theta_0,dg_1^{[n]}-\la b,dg_1^{[n]}\ra \ra\\
-(P_2\rho_0)g_0^{[n]} +(P_2\theta_0)g_1^{[n]}\end{gathered}=F_2^{k,m}(\theta,\rho,g_{i}^{[m]<[n]},\mu).
\end{aligned}\right.
\end{equation}
More generally, we consider transport equations of the form 
\begin{equation}
\label{eqn:TransportGen}
\begin{cases}
\begin{gathered}2\la ad \theta_0, d g_0\ra+2\rho_0 \la ad \rho_0, d g_1\ra+\la b,dg_0\ra \\
+\la a d \rho_0 ,d \rho_0\ra g_1+B_1 g_0+\rho_0 B_2 g_1\end{gathered}=F_1\\
{}\\
2\la a d \rho_0, d g_0\ra -2\la a d \theta_0 ,d g_1\ra-\la b,dg_1\ra  +B_2 g_0 -B_1 g_1=F_2
\end{cases}
\end{equation}
Then, these equations are equivalent to \
\begin{equation} 
\label{eqn:transNewForm}2\la a d\phi^{\pm}, g^{\pm}\ra +\la b, dg^{\pm}\ra +G^{\pm}g^{\pm}=F^{\pm}
\end{equation}
where 
$$g^{\pm}=g_0\pm (-\rho_0)^{1/2}g_1\quad \quad G^{\pm}=B_1\mp(-\rho_0)^{1/2}B_2\quad\quad F^{\pm}=F_1\mp(-\rho_0)^{1/2}F_2.$$

We use the equivalence of glancing hypersurfaces to construct solutions of the eikonal equations near a glancing point. In particular, let $p(x,\xi)$ be the symbol of $P(x,hD)$ and $B$ a hypersurface in $M$. Let $P=\{p(x,\xi)=0\}$ and $Q=\{(x,\xi):x\in B\}$ be a pair of glancing manifolds at $m=(x_0,\xi_0)\in P\cap Q$. That is, if $q(x,\xi)=q(x)$ is a defining function for $Q$
\begin{gather*}dp\text{ and }dq\text{ are linearly independent at }m\\
\{p,q\}=0\quad \quad \{p,\{p,q\}\}\neq 0\quad \quad \{q,\{q,p\}\}\neq 0
\end{gather*}
Then the equivalence of glancing hypersurfaces (see for example \cite[Theorem 21.4.8]{HOV3}) gives the existence of neighborhoods $V$ of $m$ and $U$ of $0$ and a symplectomorphism $\kappa:U\to V$ reducing $P$ and $Q$ to the normal form \eqref{eqn:normalFormGlance}. Since $Q$ is the lift of a hypersurface to $T^*\re^d$, this also induces a symplectomorphism 
$$\kappa_\partial: \gamma\to T^*B\quad \gamma:=\{(y',\eta')\in T^*\re^{d-1}:(y',y_d,\eta',\eta_d)\in U \text{ for some }\eta_d\}$$
such that $\kappa_\partial$ intertwines the billiard ball map on $T^*B$ with that on $T^*Q_{\Fried}$. 

We assume further that $H_p$ is not tangent to $T_{x}\re^d$ at $x=\pi(m)$. This allows us to conclude that 
\begin{equation}
\label{eqn:indep}\kappa_{\partial}^*(d\eta_j),\quad j=1,\dots d-1\text{ are  linearly independent on }T_{x}^*B.
\end{equation}
To see this, observe that the projection of $H_p$ onto $T^*B$ is not tangent to $T_xB$. This image is the direction of the Hamilton vector field on the fold set and hence it follows that $\partial_{y_1}$ is not tangent to 
$$\mc{H}:=\kappa_{\partial}^{-1}(T^*_xB).$$
Observe that $\mc{H}$ is Lagrangian and hence $d\eta_1\neq 0$ on $\mc{H}$. Hence, there exists a symplectic change of coordinates on leaving $(y_1,\eta_1)$ fixed such that $d\eta_j$ $j=1\dots d-1$ are independent on $\mc{H}$ and therefore that \eqref{eqn:indep} holds. This transformation can clearly be extended to leave $Q_{\Fried}$ and $P_{\Fried}$ fixed. 

Now, consider 
$$Y:P\to M\times \re^{d-1}\quad \quad P\ni p\mapsto (\pi(p),\eta_1(\kappa^{-1}(p)),\dots,\eta_{d-1}(\kappa^{-1}(p))).$$
\begin{lemma}
The map $Y$ is a fold at $m$. Moreover, the fold set meets $Q_{\Fried}$ transverally at $\xi_d=0$. 
\end{lemma}
\begin{proof}
Let $q\in C^\infty(M)$ be a defining function for $B$. Then $dq\neq 0$ on $P$ near $m$. Thus, we need only consider the restriction of $Y$ to the intersection of $P$ and $Q$. That is, 
$$Y':P\cap Q\to B\times \re^{d-1},\quad Y'=Y|_{P\cap Q}.$$
But, the map from $P\cap Q$ to $T^*B$ is a fold and $Y'$ is this projection composed with replacement of the fiber variables by $\eta_j$ $j=1,\dots d-1$ which has bijective differential. Hence, $Y'$ and $Y$ are folds with the desired properties. 
\end{proof}

The construction of solutions to the eikonal equations near a glancing point now follows from \cite[Proposition 4.3.1]{MelTayl} which we include here.
\begin{lemma}
\label{lem:EikSolve}
Let $p$ be the (real) principal symbol of a differential operator with $C^\infty $ coefficients in a neighborhood of $B\subset \re^d$ with defining function $x_d$. If $P$ and $Q$ form a glancing pair at $m$ then there exist real functions $\theta_0$ and $\rho_0$ smooth in a neighborhood, $\Sigma$, of $\pi(m)\times \{0\}\in \re^d\times \re^{d-1}$ such that 
\begin{gather*}
\rho_0=\eta_1\text{ on }\Sigma\cap(B\times \re^{d-1})\\
\theta|_B\text{ parametrizes $\kappa_\partial$, the reduction of $P$ and $Q$ to normal form}\\
d_x\partial_{\eta_j}\theta,\quad j=1\dots d-1\text{ are linearly independent on }\Sigma \\
\rho_0\text{ is a defining function for the fold}
\end{gather*}
and $\rho_0$ and $\theta_0$ solve the eikonal equations \eqref{eqn:eikEuc} in $\rho_0\leq 0$ and in Taylor series on $B$. 
\end{lemma}
\begin{proof}
Let 
$$\Lambda_{\eta'}=\{p\in P\,:\, Y(p)=(\cdot, \eta')\,:\, \eta'\in \re^{d-1}\}.$$
Then, $\Lambda_{\eta'}$ are Lagrangian submanifolds foliating $P$ near $m$. To see that they are Lagrangian, observe that
$$\kappa^{-1}(\Lambda_{\xi'})=\{((y',\eta_1+\eta_d^2),(\eta',\eta_d)):y'\in U\subset \re^{d-1},\eta_d\in V\subset\re \}$$
and hence is Lagrangian. 

This implies that the canonical one form, $\omega=\xi dx|_{\Lambda_{\eta'}}$ is closed and hence there exists $\Phi$ a smooth function on $P$ such that 
$$d(\Phi|_{\Lambda_{\eta'}})=\omega |_{\Lambda_{\eta'}}\quad \text{for $\eta'$ near }\eta_0'$$
and hence $p(x,d\Phi)=0$.

In fact, since $\Phi$ is the integral of a one form, it is locally unique up to a normalization on each $\Lambda_{\eta'}$. We fix this normalization by choosing $T\subset P$ a submanifold of dimension $d$ transverse to the fibration by $\Lambda_{\eta'}$ and contained in the fold of $Y$. We then insist that $\Phi|_{T}=0$. Now, since $Y$ is a fold 
\begin{equation}\label{eqn:projForm} Y(P)=\{\eta_1\leq x_df(x,\eta)\}
\end{equation}
with $f(m)\neq 0$.  Then, since $Y$ is a fold and $\Lambda_{\eta'}$ is Lagrangian, by \cite[Theorem 21.4.1]{HOV3}
\begin{equation}\Phi=Y^*(\theta_0\pm \frac{2}{3}(-\rho_0)^{3/2})\label{eqn:phiForm2}\end{equation}
where $\theta_0$, $\rho_0:Y(P)\to \re$ are smooth and $\rho_0$ is a defining function for the fold. Moreover, the odd part of $\Phi$ vanishes to second order at the fold since $\Phi$ is the integral of a smooth 1-form.

Next, we show that $\rho_0=\eta_1$. Notice that this is independent of the choice of the reduction to normal form, $\kappa$ and the choice of $T$. Fix $\kappa$ and suppose that $\Phi_1$ and $\Phi_2$ are two smooth solutions of $p(x,d\Phi)=0$ corresponding to different submanifolds $T_1$ and $T_2$. Then, let $w=\Phi_1-\Phi_2$. $w$ is constant on each leaf of $\Lambda_{\eta'}$ and hence is a function of only $\eta'$. Observe that the involution defined by $Y$ preserves $\Lambda_{\eta'}$ and hence the $Y$ odd (and even) part of $w$ is a function of only $\eta$. But this implies that the $Y$ odd part vanishes identically. Hence, since $\theta$ is $Y$ even, $\rho_{\Phi_1}=\rho_{\Phi_2}$.

Observe that over $B$, the involution map of $Y$ is just the projection of $P\cap Q$ to $T^*B$ and the function $\Phi$ pulls back under $\kappa$ to a solution to the same problem for the model case. Together, these imply that the odd part of $\Phi$ restricted to the boundary is independent of the choice of $T$ and $\kappa$ and hence is the same as for the model case.


Next, observe that $\theta_{\Phi_1}-\theta_{\Phi_2}$ is a function of only $\eta'$. Hence, $\partial_x(\theta_{\Phi_1}-\theta_{\Phi_2})=0$. But, as in the previous paragraph, at the boundary $B$, $\Phi$ pulls back under $\kappa$ to a solution of $p(x,d\Phi)=0$ for the model problem and hence $\partial^2_{x'\eta'}\kappa^*\theta|_B =I$. In particular, 
\begin{gather*} d_{x'} \partial_{\eta_j}\theta \text{ are linearly independent for }j=1\dots d-1\\
\end{gather*}

Now, by construction 
$$\Lambda_{\eta'}=\{(x,\partial_x\Phi(x,\eta'))\}$$
and $\kappa^{-1}(x,\partial_x\Phi(x,\eta'))=(y(x,\partial_x\Phi(x,\eta')),\eta',(y_d-\eta_1)^{1/2})$ when $(y_d-\eta_1)\geq 0$. Now, on $B$, this holds for $\eta_1\leq 0$ and $\partial_x\Phi(x,\eta')|_{B}=\partial_xY^*\theta,$. Now, if $\kappa(y,\eta)=(x,\xi)$, then, using that $\kappa$ is a symplectomorphism, we have $\left(\frac{\partial y}{\partial x}\right)^t=\frac{\partial \xi}{\partial\eta}$. Therefore, 
$$\frac{\partial y}{\partial x}(x,\partial_x\Phi(x,\eta'))|_{B}=(\partial^2_{\eta x}\theta )^t|_B.$$ Thus,
$$y=\partial_\eta \theta(x',\eta') +f(\eta')$$
and hence, using that $\kappa_{\partial}$ is symplectic, we have that $f=\partial_\eta' g$ and hence by adjusting the normalization $T$, we can arrange that $\theta|_B$ generates $\kappa_{\partial}$. 

At this point, we have solved the eikonal equations $p(x,d\phi^{\pm})=0$ with $\phi^{\pm}$ having the correct form in the region $\rho_0\leq 0$.This is a region of the form \eqref{eqn:projForm}. Our last task is to extend these solutions so that the eikonal equations continue to hold in Taylor series at $B$ and $\rho_0=0$. 

By the Malgrange preparation theorem, we can write 
$$p(x,\xi)=p'[(\xi_d-a(x,\xi'))^2-b(x,\xi')],$$
where $p'$ is nonvanishing near $m$, $a,b$ are real, and $\xi'=(\xi_1,\dots \xi_d).$ Thus, we can drop $p'$ when solving $p(x,d\phi^{\pm})=0$. Then, by the glancing hypothesis on $p$ and $q$ along with $H_p$ not tangent to the fiber at $x=\pi(m)$, 
$$\xi_d=a,\quad b=0,\quad d_{\xi'}b\neq 0\quad \text{at }m$$
with $b=0$, $x_d=0$ the glancing surface. Then, $p(x,d\phi^{\pm})=0$ becomes
\begin{equation}\label{eqn:newEik} \partial_{x_d}\phi^{\pm}-a(x,\partial_{x'}\phi^{\pm})=\pm \left(b(x,\partial_{x'}\phi^{\pm}\right)^{1/2}\quad \text{ in }\rho_0\leq 0.
\end{equation}
Then, extending $\rho_0$ and $\theta_0$ to smooth real valued functions across $\rho_0=0$ gives solutions to \eqref{eqn:newEik} in Taylor series at $\rho_0=0$. We write $\phi_1^{\pm}$ for the extended functions. Then,
\begin{equation}\label{eqn:newEik2} \partial_{x_d}\phi_1^{\pm}-a(x,\partial_{x'}\phi_1^{\pm})\mp \left(b(x,\partial_{x'}\phi_1^{\pm}\right)^{1/2}=e_{\pm}
\end{equation}
with $e_{\pm}=0$ in $\rho_0\leq 0$ and vanishing to all orders at $\rho_0=0$. Then to solve \eqref{eqn:newEik} to all orders at $x_d=0$, we add to $\phi_1$ function
$$\phi_2\sim \sum_{k=1}^\infty x_d^kg_k(x',\xi)$$
with $\phi_2$ vanishing in $\rho_0\leq 0$. 
Then, by \eqref{eqn:newEik2}, we can solve for the $g_k$ successively as functions vanishing in $\rho_0\leq 0$ and $\phi^{\pm}=\phi_1^{\pm}+\phi_2^{\pm}$ solves the required problem. 
\end{proof}

In addition, by two applications of \cite[Section 4.4]{MelTayl} (one for the real part and one for the imaginary), we have the following lemma.
\begin{lemma}
\label{lem:trans}
For $c, d\in S$ and $b_0\in \complex$, $B_1,B_2,F_1,F_2\in C^\infty$ there exist $g_0,\, g_1\in S$  in $\rho\leq 0$ solving \eqref{eqn:TransportGen}
Moreover, the equations \eqref{eqn:transEuc} can be solved in Taylor series at $\rho_0=0$ and $y=0$ and we can arrange that $g_1((0,x'),\xi)=cg_0+d$ and $g_0((0,0),0) =b_0.$.
\end{lemma}
\begin{proof}
We saw in \eqref{eqn:transNewForm} that \eqref{eqn:TransportGen} is equivalent to 
$$2\la ad_x\phi^{\pm},d_xg^{\pm}\ra +\la b,dg^{\pm}\ra +G^{\pm}g^{\pm}=F^\pm$$
where $g^{\pm}$, $G^{\pm}$, and $F^{\pm}$ are smooth in $x,\xi$ and $(-\rho_0)^{1/2}$. Hence, pulling back by $Y$, this lifts to 
$$2\la a d_x\Phi, dg\ra +\la b, dg\ra +Ga=F\quad \quad \text{on } P.$$
Then, reinterpreting this as an equation on each $\Lambda_{\eta'}$, $x$ can be used as coordinates on $\Lambda_{\eta'}$ and hence 
$$H_p=\partial_\xi p\partial_x=2\la ad\Phi,\partial_x\ra +\la b,\partial_x\ra $$
and hence $2\la ad\Phi,\cdot\ra +\la b,\cdot\ra $ is the vector field $H_p$. That is, our equation becomes 
\begin{equation} \label{eqn:transportReduce} H_pg+Gg-F.\end{equation}
We can reduce our problem to solving $H_pu=0$ by first solving 
$$H_pa_1+Ga_1=F\quad\quad a_1(m)=0$$
and 
$$H_p\alpha =G',\quad \alpha (m)=0$$
and writing 
$$a-a_1=\exp(-\alpha )u.$$
Then, the equation $H_pu=0$ just reduces to evenness of $u$ under the involution generated by projection from the manifold of bicharacteristics for $P$ to $P\cap Q$.

Our goal is to solve \eqref{eqn:transportReduce} with 
$$g_1=cg_0+d\quad \quad g_0(m)=b_0.$$
Let $\mc{I}_Q$ denote the involution on $P\cap Q$ coming from projection to $T^*B$. Then this amounts to 
$$[g]_O=c[g]_e+d,\text{ where } v_O=\frac{1}{2}(\mc{I}_Q^*v-v)\rho_0^{-1}\quad v_E=\frac{1}{2}(\mc{I}_Q^*v+v).$$
After reducing to $H_pu=0$, we have changed the boundary condition to 
$$[\exp(-\alpha)u]_O=c[\exp(-\alpha)u]_E+e'\quad u(m)=b_0$$
where $e'$ is some $\mc{I}_Q$ even function. Now, observe that 
$$[vw]_E=[v]_E[w]_E+\rho_0^2[v]_O[w]_O$$
So, we can write our boundary condition as 
$$u_O=c'u_E+f$$
where $c'$ and $f$ are $\mc{I}_Q$ even. Then, after applying $\kappa$ to reduce to normal form, we have by \cite[Proposition 2.8.2]{MelTayl} there exists such a function $u$. 

This solves the transport equations in $\rho_0\leq 0$. The extension to $\rho_0>0$ follows as in the proof of Lemma \ref{lem:EikSolve}.

\end{proof}

\subsection{Full Phase and Amplitude Functions for the Dirichlet Parametrices}
We now specialize to the case $P=-h^2\Delta-z$ and work in a neighborhood of the boundary $\partial\Omega$ of the form $O=(-a,a)\times U$ with coordinates $(y,x')$, $U$ an open set in $\partial\Omega$, and $y>0$ in $\Omega$.
Notice that in these coordinates,
$$-h^2\Delta=\la a(y,x')hD,hD\ra +h\la b(y,x'),hD\ra$$
and hence that $h\la b(y,x'),hD\ra$ term can be moved into the right hand side of the transport equations without difficulty.

By the results of Lemma \ref{lem:EikSolve} (or  \cite[Chapter 11]{TaylorPseud}, and \cite[Appendix A.II]{Gerard}), we have
\begin{lemma}
\label{lem:phase}
There exist $\theta_0,\rho_0\in C^\infty$ solving \eqref{eqn:eikEuc} for $P=-h^2\Delta-1$ for $\rho_0\leq 0$ and $((y,x'),\xi)$ near $((0,x_0'),\xi_0)$ and in Taylor series at $\rho_0=0$ and $y=0$.
Moreover,
\begin{gather}
\label{eqn:thetaNondeg}
d_x\partial_{\xi_j}\theta \text{ are linearly independent for }j=1\dots d-1\\
\frac{\partial \rho_0}{\partial y}<0,\,
\quad\quad \rho_0 =\alpha_0 \text{ on }y=0\nonumber
\end{gather}
where $\alpha_0:=\rho_0|_{y=0}=\xi_1$ and $\theta_0|_{y=0}$ parametrizes the reduction of the billiard ball map to that for the normal form \eqref{eqn:normalFormGlance} (i.e. $\kappa_\partial$).
\end{lemma}

Now that we have constructed phase functions for $z=1$, we will correct them to obtain solutions of \eqref{eqn:eikEuc} to $\O{}(h^\infty)$. To do this, let 
\begin{gather*} z= 1+i\mu,\,\quad\quad\theta=\theta_0+\sum_{n>0 }\theta_{n}\e(h)^n=:\theta_0+\theta',\\
\rho=\rho_0+\sum_{n> 0}\rho_{n}\e(h)^n=:\rho_0+\rho'\end{gather*} 
where $\theta_0$ and $\rho_0$ are the solutions found above. Then, 
\begin{align*}
2i\mu-\mu^2=&(2\la a d \theta_0,d \theta'\ra -2\rho_0\la a d\rho_0,d\rho'\ra -\rho'\la a d\rho_0,d\rho_0\ra +\la b,d\theta'\ra )\\
&
+(\la a d \theta',d\theta'\ra -2\rho'\la a d\rho_0,d\rho'\ra -\rho_0\la ad\rho',d\rho'\ra)
- \rho'\la a d \rho',d\rho'\ra \\
0=&(2\la a d \theta_0,d\rho'\ra+2\la ad\theta',d\rho_0\ra+\la b,d\rho'\ra ) +(2\la a d \theta',d \rho'\ra)
\end{align*}
where we have grouped terms according to homogeneity in $\e(h)$. Note that if $\Im z=\o{}(h)$, we have artificially introduced a perturbation of size $h$ to $\rho$ and $\theta$. 

Then, equating powers of $\e(h)$, and letting $$\theta_{<n}=\{\theta_{m}\,:\,m<n\}\,,\quad \rho_{<n}=\{\rho_{m}\,:\,m<n\},$$ we have that 
\begin{equation}
\label{eqn:phaseCorrect}
\begin{cases}
\begin{gathered}2\la a d\theta_0, d\theta_{n}\ra +2\rho_0\la ad \rho_0,d(-\rho_{n})\ra\\
+(-\rho_{n})\la ad\rho_0,d\rho_0\ra +\la b,d\theta_{n}\ra\end{gathered} &=F_{n}(\theta_{<n},\rho_{<n},\mu)\\
{}\\
2\la d\theta_{n},d\rho_0\ra -2\la d\theta_0,d(-\rho_{n})\ra-\la b,d(-\rho_{n})\ra & =G_{n}(\theta_{<n},\rho_{<n}, \mu)
\end{cases}.
\end{equation}
These equations are of the form \eqref{eqn:TransportGen} with $-\rho_{n}$ playing the role of $g_1$. Thus, appealing to Lemma \ref{lem:trans}, we can take $\rho_{1}((0,x'),\xi)=i.$
 For $n>1$, Lemma \ref{lem:trans}, implies that \eqref{eqn:phaseCorrect} can be solved with $\rho_{n}((0,x'),\xi)=0$. 
Putting this together, we have 

\begin{lemma}
\label{lem:phase2}
Let $\theta_0$ and $\rho_0$ be the functions guaranteed by Lemma \ref{lem:phase}. Then there exist $\theta,\rho\in S$ solving 
$$\begin{cases}
\la a d\theta,d\theta\ra -\rho\la a d\rho ,d\rho\ra +\la b,d\theta\ra =z+\O{}(h^\infty)\\
2\la d\theta, d\rho\ra +\la b,d\rho\ra =0
\end{cases}$$
in $\rho_0\leq 0$ and in Taylor series at $\rho_0=0$ and $y=0$.
Moreover,
$$\rho\sim \rho_0+\sum_{n>0}\rho_{n}\e(h)^n\quad\quad \theta\sim\theta_0+ \sum_{n>0}\theta_n\e(h)^n$$
with $\rho_{n},\theta_{n}\in S$, $\rho_0,\theta_0$ real valued, $\Im \theta_{1}\geq 0$ and $\rho|_{y=0}=\alpha:=\xi_1+i\e(h)$.
\end{lemma}

\begin{remark} In this way, we arrange 
$$\alpha(\xi')=\xi_1+i\e(h).$$
\end{remark}

Now, to solve for the amplitudes $g_0$ and $g_1$, we expand them as formal power series in $h^n.$
Then, the successive terms solve equations also of the form \eqref{eqn:TransportGen} (in particular, \eqref{eqn:transEuc}). Since the inhomogeneities do not appear in the equations for $g_i^{[n]}$, there are  solutions with boundary condition $g_0^{[0]}((0,x'),\xi)$ a real valued elliptic function and $g_1^{[0]}=0$.  Then, for $n> 0$, we have solutions with $g_1^{[n]}((0,x'),\xi)=0$.

\subsection{Semiclassical Fourier-Airy integral Operators}
Before proceeding, we give the necessary results on semiclassical Fourier-Airy integral operators following \cite[VIII.6 and X.2]{TaylorPseud} as well as \cite[Chapter 6]{MelTayl}. We denote $h^{-2/3}\alpha=\alpha_h$ and $h^{-2/3}\rho=\rho_h$. 

We make the following basic assumptions throughout this section. Let $\Omega\subset \re^d$ be strictly convex and $U$ be a neighborhood of $x_0\in \partial\Omega$. Suppose that $ch^M<\e(h)<Ch\log h^{-1}$ and let $\rho,\, \theta,\,\in C^\infty(U)$ and $g_0,\,g_1\in S^{\comp}_{\delta}(U).$ Suppose that $\theta_0\,,\,\rho_0\in C^\infty(U;\re)$. Suppose further that $\rho=\rho_0+\e(h)\rho'$, $\theta=\theta_0+\e(h)\theta'$ with $\rho|_{\partial\Omega}=:\alpha,$
\begin{equation}
\label{eqn:mainPhaseAssume}d_x\partial_\xi \theta_0\neq 0\,,\quad \partial_\nu\rho\leq a_0<0\text{ in case } \eqref{eqn:FAIO}\,,\quad \partial_{\nu}\rho \geq a_0>0\text{ in case }\eqref{eqn:FAIOIn}.\end{equation}
with $|\Im \rho'|>c\e(h)$ and $\theta',\, \rho' \in C^\infty(U;\complex),$
$$\theta'=\theta_1+\O{}(\e(h)),\quad \Im \theta_1((0,0),0)= 0.$$
Next, assume 
$\alpha:=\alpha_0(\xi)+\e(h)\alpha'(\xi)$ with  $\alpha_0\in C^\infty(\partial\Omega ;\re)$ and $\alpha'=i+\O{}(\e(h)).$
Then, assume that  $F\in \mc{E}'$ with
\begin{equation}\label{eqn:alphaCond} \MS(F)\subset T^*U\bigcap \left\{|\alpha_0|\leq \min\left[\gamma\left(\frac{h}{\e(h)}\right)^2,\gamma\right]\right\}\,.
\end{equation}

The fact that \eqref{eqn:FAIO} and \eqref{eqn:FAIOIn} are well defined follows from the fact that $g_0$ and $g_1$ have compact support and that $|\Im\rho'|>0$.

\begin{remark}
We could take $\alpha'=-i+\O{}(\e (h))$, but this would change the wavefront relations in Lemma \ref{lem:WFhAiry}. In particular, for $(\mc{A}_-\mc{A}i)^{-1}.$
\end{remark}

\subsection{Preliminary Estimates on Airy functions and multipliers}

We start by recalling some preliminary estimates and asymptotics for Airy functions. We have
\begin{equation}
\label{eqn:airyFormula}
A_-(z)=\Xi_-(z)e^{i2/3(-z)^{3/2}} \quad\quad |\text{Arg}(z)-\pi /3|>\delta 
\end{equation}
 where, letting $\omega:=e^{i\pi/3}$, $\Xi_-(z):=\Xi(z\omega^2)\in S^{-1/4}$, $\Xi$ has \cite[Section X.1]{TaylorPseud}
 $$\Xi(z)= z^{-1/4}\sum_{k=0}^\infty(-1)^ka_kz^{-3k/2}$$
 where $a_k>0$ and $a_0=(2\sqrt{\pi})^{-1}$. 
 and we take the branch of $z^{1/2}$ at Arg$(z)=\pi$ with $(1)^{1/2}=1$. We also write $A_+(z)=\overline{A_-(\bar{z})}$ for another solution to the Airy equation. The asymptotics for $\Xi(z)$ can be differentiated a finite number of times to obtain asymptotic expansions for $A_-^{(k)}(z)$. 

Next, recall
\begin{equation} Ai(z)=\Xi(z)e^{-2/3z^{3/2}}\,,\quad |\text{Arg}(z)-\pi |>\delta.\label{eqn:AiAsympPos}\end{equation}
Moreover, 
\begin{equation}
\label{eqn:linDepAiry}Ai(z)=\omega A_+(z)+\bar{\omega}A_-(z).
\end{equation}
So, using the asymptotics \eqref{eqn:airyFormula} and the analogous asymptotics for $A_+$, we have
\begin{equation}
\label{eqn:AiAsympNeg}
Ai(z)=\omega \Xi_+(z)e^{-2/3i(-z)^{3/2}}+\bar{\omega}\Xi_-( z)e^{2/3i(-z)^{3/2}}\,,\quad \quad |\text{Arg}(z)-\pi |<\delta 
\end{equation}
where $\Xi_+(z)=\Xi(z\bar{\omega}^2)$.

Define 
$$\phi_i(z):=\frac{Ai'(z)}{Ai(z)}\quad\quad\phi_-(z):=\frac{A_-'(z)}{A_-(z)}.$$
We will need the following lemma  (we follow the proof given in \cite[Lemma 3.1]{VodevConcave}). 
\begin{lemma}
\label{lem:airyQuotientEstimate}
Let $\phi_i$ be as above. Then there exists $\delta>0$ such that 
$$ |\phi_i(z)|\leq C\begin{cases}\la z\ra^{1/2}+|\Im z|^{-1}&|z|\geq \delta\,,\Re z<0 \\
\la z\ra^{1/2}&\text{otherwise} \end{cases}$$
and
$$|\phi_i(z)|^{-1}\leq C\begin{cases} \la z\ra^{-1/2}+|\Im z|^{-1}\la z\ra^{-1}&|z|\geq \delta\,,\Re z<0\\
\la z\ra^{-1/2}&\text{otherwise}\end{cases}.$$
\end{lemma}
\begin{proof}
Since $\phi_i$ is meromorphic and bounded above and below at $z=0$, there exists $\e_0>0$ such that for $|z|<z_0$, $0<c\leq |\phi_i|\leq C$. For $|\Arg(z)-\pi|>\delta$ and $|z|\gg 1$, the estimates follow from the asymptotics \eqref{eqn:AiAsympPos}. Thus, we need to consider the regions $\e_0<|z|<M$ and $|\Arg(z)-\pi |\leq \delta.$ 

First, we consider the region $\e_0<|z|<M$. 

Let $-\zeta_j\sim C_1j^{2/3}$ be the zeros of $Ai(z)$ and $-\zeta_j'\sim C_2 j^{2/3}$ be the zeros of $Ai'(z)$. Recall that both $\zeta_j$ and $\zeta_j'$ are positive and real for all $j$. Now, $Ai$ and $Ai'$ are entire of order $\frac{3}{2}$. Therefore, we can use the Hadamard factorization theorem to write 
$$Ai(z)=e^{C_2z+C_1}\prod_j \left(1+\frac{z}{\zeta_j}\right)e^{-\frac{z}{\zeta_j}},\quad \,\, Ai'(z)=e^{C_3z+C_4}\prod_j \left(1+\frac{z}{\zeta_j'}\right)e^{-\frac{z}{\zeta_j'}}$$
Hence taking the logarithmic derivative of $Ai$ and $Ai'$ respectively,
$$\phi_i(z)=C_2+\sum_j\frac{1}{z+\zeta_j}-\frac{1}{\zeta_j}\quad \quad z\phi_i^{-1}(z)=C_3+\sum_j\frac{1}{z+\zeta_j'}-\frac{1}{\zeta_j'}.$$
Since $\zeta_j$ are real and positive, 
$$|z+\zeta_j|^{-1}\leq \left. \begin{cases}|\Im z|^{-1}&\Re z< 0\\
C|z|^{-1}&\Re z\geq 0\end{cases}\right\}=:a(z)$$
where 
and $\zeta_j\geq 2|z|$, $|z+\zeta_j|^{-1}\leq 2|\zeta_j|^{-1}$. Thus, 
\begin{align*}|\phi_i(z)|&\leq |C_2|+\sum_{j=1}^{2|z|}(|z+\zeta_j|^{-1})+|z|\sum_{j=2|z|}^\infty |z+\zeta_j|^{-1}|\zeta_j|^{-1}\\
&\leq C(1+|z|(a(z)+1+\sum_j|\zeta_j|^{-2}))\leq Ca(z)
\end{align*}
since $\e_0<|z|<M$. By an identical argument, 
$$|z| |\phi_i|^{-1}\leq C a(z)$$
in this region.

Now, we consider the remaining region. Let $|z|\gg 1$ with $|\Arg(z)|<\delta$. First, using \eqref{eqn:linDepAiry}, we have that 
\begin{equation}\label{eqn:phiForm}\begin{split} \phi_i(-z)=\frac{A_+'(-z)}{A_+(-z)}\left(1 +\frac{A_-'(-z)}{\omega^2A_+'(-z)}\right)\left(1+\frac{A_-(-z)}{\omega^2 A_+(-z)}\right)^{-1}\\
\phi_i^{-1}(-z)=\frac{A_+(-z)}{A_+'(-z)}\left( 1+\frac{A_-(-z)}{\omega^2A_+(-z)}\right)\left(1+\frac{A_-'(-z)}{\omega^2A_+'(-z)}\right)^{-1}
\end{split}.
\end{equation}
Thus, to estimate $\phi_i$ and $\phi_i^{-1}$, we proceed by obtaining estimates on $A_+$ and $A_-$. 
Defining $\zeta=\frac{2}{3}z^{3/2}$, we have 
$$\Im \zeta=\Im z(\Re z)^{1/2}(1+\O{}(\delta))\,,\quad \quad |\Im \zeta|\geq C_\delta |\Im z||z|^{1/2}.$$
Now, let 
\begin{gather*} B_{\pm}(z):=z^{1/4}e^{\mp i\pi/12}\Xi(e^{\pm i\pi/3}z)\\
D_{\pm}(z):=\pm iz^{-1/4}e^{\mp\pi i/12}\left(\mp iz^{1/2}\Xi(e^{\pm\pi i/3}z)-\Xi'(e^{\pm \pi i/3})\right)
\end{gather*}
where $\Xi$ is as in \eqref{eqn:airyFormula} so that 
\begin{equation}\label{eqn:AiryB} A_{\pm}(-z)=z^{-1/4}e^{\pm i\pi/12}B_{\pm}(z)e^{\pm i\zeta}\,,\quad\quad A_{\pm}'(-z)=\mp iz^{1/4}e^{\pm i\pi/12}D_{\pm}(z)e^{\mp \zeta}
.\end{equation}
Then, 
\begin{gather*}B_{\pm}(z)=b_0\pm ib_1\zeta ^{-1}+\O{}(\zeta ^{-2})\,,\quad -zB_{\pm}'(z)=\pm\frac{3ib_1}{2}\zeta ^{-1}+\O{}(\zeta ^{-2})\\
D_{\pm}(z)=d_0\pm id_1\zeta ^{-1}+\O{}(\zeta ^{-2})\,,\quad -zD_{\pm}'(z)=\pm\frac{3id_1}{2}\zeta ^{-1}+\O{}(\zeta ^{-2})
\end{gather*}
where $b_i>0$, $d_i>0$ and 
\begin{gather*} \pm \Im \left(B_{\pm}(z)\overline{B_{\pm}'(z)}\right)=\frac{3b_0b_1}{2}|z|^{-5/2}(1+\O{}(\delta)+\O{}(|z|^{-3/2}))>0\\
\pm \Im \left(D_{\pm}(z)\overline{D_{\pm}'(z)}\right)=\frac{3d_0d_1}{2}|z|^{-5/2}(1+\O{}(\delta)+\O{}(|z|^{-3/2}))>0.
\end{gather*}

We first seek to show that $\pm |A_-(-z)|\leq \pm |A_+(-z)|$ in $\pm \Im z\geq 0$. To this end, define 
$$f_a(\tau)=|B_+(a+i\tau)|^2-|B_-(a+i\tau)|^2.$$
Then, 
$$f_a'(\tau)=2\Im \left(B_+(a+i\tau)\overline{B_+'(a+i\tau)}-B_-(a+i\tau)\overline{B_-'(a_i\tau)}\right)>0.$$
So taking $a=\Re z$ and using the fact that $|A_+(\Re z)|=|A_-(\Re z)|$, we have $f_{\Re z}(0)=0$ and $f_{\Re z}'(\tau)>0$ for $0\leq \tau<\delta \Re z $ and $\Re z\gg 1$. This implies
\begin{equation} 
\label{eqn:compareApm}
\pm |B_+(z)|\geq \pm |B_-(z)|\quad \quad \pm \Im z\geq 0
\end{equation}

An identical analysis with the function 
$$g_a(\tau)=|D_+(a+i\tau)|^2-|D_-(a+i\tau)|^2$$
gives 
\begin{equation}
\label{eqn:compareAppm}
\pm |D_+(z)|\geq \pm |D_-(z)|\quad \quad \pm \Im z\geq 0
\end{equation}

We now restrict our attention to $\Im z>0$ and hence $\Im \zeta>0$ since the other region is similar. By \eqref{eqn:AiryB} and \eqref{eqn:compareApm}
\begin{equation} \label{eqn:compareApm2}
\begin{gathered}
\left|\frac{A_-(-z)}{A_+(-z)}\right|=e^{-2\Im \zeta}\left|\frac{B_-(z)}{B_+(z)}\right|\leq e^{-2\Im \zeta}\,,\\
 \left|\frac{A_-'(-z)}{A_+'(-z)}\right|=e^{-2\Im \zeta}\left|\frac{D_-(z)}{D_+(z)}\right|\leq e^{-2\Im \zeta}.
 \end{gathered}
\end{equation}
Now, the asymptotics \eqref{eqn:airyFormula} imply that 
\begin{equation}\label{eqn:Abounds} 0<c\la z\ra^{1/2}\leq \left|\frac{A_{\pm}'(-z)}{A_+(-z)}\right|\leq C\la z\ra^{1/2}\,\quad\quad
0<c\leq \left|\frac{A_{\pm}(-z)}{A_+(-z)}\right|\leq C.
\end{equation}
So, using $|\Im \zeta|\geq C_\delta |\Im z||z|^{1/2}$ together with  using \eqref{eqn:compareApm2} and \eqref{eqn:Abounds} in \eqref{eqn:phiForm} gives 
\begin{align*} |\phi_i(-z)|&\leq \frac{C|z|^{1/2}}{1-e^{-2\Im \zeta}}\leq \frac{C|z|^{1/2}}{\min(1,2\Im \zeta)}\leq C|z|^{1/2}+C|\Im z|^{-1}\\
|\phi_i(-z)|^{-1}&\leq C\frac{|z|^{-1/2}}{1-e^{-2\Im \zeta}}\leq C\frac{|z|^{-1/2}}{\min(1,2\Im\zeta)}\\
&\leq \la z\ra^{-1/2}(1+|\Im z|^{-1}\la z\ra^{-1/2}) .\end{align*}

\end{proof}

The following bounds on products of Airy functions will be useful in our construction of $H_g$ and $H_d$
\begin{lemma}
\label{lem:airyBounds}
Let $\alpha$ be as in \eqref{eqn:alphaCond} and $\alpha_h=h^{-2/3}\alpha$. Then for $\gamma$ small enough and 
$$|\alpha|\leq \gamma (h\e(h)^{-1})^2$$
we have for $\Re \alpha_h\leq -\delta$ 
\begin{gather*}
Ch^{-2/3}\e(h)\leq C|\Im \alpha_h|\leq |Ai(\alpha_h)A_-(\alpha_h)|\leq C\\
c\la \alpha_h\ra^{1/2}\leq |Ai'(\alpha_h)A_-'(\alpha_h)|\leq C\la \alpha_h\ra^{1/2}\\
c(|\Im \alpha_h|^{-1}\la \alpha_h\ra^{-1}+\la \alpha_h\ra^{-1/2})^{-1}\leq |\phi_i(\alpha_h)|\leq |\Im \alpha_h|^{-1}\leq Ch^{2/3}\e(h)^{-1}\\
c\la \alpha_h\ra^{1/2}\leq |\phi_-|\leq C\la \alpha_h\ra^{1/2}
\end{gather*}
and for $\Re \alpha_h\geq -\delta$
\begin{gather*} 
Ch^{1/3}\leq |Ai(\alpha_h)A_-(\alpha_h)|\leq C\\
c(|\Im \alpha_h|^{-1}\la \alpha_h\ra^{-1}+\la \alpha_h\ra^{-1/2})^{-1}\leq |Ai'(\alpha_h)A_-'(\alpha_h)|\leq C\la \alpha_h\ra^{1/2} \\
c\la \alpha_h\ra^{1/2}\leq|\phi_i(\alpha_h)| +|\phi_-(\alpha_h)|\leq C\la \alpha_h\ra^{1/2}.
\end{gather*} 
\end{lemma}
\begin{proof}
First observe that 
\begin{equation} \label{eqn:alphaImEst}ch^{-2/3}\e(h)< |\Im \alpha_h| = \O{}(h^{-2/3}\e(h))\ll \delta\,\end{equation}
thus, either $|\alpha_h|<\delta.$ or $|\Im \alpha_h|\ll |\alpha_h|$.

The upper bounds for $Ai(\alpha_h)A_-(\alpha_h)$ and $Ai'(\alpha_h)A_-(\alpha_h)$ follow directly from the asymptotics \eqref{eqn:airyFormula}, \eqref{eqn:AiAsympPos}, and \eqref{eqn:AiAsympNeg} together with the analyticity of these functions.

In order to estimate $(A_-Ai)^{-1}$, we use the Wronskian to write
\begin{equation}\label{eqn:wronskianEst}\frac{A_-'(z)}{A_-(z)}-\frac{Ai'(z)}{Ai(z)}=\frac{W(Ai,A_-)(z)}{Ai(z)A_-(z)}=\frac{e^{-\pi i/6}}{2\pi Ai(z)A_-(z)}.
\end{equation}
Thus, to estimate $|A_-Ai|^{-1}$ it is enough to estimate $\phi_i$ and 
$$ \phi_-:=\frac{A_-'}{A_-}.$$  

Similarly, to estimate $(A_-'Ai')^{-1}$, we use the Wronskain to write
\begin{equation} \label{eqn:wronskianEstPrime}
\frac{A_-(z)}{A_-'(z)}-\frac{Ai(z)}{Ai'(z)}=-\frac{W(Ai,A_-)(z)}{Ai(z)A_-(z)}=\frac{e^{5\pi i/6}}{2\pi Ai'(z)A_-'(z)}.
\end{equation}

By Lemma \ref{lem:airyQuotientEstimate}, there exists $\delta>0$ such that 
\begin{gather*}  |\phi_i(z)|\leq C\begin{cases}|\Im z|^{-1}+\la z\ra^{1/2}&|z|\geq \delta\,,\Re z<0\\
\la z\ra^{1/2}&\text{otherwise} \end{cases}\\
|\phi_i(z)|^{-1}\leq C\begin{cases} \la z\ra^{-1/2}+|\Im z|^{-1}\la z\ra^{-1}&|z|\geq \delta\,,\Re z<0\\
\la z\ra^{-1/2}&\text{otherwise}\end{cases}.
\end{gather*} 
and, since
$$\phi_-(z)=e^{2\pi i/3}\phi_i(e^{2\pi i/3}z)$$
we also have 
\begin{gather*} |\phi_-(z)|\leq C\begin{cases}(|\Im e^{2\pi i/3}z|+\la z\ra^{1/2})&|z|\geq \delta\,\\
1&|z|\leq \delta \end{cases}\\
|\phi_-(z)|^{-1}\leq C\begin{cases} \la z\ra^{-1/2}+|\Im e^{2\pi i/3}z|^{-1}\la z\ra^{-1}&|z|\geq \delta\\
\la z\ra^{-1/2}&\text{otherwise}\end{cases}.
\end{gather*}

Now, by \eqref{eqn:alphaImEst}, either $|\alpha_h|\leq \delta$ or $|\Im e^{2\pi i/3}z|\geq \delta$, so we can estimate 
\begin{gather*}
c\la \alpha_h\ra^{1/2}\leq |\phi_-(\alpha_h)|\leq C\la \alpha_h\ra^{1/2}\\
|\phi_i|(\alpha_h)\leq \begin{cases}|\Im \alpha_h|^{-1}+\la \alpha_h\ra^{1/2}&\Re \alpha_h<0\,,|\alpha_h|\geq \delta\\
\la \alpha_h\ra^{1/2}&\text{ otherwise}\end{cases}\\
|\phi_i^{-1}|(\alpha_h)\leq \begin{cases}|\Im \alpha_h|^{-1}\la \alpha_h\ra^{-1}+\la \alpha_h\ra^{-1/2}&\Re \alpha_h<0\,,|\alpha_h|\geq \delta\\
\la \alpha_h\ra^{-1/2}&\text{ otherwise}\end{cases}.
\end{gather*}

Next, we have $$|\alpha|\leq \gamma\left(\frac{h}{\e(h)}\right)^2$$
and hence
\begin{align*}
 (1+|\alpha_h|)^{-1/2}&\geq\left\la h^{-2/3}\left(\gamma\frac{h^2}{\e(h)^2}+C\e(h)\right)\right\ra^{-1/2}\\
 &\geq \gamma^{-1/2}h^{-2/3}\e(h)\geq |\Im \alpha_h|\end{align*}
provided that $\gamma$ is small enough.  This implies $$\la \alpha_h\ra^{1/2}\leq |\Im \alpha_h|^{-1}$$ and hence 
gives the desired estimates
\end{proof}

Define the Airy multipliers:
\begin{gather*}(\mc{A}_-\mc{A}i)^{-1}F:=(2\pi h)^{-d+1}\int [Ai(\alpha_h)A_-(\alpha_h)]^{-1}e^{i\la x,\xi'\ra/h}\mc{F}_h{F}(\xi)d\xi,\\
(\mc{A}_-\mc{A}i)F:=(2\pi h)^{-d+1}\int Ai(\alpha_h)A_-(\alpha_h)e^{i\la x,\xi'\ra/h}\mc{F}_h{F}(\xi)d\xi\\
(\Phi_i)F:=(2\pi h)^{-d+1}\int \phi_i(\alpha_h)e^{i\la x,\xi'\ra/h}\mc{F}_h{F}(\xi)d\xi\\
(\Phi_i^{-1})F:=(2\pi h)^{-d+1}\int \phi_i^{-1}(\alpha_h)e^{i\la x,\xi'\ra/h}\mc{F}_h{F}(\xi)d\xi\\
(\Phi_-)F:=(2\pi h)^{-d+1}\int \phi_-(\alpha_h)e^{i\la x,\xi'\ra/h}\mc{F}_h{F}(\xi)d\xi.\\
(\Phi_-^{-1})F:=(2\pi h)^{-d+1}\int \phi_-^{-1}(\alpha_h)e^{i\la x,\xi'\ra/h}\mc{F}_h{F}(\xi)d\xi..\end{gather*}
Then the following estimates follow from Lemma \ref{lem:airyBounds}. (see also \cite[Proposition 5.3.10]{MelTayl})
\begin{lemma}
\label{lem:airyOpEst}
\begin{gather*}(\mc{A}_-\mc{A}i)^{-1}=\O{H_h^s\to H_h^s}\left(h^{-1/3}\right),\quad \mc{A}_-\mc{A}i=\O{H_h^s\to H_h^s}(1)\\
(\mc{A}'_-\mc{A}i')^{-1}=\O{H_h^s\to H_h^s}\left(h^{2/3}\e(h)^{-1}\right),\quad \mc{A}'_-\mc{A}i'=\O{H_h^s\to H_h^s}(h^{-1/3})\\
\Phi_i=\O{H_h^s\to H_h^s}(h^{-1/3})\quad \quad \Phi_i^{-1}=\O{H_h^s\to H_h^s}(h^{2/3}\e(h)^{-1})\\
\Phi_-=\O{H_h^s\to H_h^s}(h^{-1/3})\quad \quad \Phi_-^{-1}=\O{H_h^s\to H_h^s}(1).
\end{gather*}
\end{lemma}
\begin{proof}
This follows from the estimates in Lemma \ref{lem:airyBounds}.
\end{proof}

\subsection{Estimates for Fourier-Airy Integral Operators}

\subsubsection{Estimates for \eqref{eqn:FAIOIn} type Fourier Airy Integral operators}
To analyze the action of \eqref{eqn:FAIOIn}, we follow the analysis given in \cite[Chapter 6]{MelTayl}. We work in a neighborhood of the boundary $\partial\Omega$ of the form $O=[0,a)\times U$ with coordinates $(y,x')$ and define the symbol classes 
\begin{defin}
We say $p(y,x',\xi;h)\in a(h)S_{\rho,\delta,\nu}$ if 
\m |D_y^kD_{x'}^\beta (hD_{\xi})^\alpha p(y,x',\xi;h)|\leq a(h)h^{\rho |\alpha|-\delta|\beta|-\nu k}.\m
\end{defin}

Write $B_2F:=B_3 \composed (\mc{A}_i\mc{A}_-)^{-1}F$ where 
\begin{equation}
\label{eqn:tempInside}
B_3F:=(2\pi h)^{-d+1}\int [g_0Ai(\rho_h)+ih^{1/3}g_1Ai'(\rho_h)]A_-(\alpha_h)e^{i\theta/h}\mc{F}_h{F}(\xi)d\xi.
\end{equation}
Then all that remains is to analyze $B_3$.

To analyze $B_3$, we break it into several pieces that can be handled using the theory of Fourier integral operators with singular phase. Let $p_1,\,p_2,\,p_3$ have $\supp p_1\subset [C,\infty)$, $\supp p_2\subset (-2C,2C)$, $\supp p_3\subset (-\infty,-C]$ with $p_1+p_2+p_3=1$ and let $q_1=1-p_3$ where $C\gg 1$ will be chosen later. 

We first examine the case where $\Re \alpha_h>-2C$. 
\begin{lemma}
\label{lem:FAIOInEst1}
\label{lem:AiSymbol1}
\begin{gather*} y^jAi(\rho_h)A_-(\alpha_h)q_1(\Re\alpha_h)\in h^{2/3j}e^{C\e(h)/h}S_{1/3,2/3,1}\\
y^jAi'(\rho_h)A_-(\alpha_h)q_1(\Re\alpha_h)\in e^{C\e(h)/h}h^{-1/6+2/3j}S_{1/3,2/3,1}\end{gather*}
for $j\geq 0$.
\end{lemma}
\begin{proof}
We first consider the term involving $p_1$. By \eqref{eqn:airyFormula} We have that 
\begin{gather*}A_-(\alpha_h)=\Xi_-(\alpha_h)e^{2/3\alpha^{3/2}/h} \text{ if }\Re \alpha>0\,,\\Ai(\rho_h)=\Xi(\rho_h)e^{-(2/3)\rho^{3/2}/h}\text{ if } \Re \rho>0.
\end{gather*}
Thus, since $\rho_0\geq \alpha_0+cy$,
$$Ai(\rho_h)A_-(\alpha_h)p_1(\Re \alpha_h)=p_1(\Re\alpha_h)\Xi_-(\alpha_h)\Xi(\rho_h)e^{-(2/3)(\rho^{3/2}-\alpha^{3/2})/h}.$$
Write 
\m p_1(\Re\alpha_h)=\chi_1^2(\alpha_h)\chi_2^2(\rho_h)\,\,\m
where $\chi_1\,,\,\chi_2$ are supported in $\Re s\geq 1/4$ and equal to 1 for $\Re s\geq 2.$ This is possible since $\alpha\leq \rho-Cy +\O{}(h^{-2/3}\e(h)).$ It suffices to show that $ \chi_1(\alpha_h)\Xi_-(\alpha_h)\in S_{1/3,0}\,,$ $ \chi_2(\rho_h)\Xi(\rho_h)\in S_{1/3,2/3,2/3},$ and 
\begin{equation}
\label{eqn:expSymbol}
\chi_1(\Re\alpha_h)\chi_2(\Re\rho_h)e^{-(2/3)(\rho^{3/2}-\alpha^{3/2})/h}\in e^{C\e(h)/h}S_{1/3,2/3,1}.
\end{equation}
The first two estimates follow from elementary estimates on $\Xi$. 

To prove \eqref{eqn:expSymbol}, we apply the chain rule:
\begin{multline*}
D_y^kD_{x'}^\beta D_{\xi}^\gamma e^{-\frac{2}{3h}(\rho^{\frac{3}{2}}-\alpha^{\frac{3}{2}})}=\sum CD_y^{k_1}D_{x'}^{\beta_1} D_{\xi}^{\gamma_1}\left(\frac{\rho^{\frac{3}{2}}-\alpha^{\frac{3}{2}}}{h}\right)\dots \\D_y^{k_\mu}D_{x'}^{\beta_\mu}D_{\xi}^{\gamma_\mu}\left(\frac{\rho^{\frac{3}{2}}-\alpha^{\frac{3}{2}}}{h}\right)e^{-\frac{2}{3h}(\rho^{\frac{3}{2}}-\alpha^{\frac{3}{2}})}\end{multline*}
where the sum is over 
$\sum\gamma_i=\gamma,\,$ $\sum\beta_i =\beta,\,$ $\sum k_i=k.\,$
Note that \eqref{eqn:mainPhaseAssume} implies that for $y$ small on $\supp \chi_1(\alpha_h)\chi_2(\rho_h)$ 
$$\Re (\rho^{3/2}-\alpha^{3/2})\geq Cy^{3/2},\quad \Re\alpha>0.$$

Hence,
\begin{gather*}|D_y^{k}D_{x'}^\beta D_{\xi}^\gamma (\rho^{\frac{3}{2}}-\alpha^{\frac{3}{2}})e^{-\frac{2}{3h}(\rho^{\frac{3}{2}}-\alpha^{\frac{3}{2}})}|\leq e^{-c\frac{y^{3/2}}{h}}e^{\frac{C\e(h)}{h}}  C|\rho|^{\frac{3}{2}-k-|\beta|-|\gamma|},\quad k> 0,\\
|D_{x'}^\beta D_{\xi}^\gamma (\rho^{\frac{3}{2}}-\alpha^{\frac{3}{2}})|e^{-\frac{2}{3h}(\rho^{\frac{3}{2}}-\alpha^{\frac{3}{2}})}\leq Ce^{-c\frac{y^{3/2}}{h}}e^{\frac{C\e(h)}{h}}| \rho|^{\frac{3}{2}-|\beta|-|\gamma|}\quad \beta>0,\\
| D_{\xi}^\gamma (\rho^{\frac{3}{2}}-\alpha^{\frac{3}{2}})|e^{-\frac{2}{3h}(\rho^{\frac{3}{2}}-\alpha^{\frac{3}{2}})}\leq e^{-c\frac{y^{3/2}}{h}}\max\left(|\rho|^{\frac{3}{2}-|\gamma|},|\alpha|^{\frac{3}{2}-|\gamma|}\right).
\end{gather*}
But, on $\supp \chi_1(\Re\alpha_h)\chi_2(\Re \rho_h)$, $ Ch^{2/3}\leq \alpha \leq \rho.$
Thus, 
$$|D_y^{k}D_{x'}^\beta D_{\xi}^\gamma (\rho^{3/2}-\alpha^{3/2})e^{-(2/3)(\rho^{3/2}-\alpha^{3/2})}|y^j\leq Ce^{\e(h)/h}h^{-2/3(k+|\beta|+|\gamma|-3/2-j)}.$$

Now, for the term involving $p_2$, we have
$$ Ai(\rho_h)p_2(\Re \alpha_h),\quad  h^{1/3}Ai(\rho_h)p_2(\Re \alpha_h)\in e^{C\e(h)/h}S_{1/3,2/3,1}.$$
To see this observe that on $p_2(\Re \alpha_h)p_2(\Re \rho_h)$ we have $|\rho_0|,|\alpha_0|\leq Ch^{2/3}$ and hence the main term in the exponential phase is bounded independently of $h$. Moreover, since $\Re \rho_0\geq \alpha_0+Cy$, $|y|\leq h^{2/3}$ so the second statement follows.  On $p_2(\Re \alpha_h)p_1(\Re \rho_h)$, we estimate as above.

The estimate for terms involving $Ai'$ follows from the fact that 
\m Ai'(z)=\tilde{\Xi}(z)e^{-(2/3)z^{3/2}}\,\,\m
where $\tilde{\Xi}=\O{}(z^{1/4}).$
This completes the proof of the lemma.
\end{proof}

Next, we analyze the case where $\alpha_h<-C$. Write 
$$B_3^<F =(2\pi h)^{-d+1}\int [g_0Ai(\rho_h)+ih^{1/3}g_1Ai'(\rho_h)]A_-(\alpha_h)p_3(\alpha_h)e^{i\theta/h}\mc{F}_h{F}d\xi.$$

We have similar to \cite[Section 6.3]{MelTayl}
\begin{lemma}
\label{lem:singPhaseIn}
The operator defined by 
$$\mc{A}^<_-(F):=(2\pi h)^{-d+1}\int A_-(\alpha_h)p_3(\alpha_h)e^{\frac{i\la x,\xi'\ra}{h}}\mc{F}_h{F}d\xi'$$ 
is a Fourier integral operator with singular phase.
\end{lemma}

\begin{remark} For a treatment of semiclassical Fourier integral operators with singular phase see Section \ref{sec:singFIO}. 
\end{remark}

Let
$$DG=\int [g_0Ai(\rho_h)+ih^{1/3}g_1Ai'(\rho_h)]e^{i\theta/h}\mc{F}_h{G}(\xi)d\xi$$
where $G\in \mc{E}'$. Then 
$B_3^<=D\composed \mc{A}_-^<.$

Hence, we only need to analyze $D$. We decompose $D$ using $p_3(\Re \rho_h)$ and $q_1(\Re(\rho_h))$ and write the resulting operators 
$D:=D_1+D_2.$

Then, using the same analysis as in Lemma \ref{lem:AiSymbol1} we have
\begin{lemma}
\label{lem:AiyEst}
For $j\geq 0$,
\begin{gather*} \rho_0^jAi(\rho_h)q_1(\Re\rho_h)\in h^{2/3j}e^{C\e(h)/h}S_{1/3,2/3,1}\,,\\
\rho_0^jAi'(\rho_h)q_1(\Re\rho_h)\in h^{-1/6+2/3j}e^{C\e(h)/h}S_{1/3,2/3,1}.
\end{gather*}
\end{lemma}

Finally, 

\begin{lemma}
\label{lem:rhoOpIn}
$$(2\pi h)^{-d+1} \int [g_0Ai(\rho_h)+ih^{1/3}g_1Ai'(\rho_h)]e^{i\theta/h}\mc{F}_h{G}(\xi)p_3(\rho_h)d\xi=B^++B^-$$
with 
$$B^{\pm}=\begin{aligned}\omega^{\mp}(2\pi h)^{-d+1}\int &[g_0\Xi_{\pm}(\rho_h)+ih^{1/3}g_1\tilde\Xi_{\pm}(\rho_h)]\\
&\quad\quad e^{i[\theta\mp(2/3)(-\rho)^{3/2}]/h}p_3(\rho_h)\mc{F}_h{G}(\xi)d\xi\end{aligned}$$
where 
$\Xi_{\pm}(\rho_h)\in S_{1/3,2/3,2/3},$ $\tilde{\Xi}_{\pm}\in h^{-1/6}S_{1/3,2/3,2/3}.$
\end{lemma}
\begin{proof}
By \eqref{eqn:AiAsympNeg} we have
$$Ai(\rho_h)=\omega\Xi_+(\rho_h)e^{-(2i/3)(-\rho)^{3/2}/h}+\bar{\omega} \Xi_-(\rho_h)e^{(2i/3)(-\rho)^{3/2}/h}, \quad \Re \rho<0.$$
Similarly for $Ai'$. Thus, the lemma follows from symbol estimates on $\Xi_\pm$ and $\tilde{\Xi}_{\pm}.$
\end{proof}

\subsubsection{Estimates for \eqref{eqn:FAIO} type Fourier Airy Integral operators}

The analysis of \eqref{eqn:FAIO} is similar to that of \eqref{eqn:FAIOIn}. This time, we decompose $B_1$ into $\rho_h<-C$ and $\rho_h>-2C$. We have 
\begin{lemma}
\label{lem:FAIOEst1}
For $j\geq 0$,
\begin{gather*} y^jA_{-}(\rho_h)A_-(\alpha_h)^{-1}q_1(\Re\rho_h)\in e^{C\e(h)/h}h^{-1/6}h^{2/3j}S_{1/3,2/3,1}\,,\\
y^jA_-'(\rho_h)A_-(\alpha_h)^{-1}q_1(\Re \rho_h)\in e^{C\e(h)/h}h^{-1/3}h^{2/3j}S_{1/2,3/2,1}\,.\end{gather*}
\end{lemma}
\begin{proof}
Since $\rho_0\leq \alpha_0-Cy$ and $|\Re \rho|\leq ch^{2/3}$ on $\supp p_2(\Re\rho_h)$, we may analyze terms involving only $p_1$ instead of $q_1$.

By \eqref{eqn:airyFormula}, we have
 \begin{equation*}\label{eqn:outsideEstimate}\frac{A_-(\rho_h)}{A_-(\alpha_h)}=\frac{\Xi_-(\rho_h)}{\Xi_-(\alpha_h)}e^{2/3(\rho^{3/2}-\alpha^{3/2})/h}=\frac{\Xi_-(\rho_h)}{\Xi_-(\alpha_h)}e^{2/3(\rho_0^{3/2}-\alpha_0^{3/2})/h+\O{}(\e(h)/h)}.\end{equation*}
We have that $\rho_0\leq \alpha_0 -cy.$ Therefore, the estimates follow as in Lemma \ref{lem:FAIOInEst1}
\end{proof}

We have 
\begin{lemma}
\label{lem:singPhaseOut}
\begin{equation*}(\mc{A}^<_-)^{-1}F:=(2\pi h)^{-d+1}\int(A_-(\alpha_h))^{-1}p_3(\Re \alpha_h)\mc{F}_h{F}d\xi\end{equation*}
is a Fourier integral operator with singular phase.
\end{lemma}
Moreover, $(A_-(\alpha_h))^{-1}$ is bounded on $\supp q_1.$ 
Then, similar to above, we have 
\begin{lemma}
\label{lem:rhoOpOut}
$$(2\pi h)^{-d+1}\int [g_0A_-(\rho_h)+ih^{1/3}g_1A_-'(\rho_h)]e^{i\theta/h}\mc{F}_h{G}(\xi)p_3(\rho_h)d\xi=B^-$$
with 
$$B^{-}=\begin{aligned} \omega (2\pi h)^{-d+1}\int&[g_0\Xi_{-}(\rho_h)+ih^{1/3}g_1\tilde\Xi_{-}(\rho_h)]\\
&\quad\quad e^{i[\theta+(2/3)(-\rho)^{3/2}]/h}p_3(\rho_h)\mc{F}_h{G}(\xi)d\xi\end{aligned}$$
where 
$\Xi_{-}(\rho_h)\in S_{1/3,2/3,2/3},$ $\tilde{\Xi}_{-}\in h^{-1/6}S_{1/3,2/3,2/3}.$
\end{lemma}

Together with Lemma \ref{lem:FAIOEst1} and the fact that $A_-(\alpha_h)^{-1}$ is bounded on $\supp q_1(\Re \alpha_h)$, this shows that on $\supp p_3(\Re \rho_h)$, \eqref{eqn:FAIO} is a Fourier integral operator with singular phase.


\subsection{Verification of the properties \eqref{eqn:parametrixCondEuc}}
We now prove that using the phase and amplitudes constructed in the previous section that \eqref{eqn:parametrixCondEuc} is satisfied. First, we construct $F$ so that the boundary conditions are satisfied. We have that $g_1|_{\partial \Omega}=0$ and $\rho|_{\partial \Omega}=\alpha$. Hence, restricting \eqref{eqn:FAIO} or \eqref{eqn:FAIOIn} to $\partial \Omega$ gives
$$BF|_{\partial \Omega}=(2\pi h)^{-d+1}\int ge^{i\theta_b/h}\mc{F}_h{F}(\xi)d\xi$$
where $\theta_b=\theta|_{\partial \Omega}$ and $g=g_0|_{\partial \Omega}$. Now, $d_x\partial_{\xi_j}\theta_0$ are linearly independent and hence $\theta_0$ is a phase function. Fix $\delta>\delta_1>0$. Then, since $\e(h)=\O{}(h\log h^{-1})$, $e^{\frac{i}{h}\e(h)\theta'}\in S_\delta$, and shrinking the neighborhood on which we work if necessary
$$\frac{\sup|e^{\frac{i}{h}\e(h)\theta'}|}{\inf|e^{\frac{i}{h}\e(h)\theta'}|}\leq Ch^{-\delta_1}.$$ 
Thus, $J:=B|_{\partial\Omega}$ is a semiclassical Fourier integral operator that is invertible by the symbol calculus of FIOs. Hence, we just need to take $F=J^{-1}f$ to obtain the appropriate boundary conditions where $J^{-1}$ is a microlocal parametrix for $J$. Thus, we let $H_d=B_1J^{-1}$ and $H_g=B_2J^{-1}.$ We need to verify that if 
$$\MS(f)\subset \{||\xi'|_g-1|<\eta(h)\ll 1\},$$ then $$\MS(J^{-1}f)\subset\{|\xi_1|<C\eta(h)\},$$
but this follows from the fact that $\theta$ parametrizes the reduction of $\partial\Omega$ and $|\xi|^2=1$ to the normal form \eqref{eqn:normalFormGlance} combined with the wavefront set bound \eqref{e:lag-wf}.

After a change of variables near $x_0$, we may assume that locally $\Omega_1=\{y<0\}$ and $\Omega_2=\{y>0\}.$ with $x=(y,x').$

\subsubsection{Diffractive points}
Now, we have that 
$$(-h^2\Delta-z^2)B_1F=(2\pi h)^{-d+1}\int \left[a\frac{A(\rho_h)}{A(\alpha_h)}+b\frac{A'(\rho_h)}{A(\alpha_h)}\right]e^{i\theta/h}$$
where $a\sim \sum a_{j,m}h^j\e(h)^m$ and $b\sim \sum b_{j,m}h^j\e(h)^m$ such that 
\begin{equation}
\label{eqn:ampEst}\begin{gathered}a_{j,m},b_{j,m}=0\quad \text{ for }\rho_0\leq 0,\\
a_{j,m},b_{j,m}=\O{}(y^n),\quad \text{ for any }(x,\xi) \text{ and all }n>0.\end{gathered}\end{equation}
Thus, for diffractive points, by Lemma \ref{lem:FAIOEst1}
\m (-h^2\Delta-z^2)B_1F=\O{C^\infty}(h^\infty)\,\,\m
as desired.

\subsubsection{Gliding Points}
For gliding points, the verification is more complicated because $\rho_0$ may become positive away from the boundary. The case when $\alpha_0>0$  are taken care of by Lemma \ref{lem:FAIOInEst1} and the estimates \eqref{eqn:ampEst} . Suppose that $\alpha_0\leq 0$, but $\rho_0=0$ at $y_1$. Then, since the eikonal and transport equations can be solved in Taylor series at $\rho_0=0$ and $\rho_0\geq \alpha_0+Cy$, we have that $a_{j,m},b_{j,m}=c_{j,m,n}h^j\rho_0^n$, but by Lemma \ref{lem:AiyEst}, for $\alpha_0\leq 0$ and $\rho_0\geq 0$, such an integrand is $\O{}(h^\infty)$ as desired. 
Hence, we also have 
\m (-h^2\Delta-z^2)B_2F=\O{ C^\infty}(h^\infty)\,\,\m
in the gliding case.

\section{Microlocal description of $H_d$, $H_g$ and the Airy multipliers}

In \ref{sec:glanceTraj}, we need the following microlocal characterization of the operator $\mc{A}_-\mc{A}i$ similar to that in \cite[Theorem 5.4.19]{MelTayl}
\begin{lemma}
\label{lem:WFhAiry}
The Airy multipliers have wavefront set bounds as follows:
\begin{gather*} \left\{\begin{gathered}{\WFh} '(\mc{A}_-\mc{A}i)\cup{\WFh} '(\mc{A}_-\mc{A}i')\\
{\WFh} '(\mc{A}'_-\mc{A}i)\cup {\WFh} '(\mc{A}'_-\mc{A}i')\end{gathered}\right\}\subset C_{\beta}\cup\graph( \Id )=:C_b\\
{\WFh}'((\mc{A}_-\mc{A}i)^{-1})\subset
\cup_{n=0}^\infty C_{\beta^n}\cap E_+=:C_b^\infty\\
E_+:=\{\xi_1\neq 0\}\cup\{x_1\geq y_1, x_i=y_i,\, 2\leq i\leq d,\,\xi=\eta, \xi_1=0\}
\end{gather*}
where $C_{\beta^n}$ is the relation generated by $\beta^n$ and $\graph(\Id)$ denotes the graph of the identity map.
\end{lemma}
\begin{remark} Note that $C_b^\infty=\overline{\cup_{n\geq 0}C_{\beta^n}}$\end{remark}
\begin{proof}
We have that $\alpha_h=h^{-2/3}(\xi_1+\e(h)\alpha'(\xi))$ where $\xi_1$ is dual to $y$.
First, fix $\delta>0$ and suppose that $\psi_1\in S^0(\re^d)$ is a cutoff function with 
$\psi(\xi)=0\,,\,|\xi_1|\leq \delta\,,$ and $ \psi(\xi)=
1\,,\,|\xi_1|\geq 2\delta.$
Then, we show that
${\WFh}'(\psi(hD)\mc{A}_-\mc{A}i)\subset C_b,$  ${\WFh}'((\psi(hD)\mc{A}_-\mc{A}i)^{-1})\subset C_b^\infty.$
Write $\psi=:\psi_++\psi_-$ where $\supp\psi_{\pm}\subset \{\pm\xi_1>0\}$. Then, in $|\text{Arg}z|<\e$, 
\m A_-(z)Ai(z)=\Xi_-\Xi\,\,\m 
with $\Xi_-\Xi$ an elliptic symbol.
Hence, 
\begin{gather*} \psi_+A_-(\alpha_h)Ai(\alpha_h)=\psi_+\Xi_-(\alpha_h)\Xi(\alpha_h)\\
\psi_+(A_-(\alpha_h)Ai(\alpha_h))^{-1}=\psi_+\Xi_-^{-1}(\alpha_h)\Xi(\alpha_h)^{-1}
\end{gather*}
and $\psi_+(\mc{A}_-\mc{A}_i)$, $\psi_+(\mc{A}_-\mc{A}i)^{-1}$ are classical pseudodifferential operators. Thus, we have 
$${\WFh}'(\psi_+\mc{A}_-\mc{A}i)\,,{\WFh}'(\psi_+(\mc{A}_-\mc{A}i)^{-1})\,\subset \graph{Id}.$$
Now, for the term involving $\psi_-$, we use the asymptotic expansion of $Ai$ and $A_-$ to write in $|$Arg$z-\pi|<\e$,
$AiA_-(z)=\omega \Xi_+(z)\Xi_-(z)+\bar{\omega}\Xi^2_-( z)e^{4/3i(-z)^{3/2}}.$
Thus,
\begin{equation}
\label{eqn:negAsympt}
\psi_-(\xi)A_-Ai(\alpha_h)=a_1\exp\left(\frac{4}{3h}i(-\xi_1-\e(h)\alpha')^{3/2}\right)+a_2
\end{equation}
where $a_i\in h^{1/3}S^{-1/2}.$ 
Therefore $\psi_-\mc{A}_-\mc{A}i\in h^{1/3}I^0(C_b\cap \{|\xi_1|\neq 0\})$ since 
$\varphi=\la x-y,\xi\ra +\frac{4}{3}(-\xi_1)^{3/2}$
parametrizes $\beta$ for the Friedlander model and the $\alpha'$ term is a symbolic perturbation since $\e(h)=\O{}(h\log h^{-1}).$ Identical arguments give the wavefront set bound from $\mc{A}_-'\mc{A}i'$. 

Similarly, using \cite[Section 5]{MelTayl} or simply expanding in power series,
$$\psi_-((A_-Ai)^{-1}(\alpha_h))=\sum_{k\geq 0} a_k\exp\left(\frac{4k}{3h}i(-\xi_1-\e(h)\alpha'))^{3/2}\right)$$
where for any $S^{1/2}$ seminorm, $\|\cdot \|_{S^{1/2}}$,
\m \sum_{k\geq 0}\|a_k\|_{S^{1/2}}<Ch^{-1/3}.\,\,\m
Thus, $$\psi_-(\mc{A}_-\mc{A}i)^{-1}\in h^{-1/3}I^{1/3}(\re^d; C_b^\infty\cap \{|\xi_1|\neq 0\}).$$

Now, by Lemma \ref{lem:AiSymbol1}, 
\m Ai(\alpha_h)A_-(\alpha_h)q_1(\Re\alpha_h)\in e^{C\e(h)/h}S_{1/2,2/3,1}.\,\,\m
Thus, 
$$ hD_{\xi}^\beta AiA_-(\alpha_h)q_1(\Re\alpha_h(hD))=\O{}(h^{|\beta|/3})e^{C\e(h)/h}.$$
So, if $b$ is the kernel of $\mc{A}i\mc{A}_iq(\Re(\alpha_h(hD)))$, then 
\m (x_i-y_i)^kb=\O{}(h^{|\beta|/3})e^{C\e(h)/h}.\,\,\m
Hence, for any $N>0$, taking $|\beta|$ large enough and using that $\e(h)=\O{}(h\log h^{-1}).$, 
\m (x_i-y_i)^{|\beta|}b=\O{}(h^N).\,\,\m

But, $x_i-y_i$ is elliptic away from $x_i-y_i=0$. Hence,
$${\WFh}'(\mc{A}_-\mc{A}iq_1(\Re (\alpha_h(hD)))\subset \graph{(\Id)}.$$
But on $\supp( 1-q_1)$, the asymptotics \eqref{eqn:negAsympt} hold and we have studied this wavefront set. 

Next, observe that $\partial_{\xi_j}(A_-A_i(\alpha_h))^{-1}=0$ for $2\leq j\leq d$. Hence, 
\m {\WFh}'((\mc{A}_-\mc{A}i)^{-1})\subset \{x_2=y_2,\dots x_d=y_d\}.\,\,\m
The sign condition on $x_1$ follows from the fact that $(A_-A_i(h^{-2/3}\xi_1+i\e(h)))^{-1}$ is holomorphic in $\Im \xi_1> 0.$ 
Hence, by the Paley-Weiner theorem \cite[Theorem 7.3.8]{HOV1} 
$$\supp [(\mc{A}_-\mc{A}i)^{-1}\delta(x)]\subset\{ x_1>0\}.$$
\end{proof}
\begin{remark} This is where we use the assumption $\alpha'=i+\O{}(\e(h))$ rather than $\alpha'=-i+\O{}(\e(h)).$
\end{remark}

We need the following characterization of ${\WFh}'(H_d)$ \cite[Appendix A.3]{StefVod}
\begin{lemma}
\label{lem:WFhout}
$$
{\WFh}'(H_d)\subset\left\{\begin{gathered}(x,\xi,y,\eta)\in T^*\Omega_2\times T^*\partial\Omega : \\
|\xi|=1, (x,\xi)\text{ in the outgoing ray from }(y,\eta)\end{gathered}\right\}.$$
\end{lemma}
\begin{proof}
We decompose the operator into pieces where $\Re \rho\geq - 2Ch^{2/3}$ and $\Re \rho\leq -Ch^{2/3}$. When $\Re \rho\geq -2 Ch^{2/3}$, Lemma \ref{lem:FAIOInEst1} shows that in the interior of $\Omega_1$, $H_d=\O{}(h^\infty)$. When $\Re \rho \leq -Ch^{2/3}$, Lemmas \ref{lem:singPhaseOut} and \ref{lem:rhoOpOut} show that $H_dJ$ is a Fourier integral operator with singular phase
$$\psi =\theta-\frac{2}{3}\left[(-\rho)^{3/2}-(-\alpha)^{3/2}\right].$$
Thus it has 
${\WFh}'(H_dJ)|_{\Omega_2}\subset C_{\psi}$ where $C_{\psi}=\{(x,\nabla_x\psi,\nabla_{\xi}\psi,\xi)\}.$
But, this parametrizes the outgoing geodesics (\cite[Section X.4]{TaylorPseud}, \cite[Section 6.5]{MelTayl}). 

Now, at $\partial\Omega$, $H_d$ is a microlocally invertible Fourier integral operator with phase $\theta_b(x',\xi)-\theta_b(y',\xi)$. Hence, on $\partial\Omega $
\m {\WFh}'(H_d)|_{\partial\Omega}\subset \graph{\Id}.\m
\end{proof}

Similar arguments together with the wavefront set bound on $(\mc{A}_-\mc{A}i)^{-1}$ show \cite[Section 6.5]{MelTayl},
\begin{lemma}
\begin{equation*}{\WFh}'(B_3J^{-1})\subset\left\{\begin{gathered}(x,\xi,y,\eta)\in T^*\Omega_1\times T^*\partial\Omega:\\
|\xi|=1, (x,\xi)\text{ is in an outgoing ray from a point }(y,\eta)\,\end{gathered}\right\}\end{equation*}
\begin{equation*}{\WFh}'(H_g)\subset\left\{\begin{gathered}(x,\xi,y,\eta)\in T^*\Omega_1\times T^*\partial\Omega: |\xi|=1, \\
(x,\xi)\text{ is in an outgoing ray from }\overline{\cup_{n\geq 0}\beta^n((y,\eta))}\,\end{gathered}\right\}\end{equation*}\end{lemma}
\begin{proof}
We first prove wave front set bounds on operators of type \eqref{eqn:tempInside} decompose the operator into pieces where $\Re \rho\geq - 2Ch^{2/3}$ and $\Re \rho\leq -Ch^{2/3}$. When $\Re \rho\geq -2 Ch^{2/3}$, Lemma \ref{lem:FAIOInEst1} shows that in the interior of $\Omega_2$, $B_3=\O{}(h^\infty)$. When $\Re \rho \leq -Ch^{2/3}$, Lemmas \ref{lem:singPhaseIn} and \ref{lem:rhoOpIn} show that $B_3$ is a Fourier integral operator with singular phase
$$\psi =\theta-\frac{2}{3}\left[(-\rho)^{3/2}-(-\alpha)^{3/2}\right].$$
Thus it has 
$${\WFh}'(B_3)|_{\Omega_1}\subset C_{\psi}\,,\quad\text{where}\quad C_{\psi}:=\{(x,\nabla_x\psi,\nabla_{\xi}\psi,\xi\}.$$
But, this parametrizes the outgoing geodesics (\cite[Section X.4]{TaylorPseud}, \cite[Section 6.5]{MelTayl}). 

Now, at $\partial\Omega$, $B_3$ is a microlocally invertible Fourier integral operator with phase $\theta_b$. Hence, on $\partial\Omega $
$${\WFh}'(H_g)|_{\partial\Omega}\subset \graph{\Id}.$$
Combining this with the wavefront relation for $(\mc{A}_-\mc{A}i)^{-1}$ completes the proof of the lemma. 
\end{proof}


\section{Parametrix for diffractive points}
\label{sec:relationExact}

We follow \cite{StefVod} to show that the parametrices $H_d$ constructed above are $\O{C^\infty}(h^\infty)$ close to the exact solution near $\pO$. We have that for $f$ microsupported near a glancing point $(y_0,\eta_0)$
\begin{equation} 
\label{eqn:parametrixOp}
(-h^2\Delta -z^2)H_df=Kf\text{ in }U\,,\quad\quad\quad
H_df|_{\partial\Omega}=f+Sf.
\end{equation}
Here $K=\O{\mc{S}'\to C^\infty}(h^\infty)$ and $S=\O{\mc{D'}\to C^\infty}(h^\infty).$ Let $\chi\in C_0^\infty$ have $\supp \chi\subset U$ and $\chi \equiv 1$ in a neighborhood of $\partial\Omega$.

Define
\m \tilde{H}_d:=\chi H_d-R_0(\chi K-[h^2\Delta, \chi]H_d).\,\,\m
Then $\tilde{H}_d$ is $z$ outgoing and has $(-h^2\Delta-z)\tilde{H}_d=0.$ Next,
\m (\tilde{H}_df)|_{\partial\Omega}=f+Sf-\gamma R_0(\chi Kf+[-h^2\Delta, \chi]H_df).\,\,\m
The last term is the only potentially problematic term. However, since $\WFh([-h^2\Delta,\chi])$ is away from $\partial\Omega$, $H_d$ and $R_0$ are outgoing, and $\Omega$ is convex, this term is $\O{C^\infty}(h^\infty)$ when restricted to a neighborhood of $\partial\Omega$. 

Thus, writing 
\m \tilde{R}=S-\gamma R_0(\chi K+[-h^2\Delta, \chi]H_d),\,\,\m
we have that the exact solution operator is given by 
$\mc{H}_d=\tilde{H}_d(I+\tilde{R})^{-1}$
where $I+\tilde{R}$ is invertible for $h$ small since $\tilde{R}$ is $\O{C^\infty}(h^\infty)$. Hence, we have
\begin{lemma}
\label{lem:extPara}
Then the solution operator for the exterior Dirichlet problems is given by 
$$\mc{H}_d=\chi H_d-R_0(\chi K-[h^2\Delta, \chi]H_d)+\O{C^\infty}(h^\infty).$$
In a neighborhood, $U$ of $\pO$, this is 
$$\mc{H}_d|_U=\chi H_d|_U+\O{C^\infty(U)}(h^\infty).$$
\end{lemma}

\subsection{Dirichlet to Neumann Maps in the Diffractive Case}
\label{sec:glance}
Using the parametrices constructed above, we construct a microlocal representation of the Dirichlet to Neumann map near glancing. In order to do this, we simply take the normal derivative of $H$ from the previous section. That is, let ${\nu'}$ denote the inward unit normal to $\Omega$,
$$\partial_{\nu'} H_d(f)|_{\partial \Omega}=(2\pi h)^{-d+1}\int\left( g_0'+ih^{1/3}g_1'\frac{A_-'(h^{-2/3}\alpha)}{A_-(h^{-2/3}\alpha)}\right)e^{i\theta_b/h}\mc{F}_h{F}d\xi.$$
The new symbols $g_0'$ and $g_1'$ have 
$g_0'=\partial_{{\nu'}}g_0+ih^{-1}g_0\partial_{\nu'}\theta +ih^{-1}g_1\rho\partial_{{\nu'}}\rho$ and $ g_1'=\partial_{{\nu'}}g_1-ih^{-1}g_0\partial_{\nu'} \rho +h^{-1}g_1\partial_{\nu'}\theta.$
By construction $g_1$ vanishes at the boundary and, moreover $\partial_{{\nu'}}\rho\neq 0$ with $\nabla\rho=\partial_{\nu'}\rho {\nu'}$. Hence, $\partial_{\nu'}\theta=0$ by \eqref{eqn:eikEuc}. So, we have 
\begin{equation}\label{eqn:CB}g_0'=\partial_{\nu'} g_0\quad\quad\quad g_1'=-ih^{-1}g_0\partial_{{\nu'}}\rho+\partial_{\nu'} g_1.\end{equation}

Now, $g_0'\in S$ and $g_1'\in h^{-1}S$ with $g_1'$ elliptic and hence we have 
\begin{gather*} \frac{1}{(2\pi h)^{d-1}}\int g_0'e^{i\theta_b/h}\mc{F}_h{F}(\xi)d\xi=:JB(F)\\
\frac{ih^{1/3}}{(2\pi h)^{d-1}}\int g_1'\frac{A_-'(h^{-2/3}\alpha)}{A_-(h^{-2/3}\alpha)}e^{i\theta_b/h}\mc{F}_h{F}(\xi)d\xi=:Jh^{-2/3}C\Phi_-(F).
\end{gather*}
with $C\in \Psi$ elliptic, $B\in \Psi$, and $\Phi_-$ the operator defined by
$$\widehat{\Phi_- (F)}:=\frac{A'_-(\alpha_h)}{A_-(\alpha_h)}\mc{F}_h{F}=:\phi_-(\alpha_h)\mc{F}_h{F}.$$
Hence, microlocally,
\begin{equation}
\label{eqn:DtoNMicrolocal}
N_2=J(h^{-2/3}C\Phi_-+B)J^{-1}.
\end{equation}
A simple nonstationary phase argument shows that $\WFh'(\Phi_-)\subset\graph{\Id}$. This together with
the microlocal model \eqref{eqn:DtoNMicrolocal} implies the following bounds for the exterior Dirichlet to Neumann maps near glancing.

\begin{theorem}
Let $N_2$ denote the Dirichlet to Neumann map for the exterior of $\Omega$. Let $\chi\in \Cc(\re)$. Fix $0< \e <1/2$ and let $X_\e=\oph(\chi(h^{-\e}||\xi'|_g-1|))$. Then for $|\Im z|\leq Ch\log h^{-1}$, 
$$\|N_2 X_\e\|_{L^2\to L^2}\leq h^{-1+\e/2}.$$
\end{theorem}
\begin{remark}
Note that one can let $0<\e\leq 2/3$ if we apply the second microlocal calculus of \cite{SjoZwDist}.
\end{remark}

\section{Relation with exact operators in gliding case}
\label{sec:glide}
In the gliding case, we cannot make a simple wavefront set argument to show that $H_g$ is $h^\infty$ close to the exact solution operator.  Instead, we focus on constructing functions that are used in section \ref{sec:BLONearGlance} to produce microlocal descriptions of boundary layer operators and potentials near glancing. In particular, we examine operators of the form $\tilde{A}_g:=B_3J^{-1}$ where
$$B_3F:=\frac{1}{(2\pi h)^{-d+1}}\int (g_0Ai(\rho_h)+ih^{1/3}g_1Ai'(\rho_h))A_-(\alpha_h)e^{\frac{i}{h}\theta(x,\xi)}\mc{F}_h(F)(\xi)d\xi.$$
Let $(y_0,\eta_0)\in S^*\pO$ be a glancing point. Then we have that there exists $U$ a neighborhood of $y_0$ in $\Omega$ such that for $\delta$ and $\gamma$ small enough and $\psi$ with
\begin{gather*} 
\psi\equiv 1\text{ on }\{|y-y_0|<\delta,\, |\eta-\eta_0|<\delta_1,\,\left||\eta|_g-1\right|\leq \gamma h^2\e(h)^{-2}\}\\
\supp\psi \subset \{|y-y_0|<2\delta,\, |\eta-\eta_0|<2\delta_1,\,\left||\eta|_g-1\right|\leq 2\gamma h^2\e(h)^{-2}\}
\end{gather*}
\begin{equation*}\left\{\begin{aligned}(-h^2\Delta-z^2)\tilde{A}_gf&=Kf\\\
\tilde{A}_gf|_{\pO}&=J\mc{A}i\mc{A}_-J^{-1}\oph(\psi) f+Sf\\
\end{aligned}\right.
\end{equation*}
where $K=\O{\mc{D}'(\pO)\to C^\infty(U)}(h^\infty)$ and $S=\O{\mc{D}'\to C^\infty(\pO)}(h^\infty)$. Now, shrinking $\delta $ if necessary, we assume that $B(y_0,3\delta)\subset U$. 
Now, fix $\chi\in C^\infty(\Omega)$ supported in $U$ with $\chi\equiv 1$ on $B(y_0,2\delta)$. Then, using the wavefront set bound on $B_3$, we have that shrinking $\delta$ again if necessary, $\WFh(\tilde{A}_g)\cap \supp \partial\chi=\emptyset.$ So, defining $A_g:=\chi H_g$, we have 
$$\left\{\begin{aligned} 
(-h^2\Delta-z^2)A_gf&=\chi Kf+[h^2\Delta,\chi]\tilde{A}_gf=K_1f\\
A_gf|_{\pO}&=\chi J\mc{A}i\mc{A}_-J^{-1}\oph(\psi)f+\chi Sf\\
&=J\mc{A}i\mc{A}_-J^{-1}\oph(\psi)f+S_1f
\end{aligned}\right.
$$
where $K_1=\O{\mc{D}'(\pO)\to C^\infty(\Omega)}(h^\infty)$ and $S_1=\O{\mc{D}'(\pO)\to C^\infty(\pO)}(h^\infty).$

Similarly, there exists $B_g=\chi B_4J^{-1}$ with 
$$B_4F:=\frac{1}{(2\pi h)^{-d+1}}\int (g_0Ai(\rho_h)+ih^{1/3}g_1Ai'(\rho_h))A_-'(\alpha_h)e^{\frac{i}{h}\theta(x,\xi)}\mc{F}_h(F)(\xi)d\xi$$
such that 
$$\left\{\begin{aligned} 
(-h^2\Delta-z^2)B_gf&=K_2f\\
B_gf|_{\pO}&=J\mc{A}i\mc{A}'_-J^{-1}\oph(\psi)f+S_2f
\end{aligned}\right.
$$
where $K_2=\O{\mc{D}'(\pO)\to C^\infty(\Omega)}(h^\infty)$ and $S_2=\O{\mc{D}'(\pO)\to C^\infty(\pO)}(h^\infty).$

Note also that with $\nu$ the outward unit normal to $\Omega$,
$$\begin{aligned}
\partial_\nu A_gf|_{\pO}=-J(h^{-2/3}C\mc{A}i'\mc{A}_-+B\mc{A}i\mc{A}_-)J^{-1}\oph(\psi)f+S_3f\\
\partial_\nu B_gf|_{\pO}=-J(h^{-2/3}C\mc{A}i'\mc{A}_-'+B\mc{A}i\mc{A}_-')J^{-1}\oph(\psi)f+S_4f
\end{aligned}
$$
where $S_i=\O{\mc{D}'(\pO)\to C^\infty(\pO)}(h^\infty)$ and $B,C\in \Ph{}{}$ are as in \eqref{eqn:DtoNMicrolocal}. Then we have,
\begin{lemma}
\label{lem:glideBLOPara}
Near a gliding point, there exist operators $A_{i,g}$ $i=1,2$ so that 
\begin{equation*}
\left\{\begin{aligned} 
(-h^2\Delta-z^2)A_{i,g}f&=K_if\text{ in }\Omega\\
A_{1,g}|_{\pO}&=J\mc{A}i\mc{A}_-J^{-1}\oph(\psi)f+S_{1,r}f\\
A_{2,g}|_{\pO}&=J\mc{A}i\mc{A}_-'J^{-1}\oph(\psi)f+S_{2,r}f\\
\partial_\nu A_{1,g}|_{\pO}&=-J(h^{-2/3}C\mc{A}i'\mc{A}_-+B\mc{A}i\mc{A}_-)J^{-1}\oph(\psi)f+S_{1,\nu}f\\
\partial_\nu A_{2,g}|_{\pO}&=-J(h^{-2/3}C\mc{A}i'\mc{A}_-'+B\mc{A}i\mc{A}_-')J^{-1}\oph(\psi)f+S_{2,\nu}f
\end{aligned}\right.
\end{equation*}
where $K_i=\O{\mc{D}'(\pO)\to C^\infty(\Omega)}(h^\infty)$ and $S_{i,\cdot}=\O{\mc{D}'(\pO)\to C^\infty(\pO)}(h^\infty).$
\end{lemma}

\section{Wave equation parametrices}
\label{sec:waveParametrices}
Let $\Omega\subset\re^d$ be a strictly convex domain with smooth boundary. Let $\Omega_1=\Omega$ and $\Omega_2=\re^d\setminus\overline{\Omega}$. 
In order to handle the glancing region, we construct microlocal parametrices for 
$$
\begin{cases}(\partial_t^2-\Delta)u_i=0,&\,\text{ in }\Omega_i\\
u_1=u_2&\text{ on }\pO\\
\partial_{\nu_1}u_1+\partial_{\nu_2}u_2=f&\text{ on }\pO.
\end{cases}
$$

That is, we construct $H$ such that if $f$ has wavefront set in a small conic neighborhood of $(t_0,x_0,\tau_0,\xi_0)\in T^*(\re\times \partial\Omega)$ over which a glancing ray passes, then
\begin{equation}
\label{eqn:parametrixCond}
\begin{cases}(\partial_t^2-\Delta)Hf&\in C^\infty(\Omega_i)\\
(H_1f-H_2f)|_{\partial\Omega}&\in C^\infty(\partial\Omega)\\
\partial_{\nu_1} H_1f+\partial_{\nu_2}H_2f-f&\in C^\infty(\pO)\\
Hf\in C^\infty&t\ll 0
\end{cases}
\end{equation}
First, the wavefront set property for $f$ implies that $f$ is $C^\infty$ outside of a compact set in $t$. Hence, by \cite[Theorem 6.24]{HOV1}, the solution when $f$ is replaced by $\chi(t)f$ differs only by a $C^\infty$ function. Thus, without loss of generality, we assume that $f$ has compact support.

We will use the construction in Section \ref{sec:BLONearGlance}. To pass from the parametrix for $-h^2\Delta -z^2$ to a \eqref{eqn:parametrixCond}, set $z=1$, $h=\tau^{-1}$, and rescale $\xi'\to \xi'\tau$. 

\noindent That is, letting 
$$H_hf(x,h):=(2\pi h)^{-d+1}\int g(x,\xi',y,h)f(y,h)dyd\xi'\,,$$
$$\tilde{H}f(x,\tau):=(2\pi \tau^{-1})^{-d+1}\int g(x,\tau\xi', y, \tau^{-1})f(y,\tau^{-1})dyd\xi'.$$
We then have that $\tilde{H}$ acts on functions $f$ with wavefront set in $||\xi'\tau^{-1}|-1|\leq \e$ and is $\O{}(\tau^{-\infty})$ on functions with wavefront set away from this set. That is, $\tilde{H}$ acts on functions with wavefront set in a conic neighborhood of glancing.  Then, 
\m H:=\mc{F}^{-1}_{t\to \tau}\tilde{H}\mc{F}_{t\to\tau}\,\,\m
is the desired parametrix.

\section{Semiclassical Fourier integral operators with singular phase}
\label{sec:singFIO}

We now define the semiclassical analog of Fourier integral operators with singular phase. We follow the treatment in the homogeneous setting given in \cite[Section VII.6]{TaylorPseud} (For another treatment of Fourier integral operators with singular phase in the homogeneous setting, see \cite[Appendix D]{MelTayl}.).

Throughout this section, we assume that $U\subset \re^d$ is open and $\varphi\in C^\infty(U)$ is a nondegenerate phase function with the caveat that, letting $\gamma$ be a boundary defining function for $\overline{U}$ and $0\leq a<1$,  it only has
\begin{equation}
\label{eqn:phase2}
\varphi\in C^1(\overline{U})\,\quad\quad\quad\quad|D_{x}^\beta D_{\xi}^\alpha\varphi|\leq C_{\alpha,\beta}\gamma^{(1+a)-|\alpha|-|\beta|},\quad \text{ if }|\alpha|+|\beta|\geq 2.
\end{equation}
Then, let $a\in S_{\delta}(U)$ have 
\begin{equation}
\label{eqn:amp1}
\supp a\subset \{\gamma \geq ch^b\},\quad \supp a \Subset \overline{U}
\end{equation}
where $c>0$ and $0<b<1$. Here, we allow $\delta \in [0,b)$. 

A {\emph {Fourier integral operator with singular phase $\varphi$}} is an operator $Au$ defined by 
$$Au(x)=(2\pi h)^{-d}\int a(x,\xi)e^{\frac{i}{h}(\varphi(x,\xi)-\la y,\xi\ra)}u(y)dyd\xi.$$
Since $a$ has compact support, this operator is well defined. We need to prove the following lemma.
\begin{lemma}
Let $\varphi$ have \eqref{eqn:phase2} and $a\in S_{\delta}$ have \eqref{eqn:amp1}. Let $A$ be a Fourier integral operator with singular phase $\varphi$. Then
\m {\WFh}'(A)\subset\{ (x,\partial_x\varphi (x,\xi),\xi,\partial_{\xi}\varphi(x,\xi)\}.\,\,\m
\end{lemma}
\begin{proof}
To see this, consider
$$\la \chi(x)e^{-\frac{i}{h}\la x,\theta\ra},Au\ra=(2\pi h)^{-d}\int u(y)\chi(x)a(x,\xi)e^{\frac{i}{h}\Phi(x,\xi,y,\theta)}dydxd\xi$$
where $\Phi(x,\xi,y,\theta)=\varphi(x,\xi)-\la y,\xi\ra-\la x,\theta\ra.$
Then, away from $\partial_x\varphi=\partial_{\xi}\varphi=0,$ there exists $L$ such that $Le^{\frac{i}{h}\Phi}=e^{\frac{i}{h}\Phi}$. By \eqref{eqn:phase2} we have
$$(L^t)^k=\sum_{0\leq |\sigma|\leq k}h^kA_{\sigma}^k(x,\xi)D_x^\sigma.$$
Here, $A_\sigma^k$ has 
\m |D_x^\beta D_\xi ^\alpha h^k A_\sigma ^k|\leq h^k (1+\gamma ^{a-(k-|\sigma|)-|\alpha|-|\beta|}).\,\,\m
Thus, on $\supp a$, 
$$|h^kA_{\sigma}^k|\leq Ch^k(1+h^{(a-(k-|\sigma|))b})\leq Ch^{k(1-b)+ab+|\sigma|b}.$$
Thus, 
$$|(L^t)^ka|\leq \sum_{0\leq |\sigma |\leq k}Ch^{k(1-b)+ab+|\sigma|b-\delta|\sigma|)}\leq h^{k\min(1-b,b-\delta)}.$$
Since $\delta<b<1$, this gives the result.
\end{proof}

\backmatter
\bibliographystyle{amsalpha}
\bibliography{biblio}
\printindex
\end{document}